\documentclass[11pt,a4paper]{article}
\usepackage[utf8]{inputenc}

\usepackage[top=0.99in, bottom=0.99in, left=0.98in, right=0.98in]{geometry}
\usepackage{amsmath,amsthm,amssymb,amscd,dsfont,bm,epic,accents}
\usepackage{graphicx}
\usepackage[all]{xy}
\usepackage{tikz}
\usetikzlibrary{intersections}

\usepackage{multicol}
\usepackage{mathrsfs}
\usepackage{enumerate}
\usepackage{caption}
\usepackage{subcaption}
\usepackage{graphicx}
\usepackage{color}
\usepackage{tikz}
\usetikzlibrary{shapes.multipart, positioning}

\usepackage{tkz-euclide}
\usetikzlibrary{quotes}
\tikzset{every edge quotes/.style =
          { fill = white,
            sloped,
            execute at begin node = $,
            execute at end node   = $  }}

\usepackage[driverfallback=hypertex]{hyperref}
\usepackage{nameref,zref-xr}                   

\usepackage{cancel}
\usepackage[normalem]{ulem}
\usepackage{enumitem}

\newtheorem{dummy}{}[section]
\newtheorem{thm}[dummy]{Theorem}
\newtheorem{prop}[dummy]{Proposition}
\newtheorem{pr}[dummy]{Proposition}
\newtheorem{lem}[dummy]{Lemma}

\newtheorem{conj}{Conjecture}
\theoremstyle{definition}
\newtheorem{definition}[dummy]{Definition}
\newtheorem{dfn}[dummy]{Definition}
\newtheorem{nn}[dummy]{Notation}

\theoremstyle{remark}
\newtheorem{rmk}[dummy]{Remark}
\newtheorem{ex}[dummy]{Example}
\newtheorem{obs}[dummy]{Observation}

\newcommand{\consta}{{\frac{1}{2}}}

 \newcommand{\Mbarr}{{{\overline{\mathcal{M}}}^{\frac{1}{r}}_{0,k,\vec{a}}}}

 \newcommand{\CL}{{\mathbb{L}}}

\newcommand{\alt}{{\text{alt}}}

\newcommand{\tw}{\text{tw}}

 \newcommand{\Ass}{{{\,\maltese\,}}}

\newcommand{\oPM}{{\overline{\mathcal{PM}}}^{1/r}}
\newcommand{\oPMb}{{\overline{\mathcal{PM}}}}

\newcommand{\oPMr}{{{\overline{\mathcal{PM}}}^{\frac{1}{r}}_{0,k,\vec{a}}}}

\newcommand{\oQMb}{{\overline{\mathcal{QM}}}}

\newcommand{\for}{\operatorname{for}}
\newcommand{\For}{\operatorname{For}}

\newcommand{\nontilde}{}

 \newcommand{\M}{\ensuremath{\overline{\mathcal{M}}}}

 \renewcommand{\d}{\ensuremath{\partial}}

 \newcommand{\<}{\left<}
 \renewcommand{\>}{\right>}
 \DeclareMathOperator{\res}{res}

\numberwithin{equation}{section}

\usepackage{booktabs,multirow,tabularx}

\DeclareMathOperator{\Id}{Id}

\DeclareMathOperator{\rk}{rank}

\def\C{{\mathbb C}}

\def\Mbar{{\overline{\mathcal M}}}
\def\Mbarstar{{\overline{\mathcal M^*}}}

\def\GPI{{\text{GPI}}}
\def\PI{{\text{PI}}}
\def\h{{\mathfrak h}}

\def\sGPI{\text{sGPI}}
\def\PI{\text{PI}}
\def\ev{\text{ev}}

\usepackage{scalerel}

\DeclareMathOperator*{\bboxplus}{\scalerel*{\boxplus}{\sum}}

\DeclareMathOperator*{\bass}{\scalerel*{\maltese}{\sum}}

\DeclareMathOperator*{\bboxtimes}{\scalerel*{\boxtimes}{\sum}}

\title{Open $r$-spin theory in genus one, and the Gelfand--Dikii wave function}
\date{}
\author{Ran J. Tessler \and Yizhen Zhao}

\AtEndDocument{\bigskip{\footnotesize%
    \textsc{R. J. Tessler:}\par
  \textsc{Department of Mathematics, Weizmann Institute of Science, POB 26, Rehovot 7610001, Israel} \par  
  \textit{E-mail address}: \texttt{ran.tessler@weizmann.ac.il} \par
  \addvspace{\medskipamount}
  \textsc{Y. Zhao:}\par
  \textsc{Institut de Math\'ematiques de Jussieu – Paris Rive Gauche, 75005 Paris, France} \par  
  \textit{E-mail address}: \texttt{yizhen.zhao@imj-prg.fr}
}}

\begin{document}
\maketitle
\begin{abstract}
We construct the $g=1$ sector of the open $r$-spin theory, that is, an open $r$-spin theory on the moduli space of cylinders. This is the second construction of a $g>0$ open intersection theory, which includes descendents (the first is the all genus construction of the intersection theory on moduli of open Riemann surfaces with boundaries \cite{PST14,Tes15}, whose $g=1$ case equals to the $r=2$ case of our construction). Unlike the construction of \cite{Tes15}, in order to construct the $r$-spin cylinder theory we had to overcome the foundational problem of dimension jump loci, which in analogous closed theories has been treated using virtual fundamental class techniques, that are currently absent in the open setting. For this reason our construction is much more involved, and relies on the point insertion technique developed in \cite{TZ1,TZ2}.

We prove that the open $g=1$ potential equals, after a coordinate change, to the $g=1$ part of the Gelfand--Dikii wave function, thus confirming a conjecture of \cite{BCT3}.
We also prove that our $g=1$ intersection numbers satisfy a $g=1$ recursion, also predicted in \cite{BCT3,gomez2021open}. This recursion is the $g=1$ analogue of Solomon's famous $g=0$ Open WDVV equation \cite{sol_owdvv}, with descendents, and is also the universal $g=1$ recursion for $F$-Cohomological field theories \cite{alexandrov2023construction}.
Again, this is first geometric construction which is not the $g=1$ sector of \cite{PST14,Tes15}, proven to satisfy this universal recursion.
\end{abstract}

\maketitle

\setcounter{tocdepth}{1}
\tableofcontents

\section{Introduction}\label{sec intro}
\subsection{Background on closed $r$-spin theory and Witten's $r$-spin conjecture}
\subsubsection{Witten's KdV conjecture}
Witten's conjecture \cite{Witten2DGravity}, proven by Kontsevich~\cite{Kontsevich}, is one of the cornerstones in the modern study of the intersection theory on the moduli space of stable curves~$\Mbar_{g,n},$ as well as in the development of Gromov--Witten theory. Let $\psi_1, \ldots, \psi_n \in H^2(\Mbar_{g,n})$ be the first Chern classes of the relative cotangent line bundles at the $n$ marked points, and denote by
\[
F^c(t_0,t_1,\ldots,\lambda) := \sum_{\substack{g \geq 0, n\geq 1\\2g-2+n>0}} \sum_{d_1,\ldots,d_n\geq 0} \frac{\lambda^{2g-2}}{n!}\left(\int_{\Mbar_{g,n}} \psi_1^{d_1} \cdots \psi_n^{d_n} \right)t_{d_1} \cdots t_{d_n}
\]
the \emph{closed potential}, a generating function of their integrals, which are called \emph{intersection numbers}. The superscript `$c$' stands for ``closed'', as opposed to the open theory which is the focus of this paper.  $\lambda$ and $\{t_i\}_{i\ge 0}$  are formal variables.  The Witten--Kontsevich theorem reads that $\exp(F^c)$ is a tau-function of the Korteweg--de Vries (KdV) hierarchy, and equivalently that  $\exp(F^c)$ satisfies the Virasoro constraints, which is some set of partial differential equations. Each of these equivalent formulations uniquely determines all intersection numbers.
\subsubsection{Witten's $r$-spin intersection theory}
A year after the proposal of his original conjecture, Witten suggested a generalized conjecture \cite{Witten93}, in which the moduli space of curves is enhanced to the moduli space of $r$-spin curves.  An $r$-spin structure on a smooth marked curve $(C;z_1, \ldots, z_n)$ is a line bundle $S$ together with an isomorphism
\[S^{\otimes r} \cong \omega_{C}\left(-\sum_{i=1}^n a_i[z_i]\right),\]
where $a_i \in \{0,1,\ldots, r-1\}$ are called \emph{twists} and $\omega_C$ denotes the canonical bundle.  There is a natural compactification $\Mbar_{g,\{a_1, \ldots, a_n\}}^{1/r}$ of the moduli space of $r$-spin structures on smooth curves, and this space admits a virtual fundamental class $c_W$ known as \emph{Witten's class}. In genus zero, Witten's class is the Euler class of the \emph{Witten bundle} $(R^1\pi_*\mathcal{S})^{\vee},$ where $\pi: \mathcal{C} \rightarrow \Mbar_{0,\{a_1, \ldots, a_n\}}^{1/r}$ is the universal curve, $\mathcal{S}$ is the universal $r$-spin structure.  In higher genus, on the other hand, $R^1\pi_*\mathcal{S}$ is generally not a vector bundle, and defining the Witten's class involved subtle virtual fundamental class techniques \cite{PV,ChiodoWitten,Moc06,FJR,CLL}.

The {\it (closed) $r$-spin intersection numbers} are defined by
\begin{gather}\label{eq:closed r-spin}
\langle\tau^{a_1}_{d_1}\cdots\tau^{a_n}_{d_n}\rangle^{\frac{1}{r},c}_g:=r^{1-g}\int_{\M^{1/r}_{g,\{a_1, \ldots, a_n\}}} \hspace{-1cm} c_W \cup \psi_1^{d_1} \cdots \psi_n^{d_n}.
\end{gather}
The \emph{closed $r$-spin potential}
is the generating function of the closed $r$-spin intersection numbers\[
F^{\frac{1}{r},c}(t^*_*,\lambda):=\sum_{\substack{g \geq 0, n \geq 1\\2g-2+n>0}} \sum_{\substack{0 \leq a_1, \ldots, a_n \leq r-1\\ d_1, \ldots, d_n \geq 0}} \frac{\lambda^{2g-2}}{n!}\langle\tau^{a_1}_{d_1}\cdots\tau^{a_n}_{d_n}\rangle^{\frac{1}{r},c}_g t^{a_1}_{d_1} \cdots t^{a_n}_{d_n},
\]
where again $\lambda$ and $t^a_d$  for $0\le a\le r-1$ and $d\ge 0$ are formal variables.

\subsubsection{The Gelfand--Dikii (GD) integrable hierarchy and the $r$-spin conjecture}\label{subsub:rGD_closed}

Recall that a {\it pseudo-differential operator}~$A$ is a Laurent series
\[
A=\sum_{n=-\infty}^m a_n(T_1,T_2,\ldots,\lambda)\partial_x^n,
\]
where $m$ is an integer,  $T_i$ for $i\geq 1,~\lambda$ and $\partial_x$ are formal variables, and $a_n(T_1,T_2,\ldots,\lambda)\in\C[\lambda,\lambda^{-1}][[T_1,T_2,\ldots]]$.     The space of pseudo differential operators is a non-commutative associative algebra, in which the multiplication $\circ$ is induced from the usual product on $\C[\lambda,\lambda^{-1}][[T_1,T_2,\ldots]]$ and 
\[
\d_x^k\circ f:=\sum_{l=0}^\infty\frac{k(k-1)\ldots(k-l+1)}{l!}\frac{\d^lf}{\d x^l}\d_x^{k-l},
\]
where $k$ is an integer, $f\in\mathbb C[\lambda,\lambda^{-1}][[T_*]]$, and the variable $x$ is identified with $T_1$. Let $r\geq 2$ be any integer, and $A$ a pseudo-differential operator of the form
\[
A=\partial_x^r+\sum_{n=1}^\infty a_n\partial_x^{r-n},
\]
then $A$ has a unique $r$th root, meaning a unique pseudo-differential operator $A^{\frac{1}{r}}$ of the form
\[
A^{\frac{1}{r}}=\partial_x+\sum_{n=0}^\infty b_n\partial_x^{-n}
\]
satisfying $\left(A^{\frac{1}{r}}\right)^r=A$.

Let $r\geq 2$, and consider the operator
$$
L:=\partial_x^r+\sum_{i=0}^{r-2}f_i\partial_x^i,\quad f_i\in\C[\lambda^{\pm1}][[T_*]].
$$
It is not hard to check that for any $n\geq 1$, the commutator $[(L^{n/r})_+,L]$ has the form $\sum_{i=0}^{r-2}h_i\partial_x^i$ with $h_i\in\C[\lambda^{\pm1}][[T_*]]$, where $(\cdot)_+$ denotes the part of a pseudo-differential operator composed of the non-negative powers of $\partial_x$. The {\it $r$th Gelfand--Dikii hierarchy}, sometimes called the \emph{$r$KdV hierarchy} is the system of partial differential equations for $f_0,f_1,\ldots,f_{r-2}$ given by:
\begin{equation}
    \label{eq:GD_hier}
\frac{\partial L}{\partial T_n}=\lambda^{n-1}[(L^{n/r})_+,L],\quad n\geq 1.
\end{equation}

Denote by $L$ the solution of the system specified by the initial condition
\begin{equation}\label{eq:init_cond_L}
L|_{T_{\geq 2}=0}=\partial_x^r+\lambda^{-r}rx,
\end{equation}
Witten's $r$-spin conjecture states that $\frac{\partial F^{\frac{1}{r},c}}{\partial t^{r-1}_d}=0$ for $d\geq 0$ (the \emph{Ramond vanishing phenomenon}) and that under the change of variables
\begin{gather}\label{eq:closed r-spin change of variables}
T_k=\frac{1}{(-r)^{\frac{3k}{2(r+1)}-\frac{1}{2}-d}k!_r}t^a_d,\quad 0\le a\le r-2,\quad d\ge 0,
\end{gather}
where $k=a+1+rd$ and
$$
k!_r:=\prod_{i=0}^d(a+1+ri),
$$
it holds that
$$
\res L^{n/r}=\lambda^{1-n}\frac{\partial^2 F^{\frac{1}{r},c}}{\partial T_1\partial T_n}
$$
whenever $n\geq 1$ is not divisible by $r$, and $\res L^{n/r}$ denotes the coefficient of $\partial_x^{-1}$ in $L^{n/r}$.

In the language of integrable hierarchies this means that  $\exp(F^{\frac{1}{r},c})$ becomes, after a simple change of variables, a \emph{tau-function} of the $r$th Gelfand--Dikii hierarchy.

This result was proven by Faber--Shadrin--Zvonkine \cite{FSZ10}.  The case $r=2$ is equivalent to the original conjecture of Witten. Witten's $r$-spin construction initiated the study of Fan--Jarvis--Ruan--Witten \cite{FJR} theory of quantum singularities.

\subsubsection{Closed extended $g=0$ $r$-spin theory}
\cite{JKV2} showed that the genus-zero $r$-spin theory can be extended to allow a single marked point of twist $-1$. This theory was further studied in \cite{BCT_Closed_Extended}, and was shown to be related to the Gelfand--Dikii hierarchy as well. We denote the intersection numbers in this extended theory by $\langle\cdots\rangle^{\frac{1}{r},ext}.$ Extended intersection numbers which include a $-1$ twist need not to satisfy the Ramond vanishing, and are crucial for writing recursions for open intersection numbers.

\subsection{Open $r$-spin theory, and the open analogue of Witten's $r$-spin conjecture}
\subsubsection{The open analogue of Witten's KdV conjecture}
The study of intersection theory on the moduli spaces of surfaces with boundary was initiated by Pandharipande, Solomon, and the third author in \cite{PST14} (the disk case) and \cite{ST1,Tes15} (the $g>0$ case).  These works consider the moduli space $\Mbar_{g,k,l}$, which parametrizes tuples $(\Sigma; x_1, \ldots, x_k; z_1, \ldots, z_l)$ in which $\Sigma$ is a stable Riemann surface with boundary, endowed with an additional structure called \emph{graded spin structure}, $x_i \in \partial\Sigma$ are boundary marked points, and $z_j \in \Sigma \setminus \partial\Sigma$ are internal marked points. When $g=0$ the grading data is trivial and $\Sigma$ is just a stable marked disk. The intersection numbers on $\Mbar_{g,k,l}$, which should be thought of as integrals of $\psi$-classes at the internal marked points, were constructed and hence the \emph{open potential} $F^o(t_0,t_1,\ldots,s,\lambda)$ was defined as a direct generalization of $F^c$; the new formal variable $s$ tracks the number of boundary marked points. \cite{PST14} conjectured an open analogue of Witten's conjecture for these intersection numbers: That the open potential satisfies open analogues of the Virasoro constraints and of the KdV equations. The proof of these conjectures was obtained by combining the results of \cite{Bur16,Tes15,BT17}. A survey for physicists can be found in \cite{DijkWit}.

There were two conceptual difficulties in defining the open intersection numbers. The first was that the na\"ive moduli of Riemann surfaces with boundary has no canonical orientation, and is sometimes non orientable, when $g>0,$ what complicates the definition of integrals. The other, more intricate, problem is the presence of real codimension $1$ boundaries for the moduli. Because of these boundaries any reasonable definition of integral should involve an imposition of boundary conditions. The inclusion of graded spin structures solved these problems. The moduli of graded spin structure is an oriented orbifold cover of the na\"ive moduli of surfaces with boundaries, see \cite[\S 5]{Tes15}. Moreover, this structure picks, for every stratum of nodal surface, and each boundary node, a distinguished \emph{illegal} half node. The \emph{forgetful boundary conditions} on the $\psi_i$ ``classes'' are essentially that on a boundary stratum the section of $\psi_i$ is pulled back from the moduli which parametrizes surfaces obtained by normalizing the nodes of the stratum, and forgetting the illegal side. See \cite[\S 2]{Tes15} for details. 
\subsubsection{Genus-zero open $r$-spin intersection theory}
An open analogue of Witten's $r$-spin theory was constructed in genus $0,$ the disk case in \cite{BCT1,BCT2} and \cite{TZ1,TZ2}.
The work \cite{BCT1} constructed the moduli space of \emph{graded $r$-spin disks}~$\Mbar_{0,k,\{a_1, \ldots, a_l\}}^{1/r},$ which parametrizes stable marked disks with a spin structure and a grading. The twists $a_1,\ldots,a_l$ of internal points again ranged in $\{0,\ldots,r-1\},$ but boundary marked points were required to have twist $r-2.$ \cite{BCT1} also defined an open analogue of Witten's bundle $\mathcal{W}=(R^1\pi_*\mathcal{S})^{\vee}$, again in $g=0$ and for the same objects, and proved that the bundle is canonically relatively oriented.
\cite{BCT2} constructed an intersection theory for graded $r$-spin disks. As in the spinless case described above, the geometry enters the most in defining meaningful boundary conditions.
It turns out that for some strata, or, equivalently, some types of boundary nodes, the forgetful boundary condition scheme still applies. But for most of the boundary strata it fails. The main new geometric insight of \cite{BCT2} is the discovery of a hidden positivity structure, which again builds on the grading. This has led to imposing the \emph{positivity boundary conditions} on the remaining boundaries. 
The open, genus $0$ $r$-spin intersection numbers are defined by
\begin{equation}
\label{eq:intnums}
\langle\tau_{d_1}^{a_1}\cdots\tau^{a_l}_{d_l}\sigma^k\rangle^{\frac{1}{r},o}_0:=\int_{\oPM_{0,k,\{a_1, \ldots, a_l\}}} e\left( \mathcal{W} \oplus \bigoplus_{i=1}^l \mathbb{L}_i^{\oplus d_i}, \mathbf{s}_{\text{canonical}}\right),
\end{equation}
where, on a manifold with boundary $M$, the notation $e(E,s)$ denotes the Euler class of $E$ relative to the boundary condition $s$ of $E|_{\partial M}\to \partial M$. $\oPMr$ is a modification of $\Mbarr$ obtained by removing certain boundary strata.\footnote{One can also work directly on $\Mbarr$ but then the definition becomes more technically complicated.} Finally, the section $\mathbf{s}_{\text{canonical}}$ is any choice of nowhere-vanishing section of $\mathcal{W} \oplus \bigoplus_{i=1}^l\mathbb{L}_i^{\oplus d_i}|_{\partial\oPMr}$ which satisfies both the forgetful, and the positivity boundary conditions.  It was shown that any two such choices of boundary sections yield the same intersection number.
It was also shown that in $g=0$ these numbers are non zero only if
\begin{equation}\label{eq:dim_equal_rk}
2\sum_{i=1}^l d_i +\frac{2\sum_{i=1}^la_i+k(r-2)+(g-1)(r-2)}{r}=2l+k+3g-3.
\end{equation}
In particular, 
\begin{equation}\label{eq:mod_r_general_g}
{2\sum_{i=1}^la_i+k(r-2)+(g-1)(r-2)}=0~\mod~r.
\end{equation}
and
\begin{equation}\label{eq:mod_2_general_g}
\frac{2\sum_{i=1}^la_i+k(r-2)+(g-1)(r-2)}{r} = k+g-1~\mod~2
\end{equation}
hold when setting $g=0.$

The genus $0$ open $r$-spin potential is the generating function
\[
F^{\frac{1}{r},o}_0(t^*_*,s):=\sum_{\substack{k,l\geq 0\\k+2l>2}} \sum_{\substack{0 \leq a_1, \ldots, a_l \leq r-1\\ d_1, \ldots, d_l \geq 0}} \frac{1}{k!l!}\<\tau^{a_1}_{d_1}\cdots\tau^{a_l}_{d_l}\sigma^k\>^{\frac{1}{r},o}_0 t^{a_1}_{d_1} \cdots t^{a_l}_{d_l}s^k.
\]
The goal of \cite{TZ1,TZ2} was to define intersection theories which allow more types of boundary twists.
The new geometric insight was that the forgetful boundary conditions should be replaced by the new \emph{point insertion boundary conditions}.
\cite{TZ2} defined $\lfloor\frac{r}{2}\rfloor$ open $r$-spin intersection theories, indexed by $\mathfrak{h}\in\{0,1,\ldots,\lfloor\frac{r}{2}\rfloor-1\}.$ The $\mathfrak{h}=0$ zero was shown in \cite[\S 5]{TZ2} to carry the same information as the open $r$-spin theory of \cite{BCT2}, and a minor modification of it produced the exact same intersection numbers as \cite{BCT2}, given in \eqref{eq:intnums}, see \cite[Remark 5.7]{TZ2}.
This last intersection theory, defined via the point insertion and positivity boundary conditions, is the starting point of this work.   

\subsubsection{The $r$GD wave function and the conjecture of \cite{BCT3}}
With $L$ as in \S \ref{subsub:rGD_closed}, let~$\Phi(T_*,\lambda)$ be the solution of the system of equations
\begin{gather}\label{eq:wave_func_eqs}
\frac{\partial\Phi}{\partial T_n}=\lambda^{n-1}(L^{n/r})_+\Phi,\quad n\geq 1,
\end{gather}
which satisfies the initial condition
$$
\left.\Phi\right|_{T_{\geq 2}=0}=1.
$$
The system of differential equations~\eqref{eq:wave_func_eqs} for the function~$\Phi$ coincides with the system of differential equations for the wave function of the KP hierarchy.
This function $\Phi$, for $r=2$ was first studied in \cite{Bur16}. For general $r$ it was first studied in~\cite{BY15}, where the authors found an explicit formula for $\Phi$ in terms of the wave function.

Let $\phi:=\log\Phi$ and expand it as a power series in $\lambda:$
$$
\phi=\sum_{g\in\mathbb{Z}}\lambda^{g-1}\phi_g,\quad \phi_g\in\C[[T_*]].
$$
While $F^{\frac{1}{r},c}_0$, depends only on the variables $t^0_d,\ldots,t^{r-2}_d$, the function~$F^{\frac{1}{r},o}_0$ depends also on $t^{r-1}_d$ and $s$. We relate $T_{mr}$ and $t^{r-1}_{m-1}$ via
\begin{gather}\label{eq:r-1_coord_change}
T_{mr}=\frac{1}{(-r)^{\frac{m(r-2)}{2(r+1)}}m!r^m}t^{r-1}_{m-1},\quad m\ge 1.
\end{gather}

The main algebraic result of \cite{BCT2} is:

\begin{thm}\label{thm:main_BCT2}
It holds that
\begin{gather}\label{eq:main_res_BCT}
F^{\frac{1}{r},o}_0=\frac{1}{\sqrt{-r}}\phi_0\big|_{t^{r-1}_d\mapsto \frac{1}{\sqrt{-r}}(t^{r-1}_d-r\delta_{d,0}s)}-\frac{1}{\sqrt{-r}}\phi_0\big|_{t^{r-1}_d\mapsto\frac{1}{\sqrt{-r}}t^{r-1}_d}.
\end{gather}
\end{thm}
\noindent This provides the open $r$-spin version of Witten's conjecture in genus zero. A mirror theoretic interpretation of the invariants constructed in \cite{BCT2} appears in \cite{GKT}.

In the case $r=2$~\cite{Bur16} proved  an explicit formula for the spinless open potential $F^o$ of \cite{PST14,Tes15} in terms of the function $\phi=\log\Phi$.

We should point out that at the moment no analogous results are known or conjectured for the $\lfloor r/2\rfloor$ open $r$-spin theories constructed in \cite{TZ1,TZ2}, even though all $g=0$ intersection numbers of these theories are fully calculated in \cite{TZ2}, and despite seemingly related predictions from \cite{Horiprivate} and implicitly also from \cite{Hori}.

\cite{BCT3} conjectured an  all genus genus generalization of Theorem \ref{thm:main_BCT2}. More precisely, they conjectured the following. 
\begin{conj}[The rGD wave function conjecture of \cite{BCT3}]\label{conj:main_conj_BCT}
For every genus $g\geq 1$, there is a geometric construction of open $r$-spin intersection numbers
\begin{gather}\label{eq:open_numbers_all_g}
\langle\tau^{\alpha_1}_{d_1}\cdots\tau^{\alpha_l}_{d_l}\sigma^k\rangle^{\frac{1}{r},o}_g.
\end{gather}
Moreover, if we define the genus $g$ potential~$F_g^{\frac{1}{r},o}(t^*_*,s)$ by
\begin{equation}\label{eq:g_pot}
F_g^{\frac{1}{r},o}(t^*_*,s):=\sum_{l,k\ge 0}\frac{1}{l!k!}\sum_{\substack{0\le\alpha_1,\ldots,\alpha_l\le r-1\\d_1,\ldots,d_l\ge 0}}\langle\tau^{\alpha_1}_{d_1}\cdots\tau^{\alpha_l}_{d_l}\sigma^k\rangle^{\frac{1}{r},o}_g t^{\alpha_1}_{d_1}\cdots t^{\alpha_l}_{d_l}s^k.
\end{equation}
Then for every $g\geq 1$, we have
\begin{equation}\label{eq:bct_conj}
F^{\frac{1}{r},o}_g=\left.(-r)^{\frac{g-1}{2}}\phi_g\right|_{t^{r-1}_d\mapsto\frac{1}{\sqrt{-r}}(t^{r-1}_d-\delta_{d,0}rs)}.
\end{equation}
\end{conj}
\begin{rmk}
In open Gromov--Witten theory, unlike the closed theory, not much is known or even conjectured about higher-genus invariants. The conjecture above is on of the few conjectures that describes a full, all-genus open Gromov--Witten theory, and also one of the few conjectures that relate the potential to an integrable hierarchy.\footnote{As mentioned, an analogous conjecture was made in \cite{PST14} for the $r=2$ case, and it was proven in~\cite{BT17}. In~\cite{BPTZ}, the full stationary open Gromov--Witten theory for maps to $(\mathbb{CP}^1,\mathbb{RP}^1)$ is conjectured.}
\end{rmk}
\subsubsection{The universal $g=1$ open topological recursion relations, string and dilaton relations}
In \cite{sol_owdvv} Solomon showed that certain $g=0$ open GW theories satisfy a universal relation which he termed the \emph{open WDVV}. Also the genus $0$ $r$-spin theory of \cite{BCT1,BCT2,BCT3} can be shown to satisfy the open WDVV. While not all open GW and open FJRW theories are expected to satisfy the open WDVV (e.g. the rank $2$ Fermat theory studied in \cite{GKT2} does not satisfy this relation), a large class of open GW theories, which include  the OGW theory of $(\mathbb{CP}^n,\mathbb{RP}^n)$ for $n$ odd \cite{solomon2023relative},  do satisfy.

In \cite{alexandrov2023construction} Alexandrov, Basalaev and Buryak construct an algebraic extension of the open potentials of theories satisfying the open WDVV to higher genus and to include descendents. They show that their extensions satisfy certain universal equations: the \emph{open string} and \emph{open dilaton} equations for all genus, and $g=0,1$ \emph{topological recursion relations}. See also \cite{gomez2021open}.
They conjecture that there exist geometric constructions of these extensions. We refer the reader to their paper for further details.

In the case of $r$-spin theory, if we denote by $F^{\frac{1}{r}}(t_\star^\star,s)$ the conjectural extension, and $F^{\frac{1}{r}}_g$ is its genus $g$ component, then we can write these equations concretely in terms of intersection numbers as follows. We will assume in all cases that the left hand side of the equation satisfies \eqref{eq:dim_equal_rk}.
\begin{itemize}
\item[(a)] (Boundary marked point $g=0$ TRR) Suppose $l,k\geq 1$.  Then
\begin{align*}
\< \tau_{d_1+1}^{a_1}\prod_{i=2}^l\tau^{a_i}_{d_i}\sigma^k\>^{\frac{1}{r},o}_0=&
\sum_{a=-1}^{r-2}\sum_{S \sqcup R = \{2,\ldots,l\}}\left\langle \tau_0^{a}\tau_{d_1}^{a_1}\prod_{i \in S}\tau_{d_i}^{a_i}\right\rangle^{\frac{1}{r},{ext}}_0
\left\langle \tau_0^{r-2-a}\prod_{i\in R}\tau^{a_i}_{d_i}\sigma^k\right\rangle^{\frac{1}{r},o}_0+\\
&+\sum_{\substack{S \sqcup R = \{2,\ldots,l\} \\ k_1 + k_2 = k-1}} \binom{k-1}{k_1} \left\langle \tau^{a_1}_{d_1}\prod_{i \in S} \tau^{a_i}_{d_i}\sigma^{k_1}\right\rangle^{\frac{1}{r},o}_0 \left\langle \prod_{i \in R} \tau^{a_i}_{d_i} \sigma^{k_2+2}\right\rangle^{\frac{1}{r}, o}_0.
\end{align*}
\item[(b)] (Internal marked point $g=0$ TRR) Suppose $l\geq 2$.  Then
\begin{align*}
\<\tau_{d_1+1}^{a_1}\prod_{i=2}^l\tau^{a_i}_{d_i}\sigma^k\>^{\frac{1}{r},o}_0=&
\sum_{a=-1}^{r-2}\sum_{S \sqcup R = \{3,\ldots,l\}}\left\langle \tau_0^{a}\tau_{d_1}^{a_1}\prod_{i \in S}\tau_{d_i}^{a_i}\right\rangle^{\frac{1}{r},{ext}}_0
\left\langle \tau_0^{r-2-a}\tau^{a_2}_{d_2}\prod_{i\in R}\tau^{a_i}_{d_i}\sigma^k\right\rangle^{\frac{1}{r},o}_0+\\
&+\sum_{\substack{S \sqcup R = \{3,\ldots,l\} \\ k_1 + k_2 = k}} \binom{k}{k_1} \left\langle \tau^{a_1}_{d_1} \prod_{i \in S} \tau^{a_i}_{d_i}\sigma^{k_1}\right\rangle^{\frac{1}{r},o}_0 \left\langle \tau^{a_2}_{d_2}\prod_{i \in R} \tau^{a_i}_{d_i} \sigma^{k_2+1}\right\rangle^{\frac{1}{r}, o}_0.
\end{align*}
\item[(c)](Genus $1$ TRR)
\begin{equation}
\begin{split}
\left\langle \tau_{d_{1}+1}^{a_1}\prod_{i\in [l]\setminus \{1\}}\tau^{a_i}_{d_i}\sigma^k\right\rangle^{\frac{1}{r},o}_1\hspace{-0.2cm}
=&\sum_{\substack{J_1 \sqcup J_2 = [l]\setminus\{1\}\\ -1\le a \le r-2}}\hspace{-0.1cm}\left\langle \tau_0^{a}\tau_{d_{1}}^{a_{1}}\prod_{i \in J_1}\tau_{d_i}^{a_i}\right\rangle^{\frac{1}{r},\text{ext}}_0
\hspace{-0.1cm}\left\langle \tau_0^{r-2-a}\prod_{i\in J_2}\tau^{a_i}_{d_i}\sigma^k\right\rangle^{\frac{1}{r},o}_1\\
&+\hspace{-0.1cm}\sum_{\substack{J_1 \sqcup J_2 =  [l]\setminus\{1\} \\ k_1+k_2=k}} \hspace{-0.1cm} \binom{k}{k_1}\left\langle \tau^{a_{1}}_{d_{1}}\prod_{i \in J_1} \tau^{a_i}_{d_i}\sigma^{k_1}\right\rangle^{\frac{1}{r},o}_0 \hspace{-0.1cm}\left\langle \prod_{i \in J_2} \tau^{a_i}_{d_i} \sigma^{k_2+1}\right\rangle^{\frac{1}{r}, o}_1\\
&+\frac{1}{2}\left\langle \prod_{i\in [l]}\tau^{a_i}_{d_i}\sigma^{k+1}\right\rangle^{\frac{1}{r},o}_0.
    \end{split}
\end{equation}\label{eq:g=1_trr}
\item[(d)](Open String)
Assume $2g-2+k+2l>0$ then
\begin{equation}\label{eq:open_string}
\left\langle\tau^0_0\prod_{i=1}^l\tau^{a_i}_{d_i}\sigma^k\right\rangle^{\frac{1}{r},o}_g=\sum_{j=1}^l\left\langle\tau^{a_j}_{d_j-1}\prod_{i=1,i\neq j}^l\tau^{a_i}_{d_i}\sigma^k\right\rangle^{\frac{1}{r},o}_g,
\end{equation}
where an intersection number which includes $\tau^*_{-1}$ is defined to be $0.$
\item[(e)](Open Dilaton)
Assume $2g-2+k+2l>0$, then
\begin{equation}\label{eq:open_dilaton}
\left\langle\tau^1_0\prod_{i=1}^l\tau^{a_i}_{d_i}\sigma^k\right\rangle^{\frac{1}{r},o}_g=(g+l+k-1)\left\langle\prod_{i=1}^l\tau^{a_i}_{d_i}\sigma^k\right\rangle^{\frac{1}{r},o}_g
\end{equation}
\end{itemize}
The first two equations were proven to hold in \cite{BCT2}.

\subsection{This work: An $r$-spin cylinder intersection theory}
In this work we move up to genus $1,$ and construct the $g=1$ open $r$-spin theory, or equivalently the \emph{open $r$-spin cylinder theory}. This is the first rigorous construction of an open FJRW theory in positive genus.
The difficulties which appeared in the construction of the $r$-spin disk theories, that is the orientation problem and the choice of boundary conditions, appear in the cylinder theory as well, and are more challenging to overcome. In addition, we encounter a new difficulty: a dimension jump in the Witten ``bundle''. Thus, it is a true bundle only on the complement of a subspace.

To treat the dimension jump problem we first identified the dimension jump loci in the moduli, and then restricted our attention to sections of our bundles-of-interest which have a prescribed non vanishing behaviour near these loci.
This is essentially equivalent to excluding (neighbourhoods of) these loci from the moduli space, thus obtaining new boundaries, and imposing prescribed boundary conditions.
 
The orientation problem has been solved in \cite[\S 3]{TZ1}, where it was shown that away from the dimension jump loci the Witten bundle has a canonical relative orientation, which allows us to integrate.

A key difficulty in defining boundary conditions was that the positivity of the Witten bundle's sections, which played a central role in constructing the $g=0$ theory, proving independence of choices for intersection numbers, and calculating these numbers, surprisingly breaks in $g=1.$
This expected positivity has lead the first-named-author to anticipate, in \cite{BCT3}, that the (back then) yet-to-be-defined open $r$-spin theory in $g=1$ would satisfy the genus-one topological recursion relation Theorem \ref{thm:trr_g1} below.

But there is a cure for this problem, which comes from the relative cotangent lines $\CL_i.$ Dimension-versus-Rank considerations \eqref{eq:dim_equal_rk} show that only intersection numbers which contain at least one $\psi_i$ class may be non zero. Miraculously, there exist completely new, natural, boundary conditions for the $\psi_i$ classes at those boundaries that the expected positivity does not hold. Interestingly, in order to define these new boundary conditions consistently we need to work with the point-insertion analogue of the BCT $r$-spin theory mentioned above. We do not know if this is a fundamental problem or just a technicality. We refer to sections of direct sums of the Witten bundle and relative cotangent lines satisfying the above boundary conditions as \emph{canonical}.

To summarize, combining the above ideas we have shown (see Definition \ref{dfn correlator} and Theorem \ref{thm indenpdent of order} for a precise statement)
\begin{thm}\label{thm:int_nums_well_def}
Assume that the twists $a_1,\ldots,a_l$ and numbers of descendents $d_1,\ldots,d_l$ satisfy~\eqref{eq:dim_equal_rk} with $g=1$. We can define an integral
\[\int_{\widetilde{\mathcal{PM}}^{1/r}_{1,k,\{a_1, \ldots, a_l\}}} e\left( \mathcal{W} \oplus \bigoplus_{i=1}^l \mathbb{L}_i^{\oplus d_i}, \mathbf{s}_{\text{canonical}}\right),
\]
where ${\widetilde{\mathcal{PM}}^{1/r}_{1,k,\{a_1, \ldots, a_l\}}}$ is the point insertion moduli space of $r$-spin cylinders with internal points twisted by $a_1,\ldots,a_l$ and $k$ boundary markings twisted $r-2$, after removing certain boundary strata, and $\mathbf{s}_{\text{canonical}}$ is an arbitrary choice of nowhere-vanishing canonical section of $\mathcal{W} \oplus \bigoplus_{i=1}^l\mathbb{L}_i^{\oplus d_i}|_{\partial{\widetilde{\mathcal{PM}}^{1/r}_{1,k,\{a_1, \ldots, a_l\}}}}$.

Moreover, the integral is independent of the choice of the non-vanishing canonical $\mathbf{s}_{\text{canonical}}.$ 
\end{thm}
We denote the above integral by $\left\langle\tau_{d_1}^{a_1}\ldots\tau^{a_l}_{d_l}\sigma^k\right\rangle^{\frac{1}{r},o}_1.$
These numbers are taken to be $0$ if \eqref{eq:dim_equal_rk} for $g=1$ fails.

The above intersection numbers can also be calculated. The main algebraic results are
\begin{thm}\label{thm:trr_g1}
The open $g=1$ TRR \eqref{eq:g=1_trr} holds, whenever the left-hand side, satisfies \eqref{eq:dim_equal_rk}.
\end{thm}
The proof is very different from proofs of $g=0$ TRR and $g=0$ open WDVV which have appeared in the literature \cite{PST14,BCT2,TZ2,solomon2023relative} and relies strongly on $g=1$ properties.
This is the first $g=1$ open model which is proven to satisfy the universal $g=1$ TRR, except for the special case $r=2,$ which can be shown algebraically to satisfy TRR by combining the proof of the open analogue of Witten's conjecture \cite{BT17} and \cite[Theorem 1]{BCT3}.
\begin{thm}\label{thm:wave_func}
    The potential of the genus-one open $r$-spin numbers, defined via \eqref{eq:g_pot} for $g=1,$ satisfies the $g=1$ case of the rGD wave function conjecture \ref{conj:main_conj_BCT}.
\end{thm}

\begin{thm}\label{thm:open_string_dilaton}
Assume $k+2l>0,$ then the $g=1$ Open String equation \eqref{eq:open_string} and the Open Dilaton equation \eqref{eq:open_dilaton} hold as long as their left hand side satisfies \eqref{eq:dim_equal_rk}.
\end{thm}
Again, with the exception of the special case $r=2$ studied in \cite{Tes2,BT17} this is the only work so far which has shown relation between an open theory and an integrable hierarchy in $g>0,$ as well as the first work to have proven the open string and dilaton for $g>0.$
\subsection{Discussion: Possible generalizations}
This works suggests several natural research avenues. 
\begin{itemize}
    \item \underline{$g>1:$ }It is desired to construct the $g>1$ sectors of open $r$ spin as well, and to prove Conjecture \ref{conj:main_conj_BCT} for them. All difficulties which were solved in this work also make their appearance in the $g>1$ cases, in an amplified manner. In particular the dimension jump loci are less understood, hence it may be unavoidable to wait for the development of virtual cycle machinery. 
    \item \underline{Higher rank open FJRWs: }\cite{GKT2} constructed higher rank open FJRW theories. Our construction does not extend to this setting, since it may include primary $g=1$ intersection numbers, but our construction leans on the presence of $\psi_i$ classes. Still, this obstacle might be overcame in different ways. We expect the resulting intersection numbers to be subject to \emph{wall crossings}, like in \cite{GKT2}, but that some polynomial combinations of them will be invariant, and satisfy $g=1$ TRR, which is \emph{structurally different} from the TRR of Theorem \ref{thm:trr_g1}.
    \item \underline{Theories with more types of boundary insertions: }\cite{TZ2} constructs $\lfloor\frac{r}{2}\rfloor$ open $r$-spin theories with more possible boundary twists. It is natural to generalize the construction of this paper to this setting, and to seek TRRs for them. 
    \item \underline{Open Gromov--Witten (OGW) theories: }
    Finally, a very natural generalization is to Open GW theories. Also in the Open GW case dimension jumps should be milder in $g=1$ than in $g>1.$ 
    As mentioned above, large class of $g=0$ Open GW theories are subject to Solomon's OWDVV \cite{sol_owdvv,solomon2023relative}. We expect that for these theories there should be a $g=1$ extension which satisfies the analogue of our $g=1$ TRR, in which the boundary conditions for the $\psi_i$ classes should resemble those we construct here.
\end{itemize}
\subsection{Plan of the paper}

The structure of the paper is as follows. In \S\ref{sec review} we review the open $r$-spin structures, their moduli spaces, the Witten bundles, and the point insertion technique studied in $\cite{BCT1,TZ1}$. In \S \ref{sec relative  cotangent and all}, we focus on the relative cotangent line bundles $\mathbb L$: we study the relation between them and construct special sections of them. In \S \ref{sec sections def etc} we introduce the notion of canonical sections and define the genus-one open $r$-spin intersection numbers as the zero count of a canonical section. In \S \ref{sec g=1 trr  and inde of choice} we show that the genus-one open $r$-spin intersection numbers are independent of the choice of the canonical section, and prove that they satisfy the genus-one open Topological Recursion Relations. In \S\ref{sec construction of section} we prove the existence of canonical sections.

\subsection{Acknowledgements}
R.T. thanks Tyler Kelly, Jake Solomon and Sasha Buryak for interesting discussion related to the content of this work. R.T. and Y.Z. were supported by the ISF (grant No. 1729/23).

\section{Review of graded $r$-spin disks and cylinders}\label{sec review}
\label{sec mod and bundle}
In this section, we review the definition of graded $r$-spin disks and cylinders, their moduli space, the relevant bundles and the point insertion technique.  More details can be found in \cite{BCT1,TZ1}.

Throughout the paper we will use the notation $[N],$ where $N$ is a natural number, to denote the set $\{1,2,\dots,N\}.$

\subsection{The moduli space of graded $r$-spin surfaces}

The main objects of study in this paper are genus-zero and genus-one marked Riemann surfaces with boundary; we view them as arising from closed spheres and torus with an orientation reversing involution.  More precisely, a \textit{nodal marked surface} is defined as a tuple
$$(C, \phi, \Sigma, \{z_i\}_{i \in I}, \{x_j\}_{j \in B}, m^I, m^B),$$
in which
\begin{itemize}
\item  $C$ is a nodal orbifold Riemann surface, which may be composed of disconnected components, and has isotropy only at special points;
\item $\phi: C \to C$ is an anti-holomorphic involution which, from a topological perspective, realizes the coarse underlying Riemann surface $|C|$  as the union of two Riemann surfaces, $\Sigma$ and $\overline{\Sigma}=\phi(\Sigma),$ glued along their common subset $\text{Fix}(|\phi|)$;
\item $\{z_i\}_{i \in I}\subset  C$ consists of distinct \textit{internal marked points} (or \textit{internal tails}) whose images in $|C|$ lie in $\Sigma\setminus\text{Fix}(|\phi|)$. We denote by $\overline{z_i}:= \phi(z_i)$ their \textit{conjugate marked points};
\item $\{x_j\}_{j \in B} \subset \text{Fix}(\phi)$ consists of distinct \textit{boundary marked points} (or \textit{boundary tails}) whose images in $|C|$ lie in ${\partial} \Sigma$;
\item $m^I$ (respectively $m^B$) is a marking of $I$ and (respectively $B$), \textit{i.e.} a one-to-one correspondence between  $I$ (respectively $m^B$) and $\left[\lvert I \rvert\right]$ (respectively $[\lvert B \rvert]$).
\end{itemize}
The genus of a nodal marked surface is the sum of the genus of each connected component of $C$. We call a nodal marked surface a disk if it is genus-zero, and a cylinder if it is genus-one.
We say that a nodal marked surface is \textit{stable} if each genus-zero irreducible component has at least three special points, and each genus-one irreducible component has at least one special point.

\begin{rmk}
    If $\text{Fix}(|\phi|)=\emptyset$, then $\Sigma$ is a closed nodal orbifold Riemann surface, and $C=\Sigma \sqcup \overline{\Sigma}$ is a disjoint union. Notice that in this case the genus is always even, and therefore for a genus-one nodal marked surface we always have $\text{Fix}(|\phi|)\ne \emptyset$.
\end{rmk}

A node of a nodal marked surface can be internal, boundary or contracted boundary; on genus-one connected component, an internal or boundary node can be separating or non-separating, as illustrated in Figure~\ref{fig node type} by shading $\Sigma \subseteq |C|$ in each case.  Note that ${\partial}\Sigma\subset\text{Fix}(|\phi|)$ is a union of circles, and $\text{Fix}(|\phi|)\setminus{\partial}\Sigma$ is the union of the contracted boundary nodes.
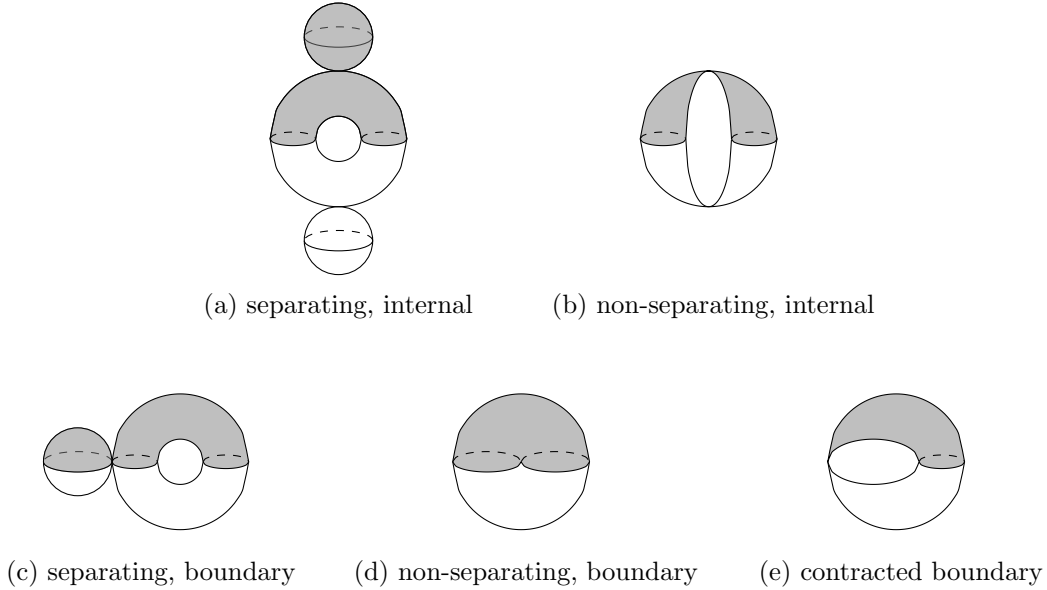
\begin{figure}[h]
\centering

  \begin{subfigure}{.3\textwidth}
  \centering

\begin{tikzpicture}[scale=0.3]
   \draw[ domain=-1:1,  smooth, variable=\x] plot ({\x}, {sqrt(1-\x*\x)});
   \draw[ domain=-3:3,   smooth, variable=\x] plot ({\x}, {sqrt(9-\x*\x)});

   \draw[ domain=-1:1, smooth, variable=\x] plot ({\x}, {-sqrt(1-\x*\x)});
   \draw[ domain=-3:3,  smooth, variable=\x] plot ({\x}, {-sqrt(9-\x*\x)});

    \draw [fill=gray,fill opacity=0.5] plot [domain=-1:1,  smooth]({\x}, {sqrt(1-\x*\x)}) --plot [domain=1:3,  smooth]({\x}, {-sqrt(1-(\x-2)*(\x-2))*0.3}) -- plot [domain=3:-3,  smooth ] ({\x}, {sqrt(9-\x*\x)})-- plot [domain=-3:-1,  smooth]({\x}, {-sqrt(1-(\x+2)*(\x+2))*0.3});

    \draw  [dashed, domain=1:3,  smooth] plot ({\x}, {sqrt(1-(\x-2)*(\x-2))*0.3});

    \draw  [dashed, domain=-3:-1,  smooth] plot ({\x}, {sqrt(1-(\x+2)*(\x+2))*0.3});

    \draw (0,4.5) circle (1.5);
    \draw (-1.5,4.5) arc (180:360:1.5 and 0.45);
    \draw[dashed] (1.5,4.5) arc (0:180:1.5 and 0.45);

     \draw (0,-4.5) circle (1.5);
    \draw (-1.5,-4.5) arc (180:360:1.5 and 0.45);
    \draw[dashed] (1.5,-4.5) arc (0:180:1.5 and 0.45);

    \draw [fill=gray,fill opacity=0.5] (0,4.5) circle (1.5);
   
\end{tikzpicture}

  \caption{separating, internal }
\end{subfigure}
  \begin{subfigure}{.3\textwidth}
  \centering

\begin{tikzpicture}[scale=0.3]

   \draw[ domain=-3:3,   smooth, variable=\x] plot ({\x}, {-sqrt(9-\x*\x)});
  \draw [fill=gray,fill opacity=0.5] plot [domain=-1:1,  smooth]({\x}, {sqrt(1-\x*\x)*3}) --plot [domain=1:3,  smooth]({\x}, {-sqrt(1-(\x-2)*(\x-2))*0.3}) -- plot [domain=3:-3,  smooth ] ({\x}, {sqrt(9-\x*\x)})-- plot [domain=-3:-1,  smooth]({\x}, {-sqrt(1-(\x+2)*(\x+2))*0.3});

  \draw  [dashed, domain=1:3,  smooth] plot ({\x}, {sqrt(1-(\x-2)*(\x-2))*0.3});

    \draw  [dashed, domain=-3:-1,  smooth] plot ({\x}, {sqrt(1-(\x+2)*(\x+2))*0.3});

    \draw[ domain=-1:1,   smooth, variable=\x] plot ({\x}, {-sqrt(1-\x*\x)*3});
\end{tikzpicture}
\vspace{0.9cm}
  \caption{non-separating, internal }
\end{subfigure}

\begin{subfigure}{.3\textwidth}
  \centering
\vspace{0.9cm}
\begin{tikzpicture}[scale=0.3]
  \draw[ domain=-1:1, smooth, variable=\x] plot ({\x}, {-sqrt(1-\x*\x)});
   \draw[ domain=-3:3,  smooth, variable=\x] plot ({\x}, {-sqrt(9-\x*\x)});

    \draw [fill=gray,fill opacity=0.5] plot [domain=-1:1,  smooth]({\x}, {sqrt(1-\x*\x)}) --plot [domain=1:3,  smooth]({\x}, {-sqrt(1-(\x-2)*(\x-2))*0.3}) -- plot [domain=3:-3,  smooth ] ({\x}, {sqrt(9-\x*\x)})-- plot [domain=-3:-1,  smooth]({\x}, {-sqrt(1-(\x+2)*(\x+2))*0.3});

    \draw  [dashed, domain=1:3,  smooth] plot ({\x}, {sqrt(1-(\x-2)*(\x-2))*0.3});

    \draw  [dashed, domain=-3:-1,  smooth] plot ({\x}, {sqrt(1-(\x+2)*(\x+2))*0.3});

    \draw (-4.5,0) circle (1.5);
    \draw (-6,0) arc (180:360:1.5 and 0.45);
    \draw[dashed] (-3,0) arc (0:180:1.5 and 0.45);
    \draw[fill = gray, opacity = 0.5] (-6,0) arc (180:360:1.5 and 0.45) arc (0:180:1.5);
\end{tikzpicture}
\vspace{0.15cm}

  \caption{separating, boundary }
\end{subfigure}
\begin{subfigure}{.3\textwidth}
  \centering

\begin{tikzpicture}[scale=0.3]
\vspace{0.15cm}

   \draw[ domain=-3:3,  smooth, variable=\x] plot ({\x}, {-sqrt(9-\x*\x)});

    \draw [fill=gray,fill opacity=0.5] plot [domain=0:3,  smooth]({\x}, {-sqrt(2.25-(\x-1.5)*(\x-1.5))*0.3}) -- plot [domain=3:-3,  smooth ] ({\x}, {sqrt(9-\x*\x)})-- plot [domain=-3:0,  smooth]({\x}, {-sqrt(2.25-(\x+1.5)*(\x+1.5))*0.3});

    \draw  [dashed, domain=0:3,  smooth] plot ({\x}, {sqrt(2.25-(\x-1.5)*(\x-1.5))*0.3});

    \draw  [dashed, domain=-3:0,  smooth] plot ({\x}, {sqrt(2.25-(\x+1.5)*(\x+1.5))*0.3});
  
\end{tikzpicture}
\vspace{0.15cm}

  \caption{non-separating, boundary}
\end{subfigure}
\begin{subfigure}{.3\textwidth}
  \centering

\begin{tikzpicture}[scale=0.3]
\vspace{0.15cm}
  \draw[ domain=-3:3,  smooth, variable=\x] plot ({\x}, {-sqrt(9-\x*\x)});

    \draw [fill=gray,fill opacity=0.5] plot [domain=1:3,  smooth]({\x}, {-sqrt(1-(\x-2)*(\x-2))*0.3}) -- plot [domain=3:-3,  smooth ] ({\x}, {sqrt(9-\x*\x)})-- plot [domain=-3:1,  smooth]({\x}, {sqrt(4-(\x+1)*(\x+1))*0.5});

    \draw  [dashed, domain=1:3,  smooth] plot ({\x}, {sqrt(1-(\x-2)*(\x-2))*0.3});

    \draw  [ domain=-3:1,  smooth] plot ({\x}, {-sqrt(4-(\x+1)*(\x+1))*0.5});
\end{tikzpicture}
\vspace{0.15cm}

  \caption{contracted boundary }
\end{subfigure}

\caption{The five types of nodes on a nodal marked cylinder.}
\label{fig node type}
\end{figure}

A \textit{rooted nodal marked disk} is a connected genus-zero nodal marked surface together with a choice of a distinguished boundary marked point (called the root). A \textit{2-rooted nodal marked disk} is a connected genus-zero nodal marked surface together with a choice of two distinguished boundary marked points (called the roots). An \textit{anchored nodal marked surface} is a nodal marked surface together with a $\phi$-invariant choice of a distinguished internal marked point (called the \textit{anchor}) on each connected component $C'$ of $C$ that is disjoint from the set $\text{Fix}(\phi)$. We denote by $Anc\subseteq I$ the set of indexes of anchors lying on $\Sigma$. Notice that for a genus-zero or genus-one anchored nodal marked surface, the anchors can only be on the genus-zero connected components.

\begin{rmk}\label{rmk root when degenerate}
We focus mainly on cylinders with boundaries and their degenerations, \textit{i.e.} connected genus-one nodal marked surfaces. (Disconnected) surfaces with marked surfaces as component will usually appear as the result of normalizations at nodes (and later also from the point insertion method we shall define). The roots and anchors help us keep track of these normalizations:
\begin{itemize}
    \item[-]  If we normalize at a separating boundary node of a connected nodal cylinder, then we get two connected components, one genus-zero and one genus-one. The genus-zero component is a rooted nodal marked disk if we regard the half-node from normalization on the genus-zero side as a root.
    \item[-] If we normalize at a non-separating boundary node of a connected nodal cylinder, then we get a connected nodal marked disk, which is a 2-rooted nodal marked disk if we regard the two half-nodes from normalization as roots.
    \item[-]  If we normalize at a pair (preserved by $\phi$) of separating internal nodes of a connected nodal surface with non-empty $\operatorname{Fix}(\phi)$, then we get three connected components, one intersects $\operatorname{Fix}(\phi)$ and a pair (preserved by $\psi$) disjoint from $\operatorname{Fix}(\phi)$. We choose the half-nodes from normalization on the sides disjoint from $\operatorname{Fix}(\phi)$ to be the anchors.
\end{itemize}
We use two similar terminologies `root' and `anchor', where `root' will only be used for half-nodes on open components and `anchor' is reserved for internal half-nodes on a closed surface. The usage of `anchor' is also consistent with the terminology in \cite{BCT1,BCT2,TZ1,TZ2}.

\end{rmk}

\subsubsection{Open $r$-spin structures}
Let $C$ be an anchored nodal marked surface with
order-$r$ cyclic isotopy groups at markings and nodes, a \textit{$r$-spin structure} on $C$ is 
\begin{itemize}
    \item 
 an orbifold complex line bundle $L$ on $C$,

\item an isomorphism 
$$\kappa:L^{\otimes r} \cong \omega_{C,log}:=\omega_{C} \otimes {\mathcal{O}}\left(\sum_{i \in I}  [z_i] + \sum_{i \in I}  [\overline{z_i}] + \sum_{j \in B} [x_j]\right),$$
where $[p]$ is the degree-$1/r$ divisor for each marking $p$. 

\item an involution $\widetilde{\phi}: L \to L$ lifting $\phi$.
\end{itemize}
The local isotopy of $L$ at a point $p$ is characterized by an integer $\operatorname{mult}_p(L)\in \{0,1,\dots,r-1\}$ in the following way: the local structure of the total space of $L$ near $p$ is $[\mathbb C^2/(\mathbb Z/r\mathbb Z)]$, where the canonical generator $\xi \in \mathbb Z/r\mathbb Z$ acts by $\xi\cdot (x,y)=(\xi x, \xi^{\operatorname{mult}_p(L)}y)$.

We denote by $RI\subseteq I$ and $RB\subseteq B$ the subsets of internal and boundary marked points $p$ satisfying $\operatorname{mult}_p(L)=0$. An \textit{associated twisted r-spin structure} $S$ is defined by
$$
S:=L\otimes \mathcal O\left( -\sum_{i \in \widetilde{RI}} r [z_i] - \sum_{i \in \widetilde{RI}} r [\overline{z_i}] - \sum_{j \in RB} r [x_j]\right),
$$
where $\widetilde{RI}\subseteq RI$ is a subset satisfying $\widetilde{RI}=RI \setminus Anc   $.

For an internal marked point $z_i$, we define the  \textit{internal twist} at $z_i$ to be $a_i:=\operatorname{mult}_{z_i}(L)-1$ if $i\in I\setminus \widetilde{RI}$
and $a_i:=r-1$ if $i\in \widetilde{RI}$. For a boundary marked point $x_i$, we define the  \textit{boundary twist} at $x_j$ as $b_j:=\operatorname{mult}_{x_j}(L)-1$ if $i\in B\setminus {RB}$ and as $b_j:=r-1$ if $j\in {RB}$. Note that all the marked points with twist $-1$ are indexed by $RI\setminus \widetilde{RI}\subseteq Anc$. When the surface $C$ is smooth, the coarse underlying bundle $|S|$ over the coarse underlying sphere $|C|$ satisfies
$$|S|^{\otimes r} \cong \omega_{|C|} \otimes {\mathcal{O}}\left(-\sum_{i \in I} a_i [z_i] - \sum_{i \in I} a_i [\overline{z_i}] - \sum_{j \in B} b_j[x_j]\right),$$
here $[p]$ is the degree-one divisor on the coarse $|C|$ corresponding to each marking $p$.

\begin{obs}
\label{obs rank open}
A connected nodal marked disk admits a twisted $r$-spin structure with twists $a_i$ and $b_j$ if and only if
\begin{equation}\label{eq mod r condition g=0}
\frac{2\sum_{i \in I} a_i + \sum_{j\in B} b_j -r+2}{r}\in \mathbb Z,
\end{equation}
and a connected nodal marked cylinder admits a twisted $r$-spin structure with twists $a_i$ and $b_j$ if and only if
\begin{equation}\label{eq mod r condition rk g=1}
\frac{2\sum_{i \in I} a_i + \sum_{j\in B} b_j}{r}\in \mathbb Z.
\end{equation}
These formulae are the specializations to our setting of the more well-known fact \cite{Witten93}: a (closed) connected nodal marked genus-$g$ curve admits a twisted $r$-spin structure twists $a_i$ if and only if 
\begin{equation}\label{eq rank close}
\frac{\sum_{i\in I} a_i +(g-1)(r-2)}{r}\in {\mathbb{Z}}.
\end{equation}
\end{obs}

We can extend the definition of twists to \textit{half-nodes}. Let $n: \widehat{C} \to C$ be the normalization morphism. For a half-node $q\in \widehat{C}$, we denote by $\sigma_0(q)$ the other half-node corresponding to the same node $n(q)$ as $q$. The isotopies of $n^*L$ at $q$ and $\sigma_0(q)$ satisfy $$\operatorname{mult}_{q}(n^*L)+\operatorname{mult}_{\sigma_0(q)}(n^*L)\equiv0 \mod r.$$ It is important to note that $n^*S$ may not be a twisted $r$-spin structure (associated with $n^*L$), because its connected components could potentially contain too many marked points with twist $-1$ (note that marked points with twist $-1$ are anchors).  Nevertheless, in genus zero or one, there is a canonical way to choose a minimal subset $\mathcal{R}$ of the half-nodes making
\begin{equation}
\label{eq normalize S}
\widehat{S}:= n^* S \otimes {\mathcal{O}}\left(-\sum_{q \in \mathcal{R}} r [q]\right)
\end{equation}
a twisted $r$-spin structure: denoting by $\mathcal T$ the set of half-nodes $q$ of $C$ where $\operatorname{mult}_{q}(n^*L)=0$, we define
\begin{equation*}
\mathcal A:= \left\{q \in \mathcal T\colon 
\begin{array}{ccc}
    \text{$n(q)$ is an internal node; after normalizing the }\\
    \text{node $n(q)$, the half-node $\sigma_0(q)$ belongs to a connected }\\
    \text{component meeting $\text{Fix}(\phi)$ or containing an anchor.}
\end{array} \right\}
\end{equation*}
and set $\mathcal R := \mathcal T \setminus \mathcal A$. Note that elements of $\mathcal R$ always come from a separating internal node and lie on the genus-zero irreducible component disjoint from $\text{Fix}(\phi)$, see \cite[\S 2.3]{BCT1} for more details. 
We define $c_t$, \textit{the twist of $S$ at a half-node} $h_t$, as $c_t:=\operatorname{mult}_{h_t}(n^*L)-1$ if $h_t \notin \mathcal R$ and as $c_t:=r-1$ if $h_t\in \mathcal R$.

For each irreducible component $C_l$ of $\widehat{C}$ with half-nodes $\{h_t\}_{t \in N_l}$, we have
$$\bigg(|\widehat{S}|\big|_{|C_l|}\bigg)^{\otimes r} \cong \omega_{|C_l|} \otimes {\mathcal{O}}\left(-\sum_{\substack{i \in I\\z_i\in C_l}} a_i[z_i] - \sum_{\substack{i \in I\\ \overline{z_i}\in C_l}} a_i [\overline{z}_i] - \sum_{\substack{j \in B\\ x_j\in C_l}} b_j[x_j] - \sum_{t \in N_l} c_t [h_t]\right);$$
note that in the case where $C_l$ intersects with $\partial \Sigma$, the set $\{h_t\}_{t \in N_l}$ is invariant under $\phi$, and the half-nodes conjugated by $\phi$ have the same twist.

Note that if $h_{t_1}=\sigma_0(h_{t_2})$, then
\begin{equation}\label{eq sum of twist at node}c_{t_1} + c_{t_2} \equiv -2 \mod r.\end{equation}
We say a node is \textit{Ramond} if one (hence both) of its half-nodes $h_t$ satisfy $c_t \equiv -1 \mod r$,  and it is said to be \textit{Neveu--Schwarz (NS)} otherwise. Note that if a node is Ramond, then both of its half-nodes lie in the set $\mathcal T$; moreover, a half-node has twist $-1$ if and only if it lies in $\mathcal A$. The set $\mathcal{R}$ in equation~\eqref{eq normalize S} is chosen in a way that each separating internal Ramond node has precisely
one half-edge in $\mathcal{R}.$

Associated to each twisted $r$-spin structure $S$, we define a Serre-dual bundle \begin{equation}\label{def of J}
    J:=S^{\vee} \otimes \omega_{C}.
\end{equation}
Note that the involutions on $C$ and $L$ induce involutions on $S$ and $J$; by abuse of notation, we denote the involutions on $S$ and $J$ also by $\widetilde{\phi}$.

\subsubsection{Gradings}
For a nodal marked surface, the boundary ${\partial}\Sigma$ of $\Sigma$ is endowed with a well-defined orientation, determined by the complex orientation on the preferred half $\Sigma \subseteq |C|$. This orientation induces the notion of positivity for $\phi$-invariant sections of $\omega_{|C|}$ over ${\partial} \Sigma$: let $p$ be a point of $\partial \Sigma$ which is not a node, we say a section $s$ is \textit{positive} at a $p$ if, for any tangent vector $v \in T_p({\partial} \Sigma)$ in the orientation direction, we have $\langle s(p), v \rangle > 0$, where $\langle -, - \rangle$ is the natural pairing between cotangent and tangent vectors.

Let $C$ be an anchored nodal marked surface, and let $A$ be the complement of the special points in $\partial \Sigma$.  We say a twisted $r$-spin structure on such $C$ is \textit{compatible on the boundary components} if there exists a $\widetilde{\phi}$-invariant section $v \in C^0\left(A, |S|^{\widetilde{\phi}}\right)$ (called a \textit{lifting} of $S$  on boundary components) such that the image of $v^{\otimes r}$ under the map on sections induced by the inclusion $|S|^{\otimes r} \to  \omega_{|C|}$ is positive.  We say $w \in C^0\left(A, |J|^{\widetilde{\phi}}\right)$ is a \textit{Serre-dual lifting of $J$ on the boundary components with respect to $v$} if $\langle w, v \rangle \in C^0(A, \omega_{|C|})$ is positive, where $\langle -, - \rangle$ is the natural pairing between $|S|^{\vee}$ and $|S|$.  This $w$ is uniquely determined by $v$ up to positive rescaling.

 Equivalence classes of liftings of $J$ (or equivalently, $S$) on the boundary components up to positive rescaling are equivalent to continuous sections of $S^0(|J|^{\widetilde\phi}),$ the $S^0$-bundle $\left(|J|^{\widetilde{\phi}}\setminus |J|_0\right)\big/ \mathbb R_+$ over $A$, where $|J|_0$ denotes the zero section of $|J|^{\widetilde{\phi}}$. Given an equivalence class $[w]$  of liftings, we say a boundary marked point or boundary half-node $x_j$ is \textit{legal}, or that $[w]$ \textit{alternates} at $x_j$, if $[w]$, as a section of $S^0(|J|^{\widetilde\phi})$, cannot be continuously extended to $x_j$.
We say an equivalence class $[w]$ of liftings of $J$ on boundaries is a \textit{grading of a twisted $r$-spin structure on boundary components} if, for every Neveu--Schwarz boundary node, one of the two half-nodes is legal and the other is illegal. 
\begin{rmk}
    The requirement that every NS boundary node has one legal and one illegal half-node arises from the behaviour of a grading on boundary components at degenerations. This condition, together with \eqref{eq sum of twist at node} allows smoothing the boundary node. See \cite{BCT1} for more details.
\end{rmk}

Let $q$ be a contracted boundary node of $C$, we say a twisted $r$-spin structure on $C$ is \textit{compatible} at $q$ if $q$ is Ramond and there exists a $\widetilde{\phi}$-invariant element $v \in |S|\big|_q$  (called a \textit{lifting} of $S$ at $q$) such that the image of $v^{\otimes r}$ under the map $|S|^{\otimes r}\big|_q \to \omega_{|C|}\big|_q$ is positive imaginary under the canonical identification of $\omega_{|C|}\big|_q$ with ${\mathbb{C}}$ given by the residue. See \cite[Definition 2.8]{BCT1} for more details.  Such a $v$ also admits a Serre-dual lifting, \textit{i.e.}  a $\widetilde{\phi}$-invariant $w \in |J|\big\vert_q$ such that $\langle v, w \rangle$ is positive imaginary.  We refer to the equivalence classes $[w]$ of such $w$ up to positive rescaling as a \textit{grading at contracted boundary node $q$}.

We say a twisted $r$-spin structure is \textit{compatible} if it is compatible on boundary components and at all contracted boundary nodes.  A (total) \textit{grading} is the collection of grading on boundary components together with a grading at each contracted boundary node. We say a grading is \textit{legal} if every boundary marked point is legal.

The grading is crucial in determining a canonical relative orientation for the Witten bundle (in \cite{TZ1}) and defining canonical boundary conditions (in \S \ref{sec positivity constraint}), which are key ingredients in defining genus-zero open $r$-spin intersection numbers.

The relation between the twists and legality, and the obstructions to having a grading, are summarized in the following proposition.
\begin{prop}[{\cite[Proposition 2.3]{BCT2}}] 
\label{prop lifting}
\begin{enumerate}
\item\label{it lift odd exist} When $r$ is odd, any twisted $r$-spin structure is compatible, and there is a unique grading.
\item\label{it lift odd legal} When $r$ is odd, a boundary marked point, or boundary half-node, $x_j$ in a twisted $r$-spin structure with a grading is legal if and only if its twist is odd.
\item\label{it lift even compatible} When $r$ is even, the boundary twists $b_j$ in a compatible twisted $r$-spin structure must be even.  

\item\label{it NS nodes}
In a graded $r$-spin structure, any Neveu--Schwarz boundary node has one legal half-node and one illegal half-node.\footnote{This item is part of the definition of grading, we put it here since it is an important property of the grading.}

\item\label{it Ramond boundary node}
Ramond boundary nodes can appear in a graded $r$-spin structure only when $r$ is odd, and in this case, their half-nodes are illegal with twists $r-1.$

\item\label{it lift compatible parity}
For a twisted $r$-spin structure over a connected marked disk, there exists grading that alternates precisely at a subset $D \subset \{x_j\}_{j \in B}$ if and only if
\begin{equation}\label{eq parity}
   \frac{2\sum a_i + \sum b_j+ 2}{r}\equiv |D| \mod 2.
\end{equation}

\item\label{it lift compatible parity g=1}
For a twisted $r$-spin structure over a connected marked cylinder with two boundary components $\partial_1 \Sigma$ and $\partial_2 \Sigma$, there exists grading that alternates precisely at a subset $D \subset \{x_j\}_{j \in B}$ if and only if
\begin{equation}\label{eq boundarywise pairty}
    |D\cap \partial_i \Sigma| \equiv  \Theta_i \mod 2
\end{equation}

for $i=1,2$, where $\Theta_i:=0$ if $|J|^{\tilde\phi}_{\partial_i\Sigma}$ is orientable and $\Theta_i:=1$ if $|J|^{\tilde\phi}_{\partial_i\Sigma}$ is not orientable.
In particular, we have
\begin{equation}\label{eq compatitable pairty g=1}
   \frac{2\sum a_i + \sum b_j}{r} \equiv |D| \mod 2.
\end{equation}

\end{enumerate}
\end{prop}

When a Ramond contracted boundary node is normalized, the grading at this boundary node induces an additional structure at the corresponding half-node. We call such half-nodes by \textit{normalized contracted boundary marked point}. Because such an additional structure is not necessary in this paper, we refer the readers to \cite[Definition 2.5]{TZ1} for the precise definition.

We can now define the primary objects of interest in this paper:

\begin{definition}
\label{def graded rspin disk}
A \textit{stable graded $r$-spin disk or cylinder} (respectively, \textit{legal stable graded $r$-spin disk or cylinder}) is a stable anchored nodal marked disk or cylinder, together with
\begin{enumerate}
\item a compatible twisted $r$-spin structure $S$ in which all contracted boundary nodes are Ramond;
\item a choice of grading (respectively, legal gradings such that all boundary marking have twist $r-2$ );

\item  an additional structure of normalized contracted boundary marked point at a subset of markings with twist $r-1$.

\end{enumerate}
\end{definition}

\begin{rmk}
   The ``legal stable graded $r$-spin disk'' here is referred to as ``legal stable graded level-$0$ $r$-spin disk'' in \cite{TZ1,TZ2}. In this paper we will focus on the case where all the legal markings have twist $r-2$ as in \cite{BCT2}. We will omit the term ``level-$0$''.
\end{rmk}

We denote by $\Mbarstar_{g,k,l}^{1/r}$ ($\Mbar_{g,k,l}^{1/r}$ respectively) the moduli space of connected stable graded $r$-spin surfaces (legal connected stable graded $r$-spin surfaces respectively) of genus $g$ with $k$ boundary and $l$ internal marked points. In \cite[Theorem 3.4]{BCT1} and \cite[Theorem 2.8]{TZ1}, $\Mbar_{g,k,l}^{1/r}$ and $\Mbarstar_{g,k,l}^{1/r}$ are shown to be a compact smooth orientable orbifold with corners, of real dimension 
\begin{equation}\label{eq dim of moduli}
\dim_{\mathbb R} \Mbarstar_{g,k,l}^{1/r}=k+2l+3g-3.
\end{equation}

\begin{nn}
    Assuming that the internal marked points $\{z_i\}_{i\in I}$ have twists $\{a_i\}_{i\in I}$, by an abuse of notation, we also denote the multiset\footnote{We use the word ``multiset'' here because the set $I$ may contain multiple $a_i$ with a same value, but we view them as different elements.} $\{a_i\}_{i\in I}$ by $I$.
    
Similarly, we denote by $B$ the multiset $\{b_j\}_{j\in B}$ equipped with a preselected legality for each of its elements. Furthermore, if $B$ is equipped with a cyclic order on $\sigma\colon B\to B$, we denote it by $\bar B$ and write $\bar{B}=\overline{\{b_1,b_2,\dots,b_{\lvert B \rvert}\}}$ to make the cyclic order manifest, where $\sigma_2(b_i)=b_{i+1}$ for $1\le i \le \lvert B \rvert$ and $\sigma_2(b_{\lvert B \rvert})=b_{1}$.
\end{nn}

We denote by $\Mbarstar_{0,B,I}^{1/r}$ the moduli space of graded $r$-spin disks with boundary points marked by $B$, and internal points marked by $I$. Note that for a graded $r$-spin disk, the canonical orientation on $\partial\Sigma$ induces a cyclic order on $B$; given a cyclic order $\sigma_2\colon B \to B$, we denote by $\Mbarstar_{0,\bar{B},I}^{1/r}\subseteq \Mbarstar_{0,{B},I}^{1/r}$ the connected component parametrizing the $r$-spin disks such that the induced cyclic order on $B$ coincides with $\sigma_2$. 
Similarly, We denote by $\Mbarstar_{1,B,I}^{1/r}$ the moduli space of graded $r$-spin cylinders with boundary points marked by $B$, and internal points marked by $I$. Note that for a graded $r$-spin cylinder with two boundary components $\partial_i\Sigma$ where $i=1,2$, the canonical orientation on $\partial_i\Sigma$ induces a cyclic order on $B^i$, where $B^i\subset B$ is the markings corresponding to marked points on $\partial_i\Sigma$; given a partition $B=B^1\sqcup B^2$, and a cyclic order $\sigma_2^i\colon B^i \to B^i$ for each $B^i$, we denote by $\Mbarstar_{1,\{\bar{B}^1,\bar B^2\},I}^{1/r}\subseteq \Mbarstar_{1,{B},I}^{1/r}$ the connected component parametrizing the $r$-spin cylinders such that the induced cyclic order on $B^i$ coincides with $\sigma^i_2$.

\begin{rmk}
   In most parts of this article, we primarily consider moduli of legal graded $r$-spin disks of cylinders. We use the notation with a superscript $*$ only when we want to emphasize that there might be illegal boundary markings. 
\end{rmk}

\subsubsection{Stable graded $r$-spin graphs}
Each connected anchored marked surface $\Sigma$ can be characterized by a decorated dual graph $\Gamma(\Sigma)$ as follow.
\begin{itemize}
\item The set of vertices of $\Gamma(\Sigma)$ is the set of irreducible components of $\Sigma$, which is decomposed into \textit{open} and \textit{closed} vertices  $V = V^O \sqcup V^C$\footnote{Since $V,V^O$ or $V^C$ depend on the graph $\Gamma(\Sigma)$, we should write them as $V(\Gamma(\Sigma)),V^O(\Gamma(\Sigma))$ or $V^C(\Gamma(\Sigma))$ to be precise. We omit $\Gamma(\Sigma)$ from the notation here (and in following items) for simplicity because there is no no ambiguity.} depending on whether the corresponding irreducible component meets $\partial \Sigma$. 
Each vertex has a genus: the genus of the irreducible component.

\item The set of half-edges $H(v)$ emanating from a vertex $v\in V$ is the set of the special points (\textit{i.e.} half-nodes and marked points) on the irreducible component corresponding to $v$. The set $H(v)$ is decomposed into \textit{boundary} and \textit{internal} half-edges $H(v) = H^B(v) \sqcup H^I(v)$ depending on whether the corresponding special point lies on $\partial \Sigma$. We write $H:=\sqcup_{v\in V}H(v)$ and $H = H^B \sqcup H^I$ the correspond decomposition. Two half-edges correspond to an (internal or boundary, separating or non-separating) edge $e$ in the set of edges $$E=E^I\sqcup E^B=E_{sp}\sqcup E_{nsp}=E^I_{sp}\sqcup E^I_{nsp}\sqcup E^B_{sp}\sqcup E^B_{nsp}$$ if their corresponding special points are two half-nodes of the same (internal or boundary, separating or non-separating) node.  We denote by $H^I,H^B,H_{sp},H_{nsp},H^I_{sp},H^I_{nsp},H^B_{sp},H^B_{nsp}$ the corresponding sets of half-edges.

\item
When $\Sigma$ is a cylinder, as mentioned in Remark \ref{rmk root when degenerate}, for each genus-zero $v \in V$, there are one or two special boundary half-edge(s) in $RT(v)\subseteq H^B(v)$ corresponding to the root(s). 

\item The set of \textit{tails} $T$ is the set of all marked points together with the contracted boundary nodes. We write $T^B:=T\cap H^B$ and $T^I:=T\cap H^I$. The set of \textit{contracted boundary tails} $H^{CB}\subseteq T^I$ corresponds to the contracted boundary node and the set $T^{anc}\subseteq T^I\setminus H^{CB}$ corresponds to the anchor.
\item 
The canonical orientation on $\partial\Sigma$ induced a map $\sigma_2\colon H^B(v) \to H^B(v)$ for each $v\in V^O$, which can be viewed as a cyclic order on each of its orbits.
\end{itemize}
We say $\Gamma(\Sigma)$ is \textit{smooth} if $E = H^{CB} = \emptyset$, or equivalently, $\Sigma$ is smooth.
If $\Sigma$ is endowed with a graded $r$-spin structure $S$, we have the additional decorations:
\begin{itemize}
    \item a map $\text{tw}: H \rightarrow \{-1,0,1,\ldots, r-1\}$ encoding the twist of $S$ at each special points;
    \item a map
$\text{alt}: H^B \rightarrow \mathbb{Z}/2\mathbb{Z}$
given by $\text{alt}(h) = 1$ if the special point corresponding to $h$ is legal and $\text{alt}(h) = 0$ otherwise.
\end{itemize}

A \textit{genus-zero (respectively, genus-one) graded $r$-spin graph} is a decorated graph for which each connected component is the dual graph of a connected graded $r$-spin disk (respectively, cylinder) as above; an intrinsic definition can be found in \cite[\S 3.2]{TZ1}. A graph $\Gamma$ is \textit{stable} if $\lvert H^B(v)\rvert +2\lvert H^I(B)(v)\ge 3$ for each genus-zero $v\in V(\Gamma)$ and $\lvert H^B(v)\rvert +2\lvert H^I(B)(v)\ge 1$ for each genus-one $v\in V(\Gamma)$.

For an edge $e$ (or a set of edges $S$) of a stable graded $r$-spin graph $\Gamma=\Gamma(\Sigma)$,  the \textit{smoothing} of $\Gamma$ along $e$ (or $S$) is the stable graded $r$-spin graph $d_e\Gamma$ (or $d_S \Gamma$) that is dual to the stable graded $r$-spin surface obtained by smoothing the node(s) in $\Sigma$ corresponding to $e$ (or $S$). The \textit{detaching} of $\Gamma$ along $e$ (or $S$) is the stable graded $r$-spin graph $\text{Detach}_e\Gamma$ (or $\text{Detach}_S\Gamma$) that is dual to the stable graded $r$-spin surface obtained by normalizing the node(s) in $\Sigma$ corresponding to $e$ (or $S$). See \cite[\S 3.2]{TZ1} for intrinsic definitions. We set
\begin{align*}
&\partial^!\Gamma = \{\Lambda \; \colon \; \Lambda \text{ stable, }\Gamma= d_S\Lambda \text{ for some } S\subset E(\Lambda)\},\\
&\partial \Gamma = \partial^!\Gamma \setminus \{\Gamma\},\\
&\partial^B \Gamma = \{\Lambda \in \partial \Gamma \; \colon \; E^B(\Lambda) \cup H^{CB}(\Lambda) \neq \emptyset\}.
\end{align*}

We say a stable graded $r$-spin graph is \textit{legal} if every boundary tail is legal, \textit{i.e.} ${\text{alt}}(t)=1~\forall t \in T^B$, and of twist $r-2$ (which is referred to as legal level-$0$ graded $r$-spin graph in \cite{BCT2}). 

 If $\Gamma$ is a connected stable graded $r$-spin graph, we denote by ${\mathcal M^*}_{\Gamma}^{1/r}\subseteq {\Mbarstar}_{g,  T^B, T^I\setminus H^{CB},}^{1/r}$ the  submoduli consisting of $r$-spin surfaces whose dual graph is exactly $\Gamma$, and by $\overline{{\mathcal M^*}}_{\Gamma}^{1/r}$ its closure. If $\Gamma$ is disconnected, we define $\Mbarstar_{\Gamma}^{1/r}$ as the product of the moduli spaces associated to its connected components. When there is no room for confusion (which is always the case, except in \S \ref{subsec decomp}), we omit the superscript $1/r$ in $\overline{{\mathcal M^*}}_{\Gamma}^{1/r}$. 
 As the twist and legality of the boundary half-edges (\textit{i.e.} the map $\tw$ and $\alt$) are parts of data of $\Gamma$, we will also omit the superscript $*$ in the notation when there is no ambiguity.
 If $\Gamma$ is single-vertex and unstable, we formally take $\Mbarstar^{1/r}_{\Gamma}$ to be a single point.

 For a stable dual graph $\Gamma$ and a subset $S\subseteq T^I(\Gamma)$ of tails, we denote by $\operatorname{for}_S(\Gamma)$ the dual graph obtained by forgetting all the tails in $S$. When $\Gamma$ is a stable graded $r$-spin graph and all the tails in $S$ are either twist-zero internal tail or twist-zero illegal boundary tail, $\operatorname{for}_S(\Gamma)$ is also a graded $r$-spin graph. Note that $\operatorname{for}_S(\Gamma)$ is not necessary stable, we denote by $\overline{\operatorname{for}}_S(\Gamma)$ the graph obtained by $\operatorname{for}_S(\Gamma)$ after contracting all the unstable vertices; $\overline{\operatorname{for}}_S(\Gamma)$ is either stable, or single-vertex unstable. We denote by 
 $$
 \operatorname{For_{\Gamma \to \overline{\operatorname{for}}_S(\Gamma)}}\colon \Mbar_{\Gamma}\to \Mbar_{\overline{\operatorname{for}}_S(\Gamma)}
 $$
 or 
 $$
 \operatorname{For_{\Gamma \to \overline{\operatorname{for}}_S(\Gamma)}}\colon \Mbarstar^{1/r}_{\Gamma}\to \Mbarstar^{1/r}_{\overline{\operatorname{for}}_S(\Gamma)}
 $$
 the corresponding forgetful morphism.

\subsection{The Witten bundle and the relative cotangent line bundles}
\label{subsec Witten bundle}
\subsubsection{Genus-zero Witten bundle}
The \textit{Witten bundle} is the protagonist of the $r$-spin theory. Roughly speaking, let $\pi: \mathcal{C} \to \Mbarstar_{0,k,l}^{1/r}$ be the universal curve and $\mathcal{S} \to \mathcal{C}$ be the twisted universal spin bundle with the universal Serre-dual bundle 
\begin{equation}\label{eq universal serre dual bundle}
\mathcal{J}:= \mathcal{S}^{\vee} \otimes \omega_{\pi},
\end{equation}
then we define a sheaf
\begin{equation}
\label{eq Witten bundle def}
{\mathcal{W}}:= (R^0\pi_*\mathcal{J})_+,
\end{equation}
where the subscript $+$ denotes invariant sections under the universal involution $\widetilde{\phi}: \mathcal{J} \to \mathcal{J}.$ 
To be more precise, defining~${\mathcal{W}}$ by \eqref{eq Witten bundle def} would require dealing with derived pushforward in the category of orbifold-with-corners. To avoid this technicality, we define ${\mathcal{W}}$ by pullback of the analogous sheaf from a subset of the closed moduli space $\Mbar_{0,k+2l}^{1/r}$; see \cite[\S 4.1]{BCT1}. 

In this $g=0$ case, $\mathcal W$ is in fact a vector bundle, this follows from a direct Riemann--Roch computation showing that 
\begin{equation}\label{eq dimension not jump}
    R^0\pi_*\mathcal{S}=0.
\end{equation}

 On a non-empty component of the moduli space which parametrizes $r$-spin disks with internal twists $\{a_i\}$ and boundary twists $\{b_j\}$, the (real) rank of the Witten bundle is
\begin{equation}\label{eq rank of witten bundle}\frac{2 \sum_{i\in I} a_i + \sum_{j\in B} b_j - (r-2)}{r}.\end{equation}

\subsubsection{Genus-one Witten bundle}
In the $g=1$ case, we still have  the universal curve $\pi: \mathcal{C} \to \Mbarstar_{1,k,l}^{1/r}$ and the twisted universal spin bundle $\mathcal{S} \to \mathcal{C}$   with the universal Serre-dual bundle 
$\mathcal{J}:= \mathcal{S}^{\vee} \otimes \omega_{\pi}$
as in $g=0$ case. However,
\eqref{eq dimension not jump} is not true in this case so $\mathcal W$ is not a vector bundle over the entire moduli. What we do is to remove the ``dimension-jump locus'' $\mathcal Z^{dj}$, \textit{i.e.,} the support of $R^0\pi_*\mathcal S$,  from the $\Mbarstar_{1,k,l}^{1/r}$, and define 
\begin{equation*}
{\mathcal{W}}:= (R^0\pi_*\mathcal{J})_+,
\end{equation*}
as a vector bundle over $\Mbarstar_{1,k,l}^{1/r}\setminus \mathcal Z^{dj}$. 

\begin{dfn}
    For a $r$-spin (nodal) cylinder $\Sigma$, we say $\Sigma$ is \textit{dimension-jump} if there exists $\Sigma'\subseteq \Sigma$ such that
    \begin{itemize}
        \item $\Sigma'$ is a union of irreducible components of $\Sigma$;
        \item $\Sigma'$ is genus-one;
        \item the restriction of $S$ to $\Sigma'$ is a trivial line bundle.
    \end{itemize}
\end{dfn}
\begin{rmk}
    Note that $S\vert_{\Sigma'}$ is trivial implies that all the internal and boundary markings on $\Sigma'$ have twist $0$. Moreover, all the half-nodes on $\Sigma'$ connecting $\Sigma'$ to $\Sigma\setminus\Sigma'$ have twist $0$.
\end{rmk}

\begin{dfn}
    We define the dimension-jump loci inside $\Mbar^{1/r}_{1,k,l}$ to be the 
    \begin{equation*}
         \mathcal Z^{dj}:=\left\{[\Sigma]\in \Mbar^{1/r}_{1,k,l} \colon \Sigma \text{ is dimension-jump}\right\}.
    \end{equation*}

    We write 
    \begin{equation*}
        \oQMb_{1,k,l}:=\Mbar^{1/r}_{1,k,l}\setminus \mathcal Z^{dj};
    \end{equation*}
    for a $r$-spin graph $\Gamma$, we also write 
    \begin{equation*}
        \mathcal Z^{dj}_{\Gamma}:=\mathcal Z^{dj}\cap \Mbar_{\Gamma}\quad\oQMb_\Gamma:=\Mbar_\Gamma\setminus \mathcal Z^{dj}_{\Gamma}.
    \end{equation*}
\end{dfn}

\begin{rmk}\label{rmk dim jump non separating}
    Note that all non-separating boundary nodes for $\Sigma$ in $\mathcal Z^{dj}$ are Ramond nodes (whose half-nodes have twist $r-1$), therefore for any $\Gamma$ with a non-separating NS edge with have $\oQMb_{\Gamma}=\Mbar_{\Gamma}$.     
   Moreover, when $r$ is even, by item \ref{it Ramond boundary node} of Proposition \ref{prop lifting}, Ramond boundary nodes can not appear; therefore, we have $\mathcal Z^{dj}=\emptyset$ when $r$ is even.
\end{rmk}

\begin{ex}
    When $r$ is odd, let $\Delta$ be a smooth $r$-spin graph with a single genus-one vertex, $k$ boundary and $l$ internal markings with twist zero. Then topologically the moduli $\Mbar_\Delta$ is the disjoint union of $r$ copies  $\Mbar_{1,k,l}$ and the dimension-jump loci $\mathcal Z^{dj}\cap\Mbar_\Delta$ is one of these $\Mbar_{1,k,l}$. Therefore in this case $\oQMb_\Delta$ is the disjoint union of $r-1$ copies  $\Mbar_{1,k,l}$.
\end{ex}

We define the Witten  bundle  $\mathcal W$ over $\oQMb_{1,k,l}$ as 
$${\mathcal{W}}:= (R^0\pi_*\mathcal{J})_+.$$
On a non-empty component of the moduli space which parametrizes $r$-spin cylinders with internal twists $\{a_i\}$ and boundary twists $\{b_j\}$, the (real) rank of the Witten bundle is
\begin{equation}\label{eq rank of witten bundle g=1}\frac{2 \sum_{i\in I} a_i + \sum_{j\in B} b_j }{r}.\end{equation}

For genus-zero $\Gamma$, we will formally write $\oQMb_\Gamma:=\Mbar_{\Gamma}$. 
In \cite[\S 3.2]{TZ1}, a canonical relative orientation $o_\Gamma$ of the Witten bundle $\mathcal W_\Gamma \to \oQMb_\Gamma$ is defined for every connected genus-zero or genus-one legal stable graded $r$-spin graph $\Gamma$.

\subsubsection{Decomposition properties of the Witten bundle}\label{subsec decomp}
In \cite{BCT1,TZ1} the Witten bundle is proven to satisfy certain decomposition properties along nodes.  We state these properties here, further details and proofs can be found in \cite{TZ1} and \cite[\S 4.2]{BCT1}.

Given a stable graded $r$-spin graph $\Gamma$ of genus-$g$ for $g=0$ or $1$, let $\widehat{\Gamma}$ be obtained by detaching either an edge or a contracted boundary tail of $\Gamma$.  In order to state the decomposition properties of the Witten bundle, we need the morphisms
\begin{equation}
\label{eq Wittendecompsequence genus one}
\oQMb_{\widehat{\Gamma}}^{1/r} \xleftarrow{q} \Mbar_{\widehat{\Gamma}} \times_{\Mbar_{\Gamma}} \oQMb_{\Gamma}^{1/r} \xrightarrow{\mu} \oQMb_{\Gamma}^{1/r} \xrightarrow{i_{\Gamma}} \overline{\mathcal{QM}^*}_{g,k,l}^{1/r},
\end{equation}
where $\Mbar_{\Gamma} \subseteq \Mbar_{g,k,l}$ is the moduli space of marked surfaces (without $r$-spin structure) corresponding to the dual graph $\Gamma$.  The morphism $q$ is defined by sending the $r$-spin structure $S$ to the $r$-spin structure $\widehat{S}$ defined by \eqref{eq normalize S}; it has degree one but is not an isomorphism because it does not induce an isomorphism between isotropy groups.  The morphism $\mu$ is the projection to the second factor in the fibre product; it is a surjective morphism, and is an isomorphism when $\Gamma$ has no non-separating edges.  The morphism $i_{\Gamma}$ is the inclusion.

We denote by ${\mathcal{W}}$ and $\widehat{{\mathcal{W}}}$ the Witten bundles on $\overline{\mathcal{QM}^*}_{g,k,l}^{1/r}$ and $\oQMb_{\widehat{\Gamma}}^{1/r}$, the decomposition properties below show how these bundles are related under pullback via the morphisms \eqref{eq Wittendecompsequence genus one}.  

\begin{pr}[{\cite[Proposition 2.20]{TZ1}}]
\label{prop decomposition}
Let $\Gamma$ be a genus-zero or genus-one stable graded $r$-spin graph with a single edge $e$, and let $\widehat{\Gamma}$ be the detaching of $\Gamma$ along $e$.  Then the Witten bundle decomposes as follows:
\begin{enumerate}
\item\label{it NS} If $e$ is a Neveu--Schwarz edge, then \begin{equation}\label{eq NSdecompses}
\mu^*i_{\Gamma}^*{\mathcal{W}} = q^*\widehat{{\mathcal{W}}}.
\end{equation}

\item\label{it decompose Ramond boundary edge} If $e$ is a Ramond boundary edge, then there is an exact sequence
\begin{equation}
\label{eq decompose}0 \to \mu^*i_{\Gamma}^*{\mathcal{W}} \to q^*\widehat{{\mathcal{W}}} \to \underline{\mathbb{R}_+} \to 0,
\end{equation}
where $\underline{\mathbb{R}_+}$ is a trivial real line bundle.

\item If $e$ is a Ramond internal edge connecting two closed genus-zero vertices, write $q^*\widehat{\mathcal{W}} = \widehat{\mathcal{W}}_1 \boxplus \widehat{\mathcal{W}}_2$, where $\widehat{\mathcal{W}}_1$ is the Witten bundle on the component containing a contracted boundary tail or the anchor of $\Gamma$, and $\widehat{\mathcal{W}}_2$ is the Witten bundle on the other component.  Then there is an exact sequence
\begin{equation}
\label{eq decompose2}
0 \to \widehat{\mathcal{W}}_2 \to \mu^*i_{\Gamma}^*{\mathcal{W}} \to \widehat{\mathcal{W}}_1 \to 0.
\end{equation}
Furthermore, if $\widehat\Gamma'$ is defined to agree with $\widehat\Gamma$ except that the twist at each Ramond tail is $r-1$, and $q': \Mbar_{\widehat{\Gamma}} \times_{\Mbar_{\Gamma}} \Mbar_{\Gamma}^{1/r} \to \Mbar_{\widehat\Gamma'}^{1/r}$ is defined analogously to $q$, then there is an exact sequence
\begin{equation}
\label{eq decompose3}
0 \to \mu^*i_{\Gamma}^*{\mathcal{W}} \to (q')^*\widehat{{\mathcal{W}}}' \to {\underline{\mathbb{C}}}^{1/r} \to 0,
\end{equation}
where $\widehat{{\mathcal{W}}}'$ is the Witten bundle on $\Mbar_{\widehat\Gamma'}^{1/r}$ and ${\underline{\mathbb{C}}}^{1/r}$ is a line bundle whose $r$-th power is trivial.

\item If $e$ is a separating Ramond internal edge connecting an open vertex to a closed vertex, write $q^*\widehat{\mathcal{W}} = \widehat{\mathcal{W}}_1 \boxplus \widehat{\mathcal{W}}_2$, where $\widehat{\mathcal{W}}_1$ is the Witten bundle on the open component (defined via $\widehat{\mathcal{S}}|_{\mathcal{C}_1}$) and $\widehat{\mathcal{W}}_2$ is the Witten bundle on the closed component.  Then the exact sequences \eqref{eq decompose2} and \eqref{eq decompose3} both hold.
\end{enumerate}
Analogously, for genus-zero $\Gamma$ which has a single vertex, no edges, and a contracted boundary tail~$t$, and $\widehat{\Gamma}$ is the detaching of $\Gamma$ along~$t$, then there is a decomposition property:
\begin{enumerate}
\setcounter{enumi}{4}
\item\label{it decompose cb tail} If ${\mathcal{W}}$ and $\widehat{\mathcal{W}}$ denote the Witten bundles on $\Mbarstar_{0,k,l}^{1/r}$ and $\Mbar_{\widehat\Gamma}^{1/r}$, respectively, then the sequence \eqref{eq decompose} holds.
\end{enumerate}
\end{pr}

\begin{rmk}\label{rmk decompose NS boundary node}
If the edge $e$ is a Neveu--Schwarz boundary edge, then the map $q$ is an isomorphism, and in this case, the proposition implies that the Witten bundle pulls back under the gluing morphism $\mu\circ q^{-1}\colon\overline{\mathcal{QM}^*}_{\widehat\Gamma}^{1/r} \to \overline{\mathcal{QM}^*}_{g,k,l}^{1/r}$.

\end{rmk}

\begin{nn}
    Assuming $\Gamma$ has only NS boundary edges and $\widehat{\Gamma}$ has $t$ connected components $\Delta_1,\dots, \Delta_t$. If we denote by $\mathcal W_{\Delta_i}\to \oQMb_{\Delta_i}$ the corresponding Witten bundles then we have $\widetilde{\mathcal W}=\bboxplus_{i=1}^t \mathcal W_{\Delta_i}$. For a multisection $s_{\Gamma}$ of $\mathcal W_{\Gamma}=i_{\Gamma}^* \mathcal W \to \oQMb_{\Gamma}$ and multisections $s_{\Delta_i}$ of $\mathcal W_{\Delta_i}$, 
    by abuse of notation, we will write 
    $
     s_{\Gamma}=\bboxplus_{i=1}^t s_{\Delta_i}
    $
    as long as
    $
     \mu^*s_{\Gamma}= q^*\left(\bboxplus_{i=1}^t s_{\Delta_i} \right).
    $
\end{nn}

\subsubsection{Coherent sections and the assembling operator}\label{sec ass} 

   With the same notation as in Proposition \ref{prop decomposition}, let $\Gamma$ be a genus-zero or genus-one connected graded $r$-spin graph and $e$ be a separating edge of $\Gamma$, let $\oQMb_{\widehat{\Gamma}_1}$ and $\oQMb_{\widehat{\Gamma}_2}$ be the two components of $\widehat\Gamma:={\text{detach} }_e{\Gamma}$. We write $q^*\widehat{\mathcal{W}} = \widehat{\mathcal{W}}_1 \boxplus \widehat{\mathcal{W}}_2$, where $\widehat{\mathcal{W}}_1$ and $\widehat{\mathcal{W}}_2$ are the (pullback of) Witten bundles on $\oQMb_{\widehat{\Gamma}_1}$ and $\oQMb_{\widehat{\Gamma}_2}$. 
   Given sections $s_1$ and $s_2$ of $\widehat{\mathcal{W}}_1$ and $\widehat{\mathcal{W}}_2$, we want to construct a section of $\mathcal W \to \oQMb_\Gamma$. In the case where $e$ is an NS boundary edge, according to Remark \ref{rmk decompose NS boundary node}, we can glue $s_1,s_2$ 
 to a section of $\mathcal W \to \oQMb_\Gamma$. However, when $e$ is an internal edge, we cannot glue $s_1$ and $s_2$ directly. In the case of an internal NS edge, this is because the automorphism groups of ${\mathcal{W}}$ and of the direct sum are not the same. Ramond internal edges introduce a more fundamental problem, since ${\mathcal{W}}$ does not decompose naturally as a direct sum of the $\widehat{\mathcal{W}}_1,\widehat{\mathcal{W}}_2,$ by Proposition \ref{prop decomposition} above. To construct a section of $\mathcal W \to \overline{\mathcal{QM}}_\Gamma$, we need the \textit{assembling operator} introduced in \cite{BCT2}, which is based on the following technical notion of \textit{coherent multisections}. 
  
  Let $ \Gamma_c$ be a genus-zero connected stable $r$-spin dual graph with an anchor $t\in T^I(\Gamma_c)$, \textit{i.e.} $ \Gamma_c$  has no open vertices or contracted boundary tails.
  \begin{itemize}
      \item If $\tw(t)=-1$, let $\mathcal J\to \Mbar_{ \Gamma_c}$ be the universal Serre-dual bundle as in \eqref{eq universal serre dual bundle}, we set $\mathcal{J}' := \mathcal{J} \otimes \mathcal O \left( r\Delta_{z_t}\right)$, where $\Delta_{z_t}$ is the divisor in the universal curve corresponding to the anchor $t$.
    We define an orbifold bundle ${\overline{\mathcal{R}}}_{\Gamma_c}$ on $\Mbar_{ \Gamma_c}$ by
    ${\overline{\mathcal{R}}}_{ \Gamma_c}:= \sigma_{z_t}^*\mathcal{J}'$,
     where $\sigma_{z_t}$ is the section corresponding to the  anchor $t$ in the universal curve. We also abuse notations to denote by ${\overline{\mathcal{R}}}_{ \Gamma_c}$ the total space of this bundle. We write
$\wp: {\overline{\mathcal{R}}}_{ \Gamma_c} \to \Mbar_{ \Gamma_c}$ for the projection.     We denote by ${\mathcal{W}}$ the  bundle $R^0\pi_*\mathcal{J}'$ on ${\overline{\mathcal{R}}}_{\Gamma_c}$.
\item If $\tw(t)\ne -1$, we set ${\overline{\mathcal{R}}}_{ \Gamma_c}:=\Mbar_{\Gamma_c}$ and set $\mathcal W\to {\overline{\mathcal{R}}}_{ \Gamma_c}$ the same as $\mathcal W\to \Mbar_{\Gamma_c}$.
  \end{itemize}

\begin{definition}[{\cite[Definition 4.2]{BCT2}}]
\label{def coherent}
 Let $ \Gamma_c$ be a connected genus-zero stable $r$-spin dual graph with an anchor $t\in T^I(\Gamma_c)$, and let $s$ be a section of $\mathcal{W}$ over a subset $U\subset{\overline{\mathcal{R}}}_{\Gamma_c}$. We say $s$ is \textit{coherent} if either the twist $\tw(t)\ne -1$, or, in the case where $\tw(t)= -1$, for any point $\zeta=(C, u_t) \in U$  corresponding to a graded $r$-spin disk $C$ and an element $u_t$ in the fibre over $z_t\in C$ of the line bundle $J':=J\otimes \mathcal O(r[z_t])\to C$, the element $s(\zeta) \in H^0(J')$ satisfies
$$\operatorname{ev}_{z_t}s(\zeta) = u_t.$$ A coherent multisection $s$ is defined as a multisection  (see \cite[Appendix A]{BCT2}) whose local branches are coherent. 

\end{definition}

Let $s$ be a coherent multisection of $\mathcal{W} \rightarrow {\overline{\mathcal{R}}}_{\Gamma_c}$. Note that, in the case $\tw(t)=-1$, for any $\zeta\in\Mbar_{\Gamma_c} \hookrightarrow {\overline{\mathcal{R}}}_{\Gamma_c}$, the evaluation $\operatorname{ev}_{z_t}(s(\zeta))$ is equal to zero, thus $s(\zeta)$ is induced by a multisection of $J$.  In other words, the restriction of a coherent multisection $s$ to $\Mbar_{\Gamma_c}$ is canonically identified with a multisection of ${\mathcal{W}} \rightarrow \Mbar_{\Gamma_c}$; we denote this induced multisection by $\overline{s}$.  In case $\tw(t)\ne -1$, we write $\overline{s} = s$.
\begin{lem}[{\cite[Lemma 2.12]{TZ2}}]
    Let $\Gamma_c$ be a connected graded $r$-spin graph with an anchor $t$, for any multisection $\hat s$ of $\mathcal W\to \Mbar_{\Gamma_c}$,
    there exist a coherent multisection $s$ of $\mathcal W\to {\overline{\mathcal{R}}}_{\Gamma_c}$ satisfying $\overline{s}=\hat s$.
\end{lem}

Given a connected graded $r$-spin graph $\Gamma$ with an open vertex or a contracted boundary, let $e$ be a separating internal edge $e\in E_{nsp}^I(\Gamma)$. We denote by $\Gamma_o$ and $\Gamma_c$ the connected components of ${\text{detach} }_e \Gamma$, where $\Gamma_o$ has an open vertex or a contracted boundary tail (hence $\Gamma_c$ has an anchor). Using the assembling operator $\Ass$ defined in \cite[\S 4.1.3]{BCT2}, we can glue a multisection $s_o$ of $\mathcal W^{}\to \oQMb_{\Gamma_o}$ and a coherent multisection $s_c$ of $\mathcal W^{}\to {\overline{\mathcal{R}}}_{\Gamma_c}$ to obtain a multisection $s_o\Ass s_c$ of $\mathcal W^{}\to \oQMb_{\Gamma}$. By construction the multisection $s_o\Ass s_c$ satisfies the following property.

\begin{lem} \label{lem zero of assembling}
    With the above notation, let $\pi_o\colon \oQMb_\Gamma \to \oQMb_{\Gamma_o}$ and $\pi_c\colon \oQMb_\Gamma \to \Mbar_{\Gamma_c}$ be the projections.
    For any multisection $s_o$ of $\mathcal W \to \oQMb_{\Gamma_o}$, and a coherent multisection $s_c$ of $\mathcal W \to \overline{\mathcal R}_{\Gamma_c}$, the multisection $s_o\Ass s_c$ vanishes at $p\in \oQMb_\Gamma$ if and only $\overline{s_c}(\pi_c(p))=0$ and $s_o(\pi_o(p))=0$.
\end{lem}

We refer the reader to \cite[\S 4.1.3]{BCT2} for further details and exact definitions.

\subsubsection{Relative cotangent line bundles}
Other important line bundles in open $r$-spin theory are the \emph{relative cotangent line bundles} or \emph{tautological line bundles} at internal marked points. These line bundles have already been defined on the moduli space $\Mbar_{g,k,l}$ of stable marked surfaces (without spin structure) in \cite{PST14}, as the line bundles with fibre $T^*_{z_i}\Sigma$. Equivalently, these line bundles are the pullback of the usual relative cotangent line bundles ${\mathbb{L}}_i\to\Mbar_{g,k+2l}$ under the doubling map $\Mbar_{g,k,l} \to \Mbar_{g,k+2l}$ that sends $\Sigma$ to $C=\Sigma\sqcup_{\partial\Sigma}\overline{\Sigma}$.  The bundle $\mathbb{L}_i\to\Mbarstar_{g,k,l}^{1/r}$ is the pullback of this relative cotangent line bundle on $\Mbar_{g,k,l}$ under the morphism $\text{For}_{\text{spin}}$ that forgets the spin structure.  Note that $\mathbb{L}_i$ is a complex line bundle, hence it carries a canonical orientation. Further discussion can be found in \S \ref{sec relative  cotangent and all}.

\subsection{Positivity boundary condition for sections of genus-zero Witten bundle}\label{sec positivity constraint}

In the construction of genus-zero $r$-spin theory \cite{BCT2,TZ2}, an important concept is the positivity boundary condition for the sections of the Witten bundle $\mathcal W$.

Let $\Gamma$ be a graded (genus-zero or genus-one) $r$-spin graph, we refer to a boundary edge $e\in E^B(\Gamma),$ or the corresponding node, as \textit{positive} if one half-edge $h_1$ of $e$ satisfies ${\text{alt}}(h_1) = 0$ and $\text{tw}(h_1) >0$.  For a graded $r$-spin graph $\Gamma$ we write $$H^+(\Gamma):=\left\{h\in H^B(\Gamma)\colon \text{either $h$ or $\sigma_1(h)$ is positive}\right\},$$ 
$$\partial^{+} \Gamma:=\{ {\Delta} \in \partial^{!} \Gamma \colon H^+({\Delta})\ne \emptyset \}.$$
$$\partial^{CB} \Gamma:=\{ {\Delta} \in \partial^{!} \Gamma \colon H^{CB}({\Delta})\ne \emptyset \}$$
and
$$
\partial^* \Gamma:= \{ {\Delta} \in \partial^! \Gamma \colon H^I_{nsp}=H^+({\Delta})= \emptyset \}
$$

For a graded $r$-spin graph $\Gamma$, we define
\[ \partial^+\Mbar_\Gamma:=\bigcup_{\Lambda\in\partial^+\Gamma}{\mathcal{M}}_\Lambda, \quad  \oPMb_\Gamma:= \oQMb_\Gamma\setminus
\partial^+\Mbar_\Gamma,\quad \partial^{CB}\Mbar_\Gamma:=\bigcup_{\Lambda\in\partial^{CB}\Gamma}{\mathcal{M}}_\Lambda\]
and similarly $\oPM_{g,B,I}\subseteq \Mbar^{1/r}_{g,B,I}$, $\oPM_{g,k,l}\subseteq \Mbar^{1/r}_{g,k,l}$ as disjoint unions over smooth $\Gamma$.

\begin{definition}
Let $C$ be a graded $r$-spin surface, $q\in C$ a point, and $v\in{\mathcal{W}}_C$. The \textit{evaluation} of $v$ at $q$ is $\text{ev}_q(v) := v(q) \in J_q$. In particular, if $q$ is a contracted boundary node, or a point on $\Sigma$ which is not a legal special point, we say $v$ \textit{evaluates positively} at $q$ if $\text{ev}_q(v)$ is positive with respect to the grading.
\end{definition}

Roughly speaking, for genus-zero $\Gamma$ a positive multisection of the  $\mathcal W\to \Mbar_\Gamma$ is a multisection that satisfies certain positivity constraints at $\partial^+\Mbar_\Gamma\cup \partial^{CB}\Mbar_\Gamma$.  The positivity constraint at a $\partial^{CB}\Mbar_\Gamma$ for a multisection $s$ is chosen to be that the evaluation of $s$ at the contracted boundary node is positive. A na\"ive definition of the positivity constraint at $\partial^+\Mbar_\Gamma$ for a multisection $s$ is that the evaluation of $s$ at the illegal half-node is positive; however, this constraint is too strong and the canonical multisection may not exist (see \cite[Example 3.23]{BCT2}).

As in \cite{BCT2}, instead of imposing 
positivity constraints at $\partial^+\Mbar_\Gamma$, we impose positive constraints at their neighbourhoods (therefore we work on $\oPMb_\Gamma$) and the positive evaluations will be required on certain ``intervals'' in the boundary of disk, which we now recall; they should be viewed as a smoothly-varying family of intervals that approximate  neighbourhoods of boundary nodes in a nodal disk. 

\begin{rmk}
    An alternative, equivalent, way for defining the boundary conditions but working on $\Mbar_{\Gamma}$ rather than $\oPMb_{\Gamma}$ is to allow vanishing of sections of the Witten bundle on the boundary, but requiring a Neumann-like boundary condition on the derivative, following \cite{BCT2}
 we chose to work on $\oPMb_{\Gamma}.$\end{rmk}

\begin{definition}[{\cite[Definition 3.4]{BCT2}}]\label{def positive neighbourhoods}
Let $\Gamma$ be a genus-zero graded $r$-spin graph and let $\Lambda \in \partial\Gamma$.  
We say an open set $U\subseteq \Mbar_\Gamma$ is a \textit{$\Lambda$-set with respect to $\Gamma$} if $U$ does not intersect with the strata ${\mathcal{M}}_\Xi$  for any $\Xi \in\partial^! \Gamma $ satisfying $\Lambda\notin \partial^!\Xi$.
We say a neighbourhood $U\subseteq\Mbar_\Gamma$ of $u\in\Mbar_\Gamma$ is a \textit{$\Lambda$-neighbourhood of $u$ with respect to $\Gamma$} if it is a $\Lambda$-set with respect to $\Gamma$.  Since there is a unique smooth graph $\Gamma$ with $\Lambda\in\partial^!\Gamma$, we refer to a $\Lambda$-neighbourhood with respect to a smooth $\Gamma$ simply as a $\Lambda$-neighbourhood.

Let $\Sigma$ be the preferred half of a graded marked disk. We write $\partial \Sigma = S/\sim$ for a space $S$ homeomorphic to $S^1$ and denote by $q\colon S \to \partial \Sigma$ the quotient map. An \emph{interval} $I$ is the image of a connected open set of $S$ under $q$. Note that the preimage of $I$ under the quotient map $q$ is the union of an open set with a finite number of isolated points.

Suppose $C\in{\mathcal{M}}_\Lambda$, and let $n_h$ be a boundary half-node corresponding to a half-edge $h\in H^B(\Lambda)$. We denote by $N: \widehat{C} \to C$ the normalization map and write $\widehat{\Sigma} = N^{-1}(\Sigma)$, We say that $n_h$ \textit{belongs} to the interval $I$ if the following two conditions hold:
\begin{itemize}
    \item the corresponding node $N(n_{h})$ lies in $I$;
    \item  $N^{-1}(I)$ contains a half-open interval with starting point $n_h$ and a half-open interval with endpoint $n_{\sigma_1(h)}$, where the starting and end points are determined by the canonical orientation of $\partial\widehat{\Sigma}$.
\end{itemize}

Let $U$ be a $\Lambda$-set. We denote by $\pi:\lvert\mathcal C\rvert\to U$ the coarse universal curve and by $\Sigma_u \subseteq \pi^{-1}(u)$ the preferred half in the fibre. A \textit{$\Lambda$-family of intervals} $\{I_{h}(u)\}_{h\in{H^+}(\Lambda),u\in U}$ for $U$ is a choice of an interval $I_{h}(u)$ for each $h\in{H^+}(\Lambda)$ and $u\in U$, such that:
\begin{enumerate}
\item
The endpoints of each $I_h(u)$ vary smoothly with respect to the smooth structure of the universal curve restricted to $U$.
\end{enumerate}
We say that $\Xi$ is \textit{a smoothing of $\Lambda$ away from $h\in H^+(\Lambda)$} if $\Lambda\in\partial^!\Xi$ and $h$ lies in the image of the injection $\iota:H^+(\Xi)\to H^+(\Lambda)$. We further require that:
\begin{enumerate}
\setcounter{enumi}{1}
\item
If $h \notin\{ h', \sigma_1(h')\}$, we have $I_h(u')\cap I_{h'}(u')=\emptyset$.  If $h=\sigma_1(h')$, we have $I_h(u')\cap I_{h'}(u')\neq\emptyset$ if and only if $u'\in{\mathcal M}_\Xi$ for some smoothing $\Xi$ of $\Lambda$ away from $h$.
  In this case, $I_h(u')\cap I_{\sigma_1(h)}(u')$ consists exactly of the node $N(n_{h})$.
\item There are no marked points in $I_h(u')$. the interval $I_h(u')$ contains at most one half-node; it contains one half-node if and only if $u'\in{\mathcal M}_\Xi$ for some smoothing $\Xi$ of $\Lambda$ away from $h$. The half-node that belongs to $I_h(u')$ is $n_{h}$ in this case.
\end{enumerate}

If the moduli point $u$ represents the stable graded disk $C$ we write $I_h(C)$ for $I_h(u).$ Figure \ref{fig intervals} shows the local picture of the intervals in the nodal and smooth disks.
\end{definition}

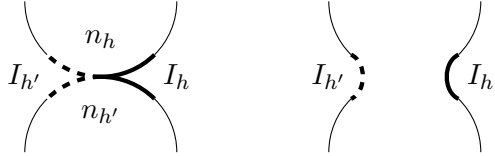
\begin{figure}[h]
\centering

\begin{tikzpicture}[scale=1]

\draw (1,1) arc (0:-45:1);
\draw[ultra thick] (0,0) arc (-90:-45:1);
\draw[ultra thick, dashed] (0,0) arc (-90:-135:1);
\draw (-1,1) arc (-180:-135:1);

\draw (1,-1) arc (0:45:1);
\draw[ultra thick] (0,0) arc (90:45:1);
\draw[ultra thick, dashed] (0,0) arc (90:135:1);
\draw (-1,-1) arc (180:135:1);

\node at (0,-0.5) {$n_{h'}$};
\node at (0,0.5) {$n_{h}$};
\node at (1,0) {$I_h$};
\node at (-1,0) {$I_{h'}$};

\draw (5,1) arc (0:-45:1);
\draw[ultra thick] (4.707,0.293) .. controls (4.5,0.2) and (4.5,-0.2) .. (4.707,-0.293);
\draw (3,1) arc (-180:-135:1);
\draw (5,-1) arc (0:45:1);
\draw[ultra thick,dashed] (3.293,0.293) .. controls (3.5,0.2) and (3.5,-0.2) .. (3.293,-0.293);
\draw (3,-1) arc (180:135:1);

\node at (5,0) {$I_h$};
\node at (3,0) {$I_{h'}$};

\end{tikzpicture}

\caption{The thicker lines representing the intervals $I_h$ and $I_{h'}$ associated to the half-nodes $n_h$ and $n_{h'}$ are drawn over the thinner boundary lines. The image on the right represents a point in the moduli space that is close to the image on the left, where the node is smoothed.}

\label{fig intervals}
\end{figure}

\begin{definition}\label{def positive section}
Let $C$ be a graded $r$-spin disk, and let $A \subseteq \partial\Sigma$ be a subset without legal special points.  Then an element $w\in{\mathcal{W}}_\Sigma$ \emph{evaluates positively at~$A$} if $\ev_x(w)$ is positive for every $x\in A$ with respect to the grading.

Let $\Gamma$ be a genus-zero graded graded graph, $U$ a $\Lambda$-set with respect to $\Gamma$, and $\{I_h\}$ a $\Lambda$-family of intervals for $U$.  Given a multisection $s$ of ${\mathcal{W}}$ defined in a subset of $\oPMb_\Gamma$ containing $U\cap\oPMb_\Gamma$, we say $s$ is \emph{$(U,I)$-positive} (with respect to $\Gamma$) if for any $u'\in U\cap \oPMb_\Gamma$, any local branch $s_i(u')$ evaluates positively at each $I_h(u')$.

We say a multisection $s$ defined in $W\cap\oPMb_\Gamma$, where $W$ is a neighbourhood of $u\in\Mbar_\Lambda$, is \emph{positive near} $u$ (with respect to $\Gamma$) if there exists a $\Lambda$-neighbourhood $U\subseteq W$ of $u$ and a $\Lambda$-family of intervals $I_*(-)$ for $U$ such that $s$ is $(U,I)$-positive.

If $W$ is a neighbourhood of $\partial^+\Mbar_\Gamma$, then a multisection $s$ defined in a set
\begin{equation}\label{eq U_+}
U_{+,\Gamma}=\left(W\cup \partial^{CB}\Mbar_{\Gamma}\right)\cap\oPMb_\Gamma
\end{equation}
is \emph{positive} (with respect to $\Gamma$) if it is positive near each point of $\partial^+\Mbar_\Gamma$ and evaluates positively at the contracted boundary nodes.  As above, we omit the phrase ``with respect to $\Gamma$'' if $\Gamma$ is smooth.
\end{definition}

For a genus-zero $r$-spin graph $\Gamma$, there always exist positive multisections of $\mathcal W_{\Gamma}\to \oPMb_{\Gamma}$ according to \cite[Proposition 3.20]{BCT2}. However, for genus-one $\Gamma$, although we can define similar positivity condition for multisections of Witten bundle, the positive multisection might not exist.

\subsection{The point insertion technique}\label{subsec PI}
Just like the closed $r$-spin theory considers an intersection theory over the moduli spaces of $r$-spin curves,  the (genus-zero or genus-one) open $r$-spin theory considers the intersection theory over the moduli of the $r$-spin surfaces $\Mbar^{1/r}_{g,B,I}$. However, since $\Mbar^{1/r}_{g,B,I}$ is a orbifold with corners, the intersection theory is not well-defined. The grading structure allows us to deal with certain type of boundaries using a notion of positivity (see \S \ref{sec positivity constraint}). The procedure of \textit{point insertion}, which also relies on the grading, was developed in \cite{TZ1} in order to treat the remaining boundaries: we can glue another moduli spaces  to $\Mbar^{1/r}_{g,B,I}$ along those strata and by that cancel them. 
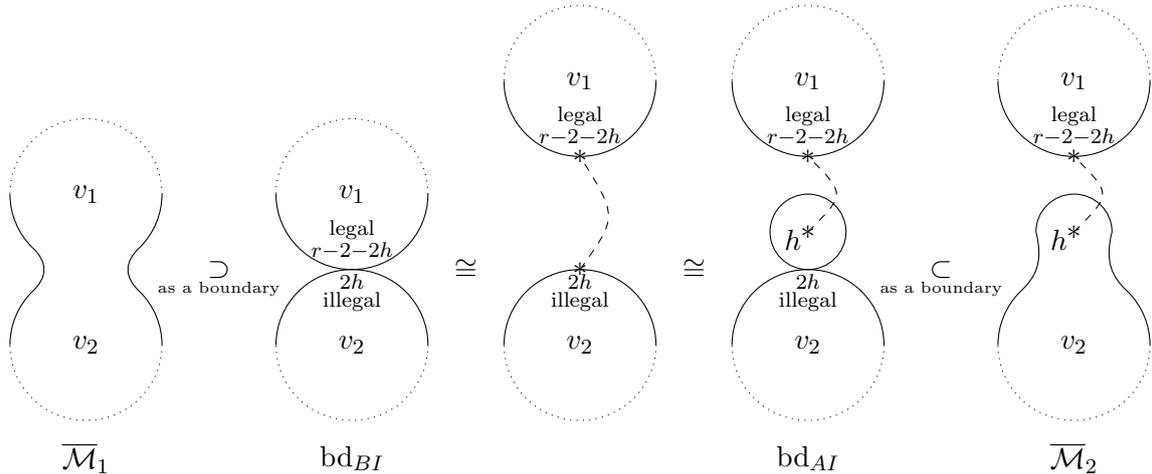
\begin{figure}[h]
    \centering
    \begin{tikzpicture}
        \draw (0,0) arc (0:180:1);
        \draw[dotted] (0,0) arc (0:-180:1);
        \node at (-1,1) {$*$};
        \node at (-1,0.7) {$\substack{2h\\ \text{illegal}}$};
        \node at (-1,0) {$v_2$};
                
        \draw[dotted] (0,3.5) arc (0:180:1);
        \draw (0,3.5) arc (0:-180:1);
        \node at (-1,2.5) {$*$};
        \node at (-1,2.9) {$\substack{\text{legal}\\r-2-2h}$};
        \node at (-1,3.5) {$v_1$};

        \draw[dashed] (-1,1) .. controls (-0.5,1.7) .. (-1,2.5);

        \draw (-3,0) arc (0:180:1);
        \draw[dotted] (-3,0) arc (0:-180:1);
        
        \node at (-4,0.7) {$\substack{2h\\ \text{illegal}}$};
        \node at (-4,0) {$v_2$};
        \node at (-4,-1.5) {$\text{bd}_{BI}$};

        \draw[dotted] (-3,2) arc (0:180:1);
        \draw (-3,2) arc (0:-180:1);
        
        \node at (-4,1.4) {$\substack{\text{legal}\\r-2-2h}$};
        \node at (-4,2) {$v_1$};

        \node at (-2.5,1) {$\cong$};

        \draw (3,0) arc (0:180:1);
        \draw[dotted] (3,0) arc (0:-180:1);
        \node at (2,1.5) {$*$};
        \node at (1.8,1.4) {$h$};
        \node at (2,0.7) {$\substack{2h\\ \text{illegal}}$};
        \node at (2,0) {$v_2$};
        \node at (2,-1.5) {$\text{bd}_{AI}$};
                
        \draw[dotted] (3,3.5) arc (0:180:1);
        \draw (3,3.5) arc (0:-180:1);
        \node at (2,2.5) {$*$};
        \node at (2,2.9) {$\substack{\text{legal}\\r-2-2h}$};
        \node at (2,3.5) {$v_1$};
        
        \draw (2.5,1.5) arc (0:360:0.5);
        \draw[dashed] (2,1.5) .. controls (2.5,2) .. (2,2.5);

        \node at (0.5,1) {$\cong$};

         \draw (-6.5,0) arc (0:45:1);
         \draw (-8.5,0) arc (180:135:1);
        \draw[dotted] (-6.5,0) arc (0:-180:1);
        
        \node at (-7.5,0) {$v_2$};
        \node at (-7.5,-1.5) {$\Mbar_1$};

        \draw[dotted] (-6.5,2) arc (0:180:1);
        \draw (-6.5,2) arc (0:-45:1);
        \draw (-8.5,2) arc (-180:-135:1);
        
        \node at (-7.5,2) {$v_1$};
        \draw (-6.793,1.293) .. controls (-7,1.1) and (-7,0.9).. (-6.793,0.707);
        \draw (-8.207,1.293) .. controls (-8,1.1) and (-8,0.9).. (-8.207,0.707);

        \node at (-5.75,1) {$\supset$};
        \node at (-5.75,0.75) {\tiny{as a boundary}};

        \draw (6.5,0) arc (0:45:1);
        \draw (4.5,0) arc (180:135:1);
        \draw[dotted] (6.5,0) arc (0:-180:1);
        \node at (5.5,1.5) {$*$};
        \node at (5.3,1.4) {$h$};
        
        \node at (5.5,0) {$v_2$};
        \node at (5.5,-1.5) {$\Mbar_2$};
                
        \draw[dotted] (6.5,3.5) arc (0:180:1);
        \draw (6.5,3.5) arc (0:-180:1);
        \node at (5.5,2.5) {$*$};
        \node at (5.5,2.9) {$\substack{\text{legal}\\r-2-2h}$};
        \node at (5.5,3.5) {$v_1$};
        
        \draw (6,1.5) arc (0:180:0.5);
        
        \draw[dashed] (5.5,1.5) .. controls (6,2) .. (5.5,2.5);
        
        \draw (6.207,0.707) .. controls (6,0.9) and (5.9,1.1).. (6,1.5);
        \draw (4.793,0.707) .. controls (5,0.9) and (5.1,1.1).. (5,1.5);

        \node at (3.75,1) {$\subset$};
        \node at (3.75,0.75) {\tiny{as a boundary}};
        
    \end{tikzpicture}
    \caption{In point insertion procedure we glue $\Mbar_1$ and $\Mbar_2$ together along their isomorphic boundaries $\text{bd}_{BI}$ and $\text{bd}_{AI}$. The first isomorphism follows from the decomposition property for boundary NS nodes; the second isomorphism holds because the moduli $\Mbar^{1/r}_{0.\{r-2-2h\},\{h\}}$ (represented by the smallest bubble in the figure) is a single point. The new markings coming from the point insertion procedure is represented by $*$; the dashed line between the new markings indicates that they come from the same node.} 
    \label{fig point insertion demonstration}
\end{figure}

More precisely, as shown in Figure \ref{fig point insertion demonstration}, let $\Mbar_1$ be a moduli of $r$-spin surfaces, and $\text{bd}_{BI}\subset \Mbar_1$ be a boundary corresponding to an NS boundary node with twist $2h$ at the illegal half-node. We can glue to $\Mbar_1$, along the boundary $\text{bd}_{BI}$, another moduli $\Mbar_2$ which has a boundary $\text{bd}_{AI}$ diffeomorphic to $\text{bd}_{BI}$. Note that $\Mbar_2$ is a moduli of possibly disconnected $r$-spin surfaces, obtained by first detaching the boundary node, then ``inserting'' the illegal twist-$2h$ boundary marked point to the interior as a twist-$h$ internal marked point. The boundary strata along which we glue the moduli spaces are called \emph{spurious boundaries}.

By applying this procedure repeatedly we get a glued moduli (see \cite[\S 4.6]{TZ1}). A point in the pre-glued space represents a disjoint union of graded $r$-spin surfaces, together with the combinatorial data of dashed lines connecting each pair of boundary marking and internal marking which appear together in the point insertion procedure. A point in the glued space represents an equivalence class of such objects under the equivalence relation induced by point insertion procedure.

In this paper, we will only do the ``$\h=0$ point insertion'' in \cite{TZ1}, \textit{i.e.}, we will  do point insertion at an NS boundary node $n$ if and only if the twist of the illegal half-node of $n$ is zero, \textit{i.e.}, in the case $h=0$ in Figure \ref{fig point insertion demonstration}.

\subsubsection{$(r,0)$-surfaces, $(r,0)$-graphs and moduli}

We now more formally describe the objects of the pre-glued moduli space.
\begin{dfn}[{\cite[Definition 4.3]{TZ1}}]\label{dfn rh disk}
    An $(r,0)$-surface is a collection of legal connected  stable graded $r$-spin surfaces (the components) such that $\text{Fix}(\phi)\ne \emptyset$ and all boundary twists are $r-2$, together with
    \begin{enumerate}
        \item a bijection (denoted by dashed lines) between a subset $B^{PI}$ of the boundary tails with twist $r-2$ and a subset $I^{PI}$ of the internal tails with twist $0$ which are not contracted boundary nodes;
        \item markings on the set of unpaired boundary tails $B^{up}$ and internal tails $I^{up}$ which are not contracted boundary nodes, \textit{i.e.} identifications $I^{up}=\{1,2,\dots,\lvert I^{up}\rvert\}$ and $B^{up}=\{1,2,\dots,\lvert B^{up}\rvert\}$.
    \end{enumerate}
    We require that, in the collection of surfaces, there exists no genus-zero stable graded $r$-spin disk with only two tails, where both of them are in $I^{PI}\sqcup B^{PI}$.
    \begin{itemize}
        \item  We say an $(r,0)$-surface is an $(r,0)$-disk (genus-zero) if all the connect surfaces in the collection are genus-zero, and the graph $\hat{\mathbf{G}}$, whose vertices are connected disks in the collection and there is an edge between two vertices if they contain a pair of points corresponding to a dashed line, is a connected genus-zero graph.
        \item  We say an $(r,0)$-surface is an $(r,0)$-cylinder (genus-one) if one of the following two condition holds. 
        \begin{itemize}
            \item All the connect surfaces in the collection are genus-zero, and the graph $\hat{\mathbf{G}}$ is a connected genus-one graph.
            \item One of the connect surfaces in the collection is genus-one while all of the other connect surfaces in the collection are genus-zero, and the graph $\hat{\mathbf{G}}$ is a connected genus-zero graph.
        \end{itemize}
    \end{itemize}

\end{dfn}

\begin{rmk}
    The $0$ in $(r,0)$ refers to $\h=0$ in the notation of \cite{TZ1,TZ2}.
\end{rmk}

The graph $\hat{\mathbf{G}}$ mentioned above characterizes the topological type of an $(r,0)$-surface. Since each graded $r$-spin surface in the collection of an $(r,0)$-surface is associated with a graded $r$-spin graph, by assigning each vertex of $\hat{\mathbf{G}}$ the corresponding graded $r$-spin graph, and specifying the corresponding pair of tails for each edge of $\hat{\mathbf{G}}$, we obtain the following combinatorial object, which can be viewed as a refined version of $\hat{\mathbf{G}}$.
\begin{dfn}[{\cite[Definition 4.7]{TZ1}}]\label{def rh graphs}
A \textit{$(r,0)$-graph}  $\mathbf{G}$ consists of 
\begin{itemize} 
    \item a set $V(\mathbf{G})$ of connected legal stable graded $r$-spin graphs with at least one open vertex or contracted boundary tail, and the boundary tails of them all have twist $r-2$,;
    \item two partitions of sets
    $$
    \bigsqcup_{\Gamma\in V(\mathbf{G})}\left(T^I(\Gamma)\backslash H^{CB}(\Gamma)\right)=I(\mathbf{G})\sqcup I^{PI}(\mathbf{G})
    $$
    and
    $$
    \bigsqcup_{\Gamma\in V(\mathbf{G})}T^B(\Gamma)=B(\mathbf{G})\sqcup B^{PI}(\mathbf{G})
    $$
    such that all $a\in I^{PI}(\bm{G})$ have twist $0$;
    \item a set of edges (the \textit{dashed lines}) $$E(\mathbf{G})\subseteq \{(a,b)\colon a\in I^{PI}(\mathbf{G}),b\in B^{PI}(\mathbf{G})\}$$ which induces an one-to-one correspondence $\delta$ between $I^{PI}(\mathbf{G})$ and $B^{PI}(\mathbf{G})$;
    \item a labelling of the set $I(\mathbf{G})$ by $\{1,2,\dots,l(\mathbf{G}):=\lvert I(\mathbf{G})\rvert\}$ and a labelling of the set $B(\mathbf{G})$ by $\{1,2,\dots,k(\mathbf{G}):=\lvert B(\mathbf{G})\rvert\}$.
\end{itemize}
    We require that 
    \begin{enumerate}
        \item there exists no $\Gamma\in V(\mathbf{G})$ satisfying 
    $
    H^I(\Gamma)\subseteq I^{PI}(\mathbf{G}), H^B(\Gamma)\subseteq B^{PI}(\mathbf{G})$ 
    and $
         \lvert H^I(\Gamma)\rvert+\lvert H^B(\Gamma)\rvert \le 2;
    $
    \end{enumerate}
    We define an auxiliary graph (in the normal sense) $\hat{\mathbf{G}}$ in the following way: the set of vertices of $\hat{\mathbf{G}}$ is $V(\mathbf{G})$, the set of edges of $\hat{\mathbf{G}}$ is $E(\mathbf{G})$; an element $(a,b)\in E(\mathbf{G})$ corresponds  to an edge between the vertices $\Gamma_a$ and $\Gamma_b$, where  $a\in T^I(\Gamma_a)$ and $b\in T^B(\Gamma_b)$. We also require that
\begin{enumerate}[resume*]
    \item  the graph $\hat{\mathbf{G}}$ is connected.
\end{enumerate}

The genus of an $(r,0)$-graph $\mathbf{G}$ is defined to be 
$$
g(\mathbf{G})=\sum_{\Gamma\in V(\bm{G})}g(\Gamma)+g(\hat{\mathbf{G}}),
$$
where $g(\Gamma)$ is the genus of graded $r$-spin graph $\Gamma\in V(\bm{G})$, and $g(\hat{\mathbf{G}})$ is the genus of the auxiliary graph $\hat{\mathbf{G}}$.

\end{dfn}

In this paper, we exclusively focus on the genus-zero and genus-one case.

\begin{dfn}
For a genus-one $(r,0)$-graph $\mathbf{G}$, we define a subset
$$I^{PI}_{sp}(\mathbf{G})\subseteq I^{PI}(\mathbf{G}) \quad\text{  (respectively } B^{PI}_{sp}(\mathbf{G})\subseteq B^{PI}(\mathbf{G})\text{)}$$ consisting all the elements $x$ in $I^{PI}(\mathbf{G})$ (respectively $B^{PI}(\mathbf{G})$) such that the graph $\operatorname{Detach}_{(x,\delta(x))}\hat{\mathbf{G}}$ (obtained by removing the edge in $\hat{\mathbf{G}}$ corresponding to the pair $(x,\delta(x))$) is non-connected. These tails are generated in point procedures for separating boundary nodes. 
We also define 
$$I^{PI}_{nsp}(\mathbf{G}):=I^{PI}(\mathbf{G})\setminus I^{PI}_{sp}(\mathbf{G}) \quad \text{  (respectively } B^{PI}_{nsp}(\mathbf{G}):=B^{PI}(\mathbf{G})\setminus B^{PI}_{sp}(\mathbf{G})\text{)},$$
which are tails generated in point insertion procedures for non-separating boundary nodes.

Moreover, we define a subset
$$I^{PI}_{sp,1}(\mathbf{G})\subseteq I^{PI}_{sp}(\mathbf{G}) \quad\text{  (respectively } B^{PI}_{sp,1}(\mathbf{G})\subseteq B^{PI}_{sp}(\mathbf{G})\text{)}$$ consisting all the elements $x$ in $I^{PI}_{sp}(\mathbf{G})$ (respectively $B^{PI}_{sp}(\mathbf{G})$) such either $\hat{\mathbf{G}}_{x}$, the connected component of $\operatorname{Detach}_{(x,\delta(x))}\hat{\mathbf{G}}$ containing $x$, is genus-one, or one of  the vertex of $\hat{\mathbf{G}}_{x}$ is a genus-one $r$-spin graph. We also define
$$I^{PI}_{sp,0}(\mathbf{G}):=I^{PI}_{sp}(\mathbf{G})\setminus I^{PI}_{sp,1}(\mathbf{G}) \quad \text{  (respectively } B^{PI}_{sp,0}(\mathbf{G}):=B^{PI}_{sp}(\mathbf{G})\setminus B^{PI}_{sp,1}(\mathbf{G})\text{)}.$$
For any $\Gamma\in V(\bm{G})$, we define the set of roots on $\Gamma$ to be
\begin{equation}\label{eq root for vertices of PI graph}
RT(\Gamma):=T(\Gamma)\cap(B^{PI}_{sp,0}(\mathbf{G})\sqcup I^{PI}_{sp,0}(\mathbf{G})\sqcup I^{PI}_{nsp}(\mathbf{G})\sqcup B^{PI}_{nsp}(\mathbf{G})).
\end{equation}
\end{dfn}

\begin{rmk}\label{rmk root for vertex of PI graph}
    For any genus-one $(r,0)$-graph $\mathbf{G}$ and $\Gamma\in V(\mathbf{G})$, if $\Gamma$ is a genus-one $r$-spin graph, then we always have $RT(\Gamma)=\emptyset$. For genus-zero $\Gamma$, there are two possibilities:
    \begin{enumerate}
        \item $RT(\Gamma)\subset I^{PI}_{nsp}(\mathbf{G})\sqcup B^{PI}_{nsp}(\mathbf{G})$ is a two-element set;
        \item $RT(\Gamma)\subset I^{PI}_{sp,0}(\mathbf{G})\sqcup B^{PI}_{sp,0}(\mathbf{G})$ is an one-element set.
    \end{enumerate}
    Therefore, for any genus-zero vertex $\Gamma$ of $\mathbf{G}$, we can regard it as a rooted or 2-rooted $r$-spin graph.
\end{rmk}

\begin{dfn}
Let $\mathbf{G}$ be an $(r,0)$-graph. Let $e$ be an edge or a contracted boundary tail of some $\Gamma\in V(\mathbf{G})$. Since $T^I(\Gamma)\backslash H^{CB}(\Gamma)=T^I(d_e \Gamma)\backslash H^{CB}(d_e \Gamma)$ and $T^B(\Gamma)=T^B(d_e \Gamma)$, we define the \textit{smoothing} of $\mathbf{G}$ along $e$ to be the $(r,0)$-graph $d_e \mathbf{G}$ obtained by replacing $\Gamma$ with $d_e \Gamma$.

We say $\mathbf{G}$ is smooth if all $\Gamma\in V(\mathbf{G})$ are smooth stable graded $r$-spin graphs.
We denote by $\GPI^{r,0}_g$ the set of all genus-$g$ $(r,0)$-graphs, by $\GPI^{r,0}_{g,B,I}$ the set of all genus-$g$ $(r,0)$-graphs $\mathbf{G}$ satisfying $I(\mathbf{G})=I$ and $B(\mathbf{G})=B$, by $\sGPI^{r,0}_{g,B,I}$ the set 
$$ 
\sGPI^{r,0}_{g,B,I}:=\{\mathbf{G}\in \GPI^{r,0}_{g,B,I}\colon \mathbf{G}\text{ smooth}\}.
$$
\end{dfn}

\begin{dfn}
    An isomorphism between two $(r,0)$-graphs $\mathbf{G_1}$ and $\mathbf{G_2}$ consists of a collection of isomorphism of stable graded $r$-spin graphs between elements of $V(\mathbf{G_1})$ and $V(\mathbf{G_2})$, which induces a bijection between $V(\mathbf{G_1})$ and $V(\mathbf{G_2})$, and preserves the partitions, dashed lines, and labellings.
\end{dfn}

For each $\mathbf{G}\in \GPI^{r,0}_g$, let $\operatorname{Aut} \mathbf{G}$ be the group of automorphisms of $\mathbf{G}\in \GPI^{r,0}_g$, then there is a natural action of $\operatorname{Aut} \mathbf{G}$ over the product $\prod_{\Gamma\in V(\mathbf{G})} \Mbar_\Gamma$. We define 
$$
\Mbar_\mathbf{G}:=\left(\prod_{\Gamma\in V(\mathbf{G})} \Mbar_\Gamma\right)\bigg\slash \operatorname{Aut} \mathbf{G}
$$
and (when $\mathbf{G}$ is genus zero or one)
$$
\oQMb_\mathbf{G}:=\left(\prod_{\Gamma\in V(\mathbf{G})} \oQMb_\Gamma\right)\bigg\slash \operatorname{Aut} \mathbf{G}
$$
and
$$
\oPMb_\mathbf{G}:=\left(\prod_{\Gamma\in V(\mathbf{G})} \oPMb_\Gamma\right)\bigg\slash \operatorname{Aut} \mathbf{G}.
$$
Notice that for genus-zero $\mathbf{G}$ we always have  $\Mbar_\mathbf{G}=\oQMb_\mathbf{G}$
\subsubsection{Witten bundle over $\Mbar_\mathbf{G}$}
Let $\mathcal W_\Gamma$ be the Witten bundle over $\oQMb_\Gamma$, we define the Witten bundle $\mathcal W_\mathbf{G}$ over $\oQMb_\mathbf{G}$ to be 
$$
 \mathcal W_\mathbf{G}:= \left(\bboxplus_{\Gamma\in V(\mathbf{G})}\mathcal W_\Gamma\right)\bigg\slash \operatorname{Aut} \mathbf{G}.
$$
\begin{rmk}\label{rmk aut trivial}
    In this paper, when we only consider genus-zero or genus-one  $(r,\h)$-graphs $\mathbf{G}$. In these two cases the automorphism groups $\operatorname{Aut}\mathbf{G}$ are always trivial. We denote by $\pi_{\mathbf{G}\to\Gamma}$ the projection maps 
    $$
\pi_{\mathbf{G}\to\Gamma}\colon \Mbar_{\mathbf{G}}\to \Mbar_\Gamma. 
$$
\end{rmk}
Let $o_\Gamma$ be the canonical relative orientation of $\mathcal W_\Gamma$ over $\oQMb_\Gamma$, we define the canonical relative orientation $o_\mathbf{G}$ of $\mathcal W_\mathbf{G}$ over $\oQMb_\mathbf{G}$ by
\begin{equation}\label{eq orientation point insertion}
    o_\mathbf{G}:=(-1)^{\lvert E(\mathbf{G})\rvert}\bboxtimes_{\Gamma\in V(\mathbf{G})}  o_\Gamma.
\end{equation}
Observe that $o_\mathbf{G}$ is independent of the order of the wedge product, since for each $\Gamma\in V(\mathbf{G})$ we have 
$$
\dim \Mbar_\Gamma \equiv \operatorname{rank} \mathcal W_\Gamma \mod 2.
$$
\begin{dfn} Given a finite set $I$ of internal markings with twist in $\{0,1,\dots,r-1\}$ and a finite set $B$ of boundary markings with twist $r-2$, we define the moduli space $\Mbar^{\frac{1}{r},0}_{g,B,I}$ of $(r,0)$-surfaces with markings $B,I$ to be
\begin{equation}\label{eq def moduli point insertion}
    \Mbar^{\frac{1}{r},0}_{g,B,I}:=\bigsqcup_{\mathbf{G}\in \sGPI^{r,0}_{0,B,I}}\Mbar_\mathbf{G}.
\end{equation}
For $g=0,1$, we also define the subspaces
\begin{equation}\label{eq def Q moduli point insertion}
    \oQMb^{\frac{1}{r},0}_{g,B,I}:=\bigsqcup_{\mathbf{G}\in \sGPI^{r,0}_{0,B,I}}\!\! \!\! \oQMb_\mathbf{G} \, \, \,\subseteq  \, \,\Mbar^{\frac{1}{r},0}_{g,B,I}
\end{equation}
and 
\begin{equation}\label{eq def P moduli point insertion}
    \oPMb^{\frac{1}{r},0}_{g,B,I}:=\bigsqcup_{\mathbf{G}\in \sGPI^{r,0}_{0,B,I}}\!\! \!\! \oPMb_\mathbf{G} \, \, \,\subseteq  \, \,\oQMb^{\frac{1}{r},0}_{g,B,I}.
\end{equation}

The Witten bundles with relative orientations over the connected components $\oQMb_\mathbf{G}$ of $\oQMb^{\frac{1}{r},0}_{g,B,I}$ induce the Witten bundle $\mathcal W^{\frac{1}{r},0}_{g,B,I}$ over $\oQMb^{\frac{1}{r},0}_{g,B,I}$ with relative orientation.
\end{dfn}
\subsubsection{Relative cotangent line bundles over $\Mbar_\mathbf{G}$}\label{sec uni cotangent line over glued moduli}
For any $i\in I$, there are three different notions of relative cotangent line bundles over $\Mbar^{\frac{1}{r},0}_{g,B,I}$.
\begin{enumerate}
    \item 
    For any $\bm{G} \in \sGPI^{r,0}_{g,B,I}$, let $\Gamma_{i,\bm{G}}\in V(\bm{G})$ be the unique vertex of $\bm{G}$ such that $i\in I \cap T^I(\Gamma_{i,\bm{G}})$, we denote by $\mathbb L'_i\to \Mbar_{\bm{G}}$ the line bundle pulled back from $\mathbb L_i\to \Mbar_{\Gamma_{i,\bm{G}}}$ via the projection $\pi_{\bm{G}\to \Gamma_{i,\bm{G}}}\colon \Mbar_{\bm{G}}\to \Mbar_{\Gamma_{i,\bm{G}}}$, and by $\mathbb L'_i\to \Mbar^{\frac{1}{r},0}_{g,B,I}$ the line bundle defined in this way on all connected components.
    \item With the above notation, we denote by $\mathcal B \Gamma_{i,\bm{G}}$ the graph obtained by forgetting all the inserted internal tails of $\Gamma_{i,\bm{G}}$ in $I^{PI}(\bm{G})\cap T^I(\Gamma_{i,\bm{G}})$ and by $\operatorname{For}_{\Gamma_{i,\bm{G}}\to \mathcal B\Gamma_{i,\bm{G}}}\colon \Mbar_{\Gamma_{i,\bm{G}}}\to \Mbar_{\mathcal B \Gamma_{i,\bm{G}}}$ the forgetful morphism (if $\mathcal B \Gamma_{i,\bm{G}}$ is unstable we formally take $\Mbar_{\mathcal B \Gamma_{i,\bm{G}}}$ to be a point). We denote by $\mathbb L_i\to \Mbar_{\bm{G}}$ the line bundle pulled back from $\mathbb L_i\to \Mbar_{\mathcal B\Gamma_{i,\bm{G}}}$ (if $\mathcal B \Gamma_{i,\bm{G}}$ is unstable we formally take it to be a trivial complex line) via 
    $$\operatorname{For}_{\Gamma_{i,\bm{G}}\to \mathcal B\Gamma_{i,\bm{G}}}\circ \pi_{\bm{G}\to \Gamma_{i,\bm{G}}} \colon \Mbar_{\bm{G}} \to \Mbar_{\mathcal B\Gamma_{i,\bm{G}}}$$ 
    and by $\mathbb L_i\to \Mbar^{\frac{1}{r},0}_{g,B,I}$ the line bundle defined in this way on all connected components.

    \item 
    In the case $g=1$, with the above notation, we denote by $\hat{\mathcal B} \Gamma_{i,\bm{G}}$ the graph obtained by forgetting all the inserted internal tails of $\Gamma_{i,\bm{G}}$ in $I^{PI}_{sp,1}(\bm{G})\cap T^I(\Gamma_{i,\bm{G}})$ and by $\operatorname{For}_{\Gamma_{i,\bm{G}}\to \hat{\mathcal B} \Gamma_{i,\bm{G}}}\colon \Mbar_{\Gamma_{i,\bm{G}}}\to \Mbar_{\hat{\mathcal B} \Gamma_{i,\bm{G}}}$ the forgetful morphism (if $\hat{\mathcal B} \Gamma_{i,\bm{G}}$ is unstable, again we formally take $\Mbar_{\hat{\mathcal B} \Gamma_{i,\bm{G}}}$ to be a point). We denote by $\mathbb L^*_i\to \Mbar_{\bm{G}}$ the line bundle pulled back from $\mathbb L_i\to \Mbar_{\hat{\mathcal B} \Gamma_{i,\bm{G}}}$ via 
    $$\operatorname{For}_{\Gamma_{i,\bm{G}}\to \hat{\mathcal B} \Gamma_{i,\bm{G}}}\circ \pi_{\bm{G}\to \Gamma_{i,\bm{G}}} \colon \Mbar_{\bm{G}} \to \Mbar_{\hat{\mathcal B} \Gamma_{i,\bm{G}}}$$ 
    and by $\mathbb L^*_i\to \Mbar^{\frac{1}{r},0}_{g,B,I}$ the line bundle defined in this way on all connected components.
    
    \end{enumerate}

 The relation between these relative cotangent line bundles will be discussed is \S \ref{sec cotangent lines relation}.

\begin{rmk}\label{rmk different notation for L in TZ}
    In \cite{TZ1,TZ2}, the notation ``$\mathbb L_i\to \Mbar^{\frac{1}{r},0}_{0,B,I}$'' refers to what we denote by ``$\mathbb L'_i\to \Mbar^{\frac{1}{r},0}_{0,B,I}$ '' in this paper. We make this change of notation because in this paper we mainly work with the line bundles $\mathbb L_i\to \Mbar^{\frac{1}{r},0}_{0,B,I}$ which `forget' all the inserted points in $I^{PI}$. 
\end{rmk}

\subsubsection{Boundary strata and point insertion}
For an $(r,0)$-graph $\bm{G}$, 
we write
$$
E(\bm{G}):=\bigsqcup_{\Gamma\in V(\bm{G})} E(\Gamma)
$$
and 
$$
H^{CB}(\bm{G}):= \bigsqcup_{\Gamma\in V(\bm{G})}  H^{CB}(\Gamma).
$$
For a set 
$S\subseteq E(\bm{G})\sqcup H^{CB}(\bm{G})$, 
we can perform a sequence of smoothings (in any order, since they lead to the same result) and obtain an $(r,0)$-graph $d_S \bm{G}$.  We set
\begin{align*}
&\partial^!\bm{G} = \{\bm{H} \; | \; \bm{G} = d_S\bm{H} \text{ for some } S\},\\
&\partial \bm{G} = \partial^!\bm{G} \setminus \{\bm{G}\},\\
&\partial^B \bm{G} = \{\bm{H} \in \partial \bm{G} \; | \; E^B(\bm{H}) \cup H^{CB}(\bm{H}) \neq \emptyset\}.
\end{align*}
For an $(r,0)$-graph $\mathbf{G}$, a boundary stratum of $\Mbar_\mathbf{G}$ corresponds to a graph in $\partial^!\mathbf{G}$, or more precisely, a choice of $\Delta_i\in \partial^! \Gamma_{i}$ for each $\Gamma_i\in V(\mathbf{G})$. In particular, a codimension-1 boundary of $\Mbar_\mathbf{G}$ for smooth $\mathbf{G}$ is determined by a choice of $\Gamma\in V(\mathbf{G})$ and a graph $\Delta \in \partial \Gamma$, where $\Delta$ has either one contracted boundary tail and no edges, or exactly one edge which is a boundary edge. There are five different types of codimension-1 boundaries of $\Mbar_\mathbf{G}$ depending on the type of the (half-)edge of $\Delta$ (or equivalently, the corresponding node of a surface $C\in \mathcal M_\Delta$):

\begin{enumerate}
    \item[{CB}] contracted boundary tails;
    \item[{R}] Ramond boundary edges;
    \item[{NS+}] NS boundary edges whose twist on the illegal side is greater than $0$;
    \item[AI] NS boundary edges whose twist on the illegal side is $0$, and the vertex containing the legal half-node only contains this half-edge and an internal tail $a\in T^I(\Delta)\cap I^{PI}(\mathbf{G})$;
    \item[BI] the remaining NS boundary edges whose twist on the illegal side is $0$.
\end{enumerate}
Therefore, the codimension-1 boundary of $\Mbar^{\frac{1}{r},0}_{g,B,I}$ is a union of five different types of boundaries. Similarly, the codimension-1 boundary of $\oQMb^{\frac{1}{r},0}_{g,B,I}$ is also a union of such five different types of boundaries since $\partial \oQMb^{\frac{1}{r},0}_{g,B,I}=\oQMb^{\frac{1}{r},0}_{g,B,I}\cap  \partial \Mbar^{\frac{1}{r},0}_{g,B,I}$, while $\oPMb^{\frac{1}{r},0}_{g,B,I}$ only has three types of boundaries since all type-NS+ and type-R boundaries are removed from $\oPMb^{\frac{1}{r},0}_{g,B,I}$. 

We will denote by $\partial^{CB} \Mbar_{\bm{G}}$, $\partial^{R} \Mbar_{\bm{G}}$, $\partial^{NS+} \Mbar_{\bm{G}}$, $\partial^{AI} \Mbar_{\bm{G}}$ and  $\partial^{BI} \Mbar_{\bm{G}}$ the corresponding boundaries in $\Mbar_{\bm{G}}$. We also write $\mathcal Z^{dj}_{\bm{G}}:=\Mbar_{\bm{G}}\setminus \oQMb_{\bm{G}}$.
\begin{rmk}
    The abbreviation ``BI'' stands for ``before-insertion'', while the abbreviation ``AI'' stands for ``after-insertion''. 
\end{rmk}

We claim that there is an one-to-one correspondence between the type-AI boundaries and the type-BI boundaries.

    \begin{thm}[{\cite[Theorem 4.12]{TZ1}}]\label{thm  PI boundaries paried}
 For fixed $I$ and $B$, there is a one-to-one correspondence $\PI$ between the type-BI boundaries and the type-AI boundaries of $\Mbar^{\frac{1}{r},0}_{g,B,I}$ (respectively $\oQMb^{\frac{1}{r},0}_{g,B,I}$ or $\oPMb^{\frac{1}{r},0}_{g,B,I}$). Two boundaries paired by the correspondence $\PI$ are canonically diffeomorphic, and this diffeomorphism can be lifted to the Witten bundles (when $g=1$ we restrict to $\oQMb^{\frac{1}{r},0}_{1,B,I}$) and the relative cotangent line bundles restricted to them. Moreover, the canonical relative orientations of the Witten bundles on the paired (spurious) boundaries induced by the canonical relative orientations are opposite to each other. 
 \end{thm}

 \begin{rmk}\label{rmk non injective in PI g>0}
     When $g=0$, the diffeomorphism a type-BI boundary and a type-AI boundary extend to a diffeomorphism between their closure. When $g\ge 1$, we can extend such diffeomorphism to a surjective morphism from the closure of type-AI boundary to the closure of type-BI boundary, which is not necessary injective on the higher-codimensional boundary strata. Such surjective morphism still pulls back Witten bundles and relative cotangent line bundles on the closures.
 \end{rmk}

 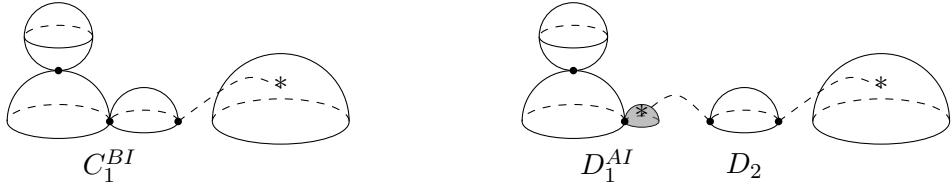
\begin{figure}[h]
         \centering

         \begin{subfigure}{.45\textwidth}
  \centering
\begin{tikzpicture}[scale=0.45]
\draw (0,-0.5) circle (1);
\draw (-1,-0.5) arc (180:360:1 and 0.333);
\draw[dashed](1,-0.5) arc (0:180:1 and 0.333);

\draw (-1.5,-3) arc (180:360:1.5 and 0.5);
\draw[dashed](1.5,-3) arc (0:180:1.5 and 0.5);
\draw(1.5,-3) arc (0:180:1.5);

\node at (0,-1.5) [circle,fill,inner sep=1pt]{};

\draw (1.5,-3) arc (180:360:1 and 0.333);
\draw[dashed](3.5,-3) arc (0:180:1 and 0.333);
\draw(3.5,-3) arc (0:180:1);

\node at (1.5,-3) [circle,fill,inner sep=1pt]{};

\draw (4.5,-3) arc (180:360:2 and 0.667);
\draw[dashed](8.5,-3) arc (0:180:2 and 0.667);
\draw(8.5,-3) arc (0:180:2);

\node at (3.5,-3) [circle,fill,inner sep=1pt]{};
\node at (6.5,-2){*};
\draw[dashed] (3.5,-3) .. controls (5.5,-1.5) ..(6.5,-2);

\node at (1.5,-4.3){$C_1^{BI}$};

\end{tikzpicture} 
\end{subfigure}
\begin{subfigure}{.45\textwidth}
  \centering
\begin{tikzpicture}[scale=0.45]
\draw (0,-0.5) circle (1);
\draw (-1,-0.5) arc (180:360:1 and 0.333);
\draw[dashed](1,-0.5) arc (0:180:1 and 0.333);

\draw (-1.5,-3) arc (180:360:1.5 and 0.5);
\draw[dashed](1.5,-3) arc (0:180:1.5 and 0.5);
\draw(1.5,-3) arc (0:180:1.5);

\node at (0,-1.5) [circle,fill,inner sep=1pt]{};

\draw (1.5,-3) arc (180:360:0.5 and 0.167);
\draw[dashed](2.5,-3) arc (0:180:0.5 and 0.167);
\draw(2.5,-3) arc (0:180:0.5);

\fill[color = gray, opacity = 0.5] (1.5,-3) arc (180:360:0.5 and 0.167) arc (0:180:0.5);

\node at (1.5,-3) [circle,fill,inner sep=1pt]{};

\draw (4,-3) arc (180:360:1 and 0.333);
\draw[dashed](6,-3) arc (0:180:1 and 0.333);
\draw(6,-3) arc (0:180:1);

\draw (7,-3) arc (180:360:2 and 0.667);
\draw[dashed](11,-3) arc (0:180:2 and 0.667);
\draw(11,-3) arc (0:180:2);

\node at (6,-3) [circle,fill,inner sep=1pt]{};
\node at (9,-2){*};
\draw[dashed] (6,-3) .. controls (8,-1.5) ..(9,-2);

\node at (4,-3) [circle,fill,inner sep=1pt]{};

\node at (2,-2.8){*};

\draw[dashed] (4,-3) .. controls (3,-2) ..(2,-2.8);

\node at (1,-4.3){$ D_1^{AI}$};
\node at (5,-4.3){$ D_2$};

\end{tikzpicture} 
\end{subfigure}

        \caption{An example of two $(r,0)$-disks lying on two boundaries $\text{bd}_{BI}$ and $\text{bd}_{AI}$ paired by $PI$. The component $C_1^{BI}$ of the $(r,0)$-disk on the left has a type-BI node, while the component $D_1^{AI}$ of the $(r,0)$-disk on the left has a type-AI node. The shaded irreducible component only contains a legal half-node (corresponding to a type-AI node) and an internal tail in $I^{PI}$. }
        \label{fig rh surface}
    \end{figure}

    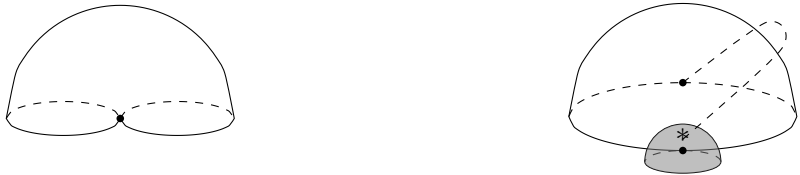
\begin{figure}
        \centering
         
         \begin{subfigure}{.45\textwidth}
  \centering
        \begin{tikzpicture}[scale=0.5]
    \vspace{0.15cm}

     \draw [] plot [domain=0:3,  smooth]({\x}, {-sqrt(2.25-(\x-1.5)*(\x-1.5))*0.3}) -- plot [domain=3:-3,  smooth ] ({\x}, {sqrt(9-\x*\x)})-- plot [domain=-3:0,  smooth]({\x}, {-sqrt(2.25-(\x+1.5)*(\x+1.5))*0.3});

    \draw  [dashed, domain=0:3,  smooth] plot ({\x}, {sqrt(2.25-(\x-1.5)*(\x-1.5))*0.3});

    \draw  [dashed, domain=-3:0,  smooth] plot ({\x}, {sqrt(2.25-(\x+1.5)*(\x+1.5))*0.3});

    \node at (0,0) [circle,fill,inner sep=1pt]{};
  
\end{tikzpicture}
\vspace{0.5cm}
        \end{subfigure}
        \begin{subfigure}{.45\textwidth}
  \centering
        \begin{tikzpicture}[scale=0.5]
    \vspace{0.15cm}

     \draw [] plot [domain=0:3,  smooth]({\x}, {-sqrt(9-(\x)*(\x))*0.3}) -- plot [domain=3:-3,  smooth ] ({\x}, {sqrt(9-\x*\x)})-- plot [domain=-3:0,  smooth]({\x}, {-sqrt(9-(\x)*(\x))*0.3});

    \draw  [dashed, domain=0:3,  smooth] plot ({\x}, {sqrt(9-(\x)*(\x))*0.3});

    \draw  [dashed, domain=-3:0,  smooth] plot ({\x}, {sqrt(9-(\x)*(\x))*0.3});

    \node at (0,0.9) [circle,fill,inner sep=1pt]{};

    \draw(1,-1.2) arc (0:180:1);
     \draw(-1,-1.2) arc (180:360:1 and 0.3);
     
     \draw[dashed](1,-1.2) arc (0:180:1 and 0.3);

    \fill[color = gray, opacity = 0.5] (-1,-1.2) arc (180:360:1 and 0.3) arc (0:180:1);

    \node at (0,-0.9) [circle,fill,inner sep=1pt]{};

    \node at (0,-0.6) {*};
    \draw [dashed] plot [smooth] coordinates {(0,-0.6) (2.6,1.8) (2.2,2.5) (0,0.9)};
  
\end{tikzpicture}
        \end{subfigure}
        \caption{An example of two $(r,0)$-cylinders paired by $\PI$, where the left one has a non-separating BI-type boundary node, and the right one has an separating AI-type boundary node, and a dashed line connecting a boundary tail and a internal tail on the same connected component.}
        \label{fig point insertion at non separationg}
    \end{figure}

  Let $\sim_{PI}$ be the equivalent relation induced by the correspondence $\PI$ on the boundaries of $\Mbar^{\frac{1}{r},0}_{g,B,I}$ (and its restriction to the boundaries of $\oQMb^{\frac{1}{r},0}_{g,B,I}$ or $\oPMb^{\frac{1}{r},0}_{g,B,I}$), where $p_{AI}$ in the closure of an AI-type boundary is equivalent to $p_{BI}$ in the closure of the corresponding BI-type boundary if $p_{BI}$ is the image of $p_{AI}$ under the surjective morphism in Remark \ref{rmk non injective in PI g>0}. Theorem \ref{thm  PI boundaries paried} shows that we can glue $\Mbar^{\frac{1}{r},0}_{g,B,I}$, $\oQMb^{\frac{1}{r},0}_{g,B,I}\subseteq \Mbar^{\frac{1}{r},0}_{g,B,I}$ and $\oPMb^{\frac{1}{r},0}_{g,B,I}\subseteq \oQMb^{\frac{1}{r},0}_{g,B,I}$ along the paired boundaries and obtain piecewise smooth glued moduli spaces
    $$
    \widetilde{\mathcal M}^{\frac{1}{r},0}_{g,B,I}:=\Mbar^{\frac{1}{r},0}_{g,B,I}\big/\sim_{PI},
    $$
    
    $$
    \widetilde{\mathcal{QM}}^{\frac{1}{r},0}_{g,B,I}:=\overline{\mathcal{QM}}^{\frac{1}{r},0}_{g,B,I}\big/\sim_{PI} \,\, \subseteq \widetilde{\mathcal M}^{\frac{1}{r},0}_{g,B,I}
    $$
    and
    $$
     \widetilde{\mathcal{PM}}^{\frac{1}{r},0}_{g,B,I}:=\overline{\mathcal{QM}}^{\frac{1}{r},0}_{g,B,I}\big/\sim_{PI} \,\,\subseteq \widetilde{\mathcal{QM}}^{\frac{1}{r},0}_{g,B,I}.
    $$
    The objects parametrized by $\widetilde{\mathcal M}^{\frac{1}{r},0}_{g,B,I}$  are called \textit{reduced $(r,0)$-surfaces} in \cite{TZ1}, they are equivalence classes of $(r,0)$-surfaces under the relation induced by $\sim_{PI}$. Note that $\widetilde{\mathcal M}^{\frac{1}{r},0}_{g,B,I}$ and $\widetilde{\mathcal {QM}}^{\frac{1}{r},0}_{g,B,I}$ have only boundaries of type CB, R, and NS+, while $\widetilde{\mathcal {PM}}^{\frac{1}{r},0}_{g,B,I}$ has only boundaries of type CB.

    The Witten bundles and the relative cotangent line bundles $\mathbb L_i$, $\mathbb L'_i$ or $\mathbb L_i^*$ over the different connected components of $\Mbar^{\frac{1}{r},0}_{g,B,I}$  can also be glued along the same boundaries into a glued Witten bundle $\widetilde{\mathcal W}\to\widetilde{\mathcal {QM}}^{\frac{1}{r},0}_{g,B,I}$ and glued relative cotangent line bundles $\widetilde{\mathbb L}_i\to\widetilde{\mathcal M}^{\frac{1}{r},0}_{g,B,I}$, $\widetilde{\mathbb L}'_i\to\widetilde{\mathcal M}^{\frac{1}{r},0}_{g,B,I}$ or $\widetilde{\mathbb L}^*_i\to\widetilde{\mathcal M}^{\frac{1}{r},0}_{g,B,I}$ with canonical complex structure.
\begin{rmk}
In the case $r=2$ and only NS insertions the Witten bundle is a trivial zero rank bundle. In this case the idea of gluing different moduli spaces to obtain an orbifold without boundary is due to Jake Solomon and the first named author \cite{ST_unpublished}. 
\end{rmk}
     By Theorem \ref{thm  PI boundaries paried} and the fact that the glued relative cotangent line bundles over $\widetilde{\mathcal M}^{\frac{1}{r},\h}_{g,B,I}$ carry canonical complex orientations, we obtain the following theorem:
    \begin{thm}\cite[Theorem 4.13]{TZ1}
      All bundles of the form
    \[\widetilde{{\mathcal{W}}}\oplus\bigoplus_{i=1}^l\widetilde{\mathbb{L}}_i^{\oplus d_i}\to \widetilde{\mathcal {QM}}^{\frac{1}{r},\h}_{g,B,I} 
    \text{ or } 
    \widetilde{{\mathcal{W}}}\oplus\bigoplus_{i=1}^l {\widetilde{\mathbb {L}'}}_i^{\oplus d_i}\to \widetilde{\mathcal {QM}}^{\frac{1}{r},\h}_{g,B,I}\text{ or } 
    \widetilde{{\mathcal{W}}}\oplus\bigoplus_{i=1}^l {\widetilde{\mathbb {L}^*_i}}^{\oplus d_i}\to \widetilde{\mathcal {QM}}^{\frac{1}{r},\h}_{g,B,I}\]are canonically relatively oriented.   
    \end{thm}

\subsection{Closed extended FJRW theories}\label{subsec:closed_ext_r_m}
In \cite{BCT_Closed_Extended,BCT2} it was observed that a certain extension of Witten's $r$-spin theory plays an important role in the open $r$-spin theory. This \emph{closed extended $r$-spin theory} is defined just like the usual genus $0$ closed $r$-spin theory, only that we allow one marking to have a negative twist of $-1.$ It was first observed in \cite{JKV2} that, as in the usual closed setting with only non-negative twists, also here $R^1\pi_*\mathcal{S}$ is an orbifold vector bundle in genus zero, so Witten's class $c_W$ can be defined by as $c_W=c_{top}((R^1\pi_*\mathcal{S})^\vee)$. The \textit{closed extended $r$-spin correlators} are
\begin{equation}\label{eq:define closed ext}
\left\langle\prod_{i\in I}\tau^{a_i}_{d_i}\right\rangle^{\frac{1}{r},\text{ext}}_0:=r\int_{\Mbar^{1/r}_{0,I}} \hspace{-0cm} c_W \cap \psi_1^{d_1} \cdots \psi_n^{d_n}.
\end{equation}
 All these numbers are calculated in \cite{BCT_Closed_Extended}. They satisfy the following relations.
 \begin{thm}[{\cite[Lemma 3.6]{BCT_Closed_Extended}}] \label{thm trr closed extended}  
 For any $i_1\in I$ and $j_1\ne j_2$ in $I\setminus \{i_1\}$, we have the Topological Recursion Relation with respect to the $(i_1,\{j_1,j_2\})$:
\begin{equation}\label{eq trr closed extended}
    \left\langle\tau^{a_{i_1}}_{d_{i_1}+1}\prod_{i\in I\setminus \{i_1\}}\tau^{a_i}_{d_i}\right\rangle^{\frac{1}{r},\text{ext}}_0=
    \sum_{\substack{R_1\sqcup R_2=I\setminus\{i_1\}\\j_1,j_2\in R_2\\-1 \le a \le r-2}} 
    \left\langle\tau^a_0 \tau^{a_{i_1}}_{d_{i_1}} \prod_{i\in R_1}\tau^{a_i}_{d_i}\right\rangle^{\frac{1}{r},\text{ext}}_0 
    \left\langle\tau^{r-2-a}_0\prod_{i\in R_2}\tau^{a_i}_{d_i}\right\rangle^{\frac{1}{r},\text{ext}}_0.
\end{equation}
     
 \end{thm}

\section{Relative cotangent line bundles and TRR multisections}\label{sec relative  cotangent and all}
The relative cotangent line bundles plays an important role in our construction of genus-one open $r$-spin theory.

\subsection{Relation between different relative cotangent line bundles}\label{sec cotangent lines relation}
Let $\Gamma$ be a dual graph, let $\pi_{\Gamma}\colon\mathcal C_{\Gamma}\to \Mbar_{\Gamma}$ be the universal curve. We denote by $\mathcal U_{\Gamma}\subseteq \mathcal C_{\Gamma}$ be the smooth locus, and by $\mathbb L_{\Gamma}\to \mathcal U_{\Gamma}$ the vertical cotangent line bundle, \textit{i.e.,} the cokernel of $d\pi^*_{\Gamma}\colon \pi^*_{\Gamma} T^*\Mbar_{\Gamma}\to T^* \mathcal U_{\Gamma}$. Then for $i\in T^I(\Gamma)$, denote by $\mu_{i,\Gamma}\colon \Mbar_{\Gamma}\to \mathcal U_{\Gamma}$ the section corresponding the internal marking marked by $i$, the relative cotangent line bundle at $i$ can be defined as $\mathbb L_{i,\Gamma}:=\mu_{i,\Gamma}^*\mathbb L_{\Gamma} \to \Mbar_{\Gamma}$. We will omit $\Gamma$ in the subscript and write $\mathbb L_{i,\Gamma}$ as $\mathbb L_{i}$ when there is no ambiguity.

Let $S^I\subset T^I(\Gamma)$ and $S^B\subset T^B(\Gamma)$ be two subsets of internal and boundary tails. Assuming $\Gamma$ is stable after forgetting all the tails in $S^I$ and $S^B$, we denote by $\for_{{S^I,S^B}}(\Gamma)$ the graph after forgetting, by $\For_{\Gamma \to \for_{{S^I,S^B}}(\Gamma)}\colon \Mbar_\Gamma\to \Mbar_{\for_{{S^I,S^B}}(\Gamma)}$ the forgetful morphism, and by $\widetilde{\For}_{\Gamma \to \for_{{S^I,S^B}}(\Gamma)}\colon \mathcal U_{\Gamma} \to \mathcal U_{\for_{{S^I,S^B}}(\Gamma)}$ the lift of $\For_{\Gamma \to \for_{{S^I,S^B}}(\Gamma)}$ to the universal curve. According to \cite[\S3.5]{PST14}, the morphism 
$$d\widetilde{\For}^*_{\Gamma \to \for_{{S^I,S^B}}(\Gamma)}\colon \widetilde{\For}_{\Gamma \to \for_{{S^I,S^B}}(\Gamma)}^*T^*\mathcal U_{\for_{{S^I,S^B}}(\Gamma)}\to T^*\mathcal U_{\Gamma}$$ induces a morphism between the vertical cotangent lines $$\mathfrak t_{\for_{{S^I,S^B}}(\Gamma)\to \Gamma}\colon \widetilde{\For}_{\Gamma\to \for_{{S^I,S^B}}(\Gamma)}^*\mathbb L_{\for_{{S^I,S^B}}(\Gamma)}\to \mathbb L_{\Gamma}$$
on $U_{\Gamma}$.
The morphism $\mathfrak t_{\for_{{S^I,S^B}}(\Gamma)\to\Gamma}$ is an isomorphism except on component of $\mathcal U_{\Gamma}$ that are contracted by $\widetilde{\For}_{\Gamma \to \for_{{S^I,S^B}}(\Gamma)}$, where it vanishes identically. 

For $i\in T^I(\Gamma)\setminus S^I$, notice that the image of $\mu_{i,\Gamma}\colon \Mbar_{\Gamma}\to \mathcal C_\Gamma$ always lies in the smooth locus $\mathcal U_{\Gamma}$, the morphism  $\mathfrak t_{\for_{{S^I,S^B}}(\Gamma)\to \Gamma}$ induce a morphism between the line bundles $$\tilde{\mathfrak t}_{i,\for_{{S^I,S^B}}(\Gamma)\to \Gamma}:=  \mu_{i,\Gamma}^*\mathfrak t_{\for_{{S^I,S^B}}(\Gamma)\to \Gamma}\colon \For_{\Gamma\to \for_{{S^I,S^B}}(\Gamma)}^*\mathbb L_{i,\for_{{S^I,S^B}}(\Gamma)}\to \mathbb L_{i,\Gamma}$$ over $\Mbar_\Gamma$ which vanishes at the locus $D_{i,S^I,\Gamma}\subset \Mbar_\Gamma$ where $i$-th internal marking lies on a component contracted by $\widetilde{\For}_{\Gamma\to \for_{{S^I,S^B}}(\Gamma)}$. So either $D_{i,S^I,\Gamma}= \Mbar_\Gamma$ and $\tilde{\mathfrak t}_{i,\for_{{S^I,S^B}}(\Gamma)\to \Gamma}$ is identically zero, or $D_{i,S^I,\Gamma}$ can be written as finite union of normal-crossing codimension-two subspaces as
$$
D_{i,S^I,\Gamma}=\bigcup_{\Delta \in \mathcal G_{i,S^I,\Gamma}} \Mbar_{\Delta} \subseteq\Mbar_{\Gamma},
$$
where $\mathcal G_{i,S^I,\Gamma}\subset \partial \Gamma$ consists of graphs such that $i$-th marking on a closed genus-zero vertex $v$, and there exist an separating internal edge $e\in E^I_{sp}(\Delta)$ such that $d_e\Delta=\Gamma$, where the connected component $\Delta_v$ of $\operatorname{Detach}_e \Delta$ containing $v$ is closed genus-zero, and  all the tails in $T^I(\Delta_v)$ are $i$, the half-edge from $e$, or elements in $S^I$. 

\begin{rmk}\label{rmk forget boundry point}
In the case $S^B=\emptyset$,  $\tilde{\mathfrak t}_{i,\Gamma \to \for_{{S^I,S^B}}(\Gamma)}$ is an isomorphism.
\end{rmk}
 We assume that $D_{i,S^I,\Gamma}\ne \Mbar_{\Gamma}$, in this case, the vanishing locus $D_{i,S^I}$ behave like a `divisor', \textit{i.e.,} it is a finite union of normal-crossing (real) codimension-two subspaces. Moreover, if we regard $\tilde{\mathfrak t}_{i,\for_{{S^I,S^B}}(\Gamma)\to \Gamma}$ as a section of $\mathcal O(D_{i,S^I,\Gamma}):=\For_{\Gamma \to \for_{{S^I,S^B}}(\Gamma)}^* \mathbb L_{i,\for_{{S^I,S^B}}(\Gamma)}^\vee \otimes \mathbb L_{i,\Gamma}$, we have the following lemma.
 
 \begin{lem}\label{lem zero locus of t between L}
Let $\Gamma$ be a smooth dual graph of genus 0 or 1 such that $\Mbar_{\Gamma}=\Mbar_{0,\overline{B},I}$ (or $\Mbar_{\Gamma}=\Mbar_{1,\{\overline{B}_1,\overline{B}_2\},I}$). The section $\tilde{\mathfrak t}_{i,\for_{{S^I,S^B}}(\Gamma)\to \Gamma}$ vanishes transversely at the codimension-two subspaces ${\mathcal M}_{\Delta} \subseteq\Mbar_{\Gamma}$ for each $\Delta \in \mathcal G_{i,S^I,\Gamma}$. 

Moreover, each $\Delta \in \mathcal G_{i,S^I,\Gamma}$ consists of one closed vertex $v^c$ and an open vertex with internal half-edges labelled by $I^o$. For any choice of $x \in \bar B$ (or $x,y$ in different $\bar B_1$ and  $\bar B_2$), the induced orientation by $\tilde{\mathfrak o}^{x}_{0,\bar{B},I}$ (or $\tilde{\mathfrak o}^{x,y}_{1,(\overline{B}_1,\overline{B}_2),I}$) and $\tilde{\mathfrak t}_{\for_{{S^I,S^B}}(\Gamma)\to \Gamma}$ on ${\mathcal M}_{\Delta}$ coincide with $\tilde{\mathfrak{o}}^x_{0,\bar B,I^o}\boxtimes \tilde{\mathfrak{o}}_{v^c}$ (or $\tilde{\mathfrak{o}}^{x,y}_{1,(\bar B_1,\bar B_2),I^o}\boxtimes \tilde{\mathfrak{o}}_{v^c}$), where $\tilde{\mathfrak{o}}$ are the orientations of moduli spaces of disks or cylinder constructed in \cite[\S 3.1]{TZ1}.
 \end{lem}
\begin{proof}
 For a subset $T\in I$, we denote by $\Delta^{T}\in \partial {\Gamma}$ the graph with two vertices $v^c$ and $v^c$ connected by a separating edge $e$, such that $H^I(v^c)=T\sqcup \{h_e\}$, where $h_e$ is the half-edge corresponding to $e$ on $v^c$ side.

We assume $S^B=\emptyset$. In the case $S^I=\{j\}$ consists of a single element, $D_{i,\{j\},\Gamma}=\Mbar_{\Delta^{\{i,j\}}}$ is a smooth codimension-two subspaces. According to \cite[Lemma 3.43]{PST14}, $\tilde{\mathfrak t}_{i,{i},\Gamma}$ is a section vanishes transversely $D_{i,\{j\},\Gamma}$, and satisfy the orientation property.

Assuming we have proved the lemma for $\lvert S^I\rvert\le n-1$, we now prove it in the case $S^I$ consists of $n\ge 2$ elements. Assuming $j\in S^I$ is one of them, we write the line bundle $\For_{\Gamma\to \for_{{S^I}}(\Gamma)}^* \mathbb L_{i,\for_{S^I}(\Gamma)}^\vee \otimes \mathbb L_{i,\Gamma}$ as
$$(\For_{\Gamma\to \for_{\{j\}}(\Gamma)}^* \mathbb L_{i,\for_{\{j\}}(\Gamma)}^\vee \otimes \mathbb L_{i,\Gamma})\otimes \For_{\Gamma\to \for_{\{j\}}(\Gamma)}^*(\For_{\for_{\{j\}}(\Gamma)\to \for_{S^I}(\Gamma)}^* \mathbb L_{i,\for_{S^I}(\Gamma)}^\vee \otimes \mathbb L_{i,\for_{\{j\}}(\Gamma)})$$
and we can write (since the forgetful morphisms compose)
$$\tilde{\mathfrak t}_{i,\for_{S^I}(\Gamma)\to\Gamma}= \tilde{\mathfrak t}_{i,\for_{\{j\}}(\Gamma)\to\Gamma}\otimes  \For_{\Gamma\to \for_{\{j\}}(\Gamma)}^* \tilde{\mathfrak t}_{i,\for_{S^I}(\Gamma)\to \for_{\{j\}}(\Gamma)}.$$
The section $\tilde{\mathfrak t}_{i,\for_{\{j\}}(\Gamma)\to\Gamma}$ vanishes at $$D_{i,\{j\},\Gamma}=\Mbar_{\Delta^{\{i,j\}}},$$ 
while $\For_{\Gamma\to \for_{\{j\}}(\Gamma)}^* \tilde{\mathfrak t}_{i,\for_{S^I}(\Gamma)\to \for_{\{j\}}(\Gamma)}$ vanishes at $$\For_{\Gamma \to \for_{\{j\}}(\Gamma)}^{-1}(D_{i,S^I\setminus\{j\},\for_{\{j\}}(\Gamma)})=\bigcup_{T\subseteq S^I, T\ne \{j\}}\Mbar_{\Delta^{T\sqcup\{i\}}}.$$
We need to prove the transversality and property for the induced orientation.  

The transversality and the orientation property at $\Mbar_{\Delta^{\{i,j\}}}$ is guaranteed by \cite[Lemma 3.43]{PST14}.

To prove the transversality and the orientation property at $\Mbar_{\Delta^{{T\sqcup\{i\}}}}$ for $T\ne \{j\}$, locally near $p\in \Mbar_{\Delta^{{T\sqcup\{i\}}}}$, we can take a small neighbourhood $p=(\For_{\{j\}}(p),0)\in U\times V \subset \Mbar_{\Gamma}$, where $U\subset \Mbar_{\for_{\{j\}}(\Gamma)}$ is a neighbourhood of $\For_{\Gamma \to \for_{\{j\}}(\Gamma)}(p)$ in $ \Mbar_{\for_{\{j\}}(\Gamma)}$ and $V$ is a neighbourhood of 0 in the complex plane, where the restriction of $\For_{\Gamma \to \for_{\{j\}}(\Gamma)}$ at $U\times V$ is the projection to $U$. Then locally near $p$, the zero locus of $\For_{\Gamma\to \for_{\{j\}}(\Gamma)}^* \tilde{\mathfrak t}_{i,\for_{S^I}(\Gamma)\to \for_{\{j\}}(\Gamma)}$ is $(\mathcal M_{\for_{\{j\}}(\Delta^{T\sqcup\{i\}})}\cap U)\times V$, and the induced orientation on it is the induced (by $\tilde{\mathfrak t}_{i,\for_{S^I}(\Gamma)\to\for_{\{j\}}(\Gamma)}$) orientation on $\mathcal M_{\for_{\{j\}}(\Delta^{T\sqcup\{i\}})}$ tensored by the canonical complex orientation of $V$. Then the lemma follows from the inductive hypothesis and the last items in \cite[Proposition 3.1 and Proposition 3.3]{TZ1}.

\end{proof}

\begin{rmk}\label{rmk divosr like section}
    For any section $s$ of a line bundle $K\to \Mbar_{\Gamma}$, there is a corresponding section $s \otimes \tilde{\mathfrak t }_{i,\for_{S^I}(\Gamma)\to\Gamma}$ of $K\otimes \mathcal O(D_{i,S^I,\Gamma})\to \Mbar_{\Gamma}$ with additional zero locus at $D_{i,S^I,\Gamma}$.

    On the other hand, if we denote by $\mathcal O(-D_{i,S^I,\Gamma})$ the dual of $\mathcal O(D_{i,S^I,\Gamma})$, then the dual section $\tilde{\mathfrak t }_{i,\for_{S^I}(\Gamma)\to\Gamma}^\vee$ has `poles' at $D_{i,S^I,\Gamma}$. We can construct a section of $\tilde{\mathfrak t' }_{i,\for_{S^I}(\Gamma)\to\Gamma}^\vee$ of $\mathcal O(-D_{i,S^I,\Gamma})$ which vanishes at $D_{i,S^I,\Gamma}$ and coincides with $\tilde{\mathfrak t }_{i,\for_{S^I}(\Gamma)\to\Gamma}^\vee$ outside a small neighbourhood of $D_{i,S^I,\Gamma}$, whose induced orientation on $\mathcal M_{\Delta^{T\sqcup\{i\}}}$ is the reverse of the orientation induced by $\tilde{\mathfrak t }_{i,\for_{S^I}(\Gamma)\to\Gamma}$. 
    
    In fact, when $\lvert S^I\rvert=1$, $D_{i,S^I,\Gamma}$ is a smooth codimension-2 subspace, we can take a small tubular neighbourhood of $D_{i,S^I,\Gamma}$ in $\Mbar_\Gamma$ which is isomorphic to $D_{i,S^I,\Gamma}\times \mathbb D_z$, where $\mathbb D_z$ is a small disk in the complex plane parametrized by $z$, and $\tilde{\mathfrak t }_{i,\for_{S^I}(\Gamma)\to\Gamma}^\vee$ on this tubular neighbourhood is given by $1/z$ under a suitable trivialization. We can take $\tilde{\mathfrak t' }_{i,\for_{S^I}(\Gamma)\to\Gamma}^\vee$ to be $\tilde{\mathfrak t }_{i,\for_{S^I}(\Gamma)\to\Gamma}^\vee$ when $\lvert z\rvert\ge \varepsilon$ for a small $\varepsilon$; when $\lvert z\rvert\le \varepsilon$ we take $\tilde{\mathfrak t' }_{i,\for_{S^I}(\Gamma)\to\Gamma}^\vee$ to be $\bar z/\varepsilon^2$.
    
    For $\lvert S^I\rvert\ge 2$,  assuming we have constructed $\tilde{\mathfrak t'}_{i,\for_{S^I}(\Gamma)\to\for_{\{j\}}(\Gamma)}^\vee$, we can construct $\tilde{\mathfrak t' }_{i,\for_{S^I}(\Gamma)\to\Gamma}^\vee$ inductively on $\lvert S^I\rvert$ by setting
    $$\tilde{\mathfrak t'}_{i,\for_{S^I}(\Gamma)\to\Gamma}^\vee= \tilde{\mathfrak t'}_{i,\for_{\{j\}}(\Gamma)\to\Gamma}^\vee\otimes  \For_{\Gamma\to \for_{\{j\}}(\Gamma)}^* \tilde{\mathfrak t'}_{i,\for_{S^I}(\Gamma)\to\for_{\{j\}}(\Gamma)}^\vee.$$ 
    The transversality and orientation property follows from the same argument as in the proof of the above lemma.
\end{rmk}

\subsection{Topological Recursion Relation multisections}\label{sec def of trr sections}

In this subsection, we construct a specific multisection of $\mathbb L_i\to \Mbar_{\Gamma}$, called \textit{Topological Recursion Relation (TRR) multisections}, for each smooth genus-zero dual graph $\Gamma$ with one or two tail(s) as root(s), and for each genus-one smooth dual graph $\Gamma$. For graded $r$-spin graphs $\Gamma$, TRR multisections of $\mathbb L_i\to \Mbar_{\Gamma}^{1/r}$  are the multisections pulled back from the spinless moduli via the forgetful morphism.

\subsubsection{Genus-zero TRR section with respect to one root}
Let $\Gamma$ be a smooth genus-zero dual graph and $i\in T^I(\Gamma)$ an internal tail, and $j\in T(\Gamma)$ be a internal or boundary tail different from $i$. We regard $j$ as a root of $\Gamma$ and define a section $\mathfrak s_{i,\Gamma,j}$ of $\mathbb L_i \to \Mbar_{\Gamma}$ in the following way.

For a smooth disk $[\Sigma]\in \mathcal M^{}_\Gamma$, we denote by $z_i$ the marking on $\Sigma$ corresponding to $i$, and by $z_j$ the (internal or boundary) marking on $\Sigma$ corresponding to $j$. We identify the preferred half $\Sigma$ with the upper half-plane and set
\begin{equation}\label{eq trr section}
\mathfrak s_{i,\Gamma,j}\left(\Sigma\right) = dz\left.\left(\frac{1}{z - z_j}-\frac{1}{z-\bar{z}_i}\right)\right|_{z = z_i} \in T_{z_i}^*C = T_{z_i}^*(\Sigma\cup \overline{\Sigma}).
\end{equation}
This fibre-by-fibre definition glues to a section of $\mathbb{L}_i\to\mathcal M^{}_\Gamma$. This section is easily seen to extend to a smooth global section $\mathfrak s_{i,\Gamma,j}$ over $\Mbar^{}_\Gamma,$: for a disk $\Sigma$ that is not necessarily smooth,  
let $\varphi_C$ be the unique meromorphic differential on $C=\Sigma\cup \overline{\Sigma}$ with simple poles at $\bar{z}_i$ and $z_j$ and at no other marked points or smooth points, such that the residue equals $1$ at $z_j$ and equals $-1$ at $\bar{z}_i$, and the residues at every pair of half-nodes sum to $0$.
Then $\mathfrak s_{i,\Gamma,j}(\Sigma)$ is the evaluation of $\varphi_C$ at $z_i\in C$. 

If $j$ is a boundary tail, in the simplest case $\Gamma=u_{0,1,1}$, a smooth genus-zero dual graph with only one internal tails $i$ and one boundary tail $j$, the section $\mathfrak s_{i,u_{0,1,1},j}$ is a non-vanishing section of $\mathbb L_i$ over the zero-dimensional moduli space $\Mbar_{u_{0,0,2}}$. For general $\Gamma$, the section $\mathfrak s_{i,\Gamma,j}$ is of the form
$$
\mathfrak s_{i,\Gamma,j}=\tilde{t}_{i,u_{0,1,1}\to \Gamma }\operatorname{For}^*_{\Gamma \to u_{0,1,1}}\mathfrak s_{i,u_{0,1,1},j}, 
$$
where $\operatorname{For}_{\Gamma \to u_{0,1,1}}\colon \Mbar_\Gamma \to \Mbar_{u_{0,1,1}}$ is the forgetful morphism forgetting all the tails except for $i$ and $j$. 

If $j$ is an internal tail, in the simplest case $\Gamma=u_{0,0,2}$, a smooth genus-zero dual graph with only two internal tails $i,j$, the section $\mathfrak s_{i,u_{0,0,2},j}$ is a section of $\mathbb L_i\to \Mbar_{u_{0,0,2}}$ which only vanishes at one endpoint (where the boundary of $\Sigma$ contracts to a contracted boundary node) of the one-dimensional moduli space $\Mbar_{u_{0,0,2}}$. For general $\Gamma$, the section $\mathfrak s_{i,\Gamma,j}$ is of the form
$$
\mathfrak s_{i,\Gamma,j}=\tilde{t}_{i,u_{0,0,2}\to \Gamma }\operatorname{For}^*_{\Gamma \to u_{0,0,2}}\mathfrak s_{i,u_{0,0,2},j}, 
$$
where $\operatorname{For}_{\Gamma \to u_{0,0,2}}\colon \Mbar_\Gamma \to \Mbar_{u_{0,0,2}}$ is the forgetful morphism forgetting all the tails except for $i$ and $j$. 

Moreover, for a graph $\Delta\in \partial \Gamma$ without internal edges, we denote by $v_{i,\Delta}\in V(\Delta)$ the unique vertex containing the tail $i$ and by $\pi_{\Delta \to v_{i,\Delta}}\colon \Mbar_{\Delta}\to \Mbar_{v_{i,\Delta}}$ the projection. If $j\in T(v_{i,\Delta})$, then we have 
$$
\mathfrak s_{i,\Gamma,j}\vert_{\Mbar_{\Delta} \subset \Mbar_{\Gamma}}= \pi^*_{\Delta \to v_{i,\Delta}}\mathfrak s_{i,v_{i,\Delta},j}.
$$
If $j\notin T(v_{i,\Delta})$, we denote by $r_j\in T^B(v_{i,\Delta})$ to boundary half-edge `closest' to $i$, \textit{i.e.,} the tails $i$ and $j$ are on different connected component after we detach the edge of $\Delta$ corresponding to $r_j$. In this case, we have 
$$
\mathfrak s_{i,\Gamma,j}\vert_{\Mbar_{\Delta} \subset \Mbar_{\Gamma}}= \pi^*_{\Delta \to v_{i,\Delta}}\mathfrak s_{i,v_{i,\Delta},r_j}.
$$

When there is no ambiguity about the root, we will simply write $\mathfrak s_{i,\Gamma,j}$ as $\mathfrak s_{i,\Gamma}$.

\subsubsection{Genus-zero TRR section with respect to two roots}
Let $\Gamma$ be a smooth genus-zero dual graph and $i\in T^I(\Gamma)$ be an internal tail, $j,k\in T(\Gamma)$ be two tails different from $i$. We regard $j$ and $k$ as two roots of $\Gamma$ and define a multisection $\mathfrak s_{i,\Gamma,j,k}$ of $\mathbb L_i \to \Mbar_{\Gamma}$ in the following way.

The multisection $\mathfrak s_{i,\Gamma,j,k}$ has two branches, each of them is of weight $1/2$. One branch is $\mathfrak s_{i,\Gamma,j}$, the other one is $\mathfrak s_{i,\Gamma,k}$. The behaviour of  $\mathfrak s_{i,\Gamma,j,k}$ when restricted to $\Mbar_{\Delta}\subset \Mbar_{\Gamma}$ follows the behaviours of its branches.

When there is no ambiguity about the root, we will simply write $\mathfrak s_{i,\Gamma,j,k}$ as $\mathfrak s_{i,\Gamma}$.

\subsubsection{Genus-one TRR multisections}\label{sec basic trr section}
Let $u_{1,0,1}$ be a smooth single-vertex open genus-one graph with one internal tail and no boundary tails, We first construct special multisections $\mathfrak s$, $s^x_{1,0,1}$ and $s^y_{1,0,1}$ of the universal line bundle $\mathbb L\to \Mbar_{u_{1,0,1}} =\Mbar_{1,0,1}$. We parametrize $\mathcal M_{1,0,1}$ by two real numbers $\tau$ and  $h$ such that $0<h<\tau$, the point $(\tau,h)$ represents a cylinder $(C,\Sigma,z_1,\bar z_1)$, where $C=\{z\in \mathbb C\colon 0\le \operatorname{Im} z\le 2\tau\}/<z \sim z+1>$, $\Sigma=\{z\in \mathbb C\colon 0\le \operatorname{Im} z\le \tau\}/<z \sim z+1>$, the mark point $z_1$ and its conjugate $\bar z_i$ are located at $h\sqrt{-1}$ and $-h\sqrt{-1}$. At each point $(\tau,h)$, $dz+d\bar z$ represents a vector in the fibre of $\mathbb L \to \mathcal M_{1,0,1}$ over $(\tau,h)$, which is a cotangent vector dual to the tangent vector `parallel to the boundaries of $\Sigma$'. Hence we can use $dz+d\bar z$ to represent a section of $\mathbb L \to \mathcal M_{1,0,1}$.  We denote by $s^y_{1,0,1}$ a multisection of $\mathbb L \to \mathcal M_{1,0,1}$, which has two branches, one branch is $dz+d\bar z$ with weight $\frac{1}{2}$, the other branch is $-dz-d\bar z$ with weight $\frac{1}{2}$. We also denote by $s^x_{1,0,1}:=\sqrt{-1}\cdot s^y_{1,0,1}$ another multisection of $\mathbb L \to \mathcal M_{1,0,1}$; note that at each fibre, the branches of $s^x_{1,0,1}$ are $\pm \sqrt{-1}\cdot (dz+d\bar z)=\pm (dz-d\bar z)$, which are cotangent vectors dual to the tangent vector `perpendicular to the boundaries of $\Sigma$'.

We consider the extension of  $s^y_{1,0,1}$ (and hence $s^x_{1,0,1}$) to $\Mbar_{1,0,1}$. Note that $\Mbar_{1,0,1}$ has three types of codimension-1 boundaries as shown in Figure \ref{fig basic moduli}.
\begin{enumerate}
    \item Boundary of the form $\Gamma_{sp}$, where $\Gamma_{sp}$ is a graph consisting of two vertices: one genus-one open vertex $v_{sp}^1$ and one genus-zero open vertex $v_{sp}^0$, where $v_{sp}^1$ and $v_{sp}^0$ are connected by an separating edge $e_{sp}$, and the unique marking (which is internal) corresponding to $z_1$ is attached to $v_{sp}^0$.

    The nodal cylinders in $\mathcal M_{\Gamma_{sp}}$ have a genus-one component and a genus-zero component. If we parametrize its genus-zero component by the unit disk in $\mathbb C$, where the marked internal point $z_1$ locates at the origin $0\in \mathbb C$ and the half-node corresponding to $e_{sp}$ locates at $\sqrt{-1}\in \mathbb C$. The limit of $s^y_{1,0,1}$ to $\Mbar_{\Gamma_{sp}}$ has two branches, under the above parametrization, one branch is represented (up to positive rescaling) by $dz-d\bar z$ with weight $\frac{1}{2}$ (\textit{i.e.} dual to the section `pointing' to the half-node), the other branch is represented (up to positive rescaling) by the negative of the previous one, \textit{i.e.} $-dz+d\bar z$ with weight $\frac{1}{2}$.
    
     \item Boundary of the form $\Gamma_{nsp}$, where $\Gamma_{nsp}$ is a graph consisting of one open genus-zero vertex $v_{nsp}$, which is connect to itself via a non-separating edge $e_{nsp}$. On $v_{nsp}$ there are three half-edges in total: two boundary half-edges $h_1$ and $h_2$ of $e_{nsp}$, and an internal half-edge corresponding to the internal marked point.

      The nodal cylinders in $\mathcal M_{\Gamma_{nsp}}$ have one genus-zero component. If we parametrize it by the unit disk in $\mathbb C$, where the marked internal point $z_1$ locates at the origin $0\in \mathbb C$ and the two boundary half-nodes corresponding to $h_1$ and $h_2$ locate at $\alpha,-1/\alpha\in \mathbb C$ for $\lvert \alpha\rvert=1$. The limit of $s^y_{1,0,1}$ to $\Mbar_{\Gamma_{nsp}}$ has two branches, under the above parametrization, one branch is represented (up to positive rescaling) by $dz-d\bar z$ with weight $\frac{1}{2}$ (\textit{i.e.} dual to the tangent vector `pointing' to the angle bisector between directions toward $\alpha$ and $-1/\alpha$), the other branch is represented (up to positive rescaling) by the negative of the previous one, \textit{i.e.} $-dz+d\bar z$ with weight $\frac{1}{2}$.
      
     \item Boundary of the form $\Gamma_{cb}$, where $\Gamma_{cb}$ is a graph consisting of one open genus-zero vertex $v_{cb}$, which is has a contracted boundary half-edge $e_{cb}$. On $v$ there is one internal half-edge corresponding to the internal marked point.

     The nodal cylinders in $\mathcal M_{\Gamma_{cb}}$ have one genus-zero component. If we parametrize it by the unit disk in $\mathbb C$, where the marked internal point $z_1$ locates at the origin $0\in \mathbb C$ and the contracted boundary node corresponding to $e_{cb}$ locates at $\lambda \sqrt{-1} \in \mathbb C$ for $0<\lambda <1$. The limit of $s^y_{1,0,1}$ to $\Mbar_{\Gamma_{cb}}$ has two branches, under the above parametrization, one branch is represented (up to positive rescaling) by $dz-d\bar z$ with weight $\frac{1}{2}$ (\textit{i.e.} dual to the tangent vector `pointing' to the contracted boundary node), the other branch is represented (up to positive rescaling) by the negative of the previous one, \textit{i.e.} $-dz+d\bar z$ with weight $\frac{1}{2}$.
\end{enumerate}

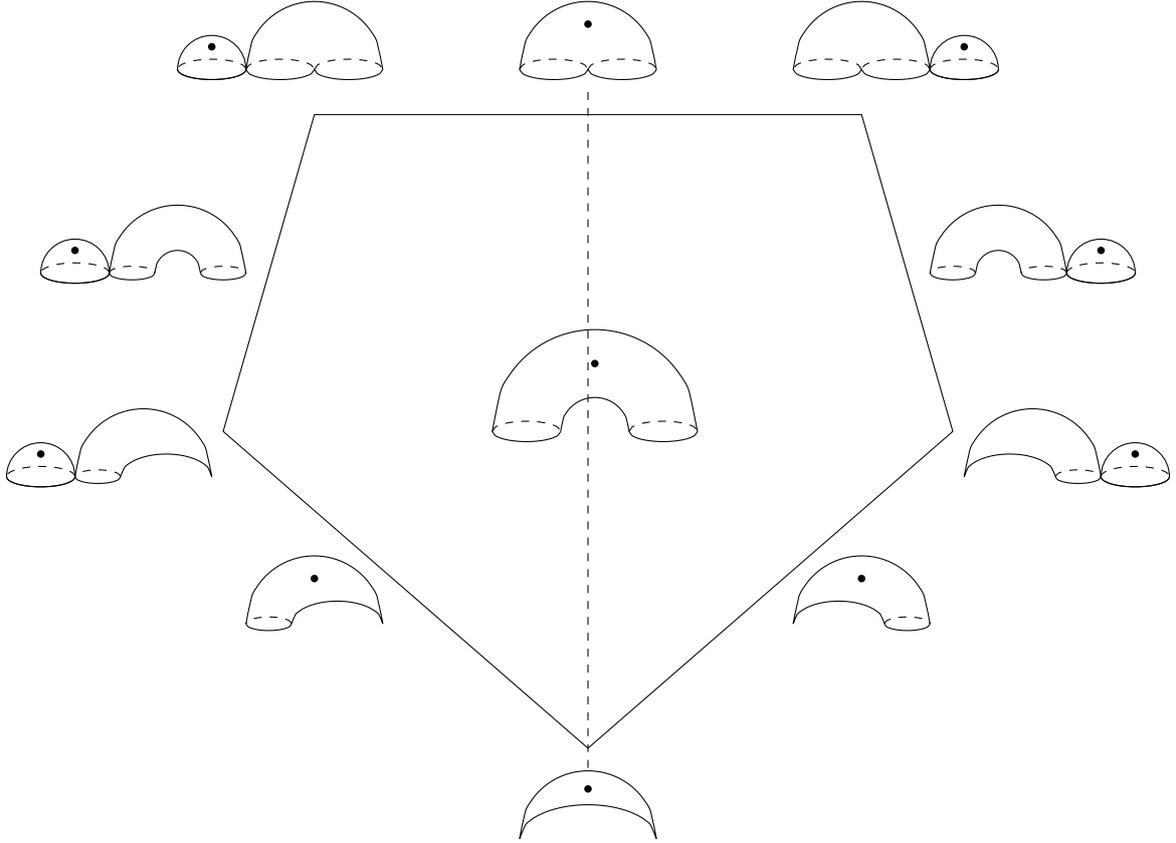
\begin{figure}[h]
\centering

\begin{tikzpicture}[scale=0.3]
   \draw (-12,14)--(12,14)--(16,0)--(0,-14)--(-16,0)--(-12,14);

   \draw[dashed] (0,15) -- (0,-15);

    \begin{scope}[shift={(-18,7)}]
    \draw [] plot [domain=-1:1,  smooth]({\x}, {sqrt(1-\x*\x)}) 
        -- plot [domain=1:3,  smooth]({\x}, {-sqrt(1-(\x-2)*(\x-2))*0.3}) 
        -- plot [domain=3:-3,  smooth ] ({\x}, {sqrt(9-\x*\x)}) 
        -- plot [domain=-3:-1,  smooth]({\x}, {-sqrt(1-(\x+2)*(\x+2))*0.3});

    \draw  [dashed, domain=1:3,  smooth] plot ({\x}, {sqrt(1-(\x-2)*(\x-2))*0.3});

    \draw  [dashed, domain=-3:-1,  smooth] plot ({\x}, {sqrt(1-(\x+2)*(\x+2))*0.3});

    \draw (-6,0) arc (180:360:1.5 and 0.45);
    \draw[dashed] (-3,0) arc (0:180:1.5 and 0.45);
    \draw[] (-6,0) arc (180:360:1.5 and 0.45) arc (0:180:1.5);

    \node at (-4.5,1) [circle,fill,inner sep=1pt]{};
\end{scope}

\begin{scope}[shift={(18,7)},xscale=-1]
    \draw [] plot [domain=-1:1,  smooth]({\x}, {sqrt(1-\x*\x)}) 
        -- plot [domain=1:3,  smooth]({\x}, {-sqrt(1-(\x-2)*(\x-2))*0.3}) 
        -- plot [domain=3:-3,  smooth ] ({\x}, {sqrt(9-\x*\x)}) 
        -- plot [domain=-3:-1,  smooth]({\x}, {-sqrt(1-(\x+2)*(\x+2))*0.3});

    \draw  [dashed, domain=1:3,  smooth] plot ({\x}, {sqrt(1-(\x-2)*(\x-2))*0.3});

    \draw  [dashed, domain=-3:-1,  smooth] plot ({\x}, {sqrt(1-(\x+2)*(\x+2))*0.3});

    \draw (-6,0) arc (180:360:1.5 and 0.45);
    \draw[dashed] (-3,0) arc (0:180:1.5 and 0.45);
    \draw[] (-6,0) arc (180:360:1.5 and 0.45) arc (0:180:1.5);

    \node at (-4.5,1) [circle,fill,inner sep=1pt]{};
\end{scope}

\begin{scope}[shift={(0.3,0)},xscale=1.5,yscale=1.5]
    \draw [] plot [domain=-1:1,  smooth]({\x}, {sqrt(1-\x*\x)}) --plot [domain=1:3,  smooth]({\x}, {-sqrt(1-(\x-2)*(\x-2))*0.3}) -- plot [domain=3:-3,  smooth ] ({\x}, {sqrt(9-\x*\x)})-- plot [domain=-3:-1,  smooth]({\x}, {-sqrt(1-(\x+2)*(\x+2))*0.3});

    \draw  [dashed, domain=1:3,  smooth] plot ({\x}, {sqrt(1-(\x-2)*(\x-2))*0.3});

     \draw  [dashed, domain=-3:-1,  smooth] plot ({\x}, {sqrt(1-(\x+2)*(\x+2))*0.3});

    \node at (0,2) [circle,fill,inner sep=1pt]{};
    
    \end{scope}

    \begin{scope}[shift={(12,-8.5)},xscale=1,yscale=1]

    \draw [] plot [domain=1:3,  smooth]({\x}, {-sqrt(1-(\x-2)*(\x-2))*0.3}) -- plot [domain=3:-3,  smooth ] ({\x}, {sqrt(9-\x*\x)})-- plot [domain=-3:1,  smooth]({\x}, {sqrt(4-(\x+1)*(\x+1))*0.5});

    \draw  [dashed, domain=1:3,  smooth] plot ({\x}, {sqrt(1-(\x-2)*(\x-2))*0.3});

    \node at (0,2) [circle,fill,inner sep=1pt]{};

    \end{scope}

    \begin{scope}[shift={(-12,-8.5)},xscale=-1,yscale=1]

    \draw [] plot [domain=1:3,  smooth]({\x}, {-sqrt(1-(\x-2)*(\x-2))*0.3}) -- plot [domain=3:-3,  smooth ] ({\x}, {sqrt(9-\x*\x)})-- plot [domain=-3:1,  smooth]({\x}, {sqrt(4-(\x+1)*(\x+1))*0.5});

    \draw  [dashed, domain=1:3,  smooth] plot ({\x}, {sqrt(1-(\x-2)*(\x-2))*0.3});

    \node at (-0,2) [circle,fill,inner sep=1pt]{};

    \end{scope}

     \begin{scope}[shift={(0,16)},xscale=1,yscale=1]

    \draw [] plot [domain=0:3,  smooth]({\x}, {-sqrt(2.25-(\x-1.5)*(\x-1.5))*0.3}) -- plot [domain=3:-3,  smooth ] ({\x}, {sqrt(9-\x*\x)})-- plot [domain=-3:0,  smooth]({\x}, {-sqrt(2.25-(\x+1.5)*(\x+1.5))*0.3});

    \draw  [dashed, domain=0:3,  smooth] plot ({\x}, {sqrt(2.25-(\x-1.5)*(\x-1.5))*0.3});

    \draw  [dashed, domain=-3:0,  smooth] plot ({\x}, {sqrt(2.25-(\x+1.5)*(\x+1.5))*0.3});

    \node at (0,2) [circle,fill,inner sep=1pt]{};
  \end{scope}

    \begin{scope}[shift={(0,-18)},xscale=1,yscale=1]

    \draw []  plot [domain=3:-3,  smooth ] ({\x}, {sqrt(9-\x*\x)})-- plot [domain=-3:3,  smooth]({\x}, {sqrt(9-\x*\x)*0.5});

    \node at (0,2.2) [circle,fill,inner sep=1pt]{};

  \end{scope}

  \begin{scope}[shift={(12,16)},xscale=-1,yscale=1]

    \draw [] plot [domain=0:3,  smooth]({\x}, {-sqrt(2.25-(\x-1.5)*(\x-1.5))*0.3}) -- plot [domain=3:-3,  smooth ] ({\x}, {sqrt(9-\x*\x)})-- plot [domain=-3:0,  smooth]({\x}, {-sqrt(2.25-(\x+1.5)*(\x+1.5))*0.3});

    \draw  [dashed, domain=0:3,  smooth] plot ({\x}, {sqrt(2.25-(\x-1.5)*(\x-1.5))*0.3});

    \draw  [dashed, domain=-3:0,  smooth] plot ({\x}, {sqrt(2.25-(\x+1.5)*(\x+1.5))*0.3});

    \draw (-6,0) arc (180:360:1.5 and 0.45);
    \draw[dashed] (-3,0) arc (0:180:1.5 and 0.45);
    \draw[] (-6,0) arc (180:360:1.5 and 0.45) arc (0:180:1.5);

    \node at (-4.5,1) [circle,fill,inner sep=1pt]{};
  \end{scope}

   \begin{scope}[shift={(-12,16)},xscale=1,yscale=1]

    \draw [] plot [domain=0:3,  smooth]({\x}, {-sqrt(2.25-(\x-1.5)*(\x-1.5))*0.3}) -- plot [domain=3:-3,  smooth ] ({\x}, {sqrt(9-\x*\x)})-- plot [domain=-3:0,  smooth]({\x}, {-sqrt(2.25-(\x+1.5)*(\x+1.5))*0.3});

    \draw  [dashed, domain=0:3,  smooth] plot ({\x}, {sqrt(2.25-(\x-1.5)*(\x-1.5))*0.3});

    \draw  [dashed, domain=-3:0,  smooth] plot ({\x}, {sqrt(2.25-(\x+1.5)*(\x+1.5))*0.3});

    \draw (-6,0) arc (180:360:1.5 and 0.45);
    \draw[dashed] (-3,0) arc (0:180:1.5 and 0.45);
    \draw[] (-6,0) arc (180:360:1.5 and 0.45) arc (0:180:1.5);

    \node at (-4.5,1) [circle,fill,inner sep=1pt]{};
  \end{scope}

   \begin{scope}[shift={(19.5,-2)},xscale=1,yscale=1]

    \draw [] plot [domain=1:3,  smooth]({\x}, {-sqrt(1-(\x-2)*(\x-2))*0.3}) -- plot [domain=3:-3,  smooth ] ({\x}, {sqrt(9-\x*\x)})-- plot [domain=-3:1,  smooth]({\x}, {sqrt(4-(\x+1)*(\x+1))*0.5});

    \draw  [dashed, domain=1:3,  smooth] plot ({\x}, {sqrt(1-(\x-2)*(\x-2))*0.3});
````
    \begin{scope}[xscale=-1]
        \draw (-6,0) arc (180:360:1.5 and 0.45);
    \draw[dashed] (-3,0) arc (0:180:1.5 and 0.45);
    \draw[] (-6,0) arc (180:360:1.5 and 0.45) arc (0:180:1.5);

    \node at (-4.5,1) [circle,fill,inner sep=1pt]{};
    \end{scope}
    \end{scope}

    \begin{scope}[shift={(-19.5,-2)},xscale=-1,yscale=1]

    \draw [] plot [domain=1:3,  smooth]({\x}, {-sqrt(1-(\x-2)*(\x-2))*0.3}) -- plot [domain=3:-3,  smooth ] ({\x}, {sqrt(9-\x*\x)})-- plot [domain=-3:1,  smooth]({\x}, {sqrt(4-(\x+1)*(\x+1))*0.5});

    \draw  [dashed, domain=1:3,  smooth] plot ({\x}, {sqrt(1-(\x-2)*(\x-2))*0.3});
````
    \begin{scope}[xscale=-1]
        \draw (-6,0) arc (180:360:1.5 and 0.45);
    \draw[dashed] (-3,0) arc (0:180:1.5 and 0.45);
    \draw[] (-6,0) arc (180:360:1.5 and 0.45) arc (0:180:1.5);

    \node at (-4.5,1) [circle,fill,inner sep=1pt]{};
    \end{scope}
    \end{scope}
   
\end{tikzpicture}

\caption{The double cover of the basic moduli $\Mbar_{1,0,1}$, \textit{i.e.,} the moduli space of pointed cylinders with labelled boundaries. We obtain $\Mbar_{1,0,1}$ after quotient the involution exchanging the label of two boundaries, \textit{i.e.,} the reflection along the vertical dashed axis.}
\label{fig basic moduli}
\end{figure}

The extensions of $s^x_{1,0,1}$ to $\Mbar_{\Gamma_{sp}}$, $\Mbar_{\Gamma_{nsp}}$ and $\Mbar_{\Gamma_{cb}}$ are the extensions of $s^y_{1,0,1}$ with an additional 90-degree rotation. We define a section $\mathfrak{s}$ (up to rescaling by positive function) on the boundaries of $\Mbar_{1,0,1}$  in the following way.
\begin{enumerate}
    \item On $\Mbar_{\Gamma_{sp}}$, if we denote by $i$ the unique internal tail (which is on $v^0_{sp}$) and by $h_0$ the boundary half-edge on $v^0_{sp}$ corresponding to $e_{sp}$, the multisection $\mathfrak s$ is pulled back from the section $\mathfrak s_{i,v^0_{sp},h}$ of $\mathbb L_i\to \Mbar_{v^0_{sp}}$ via the projection $\Mbar_{\Gamma_{sp}}\to \Mbar_{v^0_{sp}}$. Under the parametrization as above, $\mathfrak{s}$ (up to positive rescaling) has  only one branch $dz-d\bar z$ with weight 1, which is the dual to the section `pointing' to the half-node.

    \item On $\Mbar_{\Gamma_{nsp}}$, if we denote by $i$ the unique internal tail and by $h_1,h_2$ the two boundary half-edges on $v_{nsp}$, the multisection $\mathfrak s$ is pulled back from the section $\mathfrak s_{i,v_{nsp},h_1,h_2}$ of $\mathbb L_i\to \Mbar_{v_{nsp}}$ via the (isomorphic) projection $\Mbar_{\Gamma_{nsp}}\to \Mbar_{v_{nsp}}$. Then $\mathfrak{s}$ has two branches, each has weight $\frac{1}{2}$; under the parametrization as above, up to positive rescaling, one of them is $\alpha (dz+d\bar{z})$, the other one is $-\frac{1}{\alpha}(dz+d\bar{z})$, each of them is the dual of the tangent vector `pointing' to one of the boundary half-nodes.
    Notice that the restriction of such $\mathfrak s$ at the corner $\Mbar_{\Gamma_{nsp}}\cap \Mbar_{\Gamma_{sp}}$ is consistent with the $\mathfrak s$ defined on $\Mbar_{\Gamma_{sp}}$ in the previous item. Also notice that $\mathfrak s\vert_{\Mbar_{\Gamma_{nsp}}=\Mbar_{v_{nsp}}}$ constructed in this way is symmetry in the sense that it is invariant under the automorphism $\operatorname{Exc}_{h_1,h_2}\colon\Mbar_{v_{nsp}}\to\Mbar_{v_{nsp}} $ exchanging the two boundary half-edges.
    
    \item On $\Mbar_{\Gamma_{cb}}$, $\mathfrak{s}$ has two branches with weight $\frac{1}{2}$. On the banana-shape corner (see Figure \ref{fig basic moduli}), $\mathfrak{s}$ is the same as the limit of $s^x_{1,0,1}$; on the other corner (\textit{i.e.,} $\Mbar_{\Gamma_{cb}}\cap \Mbar_{\Gamma_{sp}}$), $\mathfrak s$ is already determined by its definition on $\Mbar_{\Gamma_{sp}}$, where the two branches coincide. In between these two corner, notices that the line determined by the two branches of $s^y_{1,0,1}$ divide the cotangent plane at the internal marking into two half-planes, we extend $\mathfrak s$ from the corners to $\mathcal M_{\Gamma_{cb}}$ in a way such that its two branches lie on different half-planes.

\end{enumerate}

At any point in $\Mbar_{1,0,1}$, the line determined by the two branches of $s^y_{1,0,1}$ divide the cotangent plane at the internal marking into two half-planes, and the two branches of $s^x_{1,0,1}$ lie in different half-planes. Moreover, at any point in $\partial \Mbar_{1,0,1}$, the two branches of $\mathfrak s$ either lie on in different half-planes, or they just coincide and lie on the dividing line. Therefore, there is a nowhere-vanishing homotopy between $s^x_{1,0,1}$ and $\mathfrak s$ on $\partial \Mbar_{1,0,1}$, which means $s^x_{1,0,1}$ and $\mathfrak s$ have the same winding number in the boundary $\partial \Mbar_{1,0,1}$.   Notice that $s^x_{1,0,1}$ is nowhere-vanishing in the entire $\Mbar_{1,0,1}$, we can extend $\mathfrak s$ to a nowhere-vanishing multisection $\mathfrak s$ on the entire $\Mbar_{1,0,1}$.

For any smooth single-vertex genus-one graph $u$, assuming $i$ is a internal tail of $u$, let 
$$
\For_{u\to u_{1,0,1}}\colon \Mbar_{u}\to \Mbar_{1,0,1}
$$
be the forgetful morphism forgetting all the half-edges except $i$. We define a multisection $\mathfrak{s}_{u}$ of $\mathbb L_i\to \Mbar_{u}$ by
$$
\mathfrak{s}_{u}:=\tilde{\mathfrak t}_{i,u_{1,0,1}\to u}\otimes \For_{u\to u_{1,0,1}}^*\mathfrak s.
$$
For any $\Gamma \in \partial u$ with out non-separating internal edges, assuming $i\in T^I(v)$ for an open vertex $v \in V^O(\Gamma)$, there are three possibilities.
\begin{enumerate}
    \item $v$ is a genus-one vertex. In this case, $\mathfrak s_{i,u}\vert_{\Mbar_{\Gamma}}$ is pulled back from $\mathfrak s_{i,v}$ via the projection morphism $\pi_{\Gamma\to v}\colon \Mbar_\Gamma \to \Mbar_v$, \textit{i.e.,}
    \begin{equation}\label{eq behaviour of trr genus one}
    \mathfrak s_{i,u}\vert_{\Mbar_{\Gamma}}=\pi^*_{\Gamma\to v} \mathfrak s_{i,v}.
    \end{equation}
    This is because the composition    \begin{equation*}
        \For_{v\to u_{1,0,1}}\circ\pi_{\Gamma\to v}=\For_{u\to u_{1,0,1}}\vert_{\Mbar_\Gamma}
    \end{equation*} 
    implies $\For_{u\to u_{1,0,1}}^*\mathfrak s\vert_{\Mbar_{\Gamma}}=\pi_{\Gamma\to v}^*\For_{v\to u_{1,0,1}}^*\mathfrak s$ and $\tilde{\mathfrak t}_{i,u_{1,0,1}\to u}\vert_{\Mbar_{\Gamma}}=\pi_{\Gamma\to v}^*\tilde{\mathfrak t}_{i,u_{1,0,1}\to v}$.
    \item $v$ is a rooted genus-zero vertex with an boundary root $r_v$. 
    By construction we have
    $$
     \mathfrak s_{i,v,r_v}=\tilde{\mathfrak t}_{i, v^0_{sp}\to v}\otimes \For^*_{v\to v_{sp}^0}\mathfrak s\vert_{v^0_{sp}},
    $$
    where $\For_{v\to v_{sp}^0}\colon \Mbar_v\to \Mbar_{v_{sp}^0}$ is the forgetful morphism forgetting all the half-edges of $v$ except $i$ and the root $j$, then 
    $$
    \For_{v\to v^0_{sp}}\circ\pi_{\Gamma\to v}=\pi_{\Gamma_{sp}\to v^0_{sp}}\circ\For_{u\to u_{1,0,1}}\vert_{\Mbar_\Gamma}
    $$
    implies 
    \begin{equation}\label{eq behaviour of trr 1 rooted}
    \mathfrak s_{i,u}\vert_{\Mbar_{\Gamma}}=\pi^*_{\Gamma \to v}\mathfrak s_{i,v,r_v}.
    \end{equation}
    
    \item $v$ is a 2-rooted genus-zero vertex with boundary roots $r_v^1$ and $r_v^2$.  
    By construction we have
    \begin{equation}\label{eq restriction of trr on nsp}
     \mathfrak s_{i,v,r_v^1,r_v^2}=\tilde{\mathfrak t}_{i, v_{nsp}\to v}\otimes \For^*_{v\to v_{nsp}}\mathfrak s\vert_{\Mbar_{v_{nsp}}},
    \end{equation}
    where $\For_{v\to v_{nsp}}\colon \Mbar_v \to \Mbar_{v_{nsp}}$ is a forgetful morphism forgetting all the half-edges of $v$ except $i$ and the two boundary roots $r_v^1$ and $r_v^2$. Although there are two such forgetful morphisms because both $v$ and $v_{nsp}$ have two boundary roots, \eqref{eq restriction of trr on nsp} holds for both forgetful morphisms since $\mathfrak s\vert_{\Mbar_{v_{nsp}}}$ is invariant under $\operatorname{Exc}_{h_1,h_2}$.
    Then 
    $$
    \For_{v\to v_{nsp}}\circ\pi_{\Gamma\to v}=\pi_{\Gamma_{nsp}\to v_{nsp}}\circ\For_{u\to u_{1,0,1}}\vert_{\Mbar_\Gamma}
    $$
    implies 
    \begin{equation}\label{eq behaviour of trr two rooted}
    \mathfrak s_{i,u}\vert_{\Mbar_{\Gamma}}=\pi^*_{\Gamma \to v}\mathfrak s_{i,v,r_v^1,r_v^2}.
    \end{equation}
\end{enumerate}

\begin{obs}\label{obs vanishing loci of single trr}
    For an open genus-one, single-vertex graph $u$, since $\For_{I\setminus\{i\},\bar B^1 \cup\bar B^2}^*\mathfrak s$ vanishes nowhere, the zero locus of $\mathfrak s_{i,u}$ in $\Mbar_u$ coincides with the zero locus of $\tilde{\mathfrak t}_{i,i,I\setminus\{i\},\bar B^1 \cup\bar B^2,\Gamma}$, which is of the form $\bigcup_{\Gamma\in \mathcal G_i(u)}\Mbar_\Gamma$, where the set $\mathcal G_i(u)\subset \partial u$ consists of graphs with two vertices: one genus-one open vertex $v_{\Gamma,o}$ and one genus-zero closed vertex $v_{\Gamma,c}$, and the internal half-edge labelled by $i$ is attached to $v_{\Gamma,c}$.

    Similarly, for an open genus-zero, single-vertex, rooted graph $u$ with  root $r_u$,  the zero locus of $\mathfrak s_{i,u,r_u}$  in $\Mbar_u$ is of the form $\bigcup_{\Gamma\in \mathcal G_i(u)}\Mbar_\Gamma$, where the set $\mathcal G_i(u)\subset \partial u$ consists of graphs with two vertices: one genus-zero open vertex $v_{\Gamma,o}$ and one genus-zero closed vertex $v_{\Gamma,c}$, and the  internal half-edge labelled by $i$ is attached to $v_{\Gamma,c}$, while the root $r_u$ is attached to $v_{\Gamma,o}$. 
    
    For an open genus-zero, single-vertex, 2-rooted graph $u$ with two boundary roots $r_u^1, r_u^2$, the zero locus of $\mathfrak s_{i,u,r_u^1,r_u^2}$ in $\Mbar_u$ is of the form $\bigcup_{\Gamma\in \mathcal G_i(u)}\Mbar_\Gamma$, where the set $\mathcal G_i(u)\subset \partial u$ consists of graphs with two vertices: one genus-zero open vertex $v_{\Gamma,o}$ and one genus-zero closed vertex $v_{\Gamma,c}$, and the  internal half-edge labelled by $i$ is attached to $v_{\Gamma,c}$, while at least one of the roots $r_u^1$ or $r_u^2$ is attached to $v_{\Gamma,o}$. Note that $\mathfrak s_{i,u,r_u^1,r_u^2}$ has two branches $\mathfrak s_{i,u,r_u^1}$ and $\mathfrak s_{i,u,r_u^2}$ with weight $1/2$, if only one of $r_u^1$ or $r_u^2$ is attached to $v_{\Gamma,o}$, then $\Mbar_\Gamma$ is only the zero locus of one branch $\mathfrak s_{i,u,r_u^1}$ or $\mathfrak s_{i,u,r_u^2}$; if both $r_u^1$ and $r_u^2$ are attached to $v_{\Gamma,o}$, then $\Mbar_\Gamma$ is is the zero locus of both branches.

    \begin{rmk}\label{rmk on vanishing order r}
    When the graph $u$ above is a (spinless) dual graph, the multiplicity of the zero locus in $\Mbar_u$ is one.  When $u$ is graded $r$-spin graded, the multiplicity of the zero locus in $\Mbar^{1/r}_u$ is $r$. This is because there is additional $r$-fold isotropy present in $\Mbar^{1/r}_u$ that was not present in the non-spin case, so the pulled-back section has order of vanishing $r$ in the orbifold sense.
    \end{rmk}

\end{obs}

\section{Definition of open $r$-spin correlators}\label{sec sections def etc}

    \subsection{Canonical multisections via tame perturbations of TRR}\label{sec def of tame and canonical}
    Let $\Gamma$ be a connected smooth stable graded $r$-spin graph, $I$ be a set of internal markings (not necessarily contained in $T^I(\Gamma)$), and $S\subseteq T^I\setminus I$ a subset of twist zero internal tails.
    Let $\bm{s}_\Gamma$ be a multisection of 
    $$\mathcal W_{\Gamma}\oplus \bigoplus_{i\in I\cap T^I(\Gamma)} \operatorname{For}^*_{\Gamma \to \operatorname{for}_{S}(\Gamma)}{{\mathbb L}}_{i,\operatorname{for}_{S}(\Gamma)}^{\oplus d_i}\to \oPMb_\Gamma,$$ 
    we can write it in the form 
    $$\bm{s}_{\Gamma}=s^{\mathcal W}\oplus\bigoplus_{1\le i \le \lvert I^{d>0}\rvert}\bigoplus_{1\le j \le d_{i}} s^{j}_{i},$$
    where  $s^{\mathcal W}$ is a section of $\mathcal W$ and each $s^{j}_{i}$ is a multisection of $\operatorname{For}^*_{\Gamma \to \operatorname{for}_{S}(\Gamma)}{{\mathbb L}}_{i,\operatorname{for}_{S}(\Gamma)}$. Recall that if $\operatorname{for}_{S}(\Gamma)$ is unstable we formally take ${\mathbb L}_{i,\operatorname{for}_{S}(\Gamma)}$ to be trivial complex line over the single point $\oPMb_{\operatorname{for}_{S}(\Gamma)}$.

    For any $\Delta \in \partial^* \Gamma$, we define a set $S_\Delta \subseteq S$ by 
    $$
    S_{\Delta}:=\{a\in S\colon a\in T^I(u), u\in V(\Delta), \operatorname{for}_{S\cap T^I(u)}(u)\text{ needs to be contracted to stabilize }\operatorname{for}_{S}(\Gamma)\}.
    $$
    We define $\mathcal T_S \Delta$ the graph obtained by forgetting all tails in $S_{\Delta}$, and contract all the unstable vertices repeatedly. Note that if $\operatorname{for}_{S}(\Gamma)$ is stable, then $\mathcal T_S \Delta$ is a stable graph in $\partial \operatorname{for}_{S_\Delta}(\Gamma)$; if $\operatorname{for}_{S}(\Gamma)$ is unstable, then $\mathcal T_S \Delta$ is just the same as $\operatorname{for}_{S}(\Gamma)$. 
    
    \begin{rmk}\label{rmk on forget morphism on strata}
   If $\operatorname{for}_{S}(\Gamma)$ is stable, then $\oPMb_{\mathcal T_S \Delta}\subseteq \oPMb_{\operatorname{for}_{S_\Delta}(\Gamma)}$ is the image of $\oPMb_{\Delta}\subseteq \oPMb_{\Gamma}$ under the forgetful morphism
    $\operatorname{For}_{\Gamma \to \operatorname{for}_{S_\Delta}(\Gamma)}\colon \oPMb_{\Gamma}\to \oPMb_{\operatorname{for}_{S_\Delta}(\Gamma)}$. We denote such restricted forgetful morphism by
    $$
    \operatorname{For}_{\Delta \to {\mathcal T_S \Delta}}\colon \oPMb_{\Delta}\to \oPMb_{\mathcal T_S \Delta}.
    $$
        
    \end{rmk}

    \begin{rmk}\label{rmk on contract after forget}
    In the case $\mathcal T_S \Delta$ is stable, in the contracting procedure, if a vertex $v$ with a tail labelled by $i\in I$ is unstable and needs to be contracted, then the only possibility is that $v$ is closed and has only one other half-edge connecting it to another vertex $v'$. After contracting $v$, the marking $i$ goes to $v'$ and replaces the original half-edge on $v'$. After contracting all the unstable vertices in $\operatorname{for}_{S_{\Delta}}(\Gamma)$ repeatedly,  the tail $i$ is contracted to a unique vertex in $\mathcal T_S \Delta$, denoted by $\chi(i)$,  and stops being unstable. Similarly, the vertices $u\in V(\Delta)$ not being contracted are one-to-one corespondent to the vertices of  $\xi(u) \in \operatorname{for}_{S_{\Delta}}(\Gamma)$, and there is an natural identification between $\oPMb_u$ and $\oPMb_{\xi(u)}$.
    \end{rmk}

    \begin{nn}\label{nn bad notation}
        Let $\Delta$ be a stable graded $r$-spin graph, $F_\Delta\to \oPMb_\Delta$ be a vector bundle, and $F_v\to \oPMb_v$ be vector bundle for each $v\in V(\Delta)$, such that 
        $$
        q^*(\bboxplus_{v\in V(\Delta)}F_v)=\mu^* F_\Delta
        $$
        where $q$ and $\mu$ are the morphisms in \S\ref{subsec decomp}. If a multisecion $s_\Delta$ of $F_\Delta$ and multisections $s_v$ of $F_v$ satisfy
        $$
        \mu^* s_\Delta=q^*(\bboxplus_{v\in V(\Delta)}s_v),
        $$
        we will, by abuse of notation, write
        $$
        s_\Delta=\bboxplus_{v\in V(\Delta)}s_v\quad \text{or} \quad s_\Delta=\bigoplus_{v\in V(\Delta)} \operatorname{For}^*_{\Delta \to v} s_v
        $$
        in the sense that $s_\Delta$ is the gluing of $s_v$.
    \end{nn}

    \begin{dfn}
        In the above notations, for an $r$-spin graph $\Delta \in \partial^*\Gamma$, we say the multisection $\bm{s}_{\Gamma}$ is \textit{decomposable at $\Delta$} if 
        \begin{itemize}
            \item  The restriction $\bm{s}^{\mathcal W}_{\Delta}$ of $\bm{s}^{\mathcal W}$ to $\oPMb_\Delta\subset \oPMb_\Gamma$ is of the form 
    \begin{equation}\label{eq definition of decomposable of Witten}
        \bm{s}^{\mathcal W}_{\Delta}= \bboxplus_{u\in V^O(\Delta)} \bm{s}_u^{\mathcal W}\Ass\bass_{u\in V^C(\Delta)}\bm{s}_u^{\mathcal W},
    \end{equation}
    where for $u\in V^C(\Delta)$, $\bm{s}_u^{\mathcal W}$ is a coherent multisection (see Section \ref{sec ass}) of $\mathcal W_u\to \overline{\mathcal R}_u$ such that $\overline{\bm{s}_u^{\mathcal W}}$ is a multisection of $\mathcal W_u\to \Mbar_u=\oPMb_u$; for $u\in V^O(\Delta)$, $\bm{s}_u^{\mathcal W}$ is a multisection of $\mathcal W_u\to \oPMb_u$. Recall that, unless $u$ is closed and has an anchor of twist $-1$, we have $\overline{\mathcal R}_u=\Mbar_u$ and $\overline{\bm{s}_u^{\mathcal W}}=\bm{s}_u^{\mathcal W}$ is a  multisection of $\mathcal W_u\to \oPMb_u$. 
            
    \item
    For each $s^j_i$, its restriction $s^j_{i,\Delta}$ to $\oPMb_\Delta$, is of the form
    \begin{equation}\label{eq definition of decomposable of section of Li}
        s^j_{i,\Delta}=\operatorname{For}^*_{\Delta\to \mathcal T_S \Delta}\pi^*_{\mathcal T_S \Delta \to u^{\mathcal T_S \Delta}_{i}} \,\,s^j_{i,u^{\mathcal T_S \Delta}_{i}},
    \end{equation}
    under the natural identification between $$\operatorname{For}^*_{\Gamma \to \operatorname{for}_{S}(\Gamma)}{{\mathbb L}}_{i,\operatorname{for}_{S}(\Gamma)}\big\vert_{\oPMb_{\Delta}}$$ 
    and 
    $$\operatorname{For}^*_{\Delta\to \mathcal T_S \Delta}\pi^*_{\mathcal T_S \Delta \to u^{\mathcal T_S \Delta}_{i}}\operatorname{For}^*_{u^{\mathcal T_S \Delta}_{i} \to \operatorname{for}_{S\cap T^I(u^{\mathcal T_S \Delta}_{i})}(u^{\mathcal T_S \Delta}_{i})}\mathbb L_{i,\operatorname{for}_{S\cap T^I(u^{\mathcal T_S \Delta}_{i})}(u^{\mathcal T_S \Delta}_{i})},$$ 
    where $\operatorname{For}^*_{\Delta\to \mathcal T_S \Delta}\colon \oPMb_{\Delta}\to \oPMb_{\mathcal T_S \Delta}$ is the forgetful morphism, $u^{\mathcal T_S \Delta}_{i}$ is the unique vertex of $\mathcal T_S \Delta$ containing the internal marking $i\in I$, 
    and $s^j_{i,u^{\mathcal T_S\Delta}_{i}}$ is a multisection of
    $$\operatorname{For}^*_{u^{\mathcal T_S \Delta}_{i} \to \operatorname{for}_{S\cap T^I(u^{\mathcal T_S \Delta}_{i})}(u^{\mathcal T_S \Delta}_{i})}\mathbb L_{i,\operatorname{for}_{S\cap T^I(u^{\mathcal T_S \Delta}_{i})}(u^{\mathcal T_S \Delta}_{i})}\to \oPMb_{u^{\mathcal T_S \Delta}_{i}}.$$

    \end{itemize}
    Furthermore, we say  $\bm{s}_{\Gamma}$ is \textit{transversely decomposable at $\Delta$} if for each vertices $v\in \mathcal T_S \Delta$, 
    $$
    s_v:=\overline{\bm{s}_{\xi^{-1}(v)}^{\mathcal W}} \oplus \bigoplus_{i\in I \cap T^I(v)}\bigoplus_{1\le j \le d_i} s^j_{i,v}
    $$
    is transverse to zero as a multisection of 
    $$\mathcal W_{v}\oplus \bigoplus_{i\in I \cap T^I(v)}\mathbb L_i^{\oplus d_i}\to \oPMb_v\cong \oPMb_{\xi^{-1}(v)}.$$
    
    \end{dfn}

    \begin{rmk}
        Intuitively speaking, if we ignore the subtlety caused by the assembling operator, \textit{i.e.,} if we pretend the Witten bundles decompose as direct sum, then $s$ is decomposable at $\Delta$ means we can write
        $$
        s_\Delta=\operatorname{For}^*_{\Delta \to  \mathcal T_S \Delta}\bboxplus_{v\in \mathcal T_S \Delta} s_v,
        $$
        and  $s$ is transversely decomposable at $\Delta$ means every $s_v$ is transverse to zero.
        
    \end{rmk}

    In the case $S=\emptyset$, we always have $\mathcal T_S \Delta=\Delta$. In this case, 
    let $I^{d>0}_{\Gamma}:=\{i\in T^I(\Gamma)\colon d_i\ge 1\}$ be a subset of $T^I(\Gamma)$, let $\gamma\colon \{1,2,\dots, \lvert I_{\Gamma}^{d>0}\rvert\}\to I_{\Gamma}^{d>0}$ be a bijection, we regard it as an order (also denoted by $\le_{\gamma}$) on the set $I^{d>0}_{\Gamma}$.
    To simplify the notation, we refer to the multisections 
$$s_{\gamma(1)}^1, s_{\gamma(1)}^2, \dots, s_{\gamma(1)}^{d_{\gamma(1)}}, s_{\gamma(2)}^1, \dots, s_{\gamma(2)}^{d_{\gamma(2)}}, \dots, s_{\gamma(\lvert I^{d>0}\rvert)}^1, \dots, s_{\gamma(\lvert I^{d>0}\rvert)}^{d_{\gamma(\lvert I^{d>0}\rvert)}}$$
as 
$$t_1,t_2,\dots,t_{D},$$
where $D=\sum_{i\in I} d_i$, and we set $t_0:=\bm{s}^{\mathcal W}$. We also set 
$T_s:=\bigoplus_{m=0}^s t_m.$ We define a map $\nu \colon \{1,2,\dots, D\} \to  I_{\Gamma}^{d>0}$ in the way that $t_m$ is a multisection of $\mathbb L_{\nu(m)}$. If $t_m=s^j_{\nu(m)}$ is decomposable at $\Delta$, for any $u\in V(\Delta)$, we set $t_{m,u}=s^j_{\nu(m),u}$ if $u$ is the unique vertex of $\Delta$ such that $\nu(m)\in T^I(u)$, and formally set $t_{m,u}=0$ as a section of the rank-zero bundle if $\nu(m)\notin T^I(u)$.
    \begin{dfn}\label{def tame perturbation}
    With the above notations, for a smooth stable genus-one $r$-spin graph $\Gamma$, in the case $S=\emptyset\subseteq T^I\setminus I$, we say $\bm{s}_{\Gamma}$ is \textit{$\gamma$-tame perturbation of TRR} if 
    \begin{enumerate}
        \item For all $\Delta\in \partial^* \Gamma$, $\bm{s}_{\Gamma}$ is  decomposable at $\Delta$.
        \item For all $\Delta \in \partial^* \Gamma$ and all $u\in V(\Delta)$ and $0\le j \le D$, the multisection 
        $${T_{j,u}}:=\overline{\bm{s}^{\mathcal W}_u} \oplus t_{1,u}\oplus t_{2,u}\oplus\dots \oplus t_{j,u}$$ 
        of 
        $$E_{j,u}:=\mathcal W_{u}\oplus \bigoplus_{i\in T^I(u)\cap \nu(\{1,2,\dots,j\})} \mathbb L_i\to \oPMb_u$$ is transverse to zero.
    
    \item \label{item def tame K}
    For all open vertex $u\in V(\Delta)$ as above, if $\nu(j)\in T^I(u)$, there exists an open neighbourhood $K_{j,u}$  of the zero locus (see Observation \ref{obs vanishing loci of single trr}) $\left( \bigcup_{\Lambda\in \mathcal G_{\nu(j)}(u)}\Mbar_\Lambda\right)\cap \oPMb_u$ of $\mathfrak s_{\nu(j),u}$ in $\oPMb_u$, such that for any $p\in\oPMb_u\setminus K_{j,u}$, we have $t_{j,u}(p)=\mathfrak s_{\nu(j),u}(p)$, where $\mathfrak s_{\nu(j),u}$ is the TRR multisection of $\mathbb L_{\nu(j)}\to \oPMb_u$ (with respect to the rooted or 2-rooted structure of $u$ if $u$ is genus-zero) defined in \S \ref{sec def of trr sections}.

     \item \label{item def tame non vanish in K rk>dim}
    If $\rk  E_{s,u}> \dim \oPMb_u$ for an open vertex $u\in V(\Delta)$ and $\nu(s)\in T^I(u)$, then $T_{s-1,u}$ vanishes nowhere on $K_{s,u}\subseteq\oPMb_u$.

     \item \label{item def tame non vanish in K}
    If $\rk  E_{s,u}= \dim \oPMb_u$ for an open genus-one vertex $u\in V(\Delta)$ and $\nu(s)\in T^I(u)$, then $T_{s-1,u}$ vanishes nowhere on $K_{s,u}\subseteq\oPMb_u$.
   
    \end{enumerate}
\end{dfn}

\begin{rmk} \label{rmk on restriction of tame perturbation on g=1 vertex}
    Note that, assuming $\bm{s}_{\Gamma}$ is a $\gamma$-tame perturbation of TRR, for any $\Delta \in \partial^* \Gamma$, if $u\in V(\Delta)$ is genus-one, then for any $0\le j \le D$ the multisection $T_{j,u}$ is a $\gamma\vert_{u,j}$-tame perturbation of TRR, where $\gamma\vert_{u,j}$ is the order induced by $\gamma$ on the set $\{i\in I_{\Gamma}^{d>0} \cap T^I(u)\colon i\le_{\gamma} \nu(j)\}\subseteq I_{\Gamma}^{d>0}$.
\end{rmk}

\begin{rmk}\label{rmk on non vanish tame rk >= dim}
    For any open vertex $u$ with  $\rk  E_{s,u}> \dim \oPMb_u$,  $T_{s,u}$ vanishes nowhere on $\oPMb_u$ by transversality. Moreover, as a corollary of item \ref{item def tame K} and \ref{item def tame non vanish in K}, if $\rk  E_{s,u}= \dim \oPMb_u$ for an open genus-one vertex $u\in V(\Gamma)$, then $T_{s,u}$ also vanishes nowhere on $\oPMb_u$. 
\end{rmk}

\begin{dfn}\label{dfn canonical section}
 For $g=0$ or $1$, let $I,B$ be the sets of internal and boundary markings such that $\lvert B \rvert +2 \lvert I \rvert\ge 3-2g$, we consider the vector bundle
$$E:=\mathcal W\oplus \bigoplus_{i\in I} \mathbb L_i^{\oplus d_i}\to \oPMb^{1/r,0}_{g,B,I}.$$
Let $\gamma$ be an order on the set $I^{d>0}:=\{i\in I\colon d_i\ge 1\}$. 
For a multisection $\bm{s}$ of $E$,  we write $\bm{s}=s^{\mathcal W}\oplus\bigoplus_{i\in I}\bigoplus_{1\le t \le d_i}s_i^t,$ where $s^{\mathcal W}$ is a multisection of $\mathcal W\to \oPMb^{1/r,0}_{g,B,I}$, and each $s_i^t$ is a multisection of $\mathbb L_i\to \oPMb^{1/r,0}_{g,B,I}$.
     We call $s^{\mathcal W}$ and $\bigoplus_{i\in I}\bigoplus_{1\le t \le d_i}s_i^t$ the $\mathcal W$-part and $\mathbb L$-part of $\bm{s}$ respectively.
     We say the multisection $\bm{s}$ is \textit{$\gamma$-canonical} if the following conditions hold.
     \begin{enumerate}
         \item\label{item canonical glue} The multisection $\bm{s}$ glues to a multisection of 
         $$\widetilde{E}:=\widetilde{\mathcal W}\oplus \bigoplus_{i\in I} \widetilde{\mathbb L}_i^{\oplus d_i}\to \widetilde{\mathcal {PM}}^{1/r,0}_{g,B,I},$$
         \textit{i.e.,} the restrictions of $\bm{s}$ to the type-AI boundaries of $\oPMb^{1/r,0}_{g,B,I}$ are identified with the restrictions of $\bm{s}$ to the corresponding type-BI boundaries of $\oPMb^{1/r,0}_{g,B,I}$ under $\sim_{PI}$.
         \item\label{item canonical decompose} We denote by $\bm{s}_{\bm{G}}$ the restriction of $\bm{s}$ to $\oPMb_{\bm{G}}\subseteq  \oPMb^{1/r,0}_{1,B,I}$ for  $\bm{G}\in \sGPI^{r,0}_{1,B,I}$, we require that 
         \begin{equation}\label{eq canonical on each connected component}
            \bm{s}_{\bm{G}}=\bboxplus_{\Gamma\in V(\bm{G})} \bm{s}_{\Gamma},
         \end{equation}
         where $\bm{s}_{\Gamma}=s_{\Gamma}^{\mathcal W}  \oplus \bigoplus_{i\in I\cap T^I(\Gamma)}\bigoplus_{1\le t \le d_i}s_{i,\Gamma}^t$ is a transverse multisection of the vector bundle
         $$E_{\Gamma}:=\mathcal W_{\Gamma}\oplus \bigoplus_{i\in I\cap T^I(\Gamma)} \operatorname{For}^*_{\Gamma\to \mathcal B \Gamma}\mathbb L_{i,\mathcal B \Gamma}^{\oplus d_i}\to \oPMb_\Gamma.$$

        \item \label{item canonical forget g=1} In the case $\bm{G}$ is genus-one (\textit{i.e.}, $g=1$), for any $\Gamma\in V(\bm G)$, we denote by $\hat{\mathcal B }\Gamma$ the graph obtained by forgetting all the internal tails of $\Gamma$ in $I^{PI}_{sp,1}(\mathbf{G})\cap T^I(\Gamma)\subset T^I(I)$. 
        Notice that all the tails in $I^{PI}_{sp,1}$ have twist zero, therefore  $\hat{\mathcal B }\Gamma$ is still a $r$-spin graph, and it is always stable when $T(\Gamma)\setminus RT(\Gamma) \ne \emptyset$.
        We denote by $\operatorname{For}_{\Gamma \to \hat{\mathcal B }\Gamma}\colon \oPMb_{\Gamma} \to \oPMb_{\hat{\mathcal B }\Gamma}$ the forgetful morphism. We require that $\bm{s}_{\Gamma}^{}$ is of the form 
        \begin{equation}\label{eq canonical only see roots}
        \bm{s}_{\Gamma}^{}=\operatorname{For}^*_{\Gamma \to \hat{\mathcal B }\Gamma} \bm{s}_{\hat{\mathcal B }\Gamma}^{}
        \end{equation}
        under the natural identification between $\operatorname{For}_{\Gamma \to \hat{\mathcal B }\Gamma} E_{\hat{\mathcal B }\Gamma}$ and $E_\Gamma$ (when $\hat{\mathcal B }\Gamma$ unstable these bundles are rank-zero), where $\bm{s}_{\hat{\mathcal B }\Gamma}^{}$ is a transverse multisection of the vector bundle
        $$E_{\hat{\mathcal B }\Gamma}:=\mathcal W_{\hat{\mathcal B}\Gamma}\oplus \bigoplus_{i\in I\cap T^I(\Gamma)} \operatorname{For}^*_{\hat{\mathcal B}\Gamma\to \mathcal B \Gamma}\mathbb L_{i,\mathcal B \Gamma}^{\oplus d_i}\to \oPMb_{\hat{\mathcal B }\Gamma}.$$

        \item \label{item canonical transversely decompose}
         When $\hat{\mathcal B}\Gamma$ is stable, we can decompose $\bm{s}_{\hat{\mathcal B }\Gamma}^{}$ with respect to the order by writing
         $$
         \bm{s}_{\hat{\mathcal B }\Gamma}^{}=t_{0,\hat{\mathcal B }\Gamma}\oplus t_{1,\hat{\mathcal B }\Gamma}\oplus\dots \oplus t_{D,\hat{\mathcal B }\Gamma},
         $$
            where $t_{0,\hat{\mathcal B }\Gamma}$ is a multisection of $\mathcal W_{\hat{\mathcal B }\Gamma}\to \oPMb_{\hat{\mathcal B }\Gamma}$, and for $1\le j \le D$, $t_{j,\hat{\mathcal B }\Gamma}$ is a multisection of $\operatorname{For}^*_{\hat{\mathcal B }\Gamma\to {\mathcal B }\Gamma}\mathbb L_{\nu(j),{\mathcal B }\Gamma}\to \oPMb_{\hat{\mathcal B }\Gamma}$ which pulls back to $t_j$ if $\nu(j)\in T^I(\Gamma)$, and is just a formal section of rank-zero bundle otherwise.
             Then, for any $\Delta\in \partial^* \hat{\mathcal B}\Gamma$ and any $0\le s \le D$, the multisection $$\bigoplus_{j=0}^s t_{j,\hat{\mathcal B }\Gamma}$$ is  transversely decomposable at $\Delta$. 
         
         \item \label{item canonical positive} If $\Gamma\in V(\bm G)$ is genus-zero, then the $\mathcal W$-part $\bm{s}^{\mathcal W}_\Gamma$ of section $\bm{s}_\Gamma$, as a multisection of $\mathcal W_\Gamma \to \oPMb_\Gamma$, is positive (see Definition \ref{def positive section}).
         
         \item \label{item canonical tame g=1} If $\Gamma\in V(\bm G)$ is genus-one and $\hat{\mathcal B}\Gamma$ is stable, then the multisection $\bm{s}_{\hat{\mathcal B} \Gamma}$  is a $\gamma\vert_{\hat{\mathcal B} \Gamma}$-tame perturbation of TRR section, where  $\gamma\vert_{\hat{\mathcal B} \Gamma}$ is the order on $I^{d>0}_{\hat{\mathcal B} \Gamma}=T^I(\hat{\mathcal B} \Gamma)\cap I^{d>0}\subseteq I^{d>0}$ induced by $\gamma$.
         Note that in this case we have ${\hat{\mathcal B} \Gamma}={{\mathcal B} \Gamma}$.
     \end{enumerate}
\end{dfn}

\begin{prop}\label{prop exist canonical section}
    With the above notations, for $g=0$ or $1$, for any order $\gamma$ on $I^{d>0}$, there exist at least one $\gamma$-canonical multisection of $E\to \oPMb_{g,B,I}^{1/r}$.
\end{prop}
We postpone the construction of $\gamma$-canonical multisections in the $g=1$ case to \S \ref{sec construction of section}. A construction in the $g=0$ case can be found in Proposition \ref{prop compare with BCT in g=0}.

\begin{obs} \label{obs non vanish when exist g=1 component}
    Let $\bm{s}$ be a $\gamma$-canonical multisection of $E$ in the case $\rk E=\dim\oPMb_{g,B,I}^{1/r}$, then $\bm{s}$ vanishes nowhere on $\oPMb_{\bm{G}}$ for $\bm{G}\in \sGPI^{r,0}_{g,B,I}$ unless for any $\Gamma \in V(\bm{G})$ we have 
    $\hat{\mathcal B}\Gamma=\Gamma$ and $\rk E_{\Gamma}=\dim\Mbar_{\Gamma}$. This follows from the transversality of $\bm{s}_{\hat{\mathcal B}\Gamma}$, and the equations 
    $$\sum_{\Gamma \in V(\bm{G})} \rk E_\Gamma = \rk E, \quad \sum_{\Gamma \in V(\bm{G})} \dim \oPMb_\Gamma=\dim \Mbar_{g,B,I}^{1/r},$$
    $$ \rk E_{\hat{\mathcal B}\Gamma}=\rk E_\Gamma, \text{ and } \dim \oPMb_\Gamma \ge \dim \oPMb_{\hat{\mathcal B}\Gamma}.$$
    In particular,  as a result of Remark \ref{rmk on non vanish tame rk >= dim} and item \ref{item canonical tame g=1} in Definition \ref{dfn canonical section}, $\bm{s}$ vanishes nowhere on $\oPMb_{\bm{G}}$ as long as there exist a genus-one $\Gamma \in V(\bm{G})$. 
\end{obs}

\begin{obs}\label{obs non vanish of section near boundary}
     Let $\bm{s}$ be a $\gamma$-canonical multisection of $E$ in the case $\rk E=\dim\oPMb_{g,B,I}^{1/r}$, then for any $\bm{G}\in \sGPI^{r,0}_{g,B,I}$, there exists an open neighbourhood $U_{\bm{G}}\subseteq \Mbar_{\bm{G}}$ of 
     $$\partial^{R}\Mbar_{\bm{G}}\cup \partial^{NS+}\Mbar_{\bm{G}}\cup \mathcal Z^{dj}_{\bm{G}}\subset \Mbar_{\bm{G}},$$ 
     such that $\bm{s}$ vanishes nowhere on 
     $$
     (U_{\bm{G}} \cup \partial^{CB}\Mbar_{\bm{G}})\cap \oPMb_{\bm{G}}.
     $$
     In fact, when there exist genus-one $\Gamma \in V(\bm{G})$, we can take $U_{\bm{G}}=\Mbar_{\bm{G}}$ according to Observation \ref{obs non vanish when exist g=1 component}. When all $\Gamma \in V(\bm{G})$ are genus-zero, then $\mathcal Z^{dj}_{\bm{G}}=\emptyset$ and the existence of $U_{\bm{G}}$ follows from the positivity constraint for the $\mathcal W$-part on genus-zero components (\textit{i.e.} item \ref{item canonical positive} in Definition \ref{dfn canonical section}).  Note that $\bm{s}$ also nowhere vanishes on
     $$( \partial^{BI}\Mbar_{\bm{G}}\cup \partial^{AI}\Mbar_{\bm{G}})\cap \oPMb_{\bm{G}}$$
     as consequence of item \ref{item canonical transversely decompose} in Definition \ref{dfn canonical section}. 
\end{obs}

Because $\oPMb_{\bm{G}}\setminus U_{\bm{G}}$ is compact, we can define (see \cite[Appendix A]{BCT2}) the Euler number
$$\int_{\oPMb_{\bm{G}}}e(E ; \mathbf{s})\in\mathbb{Q}$$
with respect to the multisection $\bm{s}$  as the zero count (with multiplicity) $\# Z(\bm{s})$ of $\bm{s}$ in $\oPMb_{\bm{G}}$ for each $\bm{G}\in \sGPI^{r,0}_{1,B,I}$, using the canonical relative orientation of $E$ over $\oPMb_{\bm{G}}$ constructed in \cite{TZ1}.

\begin{dfn}\label{dfn correlator}
With the above notations, for $g=0$ or $1$, if $\text{rank} E=\dim\oPMb_{g,B,I}^{1/r}$, we refer to the integral 
\[
\left\langle \prod_{i\in I}\tau^{a_i}_{d_i}\prod_{i \in B}\sigma^{b_i}\right\rangle^{\frac{1}{r},o,\gamma,\mathbf{s}}_g := \int_{\oPMb^{1/r,0}_{g,B,I}}e(E ; \mathbf{s})\in\mathbb{Q}, 
\]
where $\mathbf{s}$ is a $\gamma$-canonical multisection, 
as an \emph{open genus-$g$ $r$-spin intersection number with respect to $\gamma$ and $\bm{s}$}. For shortness we sometimes call the above number an \emph{ordered open genus-one $r$-spin intersection number}, or a \emph{correlator}. When $\text{rank}(E)\neq\dim\oPMb_{g,B,I}^{1/r}$, the integral is defined to be zero. 
\end{dfn}

\begin{rmk}\label{rmk vanish without psi g=1}
    If $d_i=0$ for all $i\in I$, then the genus-one correlators are always zero because in this case by \eqref{eq rank of witten bundle g=1} we have 
    $
    \rk E=\rk \mathcal W_{1,B,I}<\dim \oPMb_{1,B,I}^{1/r}.
    $
\end{rmk}

\subsection{Genus-zero case}
     In this subsection, we will show the above correlators are independent of the choice of the order $\gamma$ and multisection $\mathbf{s}$ in $g=0$ case. We will also compare them with the genus-zero open $r$-spin correlators defined in \cite{BCT2}. The proof of independence of choice in the $g=1$ case is postponed to \S\ref{sec g=1 indepent of choice and trr}.

     In the $g=0$ case, the $\gamma$-canonicity in Definition \ref{dfn canonical section} is in fact independent of the order $\gamma$. We simply refer to the $\gamma$-canonical multisections as \textit{canonical multisections} in $g=0$ case.
     
     \begin{prop}\label{prop independent of choice g=0}
    For any two canonical multisections $\mathbf{s}^a$ and $\mathbf{s}^b$ of 
    $$E:=\mathcal W\oplus \bigoplus_{i\in I} \mathbb L_i^{\oplus d_i}\to \oPMb^{1/r,0}_{0,B,I},$$
    we have 
    $$
    \int_{\oPMb_{0,B,I}}e(E ; \mathbf{s}^a)=\int_{\oPMb_{0,B,I}}e(E ; \mathbf{s}^b),
    $$
    \textit{i.e.,} the genus-zero open $r$-spin correlators are independent of the choice of the order $\gamma$ and canonical multisection $\bm{s}$, and we can write 
    $$\left\langle \prod_{i\in I}\tau^{a_i}_{d_i}\prod_{i \in B}\sigma^{b_i}\right\rangle^{\frac{1}{r},o}_0:=\left\langle \prod_{i\in I}\tau^{a_i}_{d_i}\prod_{i \in B}\sigma^{b_i}\right\rangle^{\frac{1}{r},o,\gamma,\mathbf{s}}_0$$
    for any such choice.
    \end{prop}
    \begin{proof}
        This is almost a special case ($\h=0$) of \cite[Theorem 3.9]{TZ2}. Although in \cite{TZ2} the relative cotangent line bundles are $\mathbb L'_i$,  the proof is identical. We outline the key ideas of the proof. 

        We regard $\mathbf{s}^a$ and $\mathbf{s}^b$ as multisections over $\widetilde{\mathcal{PM}}^{1/r,0}_{0,B,I}\subseteq \widetilde{\mathcal{M}}^{1/r,0}_{0,B,I}$. Note that $\widetilde{\mathcal{M}}^{1/r,0}_{0,B,I}$ is a compact space with boundaries of type CB, NS+ or R, while $\widetilde{\mathcal{PM}}^{1/r,0}_{0,B,I}\subseteq \widetilde{\mathcal{M}}^{1/r,0}_{0,B,I}$ only has  boundaries of type CB. We write 
        $$
        \partial \widetilde{\mathcal{M}}^{1/r,0}_{0,B,I}= \partial^{CB}\widetilde{\mathcal{M}}^{1/r,0}_{0,B,I} \cup \partial^+ \widetilde{\mathcal{M}}^{1/r,0}_{0,B,I},
        $$
        where $\partial^+ \widetilde{\mathcal{M}}^{1/r,0}_{0,B,I}$ is the union of all type-NS+ and type-R boundaries.
        
        Since both $\mathbf{s}^a$ and $\mathbf{s}^b$ are canonical, for any $0\le t \le 1$, the multisection $(1-t)\mathbf{s}^a+t\mathbf{s}^b$ does not vanish at any point  $p\in \widetilde{\mathcal{PM}}^{1/r,0}_{0,B,I} \cap \partial^{CB}\widetilde{\mathcal{M}}^{1/r,0}_{0,B,I}$. In fact, assuming $p\in \oPMb_{\Gamma^{CB}}\times \prod_{\Gamma\in V(\mathbf{G})\setminus \{\Gamma^{CB}\}}\oPMb_{\Gamma}$ where $\Gamma^{CB}$ has a contracted boundary tail, then both $\mathbf{s}^a_{\Gamma^{CB}}$,  $\mathbf{s}^b_{\Gamma^{CB}}$ and $(1-t)\mathbf{s}^a_{\Gamma^{CB}}+t\mathbf{s}^b_{\Gamma^{CB}}$
        evaluate positively at the corresponding contracted boundary node. For similar reason, there exists an open neighbourhood $U^+\subseteq \widetilde{\mathcal{M}}^{1/r,0}_{0,B,I}$ of $\partial^+ \widetilde{\mathcal{M}}^{1/r,0}_{0,B,I}\subseteq \widetilde{\mathcal{M}}^{1/r,0}_{0,B,I}$, such that for any $0\le t \le 1$, the multisection $(1-t)\mathbf{s}^a+t\mathbf{s}^b$ does not vanish at any point  $p\in \widetilde{\mathcal{PM}}^{1/r,0}_{0,B,I} \cap U^+$.

        Therefore, $(1-t)\mathbf{s}^a+t\mathbf{s}^b$ is a homotopy between $\mathbf{s}^a$ and $\mathbf{s}^a$ which vanishes nowhere on $\widetilde{\mathcal{PM}}^{1/r,0}_{0,B,I} \cap (U^+\cup \partial^{CB}\widetilde{\mathcal{M}}^{1/r,0}_{0,B,I})$, we can conclude the proof using \cite[Lemma 4.12]{BCT2}.

    \end{proof}

    In \cite{BCT1,BCT2}, an genus-zero open $r$-spin theory, to which we refer as the BCT theory, was constructed in an different way. For a graded $r$-spin graph $\Gamma,$ we denote by $\mathcal {F}\Gamma$ the graph obtained by detaching all the NS edges of $\Gamma$ whose twist at the illegal half-edge is $0$, then forgetting all the twist-$0$ illegal half-edges. Note that forgetting tails may produce unstable components $\mathcal {F} \Gamma$; we view the moduli space corresponding to an unstable component, which is a direct summand of $\Mbar_{\mathcal F \Gamma}$, as a single point.  Let $F_\Gamma\colon \Mbar_\Gamma \to \Mbar_{\mathcal F\Gamma}$ be the induced map in the level of moduli spaces.
     \cite[Proposition 6.2]{BCT2} constructed the \emph{special BCT-canonical multisections}, which are collections of transverse multisections $\hat s^{\text{BCT}}_\Gamma$ of 
     $$E_{\Gamma}:=\mathcal W_{\Gamma}\oplus \bigoplus_{i\in T^I(\Gamma)\cap I} \mathbb L_i^{\oplus d_i} \to \oPMb_\Gamma$$
     for every graded $r$-spin graph $\Gamma$, which satisfy, in addition to the same positivity as in Definition \ref{def positive section}, that for every $\Gamma\in \partial^* \Delta$ 
        \begin{equation}\label{eq bct special canonical}
         \hat s^{\text{BCT}}_\Delta\big\vert_{ \Mbar_{\Gamma_{}}}=F_\Gamma^* \hat s^{\text{BCT}}_{\mathcal F \Gamma}.
        \end{equation}
        Note that the moduli space $\oPMb^{1/r}_{0,B,I}$ is the disjoint union
        $$
        \oPMb^{1/r}_{0,B,I}=\bigsqcup_{\substack{\Gamma \text{ smooth}\\ T^I(\Gamma)=I,T^B(\Gamma)=B}} \oPMb_{\Gamma},
        $$
        the collection of the multisections $\hat s^{\text{BCT}}_\Gamma$ for such $\Gamma$ is a multisection $\hat {s}_{0,B,I}^{BCT}$ of the vector bundle 
        $$
        \mathcal W_{\Gamma}\oplus \bigoplus_{i\in T^I(\Gamma)\cap I} \mathbb L_i^{\oplus d_i} \to \oPMb^{1/r}_{0,B,I}.
        $$
        The BCT-correlators are the number of zero of such BCT-canonical multisections, \textit{i.e.,}
     \[
    \left\langle \prod_{i\in I}\tau^{a_i}_{d_i}\prod_{i \in B}\sigma^{b_i}\right\rangle^{\frac{1}{r},o,BCT}_0 = \int_{\oPMb^{1/r}_{0,B,I}}e(E ; \hat {s}_{0,B,I}^{BCT})\in\mathbb{Q}. 
    \]

    \begin{prop}\label{prop compare with BCT in g=0} 
    There exists at least one canonical multisection of $$E=\mathcal W\oplus \bigoplus_{i\in I} \mathbb L_i^{\oplus d_i}\to \oPMb^{1/r,0}_{0,B,I}.$$
    Moreover, we have

   \begin{equation}\label{eq compare with BCT}
    \left\langle \prod_{i\in I}\tau^{a_i}_{d_i}\prod_{i \in B}\sigma^{b_i}\right\rangle^{\frac{1}{r},o}_0 
    =
    \left\langle \prod_{i\in I}\tau^{a_i}_{d_i}\prod_{i \in B}\sigma^{b_i}\right\rangle^{\frac{1}{r},o,BCT}_0.
    \end{equation}
        
    \end{prop}

    \begin{proof}
        Note that $\oPMb^{1/r}_{0,B,I}$ can be regard as a subspace of $\widetilde{\mathcal{PM}}^{1/r,0}_{0,B,I}$, similar to the proof of \cite[Theorem 5.2]{TZ2}, the strategy is to extend to multisection $\mathbf{s}^{BCT}$ over $\oPMb^{1/r}_{0,B,I}\subset \widetilde{\mathcal{PM}}^{1/r}_{0,B,I}$ to a canonical multisection over the entire $\widetilde{\mathcal{PM}}^{1/r,0}_{0,B,I}$ which vanishes nowhere on $\widetilde{\mathcal{PM}}^{1/r,0}_{0,B,I} \setminus \oPMb^{1/r}_{0,B,I}$.

    For any $\oPMb_{\bm{G}}\subseteq \widetilde{\mathcal{PM}}^{1/r,0}_{0,B,I} $, recall that for $\Gamma \in V(\bm{G})$ we can forget all inserted (twist-zero) internal tails in $T^I(\Gamma)\cap I^{PI}(\bm{G})$ and obtain a graph $\mathcal B \Gamma$. We construct $\mathbf{s}_{\mathbf{G}}$ as
    $$
    \hat{\mathbf{s}}_{\mathbf{G}}:=\bboxplus_{\Gamma \in V(\bm{G})}  \operatorname{For}^*_{\Gamma \to \mathcal B \Gamma} \hat {s}^{\text{BCT}}_{\mathcal B \Gamma}.
    $$
    which is in fact a multisection of $\mathcal W \oplus \bigoplus_{i=1}^l \mathbb L_i^{\oplus d_i}\to \oPMb_{\bm{G}}$ because we are working with $\mathbb L_i$ instead of $\mathbb L'_i$ (see \S \ref{sec uni cotangent line over glued moduli}).

    The collection $\hat {\mathbf{s}}_{\mathbf{G}}$ is canonical multisection in Definition \ref{dfn canonical section}. In fact, item \ref{item canonical glue} is a consequence of \eqref{eq bct special canonical}, item \ref{item canonical decompose} follows from the construction, item \ref{item canonical positive} is the property of the BCT-canonical multisections $\hat{s}^{BCT}_{\Gamma}$, while item \ref{item canonical forget g=1} and \ref{item canonical tame g=1} are irrelevant in the $g=0$ case.

    Over $\oPMb^{1/r}_{0,B,I}\subset \widetilde{\mathcal{PM}}^{1/r,0}_{0,B,I}$, \textit{i.e.,} the union of $\oPMb_{\bm{G}}=\oPMb_{\Gamma_{\bm{G}}}$ for single-vertex $\bm{G}$ with $V(\bm{G})=\{\Gamma_{\bm{G}}\}$, the multisection $\hat{\bm{s}}_{\mathbf{G}}$ coincides with $\hat{s}^{BCT}_{\Gamma_{\bm{G}}}=\hat {s}_{0,B,I}^{BCT}\vert_{\oPMb_{\bm{G}}}$ by definition. 
    
    Over $\widetilde{\mathcal{PM}}^{1/r}_{0,B,I}\setminus \oPMb^{1/r}_{0,B,I} $, \textit{i.e.,} the union of $\oPMb_{\bm{G}}$ for $(r,0)$-graphs $\bm{G}$ with more than one vertex; we show that $\hat{\bm{s}}_{\bm{G}}$ never vanishes for such $\bm{G}$. In fact, we have
    $$
   \sum_{\Gamma\in V(\bm{G})}\rk E_{\mathcal B\Gamma}= \sum_{\Gamma\in V(\bm{G})}\rk E_{\Gamma} =\sum_{\Gamma\in V(\bm{G})} \dim \oPMb_{\Gamma}>\sum_{\Gamma\in V(\bm{G})} \dim \oPMb_{\mathcal B\Gamma}
    $$
    because at least one inserted internal tail is forgotten, therefore $\rk E_{\mathcal B\Gamma}>\dim \oPMb_{\mathcal B\Gamma}$ for at least one $\Gamma \in V(\bm{G})$, then by transversality $\hat {s}^{\text{BCT}}_{\mathcal B \Gamma}$ vanishes nowhere, hence $\hat{\bm{s}}_{\bm{G}}$ vanishes nowhere. This proves \eqref{eq compare with BCT}.

    \end{proof}

    As a consequence, these genus-zero open $r$-spin correlators satisfy the same topological recursion relations as in \cite{BCT2}.
    
    \begin{thm}[{\cite[Theorem 4.1]{BCT2}}]\label{thm TRR BCT}
    The genus-zero open $r$-spin correlators satisfy the following two types of Topological Recursion Relations (note that all boundary twists $b_i$ are $r-2$):
\begin{itemize}
\item[(a)] Suppose $\lvert I \rvert, \lvert B \rvert\ge 1$, then for any $i_1\in I$ and $j_B\in B$, we have the following TRR with respect to $(i_1,j_B)$:
\begin{equation}\label{eq: bct trr 1}
\begin{split}
&\left\langle \tau_{d_{i_1}+1}^{a_{i_1}}\prod_{i\in I\setminus \{i_1\}}\tau^{a_i}_{d_i}\prod_{i\in B}\sigma^{b_i}\right\rangle^{\frac{1}{r},o}_0\hspace{-0.2cm}\\
=&\sum_{\substack{R_1 \sqcup R_2 = I \setminus \{{i_1}\}\\ -1\le a \le r-2}}\hspace{-0.1cm}\left\langle \tau_0^{a}\tau_{d_{i_1}}^{a_{i_1}}\prod_{i \in R_1}\tau_{d_i}^{a_i}\right\rangle^{\frac{1}{r},\text{ext}}_0
\hspace{-0.1cm}\left\langle \tau_0^{r-2-a}\prod_{i\in R_2}\tau^{a_i}_{d_i}\prod_
{i\in B}\sigma^{b_i}\right\rangle^{\frac{1}{r},o}_0\\
&+\hspace{-0.1cm}\sum_{\substack{R_1 \sqcup R_2 =  I \setminus \{{i_1}\} \\ T_1 \sqcup T_2 = B\setminus \{j_B\}}} \hspace{-0.1cm} \left\langle \tau^{a_{i_1}}_{d_{i_1}}\prod_{i \in R_1} \tau^{a_i}_{d_i}\prod_{i\in T_1}\sigma^{b_i}\right\rangle^{\frac{1}{r},o}_0 \hspace{-0.1cm}\left\langle \prod_{i \in R_2} \tau^{a_i}_{d_i} \sigma^{r-2}\sigma^{b_{j_B}}\prod_{i\in T_2}\sigma^{b_i}\right\rangle^{\frac{1}{r}, o}_0.
    \end{split}
\end{equation}

\item[(b)] Suppose $\lvert I \rvert \ge 2$,  then for any $i_1\in I$ and $j_I\in I \setminus \{i_1\}$,  we have the following TRR with respect to $(i_1,j_I)$:
\begin{equation}\label{eq: bct trr 2}
\begin{split}
&\left\langle \tau_{d_{i_1}+1}^{a_{i_1}}\prod_{i\in I \setminus \{i_1\}}\tau^{a_i}_{d_i}\prod_{i\in B}\sigma^{b_i}\right\rangle^{\frac{1}{r},o}_0\hspace{-0.2cm}\\=&
\sum_{\substack{R_1 \sqcup R_2 =  I\setminus \{i_1,j_I\}\\ -1\le a \le r-2}}\hspace{-0.1cm}\left\langle \tau_0^{a}\tau_{d_{i_1}}^{a_{i_1}}\prod_{i \in R_1}\tau_{d_i}^{a_i}\right\rangle^{\frac{1}{r},\text{ext}}_0
\hspace{-0.1cm}\left\langle \tau_0^{r-2-a}\tau^{a_{j_I}}_{d_{j_I}}\prod_{i\in R_2}\tau^{a_i}_{d_i}\prod_{i\in B}\sigma^{b_i}\right\rangle^{\frac{1}{r},o}_0\\
&+\hspace{-0.1cm}\sum_{\substack{R_1 \sqcup R_2 = I \setminus \{i_1,j_I\} \\ T_1 \sqcup T_2 = B}} \hspace{-0.1cm} \left\langle \tau^{a_{i_1}}_{d_{i_1}}\prod_{i \in R_1} \tau^{a_i}_{d_i}\prod_{i\in T_1}\sigma^{b_i}\right\rangle^{\frac{1}{r},o}_0 \hspace{-0.1cm}\left\langle\tau^{a_{j_I}}_{d_{j_I}} \prod_{i \in R_2} \tau^{a_i}_{d_i} \sigma^{r-2}\prod_{i\in T_2}\sigma^{b_i}\right\rangle^{\frac{1}{r}, o}_0.
\end{split}
\end{equation}
\end{itemize}
\end{thm}

    \begin{rmk}
        A slightly different version of genus-zero open $r$-spin theory is constructed in \cite{TZ2}. The $\h=0$ theory in \cite{TZ2} is very similar to the $g=1$ theory in this paper; the only difference is that we take different relative cotangent line bundles: in \cite{TZ2} the theory is constructed using $\mathbb L'_i$ (see Remark \ref{rmk different notation for L in TZ}), while in this paper we use $\mathbb L_i$ forgetting all the inserted points. This difference results in a polynomial relation between those correlators, see   \cite[Theorem 5.6, Remark 5.7]{TZ2}.
    \end{rmk}

\section{Genus-one Topological Recursion Relation}\label{sec g=1 trr  and inde of choice}
In this section, we show that the genus-one open $r$-spin correlators also satisfy topological recursion relations, an prove that the genus-one open $r$-spin correlators are independent of the choice of the order $\gamma$ and the canonical multisection $\bm{s}$ as a corollary. We will also prove Theorems \ref{thm:wave_func} and \ref{thm:open_string_dilaton} using topological recursion relations.

\subsection{Genus-one TRR for correlators respect to orders and multisections}

Let $\gamma$ be an order on the set $I^{d>0}=\{i\in I\colon d_i\ge 1\}$. We denote the maximal element in the order $\gamma$ by $i_1=\gamma(I)$. For any subset $J\subseteq I$, we denote by $\gamma\vert_J$ the order on $\{i\in J\colon d_i\ge 1\}$ induced by $\gamma$.
\begin{prop}\label{prop trr g=1 ordered}
Let $i_1=\gamma(I)$ be the maximal element in $I^{d>0}$ with respect to the order $\gamma$ and $\mathbf{s}$ be a $\gamma$-canonical multisection, we have
\begin{equation}\label{eq trr g=1 ordered}
\begin{split}
&\left\langle \tau_{d_{i_1}+1}^{a_{i_1}}\prod_{i\in I\setminus \{i_1\}}\tau^{a_i}_{d_i}\prod_{i\in B}\sigma^{b_i}\right\rangle^{\frac{1}{r},o,\gamma,\bm{s}}_1\hspace{-0.2cm}\\
=&\sum_{\substack{R_1 \sqcup R_2 = I\setminus\{i_1\}\\ -1\le a \le r-2}}\hspace{-0.1cm}\left\langle \tau_0^{a}\tau_{d_{i_1}}^{a_{i_1}}\prod_{i \in R_1}\tau_{d_i}^{a_i}\right\rangle^{\frac{1}{r},\text{ext}}_0
\hspace{-0.1cm}\left\langle \tau_0^{r-2-a}\prod_{i\in R_2}\tau^{a_i}_{d_i}\prod_{i\in B}\sigma^{b_i}\right\rangle^{\frac{1}{r},o,\gamma\vert_{R_2},\bm{s}\vert_{1,B,R_2\sqcup\{i^{r-2-a}_I\}}}_1\\
&+\hspace{-0.1cm}\sum_{\substack{R_1 \sqcup R_2 =  I\setminus\{i_1\} \\ T_1 \sqcup T_2 =  B}} \hspace{-0.1cm} \left\langle \tau^{a_{i_1}}_{d_{i_1}}\prod_{i \in R_1} \tau^{a_i}_{d_i}\prod_{i\in T_1}\sigma^{b_i}\right\rangle^{\frac{1}{r},o}_0 \hspace{-0.1cm}\left\langle \prod_{i \in R_2} \tau^{a_i}_{d_i} \sigma^{r-2}\prod_{i\in T_2}\sigma^{b_i}\right\rangle^{\frac{1}{r}, o,\gamma\vert_{R_2},\bm{s}\vert_{1,T_2\sqcup\{i^{r-2}_B\},R_2}}_1\\
&+\consta\left\langle \prod_{i\in I}\tau^{a_i}_{d_i}\sigma^{r-2}\prod_{i\in B}\sigma^{b_i}\right\rangle^{\frac{1}{r},o}_0,
    \end{split}
\end{equation}   
where the notation $\bm{s}\vert_{1,I',B'}$ in the superscripts of the genus-one correlators refers to $\gamma\vert_{I'}$-canonical multisections (determined by $\bm{s}$) of the vector bundle
$$\mathcal W\oplus \bigoplus_{i\in I\cap I'} \mathbb L_i^{\oplus d_i}\to \oPMb^{1/r,0}_{1,B',I'},$$
the notation $i^{r-2-a}_I$ represents an internal markings with twist $r-2-a$, and $i^{r-2}_B$ represents a boundary markings with twist $r-2$.

\end{prop}

\begin{proof}
    Let 
    $$\bm{s}=s^{\mathcal W}\oplus\bigoplus_{i\in I\setminus
    \{i_1\}}\bigoplus_{1\le t \le d_i}s_i^t \oplus \bigoplus_{1\le t \le d_{i_1}+1}s_{i_1}^t$$ 
    be a $\gamma$-canonical multisection of $$E_{1,B,I}:=\mathcal W_{1,B,I}\oplus \bigoplus_{i\in I\setminus \{i_1\}} \mathbb L_i^{\oplus d_i} \oplus \mathbb L_{i_1}^{\oplus(d_{i_1}+1)}\to \oPMb^{1/r,0}_{1,B,I}.$$
    Under the similar notation as in \S\ref{sec sections def etc} we write $$\bm{s}=t_0\oplus t_1\oplus\dots \oplus t_D,$$
    where $t_D$ is a multisection of $\mathbb L_{i_1}$ for the maximal element $i_1\in I^{d>0}$, and we write 
    \begin{equation}\label{eq def remaining sec}
    \bm{s}^{re}:=t_0\oplus t_1\oplus\dots \oplus t_{D-1}
    \end{equation}
    as a multisection of 
    $$
    E^{re}_{1,B,I}=\mathcal W_{1,B,I} \oplus \bigoplus_{i\in I} \mathbb L_i^{\oplus d_i}\to \oPMb^{1/r,0}_{1,B,I}.
    $$

      For a $(r,0)$-graph $\bm{G}\in \sGPI^{r,0}_{1,B,I}$, we denote by $\Gamma_{i_1,\bm{G}}\in V(\bm{G})$ the unique vertex of $\bm{G}$ such that $i_1=\nu(D)\in T^I(\Gamma_{i_1,\bm{G}})$.   Since $\bm{s}=\bm{s}^{re}\oplus t_D$ is $\gamma$-canonical, for any $\bm{G}\in \sGPI^{r,0}_{1,B,I}$, we have a decomposition 
        $$
         \bm{s}_{\Gamma_{i_1,\bm{G}}}= \bm{s}^{re}_{\Gamma_{i_1,\bm{G}}}\oplus t_{D,\Gamma_{i_1,\bm{G}}}
        $$
        for a multisection $t_{D,\Gamma_{i_1,\bm{G}}}$ of $\operatorname{For}^*_{\Gamma_{i_1,\bm{G}}\to \mathcal B \Gamma_{i_1,\bm{G}}}\mathbb L_{i_1,\mathcal B \Gamma_{i_1,\bm{G}}}\to \oPMb_{\Gamma_{i_1,\bm{G}}}$ and a multisection $\bm{s}^{re}_{\Gamma_{i_1,\bm{G}}}$ of 
        $$
        E^{re}_{\Gamma_{i_1,\bm{G}}}:= \mathcal W_{{\Gamma_{i_1,\bm{G}}}} \oplus   \bigoplus_{i\in I \cap T^I(\Gamma_{i_1,\bm{G}})} \operatorname{For}^*_{\Gamma_{i_1,\bm{G}}\to \mathcal B \Gamma_{i_1,\bm{G}}}\mathbb L_{i_1,\mathcal B \Gamma_{i_1,\bm{G}}}^{d_{i}}
        \to \oPMb_{\Gamma_{i_1,\bm{G}}},
        $$
        such that $$\bm{s}^{re}_{\bm{G}}=\bm{s}^{re}_{\Gamma_{i_1,\bm{G}}} \boxplus \bboxplus_{\Gamma \in V(\bm{G})\setminus \{\Gamma_{i_1,\bm{G}}\}}\bm{s}_{\Gamma}.$$
        We have a similar decomposition for the bundles and multisections on $\oPMb_{\hat{\mathcal B}\Gamma_{i_1,\bm{G}}}$, \textit{i.e,}
        $$
        E_{\hat{\mathcal B}\Gamma_{i_1,\bm{G}}}=E^{re}_{\hat{\mathcal B}\Gamma_{i_1,\bm{G}}} \oplus \operatorname{For}^*_{\hat{\mathcal B}\Gamma_{i_1,\bm{G}}\to \mathcal B \Gamma_{i_1,\bm{G}}}\mathbb L_{i_1,\mathcal B \Gamma_{i_1,\bm{G}}}\to \oPMb_{\hat{\mathcal B}\Gamma_{i_1,\bm{G}}}
        $$
        and 
        $$
        \bm{s}_{\hat{\mathcal B}\Gamma_{i_1,\bm{G}}}= \bm{s}^{re}_{\hat{\mathcal B}\Gamma_{i_1,\bm{G}}} \oplus t_{D,\hat{\mathcal B}\Gamma_{i_1,\bm{G}}}.
        $$

    This proof is split into five steps:
    \begin{enumerate}
        \item  Construct a specific multisection $\nontilde{s}^{trr*}_{i_1}$ of $\mathbb L^*_{i_1}\to \oPMb^{1/r,0}_{1,B,I}$.
        \item Construct a multisection $\tilde{\mathfrak{t}}_{i_1,\mathbb L \to \mathbb L^*}$ of $\mathbb L^\vee_{i_1}\otimes \mathbb L^*_{i_1}\to \oPMb^{1/r,0}_{1,B,I}$.
        \item  Modify the multisection $\tilde{\mathfrak{t}}^\vee_{i_1,\mathbb L \to \mathbb L^*}\otimes \nontilde{s}^{trr*}_{i_1}$ with `poles' to a multisection $\nontilde{s}^{trr}_{i_1}$ of $\mathbb L_{i_1}\to \oPMb^{1/r,0}_{1,B,I}$ without `poles'.
        \item Show that $\bm{s}$ has the same number of zeros (with multiplicity) as $\bm{s}^{mod}:=\bm{s}^{re} \oplus \nontilde{s}^{trr}_{i_1}$ in $\oPMb^{1/r,0}_{1,B,I}$.
        \item Count the number of zeros (with multiplicity) of the restriction of $\bm{s}^{re}$ to the zero locus of the multisection $\nontilde{s}^{trr}_{i_1}$.
        
    \end{enumerate}

   \noindent\textbf{Step 1.} 
    We first construct a specific multisection of $\nontilde{s}^{trr*}_{i_1}$ of $\mathbb L^*_{i_1}\to \oPMb^{1/r,0}_{1,B,I}$.

        We need to construct a multisection $\nontilde{s}^{trr*}_{i_1,\mathbf{G}}$ of ${\mathbb L}^*_{i_1}\to \oPMb_{\bm{G}}$ for each smooth $\bm{G}$. Recall that in \S \ref{sec uni cotangent line over glued moduli}, the line bundle ${\mathbb L}^*_{i_1}\to \oPMb_{\bm{G}}$ is pulled back from $\mathbb L_{i_1,\hat{\mathcal B}\Gamma_{i_1,\bm{G}}}\to \oQMb_{\hat{\mathcal B}\Gamma_{i_1,\bm{G}}}$.  Since $\bm{G}$ is genus-one, the set 
        $$RT(\hat{\mathcal B}\Gamma_{i_1,\bm{G}}):=T(\hat{\mathcal B}\Gamma_{i_1,\bm{G}})\cap \left(I^{PI}_{sp,0}(\bm{G})\sqcup B^{PI}_{sp,0}(\bm{G})\sqcup I^{PI}_{nsp}(\bm{G})\sqcup B^{PI}_{nsp}(\bm{G})\right)$$ 
        is always empty for genus-one $\Gamma_{i_1,\mathbf{G}}$, and is either a single-element set (where $RT(\hat{\mathcal B}\Gamma_{i_1,\bm{G}})\subseteq T(\hat{\mathcal B}\Gamma_{i_1,\bm{G}})\cap \left(I^{PI}_{sp,0}(\bm{G})\sqcup B^{PI}_{sp,0}(\bm{G})\right)$) or a two-element set (where $RT(\hat{\mathcal B}\Gamma_{i_1,\bm{G}})=T(\hat{\mathcal B}\Gamma_{i_1,\bm{G}})\cap \left(I^{PI}_{nsp}(\bm{G})\sqcup B^{PI}_{nsp}(\bm{G})\right)$) for genus-zero $\Gamma_{i_1,\mathbf{G}}$. 

        In the case $RT(\hat{\mathcal B}\Gamma_{i_1,\bm{G}})$ is empty, \textit{i.e.,} $\Gamma_{i_1,\bm{G}}$ is genus-one, we have $\hat{\mathcal B}\Gamma_{i_1,\bm{G}}={\mathcal B}\Gamma_{i_1,\bm{G}}$. We take the restriction $\nontilde{s}^{trr*}_{i_1,\bm{G}}$ of $\nontilde{s}^{trr*}_{i_1}$ to $\oPMb_{\bm{G}}\subset \oPMb^{1/r,0}_{1,B,I}$ to be the multisection pulled back from the TRR multisection 
        $\mathfrak s_{i_1,\mathcal B\Gamma_{i_1,\bm{G}}}$ of $\mathbb L_{i_1,\mathcal B\Gamma_{i_1,\bm{G}}}\to \oPMb_{\mathcal B\Gamma_{i_1,\bm{G}}}=\oPMb_{\hat{\mathcal B}\Gamma_{i_1,\bm{G}}}$ via the projection $\pi_{\mathbf{G}\to \Gamma_{i_1,\mathbf{G}}}\colon \oPMb_{\mathbf{G}}\to \oPMb_{\Gamma_{i_1,\mathbf{G}}}$ and the forgetful morphism $\operatorname{For}_{\Gamma_{i_1,\bm{G}}\to \mathcal B \Gamma_{i_1,\bm{G}}}\colon \oPMb_{\Gamma_{i_1,\bm{G}}}\to \oPMb_{\mathcal B \Gamma_{i_1,\bm{G}}}$ forgetting all the inserted tails. 
        
        In the case $RT(\hat{\mathcal B}\Gamma_{i_1,\bm{G}})$ is a single-element set, we denote by $r_0\in RT(\hat{\mathcal B}\Gamma_{i_1,\bm{G}})$ its unique element and take the section $\nontilde{s}^{trr*}_{i_1,\mathbf{G}}$ of ${\mathbb L}^*_{i_1}\to \oPMb_{\bm{G}}$ to be the section pulled back from the TRR section $\mathfrak s_{i_1,\hat{\mathcal B}\Gamma_{i_1,\bm{G}},r_0}$ of $\mathbb L_{i_1,\hat{\mathcal B}\Gamma_{i_1,\bm{G}},r_0}\to \oPMb_{\hat{\mathcal B}\Gamma_{i_1,\bm{G}}}$ via the projection $\pi_{\mathbf{G}\to \Gamma_{i_1,\mathbf{G}}}\colon \oPMb_{\mathbf{G}}\to \oPMb_{\Gamma_{i_1,\mathbf{G}}}$ and the forgetful morphism $\operatorname{For}_{\Gamma_{i_1,\bm{G}}\to \hat{\mathcal B} \Gamma_{i_1,\bm{G}}}\colon \oPMb_{\Gamma_{i_1,\bm{G}}}\to \oPMb_{\hat{\mathcal B} \Gamma_{i_1,\bm{G}}}$ forgetting the inserted tails in $I^{PI}_{sp,0}$.  
        
        In the case $RT(\hat{\mathcal B}\Gamma_{i_1,\bm{G}})$ is a two-element set, we denote by $r_1,r_2\in RT(\hat{\mathcal B}\Gamma_{i_1,\bm{G}})$ its elements and take the multisection $\nontilde{s}^{trr*}_{i_1,\mathbf{G}}$ of ${\mathbb L}^*_{i_1}\to \oPMb_{\bm{G}}$ to be the multisection pulled back from the TRR multisection $\mathfrak s_{i_1,\hat{\mathcal B}\Gamma_{i_1,\bm{G}},r_1,r_2}$ of $\mathbb L_{i_1,\hat{\mathcal B}\Gamma_{i_1,\bm{G}},r_1,r_2}\to \oPMb_{\hat{\mathcal B}\Gamma_{i_1,\bm{G}}}$ via the projection $\pi_{\mathbf{G}\to \Gamma_{i_1,\mathbf{G}}}\colon \oPMb_{\mathbf{G}}\to \oPMb_{\Gamma_{i_1,\mathbf{G}}}$ and the morphism $\operatorname{For}_{\Gamma_{i_1,\bm{G}}\to \hat{\mathcal B} \Gamma_{i_1,\bm{G}}}\colon \oPMb_{\Gamma_{i_1,\bm{G}}}\to \oPMb_{\hat{\mathcal B} \Gamma_{i_1,\bm{G}}}$ forgetting the inserted tails in $I^{PI}_{sp,0}$.

        The multisections $\nontilde{s}^{trr*}_{i_1,\mathbf{G}}$ defined above on each $\oPMb_{\mathbf{G}}$ can be glued in the a multisection $\nontilde{s}^{trr*}_{i_1}$ of the glued line bundle $\widetilde{\mathbb L}^*_{i_1}\to \widetilde{\mathcal{PM}}^{1/r,0}_{1,B,I}$. This follows from the behaviour of the multisections $\mathfrak s$ at the boundary strata, \textit{i.e.} \eqref{eq behaviour of trr 1 rooted}, \eqref{eq behaviour of trr 1 rooted} and \eqref{eq behaviour of trr two rooted}.

        Following from the above construction and Observation \ref{obs vanishing loci of single trr}, the zero locus $Z(\nontilde{s}^{trr*}_{i_1})$ of the multisection $\nontilde{s}^{trr*}_{i_1}$ of ${\mathbb L}^*_{i_1}\to \oPMb^{1/r,0}_{1,B,I}$ (or $\widetilde{\mathbb L}^*_{i_1}\to \widetilde{\mathcal{PM}}^{1/r,0}_{1,B,I}$ after gluing) is of the form
        \begin{equation} \label{eq of a single global trr of Lstar}
        Z(\nontilde{s}^{trr*}_{i_1})=\bigcup_{\bm{G}\in \bm{\mathcal G}_{i_1,g=1} } \oPMb_{\bm{G}} \cup \bigcup_{\bm{G}\in  \bm{\mathcal G}_{i_1,sp} } \oPMb_{\bm{G}} \bigcup_{\bm{G}\in  \bm{\mathcal G}^{1/2}_{i_1,nsp} } \oPMb_{\bm{G}} \bigcup_{\bm{G}\in \bm{\mathcal G}^{1/2+1/2}_{i_1,nsp}} \oPMb_{\bm{G}},
        \end{equation}
        where 
        \begin{equation*}
             \bm{\mathcal G}_{i_1}:=\left\{ \bm{G}\in \GPI_{1,B,I}^{r,0} \colon  \begin{array}{ccc} \text{For   $\Gamma_{i_1,\bm{G}}\in V(\bm{G})$  such that  $i_1\in T^I(\Gamma_{i_1,\bm{G}})$, $\Gamma_{i_1,\bm{G}}$ consists of one closed}\\ \text{vertex $v_{\Gamma_{i_1,\bm{G}},c}$ and one open vertex $v_{\Gamma_{i_1,\bm{G}},c}$ connected by an internal}
             \\ \text{edge $e_{\bm{G}}$, where $\hat{\mathcal B} v_{\Gamma_{i_1,\bm{G}},c}$ is stable, and all $\Gamma\in V(\bm{G})\setminus\{\Gamma_{i_1,\bm{G}}\}$ are smooth}
             \end{array}\right\},
        \end{equation*}
        \begin{equation*}
             \bm{\mathcal G}^{}_{i_1,g=1}:=\left\{ \bm{G}\in \bm{\mathcal G}_{i_1} \colon  \begin{array}{c} \text{$\Gamma_{i_1,\bm{G}}$ is genus-one}
             \end{array}\right\},
        \end{equation*}
     \begin{equation*}
             \bm{\mathcal G}^{}_{i_1,sp}:=\left\{ \bm{G}\in \bm{\mathcal G}_{i_1} \colon  \begin{array}{c} \text{$\Gamma_{i_1,\bm{G}}$ is genus-zero, $\lvert RT(\Gamma_{i_1,\bm{G}})\rvert=1$, $RT(\Gamma_{i_1,\bm{G}}) \subset T(v_{\Gamma_{i_1,\bm{G}},o})$}
             \end{array}\right\},
        \end{equation*}
        \begin{equation*}
             \bm{\mathcal G}^{1/2}_{i_1,nsp}:=\left\{ \bm{G}\in \bm{\mathcal G}_{i_1} \colon  \begin{array}{c} \text{$\Gamma_{i_1,\bm{G}}$ is genus-zero, $\lvert RT(\Gamma_{i_1,\bm{G}})\rvert=2$, $\lvert RT(\Gamma_{i_1,\bm{G}}) \cap T(v_{\Gamma_{i_1,\bm{G}},o})\rvert=1$}
             \end{array}\right\},
        \end{equation*}
        \begin{equation*}
             \bm{\mathcal G}^{1/2+1/2}_{i_1,nsp}:=\left\{ \bm{G}\in \bm{\mathcal G}_{i_1} \colon  \begin{array}{c} \text{$\Gamma_{i_1,\bm{G}}$ is genus-zero, $\lvert RT(\Gamma_{i_1,\bm{G}})\rvert=2$, $\lvert RT(\Gamma_{i_1,\bm{G}}) \cap T(v_{\Gamma_{i_1,\bm{G}},o})\rvert=2$}
             \end{array}\right\}.
        \end{equation*}
        
        \begin{rmk}\label{rmk on zero locus of the single global trr of Lstar}
            Note that $\oPMb_{\bm{G}}$ is contained in the zero locus of  $\nontilde{s}^{trr*}_{i_1}$ over $\oPMb_{d_e\bm{G}}$ for $\bm{G}\in \bm{\mathcal G}_{i_1}$. 
             When $\bm{G}\in \bm{\mathcal G}^{1/2}_{i_1,nsp} \sqcup \bm{\mathcal G}^{1/2+1/2}_{i_1,nsp}$, the multisection             $\nontilde{s}^{trr*}_{i_1}$  has two branches of weight $1/2$ over $\oPMb_{d_e\bm{G}}$; for $\bm{G}\in \bm{\mathcal G}^{1/2}_{i_1,nsp}$, only one branch out of two vanishes on $\oPMb_{\bm{G}}$; for $\bm{G}\in  \bm{\mathcal G}^{1/2+1/2}_{i_1,nsp}$, both branch vanish on $\oPMb_{\bm{G}}$. 
        \end{rmk}

        A subset of $\bm{\mathcal G}_{i_1,co}\subseteq \bm{\mathcal G}^{}_{i_1,g=1} \sqcup \bm{\mathcal G}^{}_{i_1,sp} \sqcup \bm{\mathcal G}^{1/2}_{i_1,nsp} \sqcup \bm{\mathcal G}^{1/2+1/2}_{i_1,nsp}$ is particularly important in this proof, which is given by
        \begin{equation*}
            \bm{\mathcal G}_{i_1,co}:=\left\{ \bm{G}\in \bm{\mathcal G}^{}_{i_1,g=1} \sqcup \bm{\mathcal G}^{}_{i_1,sp} \sqcup \bm{\mathcal G}^{1/2}_{i_1,nsp} \sqcup \bm{\mathcal G}^{1/2+1/2}_{i_1,nsp}\colon T(v_{\Gamma_{i_1,\bm{G}},c})\cap I^{PI}(\bm{G})=\emptyset \right\}.
        \end{equation*}
        For each $\bm{G}\in \bm{\mathcal G}_{i_1}$, let $n_{\bm{G}}\in T(v_{\Gamma_{i_1,\bm{G}},c})$ be the unique internal half-edge, then by definition $T(v_{\Gamma_{i_1,\bm{G}},c})\setminus\{n_{\bm{G}}\}$ is always a non-empty subset of $I$ containing $i_1$ for $\bm{G}\in \bm{\mathcal G}_{i_1,co}$, thus we have a decomposition 
        \begin{equation*}
            \bm{\mathcal G}_{i_1,co}= \bigsqcup_{R\subseteq I\setminus \{i_1\}, R\ne \emptyset}\bm{\mathcal G}^R_{i_1,co},
        \end{equation*}
        where 
        \begin{equation*}
            \bm{\mathcal G}^R_{i_1,co}:= \left\{ \bm{G}\in \bm{\mathcal G}_{i_1,co} \colon  T(v_{\Gamma_{i_1,\bm{G}},c})=I\sqcup \{i_1,n_{\bm{G}}\}\right\}.
        \end{equation*}

    \noindent \textbf{Step 2.}
        We will modify $\nontilde{s}^{trr*}_{i_1,\mathbf{G}}$ to a  multisection $\nontilde{s}^{trr}_{i_1,\mathbf{G}}$  of $\mathbb L_{i_1}\to \oPMb^{1/r,0}_{1,B,I}$ by tensoring with a specific section of the line bundle $\mathbb L_{i_1}^{*\vee}\otimes \mathbb L_{i_1} \to \oPMb^{1/r,0}_{1,B,I}$.
         We first look at the dual line bundle $\mathbb L_{i_1}^*\otimes \mathbb L_{i_1}^\vee \to \oPMb^{1/r,0}_{1,B,I}$. When we restrict it to $\oPMb_{\mathbf{G}}\subseteq \oPMb^{1/r,0}_{1,B,I}$, by definition it is pulled back from the line bundle 
        $$\mathbb L_{i_1,\hat{\mathcal B}\Gamma_{i_1,\bm{G}}}\otimes \operatorname{For}^*_{\hat{\mathcal B}\Gamma_{i_1,\bm{G}}\to {\mathcal B}\Gamma_{i_1,\bm{G}}}\mathbb L_{i_1,{\mathcal B}\Gamma_{i_1,\bm{G}}}^\vee\to \oPMb_{\hat{\mathcal B}\Gamma_{i_1,\bm{G}}} $$ via $\pi_{\mathbf{G}\to \Gamma_{i_1,\mathbf{G}}}\colon \oPMb_{\mathbf{G}}\to \oPMb_{\Gamma_{i_1,\mathbf{G}}}$ 
        and        $\operatorname{For}_{\Gamma_{i_1,\bm{G}}\to\hat{\mathcal B}\Gamma_{i_1,\bm{G}}}\colon \oPMb_{\Gamma_{i_1,\bm{G}}}\to \oPMb_{\hat{\mathcal B}\Gamma_{i_1,\bm{G}}}$.  
        Notice that ${\mathcal B}\Gamma_{i_1,\bm{G}}$ can be obtained by forgetting all the internal tails of $\hat{\mathcal B}\Gamma_{i_1,\bm{G}}$ in $T^I(\hat{\mathcal B}\Gamma_{i_1,\bm{G}})\cap RT(\hat{\mathcal B}\Gamma_{i_1,\bm{G}})$, we have a section (see \S \ref{sec cotangent lines relation}) $\tilde{\mathfrak t}_{i_1,{\mathcal B}\Gamma_{i_1,\bm{G}}\to\hat{\mathcal B}\Gamma_{i_1,\bm{G}}}$ of the line bundle $\mathbb L_{i_1,\hat{\mathcal B}\Gamma_{i_1,\bm{G}}}\otimes \operatorname{For}^*_{\hat{\mathcal B}\Gamma_{i_1,\bm{G}}\to {\mathcal B}\Gamma_{i_1,\bm{G}}}\mathbb L_{i_1,{\mathcal B}\Gamma_{i_1,\bm{G}}}^\vee$. We pullback $\tilde{\mathfrak t}_{i_1,{\mathcal B}\Gamma_{i_1,\bm{G}}\to\hat{\mathcal B}\Gamma_{i_1,\bm{G}}}$  to a section $\tilde{\mathfrak{t}}_{i_1,\mathbb L \to \mathbb L^*,\mathbf{G}}$ of $\mathbb L_{i_1}^*\otimes \mathbb L_{i_1}^\vee \to \oPMb_{\mathbf{G}}$, \textit{i.e.,}
        $$
        \tilde{\mathfrak{t}}_{i_1,\mathbb L \to \mathbb L^*,\mathbf{G}}:=\pi^*_{\mathbf{G}\to \Gamma_{i_1,\mathbf{G}}}\operatorname{For}^*_{\Gamma_{i_1,\bm{G}}\to\hat{\mathcal B}\Gamma_{i_1,\bm{G}}}\tilde{\mathfrak t}_{i_1,{\mathcal B}\Gamma_{i_1,\bm{G}}\to\hat{\mathcal B}\Gamma_{i_1,\bm{G}}}.
        $$
        These sections $\tilde{\mathfrak{t}}_{i_1,\mathbb L \to \mathbb L^*,\mathbf{G}}$ glue to a section $\tilde{\mathfrak{t}}_{i_1,\mathbb L \to \mathbb L^*}$ of $\widetilde{\mathbb L}^*_{i_1}\otimes \widetilde{\mathbb L}_{i_1}^\vee\to \widetilde{\mathcal{PM}}^{1/r,0}_{1,B,I}$. In fact, assuming $\mathbf{G_1}$ and $\mathbf{G_2}$ are smooth $(r,0)$-graphs such that $\oPMb_{\mathbf{G_1}}$ and $\oPMb_{\mathbf{G_2}}$ are glued along their identified codimensional-1 boundaries $\overline{\text{bd}}_{BI}$ and $\overline{\text{bd}}_{AI}$, then we have 
        $$\overline{\text{bd}}_{BI}=\oPMb_{\Gamma_1^{BI}}\times \prod_{\Gamma\in V(\mathbf{G_1})\setminus\{\Gamma_1\}}\oPMb_{\Gamma}$$
        for a vertex $\Gamma_1\in V(\mathbf{G_1})$ and $\Gamma^{BI}_1\in \partial \Gamma_1$ with a unique (boundary) edge $e^{PI}$. We assume $e^{PI}$ is separating  (the case where $e^{PI}$ is non-separating is similar), and we denote by $v_0$ and $v_{r-2}$ the two vertices of $\Gamma_1^{BI}$, where the half-edge of $e^{PI}$ on $v_0$ (respectively $v_{r-2}$) is illegal with twist $0$ (respective legal with twist $r-2$). In this case we can write
        $$\overline{\text{bd}}_{AI}=\oPMb_{\Gamma_2^{AI}}\times \oPMb_{v_{r-2}}\times  \prod_{\Gamma\in V(\mathbf{G_1})\setminus\{\Gamma_1\}}\oPMb_{\Gamma},$$
        where $\Gamma_2^{AI}\in \partial \Gamma_2$ (for $\Gamma_2\in V(\bm{G}_2)=V(\bm{G}_1)\sqcup \{\Gamma_2,v_{r-2}\}\setminus \{\Gamma_1\}$) is an $r$-spin graph with two vertices, one of them is the same as $v_0$, the other vertex $v'$ has only one (inserted) internal tail and one boundary half-edge connecting to the twist-zero half-edge of $v_0$. 
        \begin{itemize}
            \item[-]
        If $i_1\in T^I(\Gamma_{i_1,\mathbf{G_1}})$ for a $\Gamma_{i_1,\mathbf{G_1}}\in V(\mathbf{G}_1)\setminus{\Gamma_1}$, then $\tilde{\mathfrak{t}}_{i_1,\mathbb L \to \mathbb L^*,\mathbf{G_1}}$ and $\tilde{\mathfrak{t}}_{i_1,\mathbb L \to \mathbb L^*,\mathbf{G_2}}$ coincide on the identified boundary since both of them are pulled back from 
        $$\tilde{\mathfrak t}_{i_1,{\mathcal B}\Gamma_{i_1,\bm{G_1}}\to\hat{\mathcal B}\Gamma_{i_1,\bm{G_1}}}=\tilde{\mathfrak t}_{i_1,{\mathcal B}\Gamma_{i_1,\bm{G_2}}\to\hat{\mathcal B}\Gamma_{i_1,\bm{G_2}}}$$
        over $\oPMb_{\hat{\mathcal B}\Gamma_{i_1,\bm{G_1}}}=\oPMb_{\hat{\mathcal B}\Gamma_{i_1,\bm{G_2}}}$.
        
        \item[-] If $i_1\in T^I(v_{r-2})$, then $\Gamma_{i_1,\mathbf{G_1}}=\Gamma_1$. By definition the restriction of $\tilde{\mathfrak{t}}_{i_1,\mathbb L \to \mathbb L^*,\mathbf{G_1}}$ to $\overline{\text{bd}}_{AI}$ is pulled back from 
        $\tilde{\mathfrak t}_{i_1,{\mathcal B}\Gamma_{1}^{BI}\to\hat{\mathcal B}\Gamma_{1}^{BI}}$ 
        via the projection $\pi_{\mathbf{G_1}\to \Gamma_1^{BI}}\colon \overline{\text{bd}}_{BI}\to \oPMb_{\Gamma_1^{BI}}$ and $\operatorname{For}_{\Gamma_{1}^{BI}\to \mathcal{B}\Gamma_{1}^{BI}}\colon \oPMb_{\Gamma_{1}^{BI}}\to \oPMb_{\mathcal{B}\Gamma_{1}^{BI}}$. The section 
        $\tilde{\mathfrak t}_{i_1,{\mathcal B}\Gamma_{1}^{BI}\to\hat{\mathcal B}\Gamma_{1}^{BI}}$ 
        itself is pulled back from $\tilde{\mathfrak t}_{i_1,{\mathcal B}v_{r-2}\to\hat{\mathcal B}\Gamma_{r-2}}$ via the projection $\pi_{\hat{\mathcal B}\Gamma_1^{BI}\to \hat{\mathcal B}v_{r-2}}\colon\oPMb_{\hat{\mathcal B}\Gamma_1^{BI}}\cong \oPMb_{\hat{\mathcal B}v_0}\times \oPMb_{\hat{\mathcal B}v_{r-2}}\to \oPMb_{\hat{\mathcal B}v_{r-2}}$ by the construction in \S \ref{sec cotangent lines relation}, under the natural identifications between $\pi^*_{\hat{\mathcal B} \Gamma_{1}^{BI}\to \hat{\mathcal B} v_{r-2}}\mathbb L_{i_1,\hat{\mathcal B} v_{r-2}}$ and $\mathbb L_{i_1,\hat{\mathcal B} \Gamma_{1}^{BI}}$, and between $$ \pi^*_{\hat{\mathcal B} \Gamma_{1}^{BI}\to \hat{\mathcal B} v_{r-2}} \operatorname{For}^*_{\hat{\mathcal B} v_{r-2} \to \mathcal B v_{r-2}} \mathbb L_{i_1,\mathcal B v_{r-2}}=\operatorname{For}^*_{\hat{\mathcal B} \Gamma_{1}^{BI} \to \mathcal B \Gamma_{1}^{BI}} \pi^*_{\mathcal B \Gamma_{1}^{BI}\to \mathcal B v_{r-2}}\mathbb L_{i_1,\mathcal B v_{r-2}}$$ and $\operatorname{For}^*_{\hat{\mathcal B} \Gamma_{1}^{BI} \to \mathcal B \Gamma_{1}^{BI}} \mathbb L_{i_1,\mathcal B \Gamma_{1}^{BI}}$. 

          On the other hand, the restriction of  $\tilde{\mathfrak{t}}_{i_1,\mathbb L \to \mathbb L^*,\mathbf{G_2}}$ to $\overline{\text{bd}}_{BI}$ is also pulled back from $\tilde{\mathfrak t}_{i_1,{\mathcal B}v_{r-2}\to\hat{\mathcal B}v_{r-2}}$ via the same map (after identify $\overline{\text{bd}}_{BI}$ and $\overline{\text{bd}}_{AI}$) by definition, thus it coincides with $\tilde{\mathfrak{t}}_{i_1,\mathbb L \to \mathbb L^*,\mathbf{G_1}}$ on this identified boundary.

        \item[-] If $i_1\in T^I(v_{0})$, similar to the previous item, the restriction of $\tilde{\mathfrak{t}}_{i_1,\mathbb L \to \mathbb L^*,\mathbf{G_1}}$ to $\overline{\text{bd}}_{BI}$ is pulled back from  $\tilde{\mathfrak t}_{i_1,{\mathcal B}v_{0}\to\hat{\mathcal B}v_{0}}$, 
         while the restriction of $\tilde{\mathfrak{t}}_{i_1,\mathbb L \to \mathbb L^*,\mathbf{G_2}}$ to $\overline{\text{bd}}_{AI}$  is also pulled back 
         from $\tilde{\mathfrak t}_{i_1,{\mathcal B}v_0\to\hat{\mathcal B}v_0}$ via the same morphism after identifying $\overline{\text{bd}}_{BI}$ and $\overline{\text{bd}}_{BI}$.

        \end{itemize}

        According to Lemma \ref{lem zero locus of t between L} and the construction above, the zero $Z(\tilde{\mathfrak{t}}_{i_1,\mathbb L \to \mathbb L^*})$ locus of the section $\tilde{\mathfrak{t}}_{i_1,\mathbb L \to \mathbb L^*}$ of  ${\mathbb L}^*_{i_1}\otimes {\mathbb L}_{i_1}^\vee\to  \oPMb^{1/r,0}_{1,B,I}$ (or $\widetilde{\mathbb L}^*_{i_1}\otimes \widetilde{\mathbb L}_{i_1}^\vee \to \widetilde{\mathcal{PM}}^{1/r,0}_{1,B,I}$ after gluing) is 
        \begin{equation*}
        Z(\tilde{\mathfrak{t}}_{i_1,\mathbb L \to \mathbb L^*})=\bigcup_{\mathbf{G}\in \bm{\mathcal F}_{i_1,sp}} \oPMb_{\mathbf{G}} \cup \bigcup_{\mathbf{G}\in \bm{\mathcal F}^1_{i_1,nsp}} \oPMb_{\mathbf{G}} \cup \bigcup_{\mathbf{G}\in\bm{\mathcal F}^2_{i_1,nsp}} \oPMb_{\mathbf{G}},
        \end{equation*}
        where, in the notations in Step 1, 
        \begin{equation*}
            \bm{\mathcal F}_{i_1,sp}:=\left\{ \bm{G}\in \bm{\mathcal G}_{i_1} \colon \text{ $RT(\Gamma_{i_1,\bm{G}})=\{i^{\Gamma_{i_1,\bm{G}}}_{sp}\}\subseteq I^{PI}_{sp,0}(\bm{G})$, $T(v_{\Gamma_{i_1,\bm{G}},c})=\{i_1,n_{\bm{G}},i^{\Gamma_{i_1,\bm{G}}}_{sp}\}$}\right \},
        \end{equation*}
         \begin{equation*}
            \bm{\mathcal F}^1_{i_1,nsp}:=\left\{ \bm{G}\in \bm{\mathcal G}_{i_1} \colon \text{$\lvert RT(\Gamma_{i_1,\bm{G}})\rvert=2$, $T(v_{\Gamma_{i_1,\bm{G}},c})=\{i_1,n_{\bm{G}},i^{\Gamma_{i_1,\bm{G}}}_{nsp}\}$ for $i^{\Gamma_{i_1,\bm{G}}}_{nsp}\in I^{PI}_{nsp}(\bm{G})$}\right \}
        \end{equation*}
        and 
        \begin{equation*}
            \bm{\mathcal F}^2_{i_1,nsp}:=\left\{ \bm{G}\in \bm{\mathcal G}_{i_1} \colon \text{$\lvert RT(\Gamma_{i_1,\bm{G}})\rvert=2$, $RT(\Gamma_{i_1,\bm{G}})\subseteq I^{PI}_{nsp}(\bm{G})$,  $T(v_{\Gamma_{i_1,\bm{G}},c})=\{i_1,n_{\bm{G}}\}\sqcup RT(\Gamma_{i_1,\bm{G}}) $}\right \}.
        \end{equation*}
        Note that for each $a\in I^{PI}_{sp,0}({\bm{G}})$, it induces  partitions of $I=I_{a,0}\sqcup I_{a,1}$ and $B=B_{a,0}\sqcup B_{a,1}$ in the following way:  when we remove the edge of $\hat{\bm{G}}$ (see Definition \ref{def rh graphs}) corresponding to the dashed line between $a$ and $\delta(a)$, we get two connected components $\hat{\bm{G}}_1$ and $\hat{\bm{G}}_0$, where $a$ is on the $\hat{\bm{G}}_0$ side. We define 
        $I_{s}(a):=I\cap \sqcup_{\Gamma\in V(\hat{\bm{G}}_s)} T^I(\Gamma)$ and $B_{s}(a):=B\cap \sqcup_{\Gamma\in V(\hat{\bm{G}}_s)} T^B(\Gamma)$ for $s=0,1$. Such partitions induces a decomposition 
        $$
        \bm{\mathcal F}_{i_1,sp}=\bigsqcup_{R\subseteq I\setminus\{i_1\}, T\subseteq B} \bm{\mathcal F}_{i_1,sp}^{R,T},
        $$
        where
        $$
       \bm{\mathcal F}_{i_1,sp}^{R,T}:=\{\bm{G} \in  \bm{\mathcal F}_{i_1,sp}\colon I_{0}(i_{sp}^{\Gamma_{i_1,\bm{G}}})=R\sqcup \{i_1\}, B_{0}(i_{sp}^{\Gamma_{i_1,\bm{G}}})=T \}.
        $$

        We consider the dual section $\tilde{\mathfrak{t}}^{\vee}_{i_1,\mathbb L \to \mathbb L^*}$ of $\tilde{\mathfrak{t}}_{i_1,\mathbb L \to \mathbb L^*}$, which is a section  of  ${\mathbb L}^{*\vee}_{i_1}\otimes {\mathbb L}_{i_1}\to  \oPMb^{1/r,0}_{1,B,I}$ (or $\widetilde{\mathbb L}^{*\vee}_{i_1}\otimes \widetilde{\mathbb L}_{i_1} \to \widetilde{\mathcal{PM}}^{1/r,0}_{1,B,I}$ after gluing) defined away from the zero locus $Z(\tilde{\mathfrak{t}}_{i_1,\mathbb L \to \mathbb L^*})$, and $Z(\tilde{\mathfrak{t}}_{i_1,\mathbb L \to \mathbb L^*})$ behaves like its `poles' (see Remark \ref{rmk divosr like section}). Then $\tilde{\mathfrak{t}}^{\vee}_{i_1,\mathbb L \to \mathbb L^*}\otimes
        \nontilde{s}^{trr*}_{i_1}$ is a multisection of  $ {\mathbb L}_{i_1}\to  \oPMb^{1/r,0}_{1,B,I}$ (or $\ \widetilde{\mathbb L}_{i_1} \to \widetilde{\mathcal{PM}}^{1/r,0}_{1,B,I}$ after gluing) defined away from $Z(\tilde{\mathfrak{t}}_{i_1,\mathbb L \to \mathbb L^*})$.

        Parts of zeros of $\nontilde{s}^{trr*}_{i_1}$ coincide with `poles' of $\tilde{\mathfrak{t}}^{\vee}_{i_1,\mathbb L \to \mathbb L^*}$: we have $ \bm{\mathcal F}^1_{i_1,nsp}\subseteq \bm{\mathcal G}^{1/2}_{i_1,nsp}$. For any $\bm{G}\in \bm{\mathcal F}^1_{i_1,nsp}$, we denote by $\bm{H}:=d_{e_{\bm{G}}}\bm{G}\in \sGPI^{r,0}_{1,B,I}$, by $r_1=i^{\Gamma_{i_1,\bm{G}}}_{nsp}$ and $r_2$ the two roots in $RT(\Gamma_{i_1,\bm{G}})=RT(d_{e_{\bm{G}}}\Gamma_{i_1,\bm{G}})=RT(\Gamma_{i_1,\bm{H}})$. The restriction of $\tilde{\mathfrak{t}}^{\vee}_{i_1,\mathbb L \to \mathbb L^*}\otimes
        \nontilde{s}^{trr*}_{i_1}$ to $\oPMb_{\bm{H}}$ has two branches with weight $1/2$. One of them is pulled back from 
        \begin{equation}\label{eq one branch that zero and pole will cancel in proof}
        \tilde{\mathfrak t}^\vee_{i_1,{\mathcal B}\Gamma_{i_1,\bm{H}}\to\hat{\mathcal B}\Gamma_{i_1,\bm{H}}}\otimes \mathfrak s_{i_1,\hat{\mathcal B}\Gamma_{i_1,\bm{H}},r_2},
        \end{equation}
         the other one is pulled back from \begin{equation}\label{eq one branch that zero and pole will not cancel in the chosen strata}
        \tilde{\mathfrak t}^\vee_{i_1,{\mathcal B}\Gamma_{i_1,\bm{H}}\to\hat{\mathcal B}\Gamma_{i_1,\bm{H}}}\otimes \mathfrak s_{i_1,\hat{\mathcal B}\Gamma_{i_1,\bm{H}},r_1}.
        \end{equation}
        Notice that by construction we have 
        $$
        \mathfrak s_{i_1,\hat{\mathcal B}\Gamma_{i_1,\bm{H}},r_2}= \tilde{\mathfrak t}_{i_1,\operatorname{for}_{r_1}(\hat{\mathcal B}\Gamma_{i_1,\bm{H}}) \to \hat{\mathcal B}\Gamma_{i_1,\bm{H}}}\otimes \operatorname{For}^*_{\hat{\mathcal B}\Gamma_{i_1,\bm{H}} \to \operatorname{for}_{r_1}(\hat{\mathcal B}\Gamma_{i_1,\bm{H}})}\mathfrak s_{i_1,\operatorname{for}_{r_1}(\hat{\mathcal B}\Gamma_{i_1,\bm{H}}),r_2},
        $$
        and 
        \begin{equation*}
        \tilde{\mathfrak t}^\vee_{i_1,{\mathcal B}\Gamma_{i_1,\bm{H}}\to\hat{\mathcal B}\Gamma_{i_1,\bm{H}}} \otimes \tilde{\mathfrak t}_{i_1, \operatorname{for}_{r_1}(\hat{\mathcal B}\Gamma_{i_1,\bm{H}})\to \hat{\mathcal B}\Gamma_{i_1,\bm{H}}}
        = 
        \operatorname{For}^*_{\hat{\mathcal B}\Gamma_{i_1,\bm{H}} \to \operatorname{for}_{r_1}(\hat{\mathcal B}\Gamma_{i_1,\bm{H}})}
        \tilde{\mathfrak t}^\vee_{i_1,\operatorname{for}_{r_1}(\hat{\mathcal B}\Gamma_{i_1,\bm{H}})\to\hat{\mathcal B}\Gamma_{i_1,\bm{H}}},
        \end{equation*}
        therefore  \eqref{eq one branch that zero and pole will cancel in proof} can extends to a multisection
        \begin{equation}\label{eq extension of one branch when zero pole cancel}
        \operatorname{For}^*_{\hat{\mathcal B}\Gamma_{i_1,\bm{H}} \to \operatorname{for}_{r_1}(\hat{\mathcal B}\Gamma_{i_1,\bm{H}})}
        (\tilde{\mathfrak t}^\vee_{i_1,\operatorname{for}_{r_1}(\hat{\mathcal B}\Gamma_{i_1,\bm{H}})\to\hat{\mathcal B}\Gamma_{i_1,\bm{H}}} \otimes \mathfrak s_{i_1,\operatorname{for}_{r_1}(\hat{\mathcal B}\Gamma_{i_1,\bm{H}}),r_2})
        \end{equation}
        which has no zeros or `poles' along $\mathcal{PM}_{\Gamma_{i_1,\bm{G}}}\subset \oPMb_{\Gamma_{i_1,\bm{H}}}$. So the zeros and `poles' of corresponding branch of $\tilde{\mathfrak{t}}^{\vee}_{i_1,\mathbb L \to \mathbb L^*}\otimes
        \nontilde{s}^{trr*}_{i_1}$  along $\mathcal{PM}_{\bm{G}}\subset \oPMb_{\bm{H}}$ cancel out: we can extend it as pullback of \eqref{eq extension of one branch when zero pole cancel}, which vanishes no where on $\mathcal{PM}_{\bm{G}}$. Meanwhile, the other branch corresponding to \eqref{eq one branch that zero and pole will not cancel in the chosen strata} has `poles' along $\mathcal{PM}_{\bm{G}}$ since $\tilde{\mathfrak{t}}^{\vee}_{i_1,\mathbb L \to \mathbb L^*}$ has `poles' and $ \nontilde{s}^{trr*}_{i_1}$ while does not vanish.

        After the above natural extension, $\tilde{\mathfrak{t}}^{\vee}_{i_1,\mathbb L \to \mathbb L^*}\otimes \nontilde{s}^{trr*}_{i_1}$, as a multisection of  $\mathbb L_{i_1}\to \oPMb^{1/r,0}_{1,B,I}$ (or  $\mathbb L_{i_1}\to \oPMb^{1/r,0}_{1,B,I}$ after gluing),  has zeros along 
        $$
        \bigcup_{\bm{G}\in \bm{\mathcal G}_{i_1,g=1} } \oPMb_{\bm{G}} \cup \bigcup_{\bm{G}\in  \bm{\mathcal G}_{i_1,sp} } \oPMb_{\bm{G}} \bigcup_{\bm{G}\in  \bm{\mathcal G}^{1/2}_{i_1,nsp} \setminus \bm{\mathcal F}^1_{i_1,nsp}} \oPMb_{\bm{G}} \bigcup_{\bm{G}\in \bm{\mathcal G}^{1/2+1/2}_{i_1,nsp}} \oPMb_{\bm{G}}
        $$
        and 'poles' along
        $$
        \bigcup_{\mathbf{G}\in \bm{\mathcal F}_{i_1,sp}} \oPMb_{\mathbf{G}} \cup \bigcup_{\mathbf{G}\in \bm{\mathcal F}^1_{i_1,nsp}} \oPMb_{\mathbf{G}} \cup \bigcup_{\mathbf{G}\in\bm{\mathcal F}^2_{i_1,nsp}} \oPMb_{\mathbf{G}}.
        $$
        \noindent\textbf{Step 3.} The following observation is useful  not only in removing the `poles' of $\tilde{\mathfrak{t}}^{\vee}_{i_1,\mathbb L \to \mathbb L^*}\otimes \nontilde{s}^{trr*}_{i_1}$, but also in the computation in Step 5 below.
        
        \begin{obs}\label{obs zero locus of modified canonical}
            For any $\bm{G}\in \left(\bm{\mathcal G}^{}_{i_1,g=1} \sqcup \bm{\mathcal G}^{}_{i_1,sp} \sqcup \bm{\mathcal G}^{1/2}_{i_1,nsp} \sqcup \bm{\mathcal G}^{1/2+1/2}_{i_1,nsp})\right)\setminus (\bm{\mathcal G}_{i_1,co}\cup \bm{\mathcal F}^1_{i_1,nsp})$, or any $\bm{G}\in \bm{\mathcal F}^2_{i_1,nsp}$, the restriction of the multisection $\bm{s}^{re}$ (see \eqref{eq def remaining sec}) vanishes nowhere on $\oPMb_{\bm{G}}$. 
            
            To see this, we  write $\bm{H}=d_{e_{\bm{G}}}\bm{G}$ and  $\Gamma_{i_1,\bm{H}}=d_{e_{\bm{G}}}\Gamma_{i_1,\bm{G}}$. Since $\bm{s}$ is $\gamma$-canonical, the restriction of  $\bm{s}^{re}$ on $\oPMb_{\bm{H}}$ is of the form
            $$\pi^*_{\bm{H}\to \Gamma_{i_1,\bm{H}}} \bm{s}^{re}_{\Gamma_{i_1,\bm{H}}}\oplus\bigoplus_{\Gamma \in V(\bm{G})\setminus\{\Gamma_{i_1,\bm{G}}\}} \pi^*_{\bm{H} \to \Gamma} \bm{s}_{\Gamma}.$$
            By transversality, for $\Gamma\in V(\bm{G})\setminus \{\Gamma_{i_1,\bm{G}}\}$,  the multisection $\bm{s}_{\Gamma}$ vanishes nowhere unless $\rk E_{\Gamma}\le \dim \oPMb_{\Gamma}$, so we can focus on the case $\rk E_{\Gamma_{i_1,\bm{H}}}\ge \dim \oPMb_{\Gamma_{i_1,\bm{H}}}$. 

            We denote by $\bar{\mathcal B}\Gamma_{i_1,\bm{G}}\in \partial^* \hat{\mathcal B}\Gamma_{i_1,\bm{H}}$ the graph obtained after stabilizing $\hat{\mathcal B}\Gamma_{i_1,\bm{G}}$ (\textit{i.e.} contract every unstable vertices $\hat{\mathcal B}\Gamma_{i_1,\bm{G}}$); it is stable since it has $i_1$ and at least one internal root as tails. Because of transversely-decomposability, the restriction of $\bm{s}^{re}_{ \Gamma_{i_1,\bm{H}}}$ on $\oPMb_{\Gamma_{i_1,\bm{G}}}\subset \oPMb_{\Gamma_{i_1,\bm{H}}}$ is pulled back from a transverse multisection of  $E^{re}_{\mathcal T_S \bar{\mathcal B}\Gamma_{i_1,\bm{G}}}\to \oPMb_{\mathcal T_S \bar{\mathcal B}\Gamma_{i_1,\bm{G}}}$ (see \S \ref{sec def of tame and canonical}), where  $S=RT(\Gamma_{i_1,\bm{G}})\cap T^I({\Gamma_{i_1,\bm{G}}})$. We always have 
            $$\rk E^{re}_{\mathcal T_S \bar{\mathcal B}\Gamma_{i_1,\bm{G}}}=\rk E^{re}_{\hat{\mathcal B}\Gamma_{i_1,\bm{H}}}=\rk E_{\Gamma_{i_1,\bm{H}}}-2,$$
            and 
            $$
            \dim \oPMb_{\mathcal T_S \bar{\mathcal B}\Gamma_{i_1,\bm{G}}}\le 
            \dim \oPMb_{ \bar{\mathcal B}\Gamma_{i_1,\bm{G}}} \le \dim \oPMb_{\Gamma_{i_1,\bm{H}}}-2,
            $$
            but the first inequality  is strict for $\bm{G}\in \bm{\mathcal F}^2_{i_1,nsp}$, while the second inequality is strict for $\bm{G}\in \left(\bm{\mathcal G}^{}_{i_1,g=1} \sqcup \bm{\mathcal G}^{}_{i_1,sp} \sqcup \bm{\mathcal G}^{1/2}_{i_1,nsp} \sqcup \bm{\mathcal G}^{1/2+1/2}_{i_1,nsp})\right)\setminus (\bm{\mathcal G}_{i_1,co}\cup \bm{\mathcal F}^1_{i_1,nsp})$,           therefore for all the $\bm{G}$ in these sets we have $\rk E^{re}_{\mathcal T_S \bar{\mathcal B}\Gamma_{i_1,\bm{G}}}> \dim \oPMb_{\mathcal T_S \bar{\mathcal B}\Gamma_{i_1,\bm{G}}}$, hence  $\bm{s}^{re}_{ \Gamma_{i_1,\bm{H}}}$ vanishes nowhere on $\oPMb_{ \Gamma_{i_1,\bm{G}}}$.

        \end{obs}

         According to the above observation, we can take an open neighbourhood $U^{2}_{nsp}$ of $$Z^2_{nsp}:=\bigcup_{\bm{G}\in \bm{\mathcal F}^2_{i_1,nsp}} \oPMb_{\bm{G}}\subset \oPMb^{1/r,0}_{1,B,I}$$ 
         which glue to an open set in $\widetilde{\mathcal{PM}}^{1/r,0}_{1,B,I}$, such that $\bm{s}^{re}$ vanishes nowhere in $U^{2}_{nsp}$.

         For fixed $R\subseteq I\setminus\{i_1\}$ and $T\subseteq B$, the subspace
         $$
         Z_{sp}^{R,T}:=\bigsqcup_{\bm{G}\in \mathcal F^{R,T}_{i_1,sp}}\oPMb_{\bm{G}}\subset \oPMb_{1,B,I}^{1/r,0}
         $$
         parametrize all the $(r,0)$-cylinders with unpaired markings $(B,I)$ which have a separating dashed line connecting a twist-zero internal point $i_{sp}^{\Gamma_{i_1,\bm{G}}}\in I^{PI}(\bm{G})$ and twist-$(r-2)$ boundary point $b^{R,T}_{sp}\in B^{PI}(\bm{G})$, where $i_{sp}^{\Gamma_{i_1,\bm{G}}}$ together with all unpaired markings in $R$ and $T$ are on the genus-zero part, and  $i_{sp}^{\Gamma_{i_1,\bm{G}}}$  lies one a closed irreducible component with only three special points: the marking $i_1$, the half-node $n_{\bm{G}}$, and $i_{sp}^{\Gamma_{i_1,\bm{G}}}$. If we denote by $a^{R,T}_{sp}$ the other half-node of $n_{\bm{G}}$ (it has the same twist as $i_1$), then after detaching the edge $e_{\bm{G}}$, except the closed $r$-spin sphere with three markings in $\{i_1,n_{\bm{G}},i_{sp}^{\Gamma_{i_1,\bm{G}}}\}$, from the above data we get an $(r,0)$-disk with unpaired marking $(T,R\sqcup \{a^{R,T}_{sp}\})$ together with an $(r,0)$-cylinder with unpaired marking $(B\sqcup \{a^{R,T}_{sp} \setminus T\},I\setminus (R\sqcup \{i_1\}))$, thus we have an detach morphism
         \begin{equation}\label{eq detach sp}
         \operatorname{Detach}_{sp}^{R,T}\colon Z^{R,T}_{sp} \to \oPMb^{1/r,0}_{1,B\sqcup \{b^{R,T}_{sp}\} \setminus T,I\setminus (R\sqcup \{i_1\})} \times \oPMb^{1/r,0}_{0,T,R\sqcup \{a^{R,T}_{sp}\}}\times \Mbar_{v_{sp,c}^{R,T}},
         \end{equation}
         where $v_{sp,c}^{R,T}$ is a closed single-vertex $r$-spin graph with three tails $i_1,n_{\bm{G}}$ and $i_{sp}^{\Gamma_{i_1,\bm{G}}}$. Note that $\Mbar_{v_{sp,c}^{R,T}}$ is a zero-dimensional orbifold with an order-$r$ isotopy group and rank-zero Witten bundle.
         Moreover, the point insertion procedure glues $Z^{R,T}_{sp}$ to a subspace $\widetilde{Z}^{R,T} \subset \widetilde{\mathcal{PM}}^{1/r,0}_{1,B,I}$, the detach morphisms $\operatorname{Detach}_{e_{\bm{G}}}$ also glue to a morphism
         $$
         \operatorname{Detach}_{sp}^{R,T}\colon \widetilde{\mathcal{PM}}^{1/r,0}_{1,B\sqcup \{b^{R,T}_{sp}\} \setminus T,I\setminus (R\sqcup \{i_1\})} \times \widetilde{\mathcal{PM}}^{1/r,0}_{0,T,R\sqcup \{a^{R,T}_{sp}\}}\times \Mbar_{v_{sp,c}^{R,T}}. 
         $$

         Similar, the subspace 
         $$
         Z_{nsp}^{1}:=\bigsqcup_{\bm{G}\in \mathcal F^{1}_{i_1,nsp}}\oPMb_{\bm{G}}\subset \oPMb_{1,B,I}^{1/r,0}
         $$
         also has a detach morphism
         \begin{equation}\label{eq detach nsp}
         \operatorname{Detach}_{nsp}\colon  Z_{nsp}^{1} \to \oPMb^{1/r,0}_{0,B\sqcup\{b_{nsp}\}, I\sqcup \{a_{nsp}\}\setminus\{i_1\}}\times \Mbar_{v_{nsp,c}},
         \end{equation}
         where $\Mbar_{v_{nsp,c}}$ is a zero-dimensional orbifold with an order-$r$ isotopy group and rank-zero Witten bundle;  the new twist-$(r-2)$ boundary marking $b_{nsp}$ was the one paired with $i_{nsp}^{\Gamma_{i_1,\bm{G}}}$, and the new internal marking $a_{nsp}$ was the other half-node of $n_{\bm{G}}$, it has the same twist as $i_1$. Under point insertion, $\operatorname{Detach}_{nsp}$ also glues to 
         \begin{equation*}
         \operatorname{Detach}_{nsp}\colon \widetilde{Z}^1_{nsp}\to  \widetilde{\mathcal{PM}}^{1/r,0}_{0,B\sqcup\{b_{nsp}\}, I\sqcup \{a_{nsp}\}\setminus\{i_1\}}\times \Mbar_{v_{nsp,c}}.
         \end{equation*}

         The locus $\widetilde{Z}^{R,T}_{sp}$ is the 'pole' of $\tilde{\mathfrak{t}}^{\vee}_{i_1,\mathbb L \to \mathbb L^*}\otimes \nontilde{s}^{trr*}_{i_1}$ because it is the `pole' of $\tilde{\mathfrak{t}}^{\vee}_{i_1,\mathbb L \to \mathbb L^*}$. Note that $\widetilde{Z}^{R,T}_{sp}$ does not intersect with any other `poles', by the method in Remark \ref{rmk divosr like section}, we can modify $\tilde{\mathfrak{t}}^{\vee}_{i_1,\mathbb L \to \mathbb L^*}$ in a small neighbourhood $U^{R,T}_{sp}\subset \widetilde{\mathcal{PM}}^{1/r,0}_{1,B,I}$ of $\widetilde{Z}^{R,T}_{sp}$ to a section with zero at $\widetilde{Z}^{R,T}_{sp}$, but with the reversed order of vanishing as $\tilde{\mathfrak{t}}_{i_1,\mathbb L \to \mathbb L^*}$.

         Similarly, the locus $\widetilde{Z}^{1}_{nsp}$ is the 'pole' of one branch of $\tilde{\mathfrak{t}}^{\vee}_{i_1,\mathbb L \to \mathbb L^*}\otimes \nontilde{s}^{trr*}_{i_1}$ because it is the `pole' of $\tilde{\mathfrak{t}}^{\vee}_{i_1,\mathbb L \to \mathbb L^*}$, and in this branch $\nontilde{s}^{trr*}_{i_1}$ does not vanish (recall that in the other branch the zero of $\nontilde{s}^{trr*}_{i_1}$ canceled with the `pole' of $\tilde{\mathfrak{t}}^{\vee}_{i_1,\mathbb L \to \mathbb L^*}$). 
         Note that $\widetilde{Z}^{1}_{nsp}$ does not intersect with any other `poles' outside the neighbourhood $U^2_{nsp}\subset \widetilde{\mathcal{PM}}^{1/r,0}_{1,B,I}$,  by the method in Remark \ref{rmk divosr like section}, we can find a small neighbourhood $U^{1}_{nsp}\subset \widetilde{\mathcal{PM}}^{1/r,0}_{1,B,I}$ of $\widetilde{Z}^{1}_{nsp}$ and modify the branch of $\tilde{\mathfrak{t}}^{\vee}_{i_1,\mathbb L \to \mathbb L^*}\otimes \nontilde{s}^{trr*}_{i_1}$ with `poles' by modifying $\tilde{\mathfrak{t}}^{\vee}_{i_1,\mathbb L \to \mathbb L^*}$ in this branch, in $U^{1}_{nsp}\setminus U^{2}_{nsp}$, to a section with zero at $\widetilde{Z}^{R,T}_{sp}\setminus U^{2}_{nsp}$, but with the reversed order of vanishing as $\tilde{\mathfrak{t}}_{i_1,\mathbb L \to \mathbb L^*}$.

         After these modifications, we get a multisection of $\mathbb L_{i_1}$ on $\widetilde{\mathcal{PM}}^{1/r,0}_{1,B,I}\setminus U^{2}_{nsp}$. We can then extend it to an arbitrary multisection of $\mathbb L_{i_1}$ over the entire  $\widetilde{\mathcal{PM}}^{1/r,0}_{1,B,I}$.   We denote such an extension by $s_{i_1}^{trr}$. 
         As the modifications and extensions can all be done first on $\oPMb_{\hat{\mathcal{B}}\Gamma_{i_1,\bm{H}}}$, and then pulled back to  $\oPMb_{\bm{H}}$, we require, for every $\bm{H}\in \sGPI_{1,B,I}^{r,0}$, that 
         $$s_{i_1,\bm{H}}^{trr}:=s_{i_1}^{trr}\vert_{\oPMb_{\bm{H}}}=\left(\operatorname{For}_{\Gamma_{i_1,\bm{H}} \to \hat{\mathcal B}\Gamma_{i_1,\bm{H}}}\circ\pi_{\bm{H}\to \Gamma_{i_1,\bm{H}}}\right)^* s_{i_1,\hat{\mathcal B}\Gamma_{i_1,\bm{H}}}^{trr},$$
             where $s$ is a multisection of $\operatorname{For}^*_{\hat{\mathcal B}\Gamma_{i_1,\bm{H}} \to {\mathcal B}\Gamma_{i_1,\bm{H}}}\mathbb L_{i_1,{\mathcal B}\Gamma_{i_1,\bm{H}}}\to \oPMb_{\hat{\mathcal{B}}\Gamma_{i_1,\bm{H}}}$, which is a modification $        \tilde{\mathfrak t}^\vee_{i_1,{\mathcal B}\Gamma_{i_1,\bm{H}}\to\hat{\mathcal B}\Gamma_{i_1,\bm{H}}}\otimes \mathfrak s_{i_1,\hat{\mathcal B}\Gamma_{i_1,\bm{H}}}$ near its ``pole''.  
             This requirement is fulfilled via an induction process very similar to \S \ref{sec extend canonical to global glued}, and we omit the details.

    \noindent\textbf{Step 4.}
        Now we show that $\bm{s}$ has the same number of zeros as $\bm{s}^{mod}:=\bm{s}^{re} \oplus \nontilde{s}^{trr}_{i_1}$ in $\oPMb^{1/r,0}_{1,B,I}$ by constructing a homotopy between $\bm{s}$ and $\bm{s}^{mod}$.

        We denote by $\bm{\mathcal H}\subseteq \sGPI^{r,0}_{1,B,I}$ the subset consisting of smooth $(r,0)$-graphs $\bm{G}$ with a genus-one vertex $\Gamma_{g=1,\mathbf{G}}$. 
        Notice that for any $\bm{G}\in \bm{\mathcal H}$, if $\Gamma\in V(\bm{G})$ is genus-zero, then $\lvert RT(\Gamma) \rvert=1$.
        We also denote by $\bm{\mathcal H}_1$ the subset such that $\Gamma_{g=1,\mathbf{G}}=\Gamma_{i_1,\mathbf{G}}$, and set $\bm{\mathcal H}_0:=\bm{\mathcal H}\setminus \bm{\mathcal H}_1$.

    \begin{obs}\label{obs on genus one K}
        Note that for any $\bm{G}\in \bm{\mathcal H}_1$, $\oPMb_{\bm{G}}$ does not intersect with $Z^{R,T}_{sp}$, $Z^1_{nsp}$ or $Z^2_{nsp}$, so in this case the restriction $s_{i_i,\bm{G}}^{trr}$ of $s_{i_i}^{trr}$ to $\oPMb_{\bm{G}}$ is just  $\tilde{\mathfrak{t}}^{\vee}_{i_1,\mathbb L \to \mathbb L^*,\bm{G}}\otimes \nontilde{s}^{trr*}_{i_1,\bm{G}}$ since we took $U^{R,T}_{sp}$, $U^1_{nsp}$ and $U^2_{nsp}$ small enough. 
        Moreover, since $\Gamma_{i_1,\bm{G}}$ is genus-one, we have $\hat{\mathcal B}\Gamma_{i_1,\bm{G}}=\mathcal B \Gamma_{i_1,\bm{G}}$, therefore $\tilde{\mathfrak{t}}_{i_1,{\mathcal B}\Gamma_{i_1,\bm{G}} \to \hat{\mathcal B}\Gamma_{i_1,\bm{G}}}$, and hence $\tilde{\mathfrak{t}}_{i_1,\mathbb L \to \mathbb L^*,\bm{G}}$, are just constant identity sections of trivial complex line bundles, so we have        $s_{i_i,\bm{G}}^{trr}=s_{i_i,\bm{G}}^{trr*}$ and both of them are the pullback of $\mathfrak s_{i_1, \hat{\mathcal B}\Gamma_{i_1,\bm{G}}}$.
        On the other hand, since $\bm{s}$ is $\gamma$-canonical, by definition the restriction $t_{D,\bm{G}}$ of $t_D$ to $\oPMb_{\bm{G}}$ is pulled back from a multisection $t_{D,\hat{\mathcal B} \Gamma_{i_1,\bm{G}}}$ via $\pi_{\bm{G}\to \Gamma_{i_1,\bm{G}}}$ and $\operatorname{For}_{\Gamma_{i_1,\bm{G}}\to \hat{\mathcal B} \Gamma_{i_1,\bm{G}}}$, 
        and $t_{D,\hat{\mathcal B} \Gamma_{i_1,\bm{G}}}$ coincide with $\mathfrak s_{i_1,\hat{\mathcal B} \Gamma_{i_1,\bm{G}}}$ in $\oPMb_{\hat{\mathcal B} \Gamma_{i_1,\bm{G}}}\setminus K_{\hat{\mathcal B} \Gamma_{i_1,\bm{G}}}$, where $K_{\hat{\mathcal B} \Gamma_{i_1,\bm{G}}}$ is a neighbourhood of the zero locus of $\mathfrak s_{i_1,\hat{\mathcal B} \Gamma_{i_1,\bm{G}}}$. 
        Notice that, according to item \ref{item def tame non vanish in K} in Definition \ref{def tame perturbation}, 
        the multisection $\bm{s}^{re}_{\hat{\mathcal B} \Gamma_{i_1,\bm{G}}}$ vanishes no where in $K_{\hat{\mathcal B }\Gamma_{i_1,\bm{G}}}\subseteq \oPMb_{\hat{\mathcal B} \Gamma_{i_1,\bm{G}}}$ as long as 
        $
        \rk E_{\hat{\mathcal B} \Gamma_{i_1,\bm{G}}}\ge \dim \oPMb_{\hat{\mathcal B} \Gamma_{i_1,\bm{G}}}.
        $
    
        \end{obs}
        
        Now we construct a homotopy $H$ between $\nontilde{s}^{trr}_{i_1}$ and $t_D$ as a multisection of 
        $$\pi^*_{\mathcal M}\mathbb L_{i_1}\to \oPMb^{r,0}_{1,B,I}\times[0,1];$$ 
        where $\pi_{\mathcal M}$ is the projection to the first component; we require $H$ satisfy the following conditions.
       \begin{enumerate}
            \item $H$ glue to a global multisection of $\pi^*_{\mathcal M}\mathbb L_{i_1}\to \widetilde{\mathcal{PM}}^{r,0}_{1,B,I}\times[0,1]$.
            \item For each $\bm{G}\in \sGPI^{r,0}_{1,B,I}$, the restriction $H_{\bm{G}}$ of $H$ to $\oPMb_{\bm{G}}\times [0,1]$ is of the form 
           \begin{equation}\label{eq homotopy is pulled back}
            H_{\bm{G}}= (\pi_{\bm{G}\to \Gamma_{i_1,\bm{G}}}\times \Id_{[0,1]})^* H_{\Gamma_{i_1,\bm{G}}},
            \end{equation}
            where $H_{\Gamma_{i_1,\bm{G}}}$ is a multisection of $\pi^*_{\mathcal M} \operatorname{For}^*_{\Gamma_{i_1,\bm{G}}\to \mathcal B \Gamma_{i_1,\bm{G}}}\mathbb L_{i_1,\mathcal B \Gamma_{i_1,\bm{G}}}\to \oPMb_{\Gamma_{i_1,\bm{G}}}\times [0,1]$. 
            \item 
            When $\rk E_{\Gamma_{i_1,\bm{G}}}> \dim \oPMb_{\Gamma_{i_1,\bm{G}}}$, or $\rk E_{\Gamma_{i_1,\bm{G}}}= \dim \oPMb_{\Gamma_{i_1,\bm{G}}}$ in the case $\Gamma_{i_1,\bm{G}}$ is genus-one, the homotopy $H_{\Gamma_{i_1,\bm{G}}}$ vanishes nowhere on $Z\left(\bm{s}^{re}_{\Gamma_{i_1,\bm{G}}}\right)\times [0,1]\subset \oPMb_{\Gamma_{i_1,\bm{G}}}\times [0,1]$,
            where $Z\left(\bm{s}^{re}_{\Gamma_{i_1,\bm{G}}}\right)\subset \oPMb_{{\Gamma_{i_1,\bm{G}}}}$ of is the zero locus of $\bm{s}^{re}_{\Gamma_{i_1,\bm{G}}}$.
        \end{enumerate}

        To construct such $H$, we first take $H'$ to be the linear homotopy between $s^{trr}_{i_1}$ and $t_D$, \textit{i.e.,} $H'(p,\lambda):=(1-\lambda)\cdot s^{trr}_{i_1}+\lambda \cdot t_D$ for any $\lambda\in [0,1]$. Note that $H'$ satisfies the first two requirements. If $\Gamma_{i_1,\bm{G}}$ is genus-one, the third requirement is also satisfied: by Observation \ref{obs on genus one K} we have 
        $$Z\left(\bm{s}^{re}_{\Gamma_{i_1,\bm{G}}}\right)\subset \oPMb_{\Gamma_{i_1,\bm{G}}}\setminus \operatorname{For}^{-1}_{\Gamma_{i_1,\bm{G}} \to \hat{\mathcal B}\Gamma_{i_1,\bm{G}}} \left(K_{\hat{\mathcal B}\Gamma_{i_1,\bm{G}}}\right),$$
        while for any $p\in \oPMb_{\Gamma_{i_1,\bm{G}}}\setminus \operatorname{For}^{-1}_{\Gamma_{i_1,\bm{G}} \to \hat{\mathcal B}\Gamma_{i_1,\bm{G}}} \left(K_{\hat{\mathcal B}\Gamma_{i_1,\bm{G}}}\right)$ and $\lambda\in [0,1]$ we have $$H'_{\Gamma_{i_1,\bm{G}}}(p,\lambda)=(1-\lambda)\cdot \operatorname{For}^*_{\Gamma_{i_1,\bm{G}} \to \hat{\mathcal B}\Gamma_{i_1,\bm{G}}} \mathfrak s_{i_1,\hat{\mathcal B}\Gamma_{i_1,\bm{G}}}(p) +\lambda \cdot t_{D,\Gamma_{i_1,\bm{G}}}(p)=\operatorname{For}^*_{\Gamma_{i_1,\bm{G}} \to \hat{\mathcal B}\Gamma_{i_1,\bm{G}}} \mathfrak s_{i_1,\hat{\mathcal B}\Gamma_{i_1,\bm{G}}}(p)\ne 0$$
        as the zero locus of $\mathfrak s_{i_1,\hat{\mathcal B}\Gamma_{i_1,\bm{G}}}$ is contained in $K_{\hat{\mathcal B}\Gamma_{i_1,\bm{G}}}$.

        When $\Gamma_{i_1,\bm{G}}$ is genus-zero and $\rk E_{\Gamma_{i_1,\bm{G}}}> \dim \oPMb_{\Gamma_{i_1,\bm{G}}}$, the homotopy $H'_{\Gamma_{i_1,\bm{G}}}$ might vanish on $Z\left(\bm{s}^{re}_{\Gamma_{i_1,\bm{G}}}\right)\times [0,1]$; however, this can only happen when  $\rk E_{\Gamma_{i_1,\bm{G}}}= \dim \oPMb_{\Gamma_{i_1,\bm{G}}}+2$, because otherwise we have  $\rk E^{re}_{\Gamma_{i_1,\bm{G}}}= \rk E_{\Gamma_{i_1,\bm{G}}}-2> \dim \oPMb_{\Gamma_{i_1,\bm{G}}}$,\footnote{Since every boundary marking of $\Gamma_{i_1,\bm{G}}$ is legal,  by \eqref{eq parity}, \eqref{eq rank of witten bundle} and \eqref{eq dim of moduli} we have $\rk E_{\Gamma_{i_1,\bm{G}}}\equiv \dim \oPMb_{\Gamma_{i_1,\bm{G}}} \mod 2$.} the transversality of $\bm{s}^{re}_{\Gamma_{i_1,\bm{G}}}$ implies $Z\left(\bm{s}^{re}_{\Gamma_{i_1,\bm{G}}}\right)=\emptyset$. So we only need to modify $H'_{\Gamma_{i_1,\bm{G}}}$ for genus-zero $\Gamma_{i_1,\bm{G}}$ with $\rk E^{re}_{\Gamma_{i_1,\bm{G}}}= \dim \oPMb_{\Gamma_{i_1,\bm{G}}}$. By the transversely decomposability of $\bm{s}^{re}_{\Gamma_{i_1,\bm{G}}}$ (item \ref{item canonical transversely decompose} in Definition \ref{dfn canonical section}), and the positive constraint for its $\mathcal W$-part (item \ref{item canonical positive} in Definition \ref{dfn canonical section}), the zero locus $Z\left(\bm{s}^{re}_{\Gamma_{i_1,\bm{G}}}\right)$ is a finite union of dimensional-zero points in $\mathcal{PM}_{\Gamma_{i_1,\bm{G}}}$. We then perturb $H'_{\Gamma_{i_1,\bm{G}}}$ in small neighbourhood of each of them to a homotopy $H_{\Gamma_{i_1,\bm{G}}}$, which coincides with $H'_{\Gamma_{i_1,\bm{G}}}$ on  $\oPMb_{\Gamma_{i_1,\bm{G}}}\times \{0,1\}$ and vanishes nowhere on $Z\left(\bm{s}^{re}_{\Gamma_{i_1,\bm{G}}}\right)\times [0,1]$, this is always possible because the transversality of $\bm{s}_{\Gamma_{i_1,\bm{G}}}$ and the transversely-decomposability of $\bm{s}^{re}_{\Gamma_{i_1,\bm{G}}}$ guarantee $H'_{\Gamma_{i_1,\bm{G}}}$ dose not vanish on the end points, while
        $$\rk \left(\pi^*_{\mathcal M} \operatorname{For}^*_{\Gamma_{i_1,\bm{G}}\to \mathcal B \Gamma_{i_1,\bm{G}}}\mathbb L_{i_1,\mathcal B \Gamma_{i_1,\bm{G}}}\right)=2>1=\dim \left(Z\left(\bm{s}^{re}_{\Gamma_{i_1,\bm{G}}}\right)\times [0,1]\right),$$
        guarantee that we can perturb away from the end points to avoid vanishing.
        
        The collection of these perturbed $H_{\Gamma_{i_1,\bm{G}}}$, together with those $H'_{\Gamma_{i_1,\bm{G}}}$ who do not need to be perturbed, determine  a new homotopy $H$ between $\nontilde{s}^{trr}_{i_1}$ and $t_D$ via \eqref{eq homotopy is pulled back}. The first requirement is still satisfied because the perturbations are only made in the interior  $\mathcal{PM}_{\Gamma_{i_1,\bm{G}}}$ of $\oPMb_{\Gamma_{i_1,\bm{G}}}$, without changing anything on $\partial\oPMb_{\Gamma_{i_1,\bm{G}}}$.

        Now we have the desired homotopy $H$, then $\pi^*_{\mathcal M} \bm{s}^{re}\oplus H$ is a homotopy between $\bm{s}=\bm{s}^{re}\oplus t_D$ and $\bm{s}^{mod}:=\bm{s}^{re} \oplus \nontilde{s}^{trr}_{i_1}$. 
        We can show that for each $\bm{G}\in \sGPI^{r,0}_{1,B,I}$, there exists an open neighbourhood $U_{\bm{G}}\subseteq \Mbar_{\bm{G}}$ of 
     $$\partial^{R}\Mbar_{\bm{G}}\cup \partial^{NS+}\Mbar_{\bm{G}}\cup \mathcal Z^{dj}_{\bm{G}}\subset \Mbar_{\bm{G}},$$ 
     such that $\pi^*_{\mathcal M} \bm{s}^{re}\oplus H$ vanishes nowhere on 
     $$
     \left((U_{\bm{G}} \cup \partial^{CB}\Mbar_{\bm{G}})\cap \oPMb_{\bm{G}}\right)\times [0,1].
     $$
     
     In fact, for any $\bm{G}\in \bm{\mathcal H}$, we can take $U_{\bm{G}}=\Mbar_{\bm{G}}$. This is because $\rk E_{\bm{G}}=\dim \oPMb_{\bm{G}}$ implies that (see Observation \ref{obs non vanish when exist g=1 component}) we must have a vertex $\Gamma\in V(\bm{G})$ which such that either $\rk E_{\Gamma}>\dim \oPMb_{\Gamma}$, or $\rk E_{\Gamma}=\dim \oPMb_{\Gamma}$ and $\Gamma$ is genus-one. If $i_1 \notin T^I(\Gamma)$, then $\bm{s}_{\Gamma}$ vanishes nowhere on $\oPMb_{\Gamma}$ by transversality or Remark \ref{rmk on non vanish tame rk >= dim};  If $i_1 \in T^I(\Gamma)$, then $\pi^*_{\mathcal M}\bm{s}^{re}_{\Gamma}\oplus H$ vanishes nowhere on $\oPMb_{\Gamma}\times [0,1]$ by the third requirement above. In any case $\pi^*_{\mathcal M} \bm{s}^{re}\oplus H$ has a nowhere-vanishing direct summand so it is nowhere vanishing on the entire $\oPMb_{\bm{G}}$.

     If $\bm{G}\in \sGPI_{1,B,I}^{r,0}$, then any vertex of $\bm{G}$ is genus-zero, therefore $\mathcal Z^{dj}_{\bm{G}}=\emptyset$. In this case the existence of such $U_{\bm{G}}$ follows from the positivity constraint (item \ref{item canonical positive} in Definition \ref{dfn canonical section}) on the $\mathcal{W}$-part of $\bm{s}_{\Gamma}$ or $\bm{s}^{re}_{\Gamma}$ for all (genus-zero) vertices $\Gamma\in V(\bm{G})$.

     Since after gluing we have $\partial \widetilde{\mathcal{PM}}^{1/r,0}_{1,B,I}=\partial^{CB}\widetilde{\mathcal{PM}}^{1/r,0}_{1,B,I}$, the collection $\{U_{\bm{G}}\}$ induces an open neighbourhood $\widetilde{U}\subset \widetilde{\mathcal{PM}}^{1/r,0}_{1,B,I}$ of $\Mbar^{1/r,0}_{1,B,I}\setminus \widetilde{\mathcal{PM}}^{1/r,0}_{1,B,I}$, such that the glued homotopy $H$ between glued $\bm{s}$ and glued $\bm{s}^{re}$ vanishes nowhere in $$\left(\left(\widetilde{U}\cap \widetilde{\mathcal{PM}}^{1/r,0}_{1,B,I}\right)\cup \partial \widetilde{\mathcal{PM}}^{1/r,0}_{1,B,I}\right)\times [0,1].$$
     Note that $\widetilde{\mathcal{PM}}^{1/r,0}_{1,B,I}\setminus \left(\widetilde{U}\cap \widetilde{\mathcal{PM}}^{1/r,0}_{1,B,I}\right)= \Mbar^{1/r,0}_{1,B,I}\setminus \widetilde{U}$ is compact, by \cite[Lemma 4.12]{BCT2}, $\bm{s}$ has the same zero count (with multiplicity) as $\bm{s}^{mod}$ in $\widetilde{\mathcal{PM}}^{1/r,0}_{1,B,I}$, \textit{i.e.},
     \begin{equation}\label{eq zero count coincide}
     \# Z(\bm{s})  = \# Z(\bm{s}^{mod}).     
     \end{equation}

    \noindent\textbf{Step 5.} 
        Finally we count the number of zeros $Z(\bm{s}^{mod})$ of the multisection $\bm{s}^{mod}:=\bm{s}^{re} \oplus \nontilde{s}^{trr}_{i_1}$ in $\oPMb^{1/r,0}_{1,B,I}$. According the Observation \ref{obs zero locus of modified canonical}, the zeros of $\bm{s}^{mod}$ are contained in the loci 
         $$
         Z_{sp}^{R,T}:=\bigsqcup_{\bm{G}\in \mathcal F^{R,T}_{i_1,sp}}\oPMb_{\bm{G}}\subset \oPMb_{1,B,I}^{1/r,0}
         $$
         for $R\subseteq I\setminus\{i_1\}$ and $T\subseteq B$,
          $$
         Z_{nsp}^{1}:=\bigsqcup_{\bm{G}\in \mathcal F^{1}_{i_1,nsp}}\oPMb_{\bm{G}}\subset \oPMb_{1,B,I}^{1/r,0},
         $$
         together with
         $$
         Z_{co}^{R}:=\bigsqcup_{\bm{G}\in \mathcal G^{R}_{i_1,co}}\oPMb_{\bm{G}}\subset \oPMb_{1,B,I}^{1/r,0}
         $$
        for $R\subseteq I\setminus\{i_1\}$, $R\ne \emptyset$.
         
         The transversely-decomposability guarantees that $\bm{s}^{re}$ is transverse when restricted to these loci, and does not vanish on the intersection of any two of them. 

         We first count the  number of zeros (with multiplicity) of the multisection $\bm{s}^{re}$ on the locus $Z_{sp}^{R,T}$ for  $R\subseteq I\setminus\{i_1\}$ and $T\subseteq B$. 
          Recall that we have a detach morphism 
         $$
         \operatorname{Detach}_{sp}^{R,T}\colon Z^{R,T}_{sp} \to \oPMb^{1/r,0}_{1,B\sqcup \{b^{R,T}_{sp}\} \setminus T,I\setminus (R\sqcup \{i_1\})} \times \oPMb^{1/r,0}_{0,T,R\sqcup \{a^{R,T}_{sp}\}}\times \Mbar_{v_{sp,c}^{R,T}}.
         $$
         We denote by $\pi_{o}$ the projection
         \begin{equation*}
              \oPMb^{1/r,0}_{1,B\sqcup \{b^{R,T}_{sp}\} \setminus T,I\setminus (R\sqcup \{i_1\})} \times \oPMb^{1/r,0}_{0,T,R\sqcup \{a^{R,T}_{sp}\}}\times \Mbar_{v_{nsp,c}} \to  \oPMb^{1/r,0}_{1,B\sqcup \{b^{R,T}_{sp}\} \setminus T,I\setminus (R\sqcup \{i_1\})} \times \oPMb^{1/r,0}_{0,T,R\sqcup \{a^{R,T}_{sp}\}}
         \end{equation*}
         to the first two factors, it is a morphism of degree-$1/r$.  

         On the other hand, note that for any $\bm{G}\in \mathcal F_{i_1,sp}^{R,T}$, if we use the notation of Step 2 and take $S=\left\{i_{sp}^{\Gamma_{i_1,\bm{G}}}\right\}$, by construction in \S \ref{sec def of tame and canonical} we have 
         $\mathcal T_{S}\Gamma_{i_1,\bm{G}}=\operatorname{for}_{S}(d_{e_{\bm{G}}}\Gamma_{i_1,\bm{G}})$, thus 
         $$
         \operatorname{For}_{\Gamma_{i_1,\bm{G}}\to \mathcal T_{S}\Gamma_{i_1,\bm{G}}}\colon \oPMb_{\Gamma_{i_1,\bm{G}}}\to \oPMb_{\mathcal T_{S}\Gamma_{i_1,\bm{G}}}=\oPMb_{\operatorname{for}_{S}(d_{e_{\bm{G}}}\Gamma_{i_1,\bm{G}})}
         $$ 
         is just the restriction of 
         $\operatorname{For}_{d_{e_{\bm{G}}}\Gamma_{i_1,\bm{G}}\to \operatorname{for}_{S}(d_{e_{\bm{G}}}\Gamma_{i_1,\bm{G}})}$ to $\oPMb_{\Gamma_{i_1,\bm{G}}}\subset \oPMb_{d_{e_{\bm{G}}}\Gamma_{i_1,\bm{G}}}$. These forgetful morphisms for all $\bm{G}\in \mathcal F_{i_1,sp}^{R,T}$ induce a morphism
         $$
         \operatorname{For}_{i_{sp}}^{R,T}\colon Z_{sp}^{R,T}\to \oPMb^{1/r,0}_{1,B\sqcup \{b^{R,T}_{sp}\} \setminus T,I\setminus (R\sqcup \{i_1\})} \times \oPMb^{1/r,0}_{0,T,R\sqcup \{a^{R,T}_{sp}\}},
         $$
         and we have 
         $$
         \operatorname{For}_{i_{sp}}^{R,T}=\pi_o\circ  \operatorname{Detach}_{sp}^{R,T}
         $$
         after identifying the internal marking $a^{R,T}_{sp}$ in $ \oPMb^{1/r,0}_{0,T,R\sqcup \{a^{R,T}_{sp}\}}$ as the internal marking $i_1$ in $ \oPMb^{1/r,0}_{0,T,R\sqcup \{i_1\}}$; it is a degree-$1/r$ morphism.

         By the transversely-decomposability, the restriction of $\bm{s}^{re}$ to ${Z}^1_{nsp}$ is of the form
         \begin{equation}\label{eq split of remaining section at sp locus}
         \bm{s}^{re}\vert_{Z_{sp}^{R,T}}= \operatorname{For}_{i_{sp}}^{R,T*} \left( \bm{s}\vert_{1,B\sqcup \{b^{R,T}_{sp}\} \setminus T,I\setminus (R\sqcup \{i_1\})} \boxplus \bm{s}\vert_{0,T,R\sqcup \{i_1\}} \right),
         \end{equation}
         where $\bm{s}\vert_{1,B\sqcup \{b^{R,T}_{sp}\} \setminus T,I\setminus (R\sqcup \{i_1\})}$ is a multisection of 
         $$
         E_{1,B\sqcup \{b^{R,T}_{sp}\} \setminus T,I\setminus (R\sqcup \{i_1\})} := \mathcal W_{1,B\sqcup \{b^{R,T}_{sp}\} \setminus T,I\setminus (R\sqcup \{i_1\})}\oplus \bigoplus_{i\in I\setminus (R\sqcup \{i_1\})} \mathbb L_i^{\oplus d_i}\to \oPMb^{1/r,0}_{1,B\sqcup \{b^{R,T}_{sp}\} \setminus T,I\setminus (R\sqcup \{i_1\})}
         $$
         and 
          $\bm{s}\vert_{0,T,R\sqcup \{i_1\}}$ is a multisection of 
         $$
         E_{0,T,R\sqcup \{i_1\}} := \mathcal W_{0,T,R\sqcup \{i_1\}} \oplus \bigoplus_{i \in R\sqcup \{i_1\}} \mathbb L_i^{\oplus d_i}\to \oPMb^{1/r,0}_{0,T,R\sqcup \{i_1\}}.
         $$
         Moreover, because $\bm{s}$ is $\gamma$-canonical, the multisection $\bm{s}\vert_{0,T,R\sqcup \{i_1\}}$ is canonical, and the multisection $\bm{s}\vert_{1,B\sqcup \{b^{R,T}_{sp}\} \setminus T,I\setminus (R\sqcup \{i_1\})}$ is $\gamma\vert_{I\setminus (R\sqcup \{i_1\})}$-canonical (see Remark \ref{rmk on restriction of tame perturbation on g=1 vertex}).

         Now we show that the relative orientation of  $E^{re}_{1,B,I}\vert_{Z_{sp}^{R,T}}\to Z_{sp}^{R,T}$ induced by the canonical relative orientation (see \eqref{eq orientation point insertion}) of $E_{1,B,I}\to \oPMb^{r,0}_{1,B,I}$ and $s^{trr}_{i_1}$ (away from $U^2_{nsp}$), is the same as the one induced by the canonical relative orientations of $E_{1,B\sqcup \{b^{R,T}_{sp}\} \setminus T,I\setminus (R\sqcup \{i_1\})} \to \oPMb^{1/r,0}_{1,B\sqcup \{b^{R,T}_{sp}\} \setminus T,I\setminus (R\sqcup \{i_1\})})$ and $E_{0,T,R\sqcup \{i_1\}} \to \oPMb^{1/r,0}_{0,T,R\sqcup \{i_1\}}$ via $\operatorname{For}_{i_{sp}}^{R,T*}$. In fact, we only need to check it for the relative orientation of the corresponding Witten bundles as all the relative cotangent lines are orientated by their complex structure. For $\bm{G}\in \mathcal F_{i_1,sp}^{R,T}$, the relative orientation of $\mathcal W_{d_{e_{\bm{G}}}\bm{G}}\to \oPMb_{d_{e_{\bm{G}}}\bm{G}}$ is (see \eqref{eq orientation point insertion}) 
         $$
         (-1)^{\lvert E(d_{e_{\bm{G}}}\bm{G})\rvert}\bboxtimes_{\Gamma\in V(d_{e_{\bm{G}}}\bm{G})} o_{\Gamma},
         $$
         by  Lemma \ref{lem zero locus of t between L} and \cite[Proposition 3.1 and Corollary 3.20]{TZ1}
         the relative orientation on $\oPMb_{\bm{G}}\subset  \oPMb_{d_{e_{\bm{G}}}\bm{G}}$ induced by $s^{trr}_{i_1,d_{e_{\bm{G}}}\bm{G}}$ is 
         \begin{equation}\label{eq induce rative orientation sp case}
         (-1)^{\lvert E(d_{e_{\bm{G}}}\bm{G})\rvert+1}o_{v_{\Gamma_{i_1,\bm{G},c}}}\boxtimes o_{v_{\Gamma_{i_1,\bm{G},o}}} \boxtimes\bboxtimes_{\Gamma\in V(d_{e_{\bm{G}}}\bm{G})\setminus \{d_{e_{\bm{G}}}\Gamma_{i_1,\bm{G}}\}} o_{\Gamma},
         \end{equation}
         where we have an additional factor $-1$ because the orientation induced by $s^{trr}_{i_1,\bm{G}}$ on $\oPMb_{\bm{G}}$ is opposite to the one induced by $\tilde{\mathfrak t}_{i_1,\mathbb L \to \mathbb L^*,\bm{G}}$. One the other hand, assuming $\operatorname{For}_{i_{sp}}^{R,T}(\oPMb_{\bm{G}})=\oPMb_{\bm{G}_1}\times \oPMb_{\bm{G}_0}$ for $\bm{G}_1\in \sGPI^{r,0}_{1,B\sqcup \{b^{R,T}_{sp}\} \setminus T,I\setminus (R\sqcup \{i_1\})}$ and $\bm{G}_0\in \sGPI^{r,0}_{0,T,R\sqcup \{i_1\}}$, then the orientation on  $\oPMb_{\bm{G}}$ induces by $o_{\bm{G}_1}$ and $o_{\bm{G}_0}$ is 
         \begin{equation}\label{eq orientation product after remove dash line}
             (-1)^{\lvert E(\bm{G}_1)\rvert}\bboxtimes_{\Gamma\in V(\bm{G}_1)} o_{\Gamma} \boxtimes (-1)^{\lvert E(\bm{G}_0)\rvert}\bboxtimes_{\Gamma\in V(\bm{G}_0)} o_{\Gamma}.
         \end{equation}
         The two relative orientations \eqref{eq induce rative orientation sp case} and \eqref{eq orientation product after remove dash line} coincide because we have
         $$V(\bm{G}_1)\sqcup V(\bm{G}_2)=\{v_{\Gamma_{i_1,\bm{G},o}}\}\sqcup V(d_{e_{\bm{G}}}\bm{G})\setminus \{d_{e_{\bm{G}}}\Gamma_{i_1,\bm{G}}\},$$  
         $\Mbar_{v_{\Gamma_{i_1,\bm{G},c}}}$ is dimension-zero with rank-zero Witten bundle, and, since we removed a dashed line,
         $$\lvert E(\bm{G}_1)\rvert+\lvert E(\bm{G}_0)\rvert=\lvert E(d_{e_{\bm{G}}}\bm{G})\rvert-1.$$
        Under this relative orientation, \eqref{eq split of remaining section at sp locus} implies (recall that $\operatorname{For}_{i_{sp}}^{R,T}$ is degree-$1/r$)
         $$\# Z\left( \bm{s}^{re}\vert_{Z_{sp}^{R,T}}\right)=\frac{1}{r}\#Z(\bm{s}\vert_{1,B\sqcup \{b^{R,T}_{sp}\} \setminus T,I\setminus (R\sqcup \{i_1\})})\cdot \#Z(\bm{s}\vert_{0,T,R\sqcup \{i_1\}}).$$
         Because the vanishing order of $s^{trr}_{i_1}$ along $Z_{sp}^{R,T}$ is $r$ (see Remark \ref{rmk on vanishing order r}), the count of zeros of $\bm{s}^{mod}=\bm{s}^{re}\oplus s^{trr}_{i_1}$ in the locus $Z_{sp}^{R,T}$ with multiplicity is
         \begin{equation}\label{eq trr sp terms}
             \left\langle \tau^{a_{i_1}}_{d_{i_1}}\prod_{i \in R} \tau^{a_i}_{d_i}\prod_{i\in T}\sigma^{b_i}\right\rangle^{\frac{1}{r},o}_0 \hspace{-0.1cm}\left\langle \prod_{i \in I\setminus(R\sqcup \{i_1\})} \tau^{a_i}_{d_i} \sigma^{r-2}\prod_{i\in B\setminus T}\sigma^{b_i}\right\rangle^{\frac{1}{r}, o,\gamma\vert_{I\setminus(R\sqcup \{i_1\})},\bm{s}\vert_{1,B\sqcup \{b^{R,T}_{sp}\} \setminus T,I\setminus (R\sqcup \{i_1\})}}_1.
         \end{equation}

         By the same argument, the count of zeros of $\bm{s}^{mod}=\bm{s}^{re}\oplus s^{trr}_{i_1}$ in the locus $Z_{nsp}^{1}$ with multiplicity is
         \begin{equation}\label{eq trr nsp term}
             \consta\left\langle \prod_{i\in I}\tau^{a_i}_{d_i}\sigma^{r-2}\prod_{i\in B}\sigma^{b_i}\right\rangle^{\frac{1}{r},o}_0,
         \end{equation}
         there is an additional coefficient $1/2$ because $Z_{nsp}^{1}$ is the zero locus of only one branch of $s^{trr}_{i_1}$, which has weight $1/2$.
         
         The count of zeros of $\bm{s}^{mod}=\bm{s}^{re}\oplus s^{trr}_{i_1}$ in the locus $Z_{co}^{R}$ is 
             \begin{equation}\label{eq trr co terms}
                \hspace{-0.1cm}\left\langle \tau_0^{a}\tau_{d_{i_1}}^{a_{i_1}}\prod_{i \in R}\tau_{d_i}^{a_i}\right\rangle^{\frac{1}{r},\text{ext}}_0
                \hspace{-0.1cm}\left\langle \tau_0^{r-2-a}\prod_{i\in I\setminus(R \sqcup\{i_1\})}\tau^{a_i}_{d_i}\prod_{i\in B}\sigma^{b_i}\right\rangle^{\frac{1}{r},o,\gamma\vert_{I\setminus(R \sqcup\{i_1\})},\bm{s}\vert_{1,B,I\sqcup\{a_{co}^R\}\setminus(R \sqcup\{i_1\})\}}}_1.
            \end{equation}
         In fact, similar to \eqref{eq detach sp} and \eqref{eq detach nsp}, we have a detach morphism
        \begin{equation}\label{eq detach co}
         \operatorname{Detach}_{co}^{R}\colon Z^{R}_{co} \to \Mbar_{v_{co,c}^R}\times \oPMb^{1/r,0}_{1,B,I\sqcup \{a_{co}^R\}\setminus (R\sqcup \{i_1\})},
         \end{equation}
         where $v_{co,c}^R$ is a closed single-vertex $r$-spin graph with tails in $R\sqcup \{i_1, n_{\bm{G}}\}$, $n_{\bm{G}}$ and $a_{co}^R$ are the internal markings corresponding to the two half-nodes of $e_{\bm{G}}$ after detaching. Then \eqref{eq trr co terms} can be obtained by a similar argument as \eqref{eq trr sp terms} and \eqref{eq trr nsp term}; a fully detailed argument genus-zero case (which is identical) can be found in \cite[Proof of Lemma 4.14]{BCT2}.

         We can compute $\#Z(\bm{s}^{mod})$ by summing \eqref{eq trr co terms}, \eqref{eq trr sp terms} and \eqref{eq trr nsp term} over all possible partitions, then we get the right-hand side of \eqref{eq trr g=1 ordered}. The left-hand side of \eqref{eq trr g=1 ordered} equals $\#Z(\bm{s})$ by definition, therefore \eqref{eq trr g=1 ordered} follows from \eqref{eq zero count coincide}.

\end{proof}

\begin{rmk}
    In the genus-zero case, Theorem \ref{thm TRR BCT} can be proven directly using a similar argument as above.
\end{rmk}

\subsection{Independent of choice for genus-one open $r$-spin correlators}\label{sec g=1 indepent of choice and trr}
\begin{thm}\label{thm indenpdent of order}
     For any two orders $\gamma_1$ and $\gamma_2$ on the set $\{i\in I\colon d_i\ge 1\}$, and two choices of $\gamma_1$-canonical multisection $\bm{s}^1$ and $\gamma_2$-canonical multisection $\bm{s}^2$, we have
  \begin{equation}\label{eq indep of choice g=1}
  \left\langle \prod_{i\in I}\tau^{a_i}_{d_i}\prod_{i\in B}\sigma^{b_i}\right\rangle^{\frac{1}{r},o,\gamma_1,\bm{s}^1}_1=\left\langle \prod_{i\in I}\tau^{a_i}_{d_i}\prod_{i\in B}\sigma^{b_i}\right\rangle^{\frac{1}{r},o,\gamma_2,\bm{s}^2}_1.
  \end{equation}
\end{thm}

Theorem \ref{thm indenpdent of order} (which we will prove in the end of this subsection) indicates the genus-one $r$-spin correlators with respect to $\gamma$ and $\bm{s}$ are in fact independent of any choice of order $\gamma$ or $\gamma$-canonical multisection $\bm{s}$, therefore we can make the following definition.
\begin{definition}\label{def correlator unordered}
The \textit{genus-one $r$-spin invariants (or correlators)} are defined to be
    $$\left\langle \prod_{i\in I}\tau^{a_i}_{d_i}\prod_{i\in B}\sigma^{b_i}\right\rangle^{\frac{1}{r},o}_1:=\left\langle \prod_{i\in I}\tau^{a_i}_{d_i}\prod_{i\in B}\sigma^{b_i}\right\rangle^{\frac{1}{r},o,\gamma,\bm{s}}_1
  $$
for any order $\gamma$ on the set $\{i\in I\colon d_i\ge 1\}$ and and choice of $\gamma$-canonical multisections $\bm{s}$.
  \end{definition}
  \begin{rmk}
    Since every boundary twist $b_i$ is $r-2$ in this paper, we can simply write $\sigma$ in stead of $\sigma^{b_i}$. Following the convention in \cite{BCT2}, (and by an abuse of notation,) the open correlators $\left\langle \prod_{i\in I}\tau^{a_i}_{d_i}\prod_{i\in B}\sigma^{b_i}\right\rangle^{\frac{1}{r},o}_g$ is also be written as $\left\langle \prod_{i=1}^{\lvert I \rvert}\tau^{a_i}_{d_i}\sigma^{\lvert B \rvert}\right\rangle^{\frac{1}{r},o}_g$ in \S \ref{sec intro} for simplicity.
\end{rmk}

As a corollary of Proposition \ref{prop trr g=1 ordered} and Theorem \ref{thm indenpdent of order}, we have the following Topological Recursion Relation for $g=1$ $r$-spin theory.
\begin{thm}[Topological Recursion Relation for $g=1$ $r$-spin theory]\label{thm trr g=1 unordered}
For any $i_1\in I$ we have a TRR with respect to $i_1$:
\begin{equation}\label{eq trr g=1 unordered}
\begin{split}
&\left\langle \tau_{d_{i_1}+1}^{a_{i_1}}\prod_{i\in I\setminus \{i_1\}}\tau^{a_i}_{d_i}\prod_{i\in B}\sigma^{b_i}\right\rangle^{\frac{1}{r},o}_1\hspace{-0.2cm}\\
=&\sum_{\substack{R_1 \sqcup R_2 = I\setminus\{i_1\}\\ -1\le a \le r-2}}\hspace{-0.1cm}\left\langle \tau_0^{a}\tau_{d_{i_1}}^{a_{i_1}}\prod_{i \in R_1}\tau_{d_i}^{a_i}\right\rangle^{\frac{1}{r},\text{ext}}_0
\hspace{-0.1cm}\left\langle \tau_0^{r-2-a}\prod_{i\in R_2}\tau^{a_i}_{d_i}\prod_{i\in B}\sigma^{b_i}\right\rangle^{\frac{1}{r},o}_1\\
&+\hspace{-0.1cm}\sum_{\substack{R_1 \sqcup R_2 =  I\setminus\{i_1\} \\ T_1 \sqcup T_2 =  B}} \hspace{-0.1cm} \left\langle \tau^{a_{i_1}}_{d_{i_1}}\prod_{i \in R_1} \tau^{a_i}_{d_i}\prod_{i\in T_1}\sigma^{b_i}\right\rangle^{\frac{1}{r},o}_0 \hspace{-0.1cm}\left\langle \prod_{i \in R_2} \tau^{a_i}_{d_i} \sigma^{r-2}\prod_{i\in T_2}\sigma^{b_i}\right\rangle^{\frac{1}{r}, o}_1\\
&+\consta\left\langle \prod_{i\in I}\tau^{a_i}_{d_i}\sigma^{r-2}\prod_{i\in B}\sigma^{b_i}\right\rangle^{\frac{1}{r},o}_0.
    \end{split}
\end{equation}
\end{thm}

\begin{proof}[Proof of Theorem \ref{thm indenpdent of order} and Theorem \ref{thm trr g=1 unordered}]
    We prove the \eqref{eq indep of choice g=1}, and hence \eqref{eq trr g=1 unordered}, by an induction on $D=\sum_{i\in I} d_i$.

    The case $D=0$ is trivially true: according to Remark \ref{rmk vanish without psi g=1} the correlators in this case are always zero.

    In the case $D=1$, there is only one element $i_1\in I$ with $d_{i_1}=1>0$. Note that in this case $I^{d>0}=\{i_1\}$ is a single-element set, so the choice of the order $\gamma$ is unique. We apply \eqref{eq trr g=1 ordered}  to the correlator $\left\langle \tau_{1}^{a_{i_1}}\prod_{i\in I\setminus \{i_1\}}\tau^{a_i}_{0}\prod_{i\in B}\sigma^{b_i}\right\rangle^{\frac{1}{r},o,\gamma,\bm{s}}_1$ and express it as 
    \begin{equation}\label{eq apply trr D=1}
        \begin{split}
    &\sum_{\substack{R_1 \sqcup R_2 = I\setminus\{i_1\}\\ -1\le a \le r-2}}\hspace{-0.1cm}\left\langle \tau_0^{a}\tau_{0}^{a_{i_1}}\prod_{i \in R_1}\tau_{0}^{a_i}\right\rangle^{\frac{1}{r},\text{ext}}_0
    \hspace{-0.1cm}\left\langle \tau_0^{r-2-a}\prod_{i\in R_2}\tau^{a_i}_{0}\prod_{i\in B}\sigma^{b_i}\right\rangle^{\frac{1}{r},o,\gamma\vert_{R_2},\bm{s}\vert_{R_2\sqcup\{i^{r-2-a}_I\},B}}_1\\
    &+\hspace{-0.1cm}\sum_{\substack{R_1 \sqcup R_2 =  I\setminus\{i_1\} \\ T_1 \sqcup T_2 =  B}} \hspace{-0.1cm} \left\langle \tau^{a_{i_1}}_{0}\prod_{i \in R_1} \tau^{a_i}_{0}\prod_{i\in T_1}\sigma^{b_i}\right\rangle^{\frac{1}{r},o}_0 \hspace{-0.1cm}\left\langle \prod_{i \in R_2} \tau^{a_i}_{0} \sigma^{r-2}\prod_{i\in T_2}\sigma^{b_i}\right\rangle^{\frac{1}{r}, o,\gamma\vert_{R_2},\bm{s}\vert_{R_2,B\sqcup\{i^{r-2}_B\}}}_1\\
    &+\consta\left\langle \prod_{i\in I}\tau^{a_i}_{0}\sigma^{r-2}\prod_{i\in B}\sigma^{b_i}\right\rangle^{\frac{1}{r},o}_0.
    \end{split}
    \end{equation}
    By Remark \ref{rmk vanish without psi g=1} all the genus-one factors in \eqref{eq apply trr D=1} vanishes, the only term left is 
    $$
    \consta\left\langle \prod_{i\in I}\tau^{a_i}_{0}\sigma^{r-2}\prod_{i\in B}\sigma^{b_i}\right\rangle^{\frac{1}{r},o}_0
    $$
    which is independent of the choice of $\bm{s}$.

    Assuming we have proven \eqref{eq indep of choice g=1} and  \eqref{eq trr g=1 unordered} in in all the cases where $\sum_i d_i\le D_0$ for a integer $D_0\ge 1$,  we now prove \eqref{eq indep of choice g=1} and  \eqref{eq trr g=1 unordered} in the case $\sum_i d_i= D_0+1$. We take $i_1=\gamma_1(I)$ and use \eqref{eq trr g=1 ordered}  to the correlator 
    $\left\langle \tau_{d_{i_1}}^{a_{i_1}}\prod_{i\in I\setminus \{i_1\}}\tau^{a_i}_{d_i}\prod_{i\in B}\sigma^{b_i}\right\rangle^{\frac{1}{r},o,\gamma_1,\bm{s}^1}_1$ on the left-hand side of \eqref{eq indep of choice g=1}  as
    \begin{equation}\label{eq apply trr D>1}
        \begin{split}
    &\sum_{\substack{R_1 \sqcup R_2 = I\setminus\{i_1\}\\ -1\le a \le r-2}}\hspace{-0.1cm}\left\langle \tau_0^{a}\tau_{d_{i_1}-1}^{a_{i_1}}\prod_{i \in R_1}\tau_{d_i}^{a_i}\right\rangle^{\frac{1}{r},\text{ext}}_0
    \hspace{-0.1cm}\left\langle \tau_0^{r-2-a}\prod_{i\in R_2}\tau^{a_i}_{d_i}\prod_{i\in B}\sigma^{b_i}\right\rangle^{\frac{1}{r},o,\gamma_1\vert_{R_2},\bm{s}^1\vert_{R_2\sqcup\{i^{r-2-a}_I\},B}}_1\\
    &+\hspace{-0.1cm}\sum_{\substack{R_1 \sqcup R_2 =  I\setminus\{i_1\} \\ T_1 \sqcup T_2 =  B}} \hspace{-0.1cm} \left\langle \tau^{a_{i_1}}_{d_{i_1}-1}\prod_{i \in R_1} \tau^{a_i}_{d_i}\prod_{i\in T_1}\sigma^{b_i}\right\rangle^{\frac{1}{r},o}_0 \hspace{-0.1cm}\left\langle \prod_{i \in R_2} \tau^{a_i}_{d_i} \sigma^{r-2}\prod_{i\in T_2}\sigma^{b_i}\right\rangle^{\frac{1}{r}, o,\gamma_1\vert_{R_2},\bm{s}^1\vert_{R_2,B\sqcup\{i^{r-2}_B\}}}_1\\
    &+\consta\left\langle \tau_{d_{i_1}-1}^{a_1}\prod_{i\in I\setminus\{i_1\}}\tau^{a_i}_{d_i}\sigma^{r-2}\prod_{i\in B}\sigma^{b_i}\right\rangle^{\frac{1}{r},o}_0.
    \end{split}
    \end{equation}
    Note that every genus-one factor in \eqref{eq apply trr D>1} is independent of any choice of the order or the canonical multisection by inductive hypothesis as $\sum_{i\in R_2}d_i \le D_0$, therefore we can rewrite \eqref{eq apply trr D>1} as 
    \begin{equation}\label{eq apply trr D>1 v1}
        \begin{split}
    &\sum_{\substack{R_1 \sqcup R_2 = I\setminus\{i_1\}\\ -1\le a \le r-2}}\hspace{-0.1cm}\left\langle \tau_0^{a}\tau_{d_{i_1}-1}^{a_{i_1}}\prod_{i \in R_1}\tau_{d_i}^{a_i}\right\rangle^{\frac{1}{r},\text{ext}}_0
    \hspace{-0.1cm}\left\langle \tau_0^{r-2-a}\prod_{i\in R_2}\tau^{a_i}_{d_i}\prod_{i\in B}\sigma^{b_i}\right\rangle^{\frac{1}{r},o}_1\\
    &+\hspace{-0.1cm}\sum_{\substack{R_1 \sqcup R_2 =  I\setminus\{i_1\} \\ T_1 \sqcup T_2 =  B}} \hspace{-0.1cm} \left\langle \tau^{a_{i_1}}_{d_{i_1}-1}\prod_{i \in R_1} \tau^{a_i}_{d_i}\prod_{i\in T_1}\sigma^{b_i}\right\rangle^{\frac{1}{r},o}_0 \hspace{-0.1cm}\left\langle \prod_{i \in R_2} \tau^{a_i}_{d_i} \sigma^{r-2}\prod_{i\in T_2}\sigma^{b_i}\right\rangle^{\frac{1}{r}, o}_1\\
    &+\consta\left\langle \tau_{d_{i_1}-1}^{a_1}\prod_{i\in I\setminus\{i_1\}}\tau^{a_i}_{d_i}\sigma^{r-2}\prod_{i\in B}\sigma^{b_i}\right\rangle^{\frac{1}{r},o}_0.
    \end{split}
    \end{equation} 
    This implies that the correlator 
    $\left\langle \tau_{d_{i_1}}^{a_{i_1}}\prod_{i\in I\setminus \{i_1\}}\tau^{a_i}_{d_i}\prod_{i\in B}\sigma^{b_i}\right\rangle^{\frac{1}{r},o,\gamma_1,\bm{s}^1}_1$ does not depend on the choice of $\bm{s}^1$. Note that for the moment it still depends on $\gamma_1$ since $i_1$ in \eqref{eq apply trr D>1 v1} is chosen as $i_1=\gamma_1(I)$.

    We can also take $i_2:=\gamma_2(I)$ and express the correlator 
    $\left\langle \tau_{d_{i_2}}^{a_{i_2}}\prod_{i\in I\setminus \{i_2\}}\tau^{a_i}_{d_i}\prod_{i\in B}\sigma^{b_i}\right\rangle^{\frac{1}{r},o,\gamma_2,\bm{s}^2}_1$ on the right-hand side of \eqref{eq indep of choice g=1} using \eqref{eq trr g=1 ordered} as
    \begin{equation}\label{eq apply trr D>1 v2}
        \begin{split}
    &\sum_{\substack{R_1 \sqcup R_2 = I\setminus\{i_2\}\\ -1\le a \le r-2}}\hspace{-0.1cm}\left\langle \tau_0^{a}\tau_{d_{i_2}-1}^{a_{i_2}}\prod_{i \in R_1}\tau_{d_i}^{a_i}\right\rangle^{\frac{1}{r},\text{ext}}_0
    \hspace{-0.1cm}\left\langle \tau_0^{r-2-a}\prod_{i\in R_2}\tau^{a_i}_{d_i}\prod_{i\in B}\sigma^{b_i}\right\rangle^{\frac{1}{r},o}_1\\
    &+\hspace{-0.1cm}\sum_{\substack{R_1 \sqcup R_2 =  I\setminus\{i_2\} \\ T_1 \sqcup T_2 =  B}} \hspace{-0.1cm} \left\langle \tau^{a_{i_2}}_{d_{i_2}-1}\prod_{i \in R_1} \tau^{a_i}_{d_i}\prod_{i\in T_1}\sigma^{b_i}\right\rangle^{\frac{1}{r},o}_0 \hspace{-0.1cm}\left\langle \prod_{i \in R_2} \tau^{a_i}_{d_i} \sigma^{r-2}\prod_{i\in T_2}\sigma^{b_i}\right\rangle^{\frac{1}{r}, o}_1\\
    &+\consta\left\langle \tau^{a_{i_2}}_{d_{i_2}-1}\prod_{i\in I\setminus \{i_2\}}\tau^{a_i}_{d_i}\sigma^{r-2}\prod_{i\in B}\sigma^{b_i}\right\rangle^{\frac{1}{r},o}_0,
    \end{split}
    \end{equation}
    which only depends on $i_2=\gamma_2(I)$. Therefore \eqref{eq indep of choice g=1} is proven in the case $\gamma_1(I)=\gamma_2(I)$.

    To conclude the proof, we need to show that \eqref{eq apply trr D>1 v1} coincides with \eqref{eq apply trr D>1 v2} even if $i_1=\gamma_1(I)\ne \gamma_2(I) = i_2$. This can be done by applying the Topological Recursion Relations \eqref{eq trr closed extended}, \eqref{eq: bct trr 2} and \eqref{eq trr g=1 unordered} to the closed extended, open genus-zero, and open genus-one correlators in  \eqref{eq apply trr D>1 v1} and \eqref{eq apply trr D>1 v2}, in the following way. 
    \begin{itemize}
        \item[-] For the closed extended correlators $\left\langle \tau_0^{a}\tau_{d_{i_1}-1}^{a_{i_1}}\prod_{i \in R_1}\tau_{d_i}^{a_i}\right\rangle^{1/r,\text{ext}}_0$ in \eqref{eq apply trr D>1 v1}, if $i_2\notin R_1$, we leave it unchanged; if $i_2\in R_1$, we apply \eqref{eq trr closed extended} with respect to $(i_2,\{i_1,i_I^a\})$, where $i_I^a$ is the internal tail corresponding to $\tau_0^a$ in the correlator.
        \item[-] For the genus-zero open correlators of the form $\left\langle \tau_{d_{i_1}-1}^{a_{i_1}}\prod_{i \in I'}\tau_{d_i}^{a_i} \prod_{i\in B'}\sigma^{b_i}\right\rangle^{1/r,o}_0$ in \eqref{eq apply trr D>1 v1}, if $i_2\notin I'$, we leave it unchanged; if $i_2\in I'$, we apply \eqref{eq: bct trr 2} with respect to $(i_2,i_1)$.
        \item[-] For the genus-one open correlators of the form $\left\langle \prod_{i \in I'}\tau_{d_i}^{a_i} \prod_{i\in B'}\sigma^{b_i}\right\rangle^{1/r,o}_1$ in \eqref{eq apply trr D>1 v1}, if $i_2\notin I'$, we leave it unchanged; if $i_2\in I'$, we apply \eqref{eq trr g=1 unordered} with respect to $i_2$. Note that \eqref{eq trr g=1 unordered} is already proven in this case by inductive hypothesis. 
    \end{itemize}

    After substituting these correlators with the results after applying TRRs, we get a new expression for \eqref{eq apply trr D>1 v1}. Note that in this expression $i_1$ and $i_2$ are symmetry.

    We can also apply TRRs to the correlators in \eqref{eq apply trr D>1 v2} in the same way after exchanging the roles of $i_1$ and $i_2$. In this way we get a new expression for \eqref{eq apply trr D>1 v2}, which is exactly the same as the new expression for \eqref{eq apply trr D>1 v1}. This proves \eqref{eq indep of choice g=1}.

\end{proof}

\subsection{Proof of Theorems \ref{thm:wave_func} and \ref{thm:open_string_dilaton}}
\begin{proof}[Proof of Theorem \ref{thm:wave_func}]
It was proven in \cite[\S 3.4]{BCT3} that the right hand side of \eqref{eq:bct_conj} satisfies the open $g=1$ TRR, Theorem~\ref{thm:trr_g1}. The TRR clearly fixed all coefficients of $\phi_1,$ hence Theorem~\ref{thm:wave_func} follows from Theorem~\ref{thm:trr_g1}.
\end{proof}
\begin{proof}[Proof of Theorem \ref{thm:open_string_dilaton}]
It was proven in \cite[\S 3.3]{BCT3} that the conjectural all genus potential given in Conjecture \ref{conj:main_conj_BCT} satisfies the Open String and Dilaton equations. Since we have proven this conjecture for $g=1,$ in Theorem \ref{thm:wave_func} it follows that these equations indeed hold for the open $r$-spin cylinder theory.
\end{proof}

\section{Proof of Proposition \ref{prop exist canonical section}}\label{sec construction of section}
In this section we construct a $\gamma$-canonical multisection (Definition \ref{dfn canonical section}) of
$$E=\mathcal W\oplus \bigoplus_{i\in I} \mathbb L_i^{\oplus d_i}\to \oPMb^{1/r,0}_{1,B,I}$$
for any order $\gamma$ on the set $I^{d>0}=\{i\in I\colon d_i\ge 1\}$. This will prove Proposition \ref{prop exist canonical section} in $g=1$ cases. Note that the $g=0$ cases are contained in Proposition \ref{prop compare with BCT in g=0}.

\subsection{Construction of $\gamma$-tame perturbations of TRR on $\oPMb^{1/r}_{1,B,I}$}\label{sec construction tame section}
 According to item \ref{item canonical tame g=1} in Definition \ref{dfn canonical section}, we need to construct a $\gamma\vert_{\hat{\mathcal B}\Gamma}$-tame perturbation of TRR for each genus-one $\Gamma\in V(\bm{G})$, $\bm{G}\in \sGPI_{1,B,I}^{r,0}$. In this subsection, we will construct such perturbations in the cases where $\bm{G}$ has only one vertex $\Gamma$ which is a smooth graded $r$-spin graph with internal and boundary markings $I$ and $B$. In other words, we will construct a multisection of $\mathcal W \oplus \bigoplus_{i\in I} \mathbb L_i^{\oplus d_i}\to \oPMb_{1,B,I}^{1/r}$ since each connected component of $\oPMb_{1,B,I}^{1/r}$ is just $\oPMb_{\Gamma}$ for one of such $\Gamma$.
 Note that in these cases we have $\hat{\mathcal B}\Gamma=\Gamma$ and $I^{d>0}\subseteq T^I(\hat{\mathcal B}\Gamma)$, thus we can write $\gamma\vert_{\hat{\mathcal B}\Gamma}$ as $\gamma$ for simplicity. 

A multisection of $\mathcal W \oplus \bigoplus_{i\in I} \mathbb L_i^{\oplus d_i}\to \oPMb_{1,B,I}^{1/r}$ is of the form
$$
\bm{s}^{\mathcal W}\oplus \bigoplus_{i\in I} \bigoplus_{j=1}^{d_i} s_i^j
$$
where $\bm{s}^{\mathcal W}$ is a multisection of $\mathcal W$ and each $s_i^j$ is a multisection of $\mathbb L_i$. 
We will first construct the $\mathcal W$-part $\bm{s}^{\mathcal W}$, then construct the $\mathbb L$-part $s_i^j$ one-by-one, in the order 
$$s_{\gamma(1)}^1, s_{\gamma(1)}^2, \dots, s_{\gamma(1)}^{d_{\gamma(1)}}, s_{\gamma(2)}^1, \dots, s_{\gamma(2)}^{d_{\gamma(2)}}, \dots, s_{\gamma(\lvert I^{d>0}\rvert)}^1, \dots, s_{\gamma(\lvert I^{d>0}\rvert)}^{d_{\gamma(\lvert I^{d>0}\rvert)}},$$
where we regard $\gamma$ as a bijection $\gamma \colon \{1,2,\dots, \lvert I^{d>0}\rvert\} \to I^{d>0}$.

\begin{nn}\label{notation collection of vertices}
    We denote by $G_{1,B,I}^*$ the set of graded $r$-spin graph $\Delta$ without  non-separating internal edge satisfying $\oPMb_\Delta\subseteq \oPMb^{1/r}_{1,B,I}$ and $\oPMb_\Delta\ne \emptyset$ (hence $\Delta$ has no positive boundary edges). 
    Notices that each genus-zero vertex of $\Delta \in G_{1,B,I}^*$  has a natural rooted or 2-rooted structure (see Remark \ref{rmk root when degenerate}), 
we denote by $\mathcal V_{1,B,I}^{root}$ the collection of vertices (as single-vertex graph) of all the graphs $\Delta\in G^*_{1,B,I}$ after forgetting a (possibly empty) subset of twist-zero illegal half-edges which are not roots.
We also denote by $\mathcal V_{1,B,I}$ the collection of vertices of all the graphs $\Delta\in G^*_{1,B,I}$ after forgetting the rooted or 2-rooted structure for open genus-zero vertices, and a arbitrary subset of twist-zero illegal half-edges.
\end{nn}

\subsubsection{Construction of the $\mathcal W$-part of $\gamma$-tame perturbations of TRR}\label{sec construction of tame W part}

We first construct the $\mathcal W$-part $\bm{s}^{\mathcal W}$ satisfying the following requirements.

\begin{enumerate}
    \item For each graph $\Delta\in G_{1,B,I}^*$, we require that $\bm{s}^{\mathcal W}_{\Delta}$, the restriction of $\bm{s}^{\mathcal W}$ to $\oPMb_\Delta$, is of the form
    \begin{equation}\label{eq docomposition of section of witten on central}
        \bm{s}^{\mathcal W}_{\Delta}= \bboxplus_{u\in V^O(\Delta)} \bm{s}_u^{\mathcal W}\Ass\bass_{u\in V^C(\Delta)}\bm{s}_u^{\mathcal W},
    \end{equation}
    where for $u\in V^C(\Delta)$, $\bm{s}_u^{\mathcal W}$ is a coherent multisection (see Section \ref{sec ass}) of $\mathcal W_u\to \overline{\mathcal R}_u$ such that $\overline{\bm{s}_u^{\mathcal W}}$ is transverse as a multisection of $\mathcal W_u\to \Mbar_u$; for $u\in V^O(\Delta)$, $\bm{s}_u^{\mathcal W}$ is a transverse multisection of $\mathcal W_u\to \oPMb_u$. Recall that, unless $u$ is closed and has an anchor of twist $-1$, we have $\overline{\mathcal R}_u=\Mbar_u$ and $\overline{\bm{s}_u^{\mathcal W}}=\bm{s}_u^{\mathcal W}$ is a transverse multisection of $\mathcal W_u\to \oPMb_u$.

    \item For $u\in V^O(\Delta)$ we denote by $\mathcal B u$ the (single-vertex) graph obtained by forgetting all the twist-zero illegal half-edges. The multisection is $\bm{s}_u^{\mathcal W}$ pulled-back from a multisection section $\bm{s}_{\mathcal B u}^{\mathcal W}$ of the bundle $\mathcal W_{\mathcal B u}\to \oPMb_{\mathcal B u}$ via the forgetful morphism $\oPMb_{u}\to \oPMb_{\mathcal B u}$.
    \item If $u\in V^O(\Delta)$ is genus-zero, we further require that $\bm{s}_{\mathcal B u}^{\mathcal W}$ is positive in the BCT sense.     
\end{enumerate}

 We will construct the multisection $\bm{s}^{\mathcal W}_{\mathcal B u}$ (hence $\bm{s}^{\mathcal W}_{ u}$) for all $u\in \mathcal V_{1,B,I}$ in a consistent way. We only do it when $\mathcal B u$ is stable since otherwise $\mathcal W_u$ is rank-zero.

We start from $\bm{s}^{\mathcal W}_{\mathcal B u}$ for (closed or open) genus-zero $u$. Notices such multisection are \textit{special canonical} in BCT sense, they are constructed inductively in \cite[\S 6.2]{BCT2}.

After constructing $\bm{s}^{\mathcal W}_{\mathcal B u}$ for genus-zero $u$, we now construct $\bm{s}^{\mathcal W}_{\mathcal B u}$ for genus-one open $u$. We induct on the dimension of $\oPMb_{\mathcal B u}$. Assuming we have constructed $\bm{s}^{\mathcal W}_{\mathcal B u}$ for all open genus-zero $u$ with $\dim \oPMb_{\mathcal B u}\le n-1$. When $\dim \oPMb_{\mathcal B u} = n$, for all $\Delta \in \partial^* \mathcal B u$, the restricted multisections $\bm{s}^{\mathcal W}_{\Delta}:=\bm{s}^{\mathcal W}_{\mathcal B u}\vert_{\oPMb_\Delta}$ are determined by \eqref{eq docomposition of section of witten on central}, and such restricted multisections are consistent in the sense they can be glued into a multisection on $\bigcup_{\Delta\in \partial^* \mathcal B u}\oPMb_\Delta \subset \oPMb_{\mathcal B u}$, thanks to the inductive procedure. We extend this glued multisection to a transverse multisection on entire $\oPMb_{\mathcal B u}$ using Lemma \ref{lem extend section from sub} to obtain $\bm{s}^{\mathcal W}_{\mathcal B u}$.

\begin{lem}\label{lem extend section from sub}
    Let $M$ be a smooth orbifold and $E\to M$ be a vector bundle. For smooth suborbifolds $S_1,S_2,\dots S_k\subset M$ which intersect transversally, assuming $s$ is a section of $E\to \bigcup_{i=1}^k S_i$ such that the restriction $s\vert_{\bigcap_{j\in J}S_j}$ is transverse for every subset $J\subseteq \{1,2,\dots,k\}$, then we can extend $s$ to a transverse section over the entire $M$. 
\end{lem}
\begin{proof}
    Any extension of $s$ to a sufficient small neighbourhood of $\bigcup_{i=1}^k S_i$ is transverse. We then extend it to a global section on $M$, and perturb away from $\bigcup_{i=1}^k S_j$ to make it transverse.
\end{proof}

\subsubsection{Construction of the $\mathbb L$-part of $\gamma$-tame perturbations of TRR} \label{sec construction of tame L part}
Now we construct the multisections $s_i^j$ one-by-one. To simplify the notation, as in Definition \ref{def tame perturbation}, we refer to the multisections 
$$s_{\gamma(1)}^1, s_{\gamma(1)}^2, \dots, s_{\gamma(1)}^{d_{\gamma(1)}}, s_{\gamma(2)}^1, \dots, s_{\gamma(2)}^{d_{\gamma(2)}}, \dots, s_{\gamma(\lvert I^{d>0}\rvert)}^1, \dots, s_{\gamma(\lvert I^{d>0}\rvert)}^{d_{\gamma(\lvert I^{d>0}\rvert)}}$$
as 
$$t_1,t_2,\dots,t_{D},$$
where $D=\sum_{i\in I} d_i$, and we refer to $\bm{s}^{\mathcal W}$ as $t_0$. We also denote by 
$T_s:=\bigoplus_{j=0}^s t_j.$
We define a map $\nu \colon \{1,2,\dots, D\} \to  I^{d>0}$ in the way that $t_j$ is a multisection of $\mathbb L_{\nu(j)}$.

We construct multisections $t_j$ satisfying the following requirements, which are stronger than the requirements in Definition \ref{def tame perturbation}. 
\begin{enumerate}
    \item For each graph $\Delta \in G_{1,B,I}^*$ and $1\le j \le D$, we require that $t_{j,\Delta}$, the restriction of $t_j$ to $\oPMb_\Delta$, is of the form
    \begin{equation}\label{eq recursive of section of Li}
        t_{j,\Delta}=\pi^*_{\Delta \to u^{\Delta}_{\nu(j)}} t_{j,u^{\Delta}_{\nu(j)}},
    \end{equation}
    where $u^{\Delta}_{\nu(j)}$ is the unique vertex of $\Delta$ containing the internal marking $\nu(j)\in I$, 
    and $t_{j,u^{\Delta}_{\nu(j)}}$ is a multisection of $\mathbb L_{\nu(j)}\to \oPMb_{u^{\Delta}_{\nu(j)}}$. Note that if we formally set $t_{j,u}=0$ for vertices $u\in V(\Delta)$ not containing $\nu(j)$, as a section of a rank-zero bundle, then we can write 
     \begin{equation}
        t_{j,\Delta}=\bboxplus_{u\in V(\Delta)}t_{j,u}.
    \end{equation}

    \item 
    If $u^{\Delta}_{\nu(j)}$ is open, we further require that 
    \begin{equation}\label{eq section of Li pull back from modified base}
        t_{j,\nu(j)}=\operatorname{For}^*_{u^{\Delta}_{\nu(j)} \to \hat{\mathcal B}u^{\Delta}_{\nu(j)}} t_{j,\hat{\mathcal B}u^{\Delta}_{\nu(j)}}
    \end{equation}
     as multisections of $\mathbb L_{j,u^{\Delta}_{\nu(j)}} \cong \operatorname{For}^*_{u^{\Delta}_{\nu(j)} \to \hat{\mathcal B}u^{\Delta}_{\nu(j)}} \mathbb L_{j,\hat{\mathcal B}u^{\Delta}_{\nu(j)}}$ (see Remark \ref{rmk forget boundry point}),
    where $\hat{\mathcal B}u^{\Delta}_{\nu(j)}$ is the single-vertex graph obtained by forgetting all the twist-zero illegal boundary half-edges of $u^{\Delta}_{\nu(j)}$ except the root(s), and $\operatorname{For}_{u^{\Delta}_{\nu(j)} \to \hat{\mathcal B}u^{\Delta}_{\nu(j)}}\colon \oPMb_{u^{\Delta}_{\nu(j)}} \to \oPMb_{\hat{\mathcal B}u^{\Delta}_{\nu(j)}}$ is the corresponding forgetful morphism. 
    Notice that $\hat{\mathcal B}u^{\Delta}_{\nu(j)}$ is always stable because it has an internal marking $\nu(j)$, and if it is genus-zero it has at least a root.

    \item \label{item transverse of L together with W in construction}
    For all $\Delta \in G_{1,B,I}^*$, all $u\in V(\Delta)$ and $0\le j \le D$, the multisection 
        $${T_{j,u}}:=\overline{\bm{s}^{\mathcal W}_u} \oplus t_{1,u}\oplus t_{2,u}\oplus\dots \oplus t_{j,u}$$ 
        of 
        $$E_{j,u}:=\mathcal W_{u}\oplus \bigoplus_{i\in T^I(u)\cap \nu(\{1,2,\dots,j\})} \mathbb L_i\to \oPMb_u$$ is transverse to zero.
    
    \item \label{item construction tame K}
    For all open vertex $u\in V(\Delta)$ as above, if $\nu(j)\in T^I(u)$, there exists an open neighbourhood $K_{j,u}$  of the zero locus (see Observation \ref{obs vanishing loci of single trr}) $\left( \bigcup_{\Lambda\in \mathcal G_{\nu(j)}(u)}\Mbar_\Lambda\right)\cap \oPMb_u$ of $\mathfrak s_{\nu(j),u}$ in $\oPMb_u$, such that for any $p\in\oPMb_u\setminus K_{j,u}$, we have $t_{j,u}(p)=\mathfrak s_{\nu(j),u}(p)$, where $\mathfrak s_{\nu(j),u}$ is the TRR multisection of $\mathbb L_{\nu(j)}\to \oPMb_u$ (with respect to the rooted or 2-rooted structure of $u$ if $u$ is genus-zero) defined in \S \ref{sec def of trr sections}.

    \item \label{item lower bound rk>dim}
    If $\rk  E_{s,u}> \dim \oPMb_u$ for an open  vertex $u\in V(\Delta)$ and $\nu(s)\in T^I(u)$, then $T_{s-1,u}$ vanishes nowhere on $K_{s,u}\subseteq\oPMb_u$.

     \item \label{item lower bound rk>=dim}
    If $\rk  E_{s,u}= \dim \oPMb_u$ for an open genus-one vertex $u\in V(\Delta)$ and $\nu(s)\in T^I(u)$, then $T_{s-1,u}$ vanishes nowhere on $K_{s,u}\subseteq\oPMb_u$; as a consequence, $T_{s,u}=T_{j-1,u}\oplus t_{j,u}$ vanishes nowhere on $\oPMb_u$ in this case.
 \end{enumerate}
 \begin{rmk}\label{rmk on non vanishment of section}
     For any vertex $u\in \mathcal V^{root}_{1,B,I}$, if $\rk  E_{s,u}> \dim \oPMb_u$, then then $T_{s,u}$ also vanishes nowhere on $\oPMb_u$.
     This is a corollary of the transversality in item \ref{item transverse of L together with W in construction}.

Notices that since we set $t_0=T_0= \bm{s}^{\mathcal W}$, we have $T_{0,u}=\bm{s}^{\mathcal W}_u$ for genus-one open $u$, they already (trivially) satisfy 
item \ref{item lower bound rk>=dim} because for all open genus-one vertices $u$ we have $\rk \mathcal W_u< \dim \oPMb_u$ according to \eqref{eq rank of witten bundle g=1}.
 \end{rmk}

Assuming we have already constructed $t_0=\bm{s}^{\mathcal W}, t_1, \dots, t_{j-1}$ satisfying the requirements above, we now construct $t_j$. 

\noindent\textbf{Step 1.} We start from constructing $t_{j,u}$ for closed vertices $u\in \mathcal V^{root}_{1,B,I}$. If the marking $\nu(j)\in I$ is not attached to $u$, we formally take $t_{j,u}=0$ as a section the rank-zero bundle. If  $\nu(j)\in I\cap T^I(u)$, we construct $t_{j,u}$ by an induction on $\dim \Mbar_u$, such that:
\begin{itemize}
    \item For any $\Delta \in \partial u$, \eqref{eq recursive of section of Li} is satisfied.
    \item The multisection $\overline{T_{j,u}}:=\overline{\bm{s}^{\mathcal W}_u} \oplus t_{1,u}\oplus t_{2,u}\oplus\dots \oplus t_{j,u}$ is transverse to zero.
    \end{itemize}

For dimension-zero $\Mbar_u$, we take $t_{j,u}$ to be any non-zero multisection. 

Assuming we have constructed $t_{j,v}$ for all closed $v\in \mathcal V^{root}_{1,B,I}$ with $\dim \Mbar_v \le n-1$, we now construct a multisection $t_{j,u}$ for $u\in \mathcal V^{root}_{1,B,I}$ with $\dim \Mbar_u = n$. For any $\Delta \in \partial u$, the restriction $t_{j,\Delta}$ of $t_{j,u}$ on $\Mbar_{\Delta}$ is already determined by \eqref{eq recursive of section of Li} in a consistent way: for different $\Delta_1$ and $\Delta_2$, we have $t_{j,\Delta_1}\vert_{\Mbar_{\Delta_1}\cap \Mbar_{\Delta_2}}=t_{j,\Delta_2}\vert_{\Mbar_{\Delta_1}\cap \Mbar_{\Delta_2}}$ by construction. Therefore, the multisection $t_{j,u}$ is already defined on $\bigcup_{\Delta \in \partial u}\Mbar_\Delta \subset \Mbar_u$, then we  extend it to a transverse multisection $t_{j,u}$ on the entire $\Mbar_u$ using Lemma \ref{lem extend section from sub}.

\noindent\textbf{Step 2.} 
Now we construct $t_{j,u}$ and $K_{j,u}$ for open vertices $u\in \mathcal V^{root}_{1,B,I}$. Again, if the marking $\nu(j)\in I$ is not attached to $u$, we formally take $t_{j,u}=0$ as a section the rank-zero bundle, this includes all the cases where $\hat{\mathcal B}u$ is unstable. If  $\nu(j)\in I\cap T^I(u)$  (hence $\hat{\mathcal B}u$ is stable), we construct $t_{j,u}$ and $K_{j,u}$ such that:
\begin{itemize}
    \item For any $\Delta \in \partial^* u$, \eqref{eq recursive of section of Li} and \eqref{eq section of Li pull back from modified base} are satisfied.
    \item The multisection $\overline{T_{j,u}}:=\overline{\bm{s}^{\mathcal W}_u} \oplus t_{1,u}\oplus t_{2,u}\oplus\dots \oplus t_{j,u}$ is transverse to zero.
    \item The open neighbourhood $K_{j,u}$  of the zero locus $\left( \bigcup_{\Lambda\in \mathcal G_{\nu(j)}(u)}\Mbar_\Lambda\right)\cap \oPMb_u$ of $\mathfrak{s}_{\nu(j),u}$ in $\oPMb_u$ satisfies the following conditions.
    \begin{enumerate}[label=(\alph*)]
        \item \label{item psi non vanish away from zero locus} 
        For any $p\in\oPMb_u\setminus K_{j,u}$, we have $t_{j,u}(p)=\mathfrak{s}_{\nu(j),u}(p)$.

        \item $K_{j,u}=\operatorname{For}^{-1}_{u\to \hat{\mathcal B}u}(K_{j,\hat{\mathcal B}u})$.
        
        \item\label{item recursive perturbed area} 
        For any $\Delta \in \partial^* u$, if $\pi_{\Delta \to u^{\Delta}_{\nu(j)}}(p)\in K_{j,u^{\Delta}_{\nu(j)}}$\footnote{More precisely, by $\pi_{\Delta \to u^{\Delta}_{\nu(j)}}(p)\in K_{j,u^{\Delta}_{\nu(j)}}$ we mean $p$ is in the image of $K_{j,u^{\Delta}_{\nu(j)}}\times \prod_{v\in V(\Delta)\setminus \{u^{\Delta}_{\nu(j)}\}}\oPMb_v$ under the gluing morphism.}   for some $p\in \oPMb_{\Delta}$, then $p\in \oPMb_{\Delta}\cap K_{j,u} \subseteq \oPMb_u$, where $u^{\Delta}_{\nu(j)}$ is the vertex of $\Delta$ containing the marking $\nu(j)$, and we take $K_{j,u^{\Delta}_{\nu(j)}}=\Mbar_{u^{\Delta}_{\nu(j)}}$ if $u^{\Delta}_{\nu(j)}$ is a closed vertex.
        \item\label{item remaining part non vanish rk>dim} 
        If $\rk E_{j,u}> \dim \oPMb_u$, for any $p\in K_{j,u}$ we have $ T_{j-1,u}(p) \ne 0$.
        \item\label{item remaining part non vanish rk=dim} For open genus-one $u$, if $\rk E_{j,u}= \dim \oPMb_u$, for any $p\in K_{j,u}$ we also have $ T_{j-1,u}(p) \ne 0$. As a consequence, $T_{s,u}=T_{j-1,u}\oplus t_{j,u}$ vanishes nowhere on $\oPMb_u$.
    \end{enumerate}
    \end{itemize}

    We construct $t_{j,{\hat{\mathcal B}u}}$ and $K_{j,{\hat{\mathcal B}u}}$ (hence also $t_{j,u}$ and $K_{j,u}$) by an induction on $\dim \oPMb_{\hat{\mathcal B}u}$. For dimension-zero $\oPMb_{\hat{\mathcal B}u}$ (this only happens when $u$ is genus-zero), 
    we simply take $K_{j,u}=\emptyset$ and therefore $t_{j,\hat{\mathcal B}u}=\mathfrak{s}_{\nu(j),\hat{\mathcal B}u}$, $t_{j,u}=\mathfrak{s}_{\nu(j),u}$.

    Assuming we have constructed $t_{j,{\hat{\mathcal B}v}}$ and $K_{j,{\hat{\mathcal B}v}}$ for all open (and closed, constructed in Step 1) vertices $v\in \mathcal V^{root}_{1,B,I}$ with $\dim \Mbar_{\hat{\mathcal B}v} \le n-1$, we now construct  $t_{j,{\hat{\mathcal B}u}}$ and $K_{j,{\hat{\mathcal B}u}}$ for open $u\in \mathcal V^{root}_{1,B,I}$  with $\dim \Mbar_{\hat{\mathcal B}u} = n$.

    We first construct the set $K_{j,\hat{\mathcal B}u}$ (and hence $K_{j,u}$) satisfying items \ref{item recursive perturbed area}, \ref{item remaining part non vanish rk>dim}, and, in the case $\rk E_{j,{\hat{\mathcal B}u}}=\dim \oPMb_{\hat{\mathcal B}u}$ for genus-one $u$, item \ref{item remaining part non vanish rk=dim}.  $K_{j,\hat{\mathcal B}u}$ should contain the zero locus of $\mathfrak {s}_{\nu(j),\hat{\mathcal B}u}$ in $\oPMb_{\hat{\mathcal B}u}$, as well as some extra parts determined by item \ref{item recursive perturbed area}. We define the unions of such subspaces by
    $$\tilde{K}_{j,\hat{\mathcal B}u}:=\left( \bigcup_{\Lambda\in \mathcal G_{\nu(j)}({\hat{\mathcal B}u})}\Mbar_\Lambda\cap \oPMb_{\hat{\mathcal B}u}\right)\cup \bigcup_{\Delta \in \partial^* {\hat{\mathcal B}u}}\{p\in \oPMb_\Delta\subset \oPMb_{\hat{\mathcal B}u} \colon \pi_{\Delta \to u^{\Delta}_{\nu(j)}}(p)\in K_{j,u^{\Delta}_{\nu(j)}}\}.$$
    In the case ${\hat{\mathcal B}u}$ is genus-zero and $\rk E_{j,{\hat{\mathcal B}u}}\le \dim \oPMb_{\hat{\mathcal B}u}$, we can take $K_{j,u}$ to be any open neighbourhood of $\tilde{K}_{\hat{\mathcal B}u}$ in $\oPMb_{\hat{\mathcal B}u}$. In the case $\rk E_{j,{\hat{\mathcal B}u}}> \dim \oPMb_{\hat{\mathcal B}u}$ (or $\rk E_{j,{\hat{\mathcal B}u}}= \dim \oPMb_{\hat{\mathcal B}u}$ for genus-one open ${\hat{\mathcal B}u}$), we need to take $K_{j,\hat{\mathcal B}u}$ satisfying item \ref{item remaining part non vanish rk>dim} or \ref{item remaining part non vanish rk=dim}. This can be done due to the following lemma. 
    \begin{lem}
        In the case $\rk E_{j,u}> \dim \oPMb_u$ for any open $u$, or $\rk E_{j,u}= \dim \oPMb_u$ for genus-one open $u$, for any $p\in \tilde{K}_{j,u}$, we have 
        $T_{j-1,u}(p) \ne 0.$
    \end{lem}
    \begin{proof}
        We focus on the case $\rk E_{j,u}= \dim \oPMb_u$ with genus-one open $u$. The argument for the case $\rk E_{j,u}> \dim \oPMb_u$ for all open $u$ is identical.

        If $p\in \Mbar_\Lambda\cap \oPMb_u$ for a graph $\Lambda\in \mathcal G_{\nu(j)}(u)$, then $\Lambda$ consists of an open genus-one vertex $v_{\Lambda,o}$ and a closed genus-zero vertex $v_{\Lambda,c}$. We have $$\dim \oPMb_{v_{\Lambda,o}}+\dim \Mbar_{v_{\Lambda,c}}=\dim \oPMb_u -2,$$ and 
        $$
            \rk E_{j-1,v_{\Lambda,o}}+ \rk E_{j-1,v_{\Lambda,c}}= \rk E_{j-1,u}= \rk E_{j,u}-2.
        $$
        Then $\rk E_{j,u}= \dim \oPMb_u$ implies we have either $$\rk E_{j-1,v_{\Lambda,o}}\ge \dim \oPMb_{v_{\Lambda,o}}$$ which means $T_{j-1,v_{\Lambda,o}}\left(\pi_{\Lambda\to v_{\Lambda,o}}(p)\right)\ne 0$ by inductive hypothesis (or transversality), or $$\rk E_{j-1,v_{\Lambda,c}}> \dim \oPMb_{v_{\Lambda,c}}$$
        which means $T_{j-1,v_{\Lambda,c}}\left(\pi_{\Lambda\to v_{\Lambda,c}}(p)\right)\ne 0$ by transversality. In any case we have $T_{j-1,u}\left(p\right)\ne 0$ by \eqref{eq docomposition of section of witten on central}, \eqref{eq recursive of section of Li} and Lemma \ref{lem zero of assembling}.

        If $p\in \oPMb_\Delta$ for some $\Delta \in \partial^* u$ and $\pi_{\Delta \to u^{\Delta}_{\nu(j)}}(p)\in K_{j,u^{\Delta}_{\nu(j)}}$, we have 
        $$
        \sum_{v\in \Delta} \dim \Mbar_v=\dim \oPMb_{u}-\lvert E^B(\Delta)\rvert-2 \lvert E^I(\Delta)\rvert<\dim \oPMb_{u},
        $$
        and
        $$
        \rk E_{j,u^{\Delta}_{\nu(j)}}+\sum_{v\in V(\Delta)\setminus\{u^{\Delta}_{\nu(j)}\}}\rk E_{j-1,v}= \rk E_{j,u}.
        $$
        Note that if $u^{\Delta}_{\nu(j)}$ is closed, then we are in the situation where $\Delta\in \mathcal G_{\nu(j)}(u)$, which is covered by the previous case. Therefore we can assume $u^{\Delta}_{\nu(j)}$ is open,         then $\rk E_{j,u}= \dim \oPMb_u$ implies we have either $$\rk E_{j,u^{\Delta}_{\nu(j)}}> \dim \oPMb_{u^{\Delta}_{\nu(j)}}$$ which means $T_{j-1,u^{\Delta}_{\nu(j)}}\left(\pi_{\Delta\to u^{\Delta}_{\nu(j)}}(p)\right)\ne 0$ by inductive hypothesis, or there exists $v\in V(\Delta)\setminus\{u^{\Delta}_{\nu(j)}\}$ such that
         $$\rk E_{j-1,v}> \dim \Mbar_{v}$$
        which means $T_{j-1,v}\left(\pi_{\Delta\to v}(p)\right)\ne 0$ by transversality. In any case we have $T_{j-1,u}\left(p\right)\ne 0$ by \eqref{eq docomposition of section of witten on central}, \eqref{eq recursive of section of Li} and Lemma \ref{lem zero of assembling}.
        \end{proof}
    Therefore, we can take $K_{j,\hat{\mathcal B}u}$ to be a small open neighbourhood of $\tilde{K}_{j,\hat{\mathcal B}u}$ where $T_{j-1,{\hat{\mathcal B}u}}(p)$ does not vanish.

    Now we construct the multisection $t_{j,{\hat{\mathcal B}u}}$ (and hence $t_{j,u}$). 
    On each $\oPMb_\Delta\subset \oPMb_{\hat{\mathcal B}u}$ for $\Delta\in \partial^*{\hat{\mathcal B}u}$, we take $t_{j,{\hat{\mathcal B}u}}$ according to \eqref{eq recursive of section of Li}. 
    On $\oPMb_{\hat{\mathcal B}u}\setminus K_{j,\hat{\mathcal B}u}$, we take  $t_{j,{\hat{\mathcal B}u}}$ to be $\mathfrak{s}_{\nu(j),\hat{\mathcal B}u}$. Such choices are consistent by the inductive construction of $t_{j,{\hat{\mathcal B}v}}$ and $K_{\hat{\mathcal B}v}$. 
    We need to extend such multisection to the entire $\oPMb_{\hat{\mathcal B}u}$ which is transverse to the zero locus of $T_{j-1,{\hat{\mathcal B}u}}$. 
    Notice that, by inductive construction, $t_{j,{\hat{\mathcal B}u}}$ is transverse to the zero locus of $T_{j-1,{\hat{\mathcal B}u}}$ when restricted to $\oPMb_\Delta\subset \oPMb_{\hat{\mathcal B}u}$, as well as when restricted to $\oPMb_{\hat{\mathcal B}u}\setminus K_{j,\hat{\mathcal B}u}$ (since $\mathfrak{s}_{\nu(j),\hat{\mathcal B}u}$ does not vanish outside $K_{j,\hat{\mathcal B}u}$). 
    Therefore, for any extension $\tilde{t}_{j,{\hat{\mathcal B}u}}$ of $t_{j,{\hat{\mathcal B}u}}$, there exists an open set $U$ containing  $\left(\oPMb_{\hat{\mathcal B}u}\setminus K_{j,\hat{\mathcal B}u}\right)\cup \bigcup_{\Delta\in  \partial^*{\hat{\mathcal B}u}}\oPMb_\Delta\subset \oPMb_{\hat{\mathcal B}u}$, such that $\tilde{t}_{j,{\hat{\mathcal B}u}}$ is transverse to the zero locus of $T_{j-1,{\hat{\mathcal B}u}}$ when restricted to $U$. 
    Then we can perturb $\tilde{t}_{j,{\hat{\mathcal B}u}}$ by a section supported in $\oPMb_{\hat{\mathcal B}u}\setminus U$ to make it transverse to the zero locus of $T_{j-1,{\hat{\mathcal B}u}}$ on the entire $\oPMb_{\hat{\mathcal B}u}$.

\subsection{Construction of a $\gamma$-canonical multisection on $\oPMb^{1/r,0}_{1,B,I}$} \label{sec extend canonical to global glued}

In this subsection, we construct a $\gamma$-canonical multisection of $E=\mathcal W \oplus \bigoplus_{i\in I} \mathbb L_i^{\oplus d_i}\to \oPMb^{1/r,0}_{1,B,I}$ using the $\gamma$-tame perturbations of TRR constructed in \S \ref{sec construction tame section}. 

Under the notation in Definition \ref{dfn canonical section}, in order to guarantee item \ref{item canonical forget g=1} and \ref{item canonical transversely decompose} in Definition \ref{dfn canonical section}, we need to construct a multisection $\tilde{\bm{s}}_{\hat{\mathcal B}\Gamma}$ of 
$$
E_{\hat{\mathcal B}\Gamma}=\mathcal W_{\hat{\mathcal B}\Gamma}\oplus \bigoplus_{i\in I\cap T^I(\hat{\mathcal B}\Gamma)} \operatorname{For}^*_{\hat{\mathcal B}\Gamma\to \mathcal B \Gamma}\mathbb L_{i,\mathcal B \Gamma}^{\oplus d_i}
$$
for each $\Gamma \in V(\bm{G})$, $\bm{G}\in \sGPI^{r,0}_{1,B,I}$ in a transversely-decomposable way.

\begin{rmk}\label{rmk consistent notation of base}
Note that every such (single-vertex) $r$-spin graph $\Gamma$ can be obtained from a (not necessary unique) graph $\mathcal P\Gamma\in \mathcal V_{1,B,I}^{root}$ (see Notation \ref{notation collection of vertices}) after turning every twist-zero illegal boundary tail of $\mathcal P\Gamma$ into an inserted twist-zero internal tail (\textit{i.e.,} point insertion procedure). The graph $\Gamma$ and $\mathcal P\Gamma$ coincide after forgetting these twist-zero  tails. 

Recall that (see \S \ref{sec uni cotangent line over glued moduli}) $\mathcal B \Gamma$ is the graph forgetting all inserted internal tails in $I^{PI}(\bm G)\cap T^I(\Gamma)$, and $\hat{\mathcal B} \Gamma$ is the graph forgetting all inserted internal tails in $I^{PI}_{sp,1}(\bm G)\cap T^I(\Gamma)$. Such notations are consistent with the notations in \S \ref{sec construction of tame W part} and \S \ref{sec construction of tame L part} (which forget the twist-zero illegal boundary tails) in the sense that all inserted internal tails in $I^{PI}(\bm G)$ come from twist-zero illegal tails, while $I^{PI}_{sp,1}(\bm G)\subseteq I^{PI}(\bm G)$ correspond to twist-zero illegal boundary tails which are not roots.

Moreover, for any $\Lambda\in \partial^* \Gamma$, and any $u\in V(\Lambda)$, we denote by $\hat{\mathcal B}\Lambda$ or $\hat{\mathcal B} u$ the graph obtained by forgetting all the internal tails in $I^{PI}_{sp,1}$ and all the twist-zero illegal boundary tails which are not roots (\textit{i.e.,} lies in $I^{PI}_{sp,1}$ after point insertion) from $\Lambda$ or $u$. Note that $\hat{\mathcal B} u$ is always stable if $T^I(u)\cap I\ne \emptyset$.
We also denote by $\mathcal B \Lambda$ or $\mathcal B u$ the graph obtained by forgetting all the internal tails in $I^{PI}$ and all the twist-zero illegal boundary tails from $\Lambda$ or $u$.

Note that we can always regard $\mathcal B u$ as element in $\mathcal V_{1,B,I}$. If $u$ has no internal roots (when $\Gamma$ is genus-one, we regard half-edges of $u$ as roots if they are either in $RT(\Gamma)$, or connect $u$ to a component containing tails in $RT(\Gamma)$), we can also regard  $\hat{\mathcal B} u$ as element in $\mathcal V_{1,B,I}^{root}$.
\end{rmk}

 We decompose $\tilde{\bm s}_{\hat{\mathcal B} \Gamma}$ into $\mathcal W$-part and $\mathbb L$-part as 
    $\tilde{\bm s}_{\hat{\mathcal B} \Gamma}=\tilde{\bm s}_{\hat{\mathcal B} \Gamma}^{\mathcal W}\oplus \tilde{\bm s}_{\hat{\mathcal B} \Gamma}^{\mathbb L}$. In the case $\mathcal B u$ is unstable, we don't need to construct $\tilde{\bm s}_{\hat{\mathcal B} \Gamma}^{\mathcal W}$ as $\mathcal W_{\hat{\mathcal B}u}$ is rank-zero. If  $\mathcal B u$ is stable, then we simply take the $\mathcal W$-part $\tilde{\bm s}_{\hat{\mathcal B} u}^{\mathcal W}$ to be
    $$
    \hat{\bm s}_{\hat{\mathcal B} u}^{\mathcal W}:= \operatorname{For}^*_{\hat{\mathcal B} u \to {\mathcal B} u} \bm{s}_{{\mathcal B} u}^{\mathcal W}
    $$
    as a multisection of the Witten bundle $\operatorname{For}^*_{\hat{\mathcal B} u \to {\mathcal B} u} \mathcal W_{{\mathcal B} u}\cong \mathcal W_{\hat{\mathcal B} u}\to \oPMb_{\hat{\mathcal B} u}$, where $\bm{s}_{{\mathcal B} u}^{\mathcal W}$ is the multisection of $\mathcal W_{\mathcal B u}\to \oPMb_{{\mathcal B} u}$ constructed in \S \ref{sec construction of tame W part} after regarding $\mathcal Bu$ as element in $\mathcal V_{1,B,I}$ as in Remark \ref{rmk consistent notation of base}. Such choices will guarantee item \ref{item canonical positive} in Definition \ref{dfn canonical section}.

    Then we need to construct the $\mathbb L$-part $\tilde{\bm s}_{\hat{\mathcal B} \Gamma}^{\mathbb L}$ for each $\Gamma$ such that $\hat{\mathcal B} \Gamma$ is stable (otherwise the bundle is rank-zero). In fact, we will construct a collection of multisection $\{\tilde{\bm{s}}^{\mathbb L}_{j,v}\}$ of 
    $$
     \operatorname{For}^*_{v\to \mathcal B v}\mathbb L_{\nu(j),\mathcal B v} \to \oPMb_{v}
    $$
    for each $v\in V(\mathcal T_{RT(\Gamma)\cap T^I(\Gamma)} \Delta)$, $ \Delta\in \partial^* \hat{\mathcal B} \Gamma$, $\Gamma \in V(\bm{G})$, $\bm{G}\in \sGPI^{r,0}_{1,B,I}$, $1\le j \le D=\sum_{i\in I} d_i$, $\nu(j)\in T^I(v)$, where $\mathcal T_{RT(\Delta\cap T^I(\Delta))} \Delta$ is the graph defined in \S \ref{sec def of tame and canonical} when taking $S=RT(\Gamma)\cap T^I(\Gamma)$,  $\mathcal B v$ is the graph forgetting all the tails in $T^I(v)\cap I^{PI}(\bm{G})$ and illegal twist-zero half-edges, and $\nu\colon \{1,2,\dots,D\}\to I^{d>0}$ is the same map in \S \ref{sec construction of tame L part}. By construction in \S \ref{sec def of tame and canonical} such $v$ is always stable. 

     Notice that each genus-zero vertex $u\in V(\hat{\mathcal B}\Gamma)$ has one or two half-edges as root(s): they are determined by the topological structure of the cylinder as in Remark \ref{rmk root when degenerate} when $\Gamma$ is genus-one, and determined by $RT(\Gamma)$ in a similar way when $\Gamma$ is genus-one. Moreover, there are also rooted or 2-rooted structures for $v\in V(\mathcal T_{S} \Delta)$: whenever a vertex with root(s) on it is contracted (see \S \ref{sec def of tame and canonical} for the construction of $\mathcal T_{S} \Delta$ via contracting unstable vertices), the nearby half-edges become new roots (after possible detaching).
    
    The collection $\{\tilde{\bm{s}}^{\mathbb L}_{j,v}\}$ will satisfy the following conditions.

    \begin{enumerate}
    \item For any $1\le s \le D$, the multisection 
    $$\bigoplus_{\substack{1\le j \le s\\ \nu(j)\in T^I(v)}}\tilde{\bm{s}}^{\mathbb L}_{j,v}$$ 
    is transverse to zero when restricted to the zero loci of $\bm{s}^{\mathcal W}_{\xi^{-1}(v)}$ in $\oPMb_{\xi^{-1}(v)}\cong \oPMb_v$ (see Remark \ref{rmk on contract after forget}) when $v$ is open, and is transverse to zero when restricted to the zero loci of $\overline{\bm{s}^{\mathcal W}_{\xi^{-1}(v)}}$ in $\Mbar_{ v}$ when $v$ is closed. This will guarantee the transversality requirement in item \ref{item canonical decompose} and \ref{item canonical transversely decompose} in Definition \ref{dfn canonical section}.
    
    \item If no internal tails of $v$ (except the anchor when $v$ is closed) are  internal roots,  we require that
    \begin{equation}\label{eq canonical L part in case when no internal root}
    \tilde{\bm s}^{\mathbb L}_{j,v}=  t_{j, \xi^{-1}(v)},
    \end{equation}
    for any $j$ such that $\nu(j)\in T^I(v)$ ,where $t_{j, \xi^{-1}(v)}$ is the multisections of $\mathbb L_{\nu(j),\xi^{-1}(v)}\to \oPMb_{\xi^{-1}(v)}\cong \oPMb_v$ constructed in \S \ref{sec construction of tame L part} after regarding $\xi^{-1}(v)$ as an element in $\mathcal V_{1,B,I}^{root}$ as in Remark \ref{rmk consistent notation of base}.
    Note that in this case $ \operatorname{For}^*_{v\to \mathcal B v}\mathbb L_{\nu(j),\mathcal B v}$ coincide with $ \mathbb L_{\nu(j), \xi^{-1}(v)}=\mathbb L_{\nu(j), v}$.
    This requirement will guarantee item \ref{item canonical tame g=1} in Definition \ref{dfn canonical section}.

      \item For any $\Lambda\in \partial^* v$, we denote by $S'=T^I(v)\cap RT(\Gamma)$ the set of original internal roots of $\Gamma$ on $v$ (not including the new roots from contracting), we require that  the restriction of the  
        $\tilde{\bm s}_{j,v}^{\mathbb L}$ to $\oPMb_{\Lambda}\subseteq \oPMb_{ v}$ satisfies
    \begin{equation}\label{eq decomp L part of canonical with roots}
    \tilde{\bm s}_{j,\Lambda}^{\mathbb L}:=\tilde{\bm s}_{j,v}^{\mathbb L}\vert_{\oPMb_{ \Lambda}}= \operatorname{For}^*_{\Lambda \to \mathcal T_{S'}\Lambda} 
     \pi^*_{\mathcal T_{S'}\Lambda \to w}   \tilde{\bm s}_{j,w}^{\mathbb L}
    \end{equation}
    as a multisection of 
    $$
     \operatorname{For}^*_{v\to \operatorname{for}_{S'}(v)}\mathbb L_{\nu(j),\operatorname{for}_{S'}(v)}\big\vert_{\oPMb_{\Lambda}} 
     \cong  
     \operatorname{For}^*_{\Lambda \to \mathcal T_{S'}\Lambda} 
     \pi^*_{\mathcal T_{S'}\Lambda \to w} \operatorname{For}^*_{w\to \mathcal B w}\mathbb L_{\nu(j),\mathcal B w} \to \oPMb_{ \Lambda},
    $$
    where $w\in V(\mathcal T_{S'}\Lambda)$ is the unique vertex of $\mathcal T_{S'}\Lambda$ such that $\nu(j)\in T^I(w)$.
    This will guarantee item \ref{item canonical glue} and \ref{item canonical transversely decompose} in Definition \ref{dfn canonical section}.
    \end{enumerate}

    Similar to the previous subsections, we construct $ \tilde{\bm s}^{\mathbb L}_{j,v}$ by an induction on $j$ and $\dim \oPMb_v$.

    Assuming we have constructed $ \tilde{\bm s}^{\mathbb L}_{s,w}$ for all $s<j$ and any $w$,  as well as  $s=j$ and all $w$  with $\dim \oPMb_w \le n$, satisfying \eqref{eq canonical L part in case when no internal root} and \eqref{eq decomp L part of canonical with roots}. Now we construct $ \tilde{\bm s}^{\mathbb L}_{j,v}$ for a $v$ with $\dim \oPMb_v = n+1$. Note that \eqref{eq decomp L part of canonical with roots} have already determined the restriction of  $ \tilde{\bm s}^{\mathbb L}_{j,v}$ to 
    $$\bigcup_{\Lambda\in \partial^* v}\oPMb_{\Lambda}\subset \oPMb_{v},$$ 
    we need to extend such multisection to the entire $\oPMb_{v}$. 

In the case where no internal tails of $v$ (except the anchor when $v$ is closed) are  internal roots, we are forced to take $ \tilde{\bm s}^{\mathbb L}_{j,v}$ according to \eqref{eq canonical L part in case when no internal root}.  Note that in this case,  for all $w\in V(\mathcal T_{S'}\Lambda)$ appearing in the right-hand side of \eqref{eq decomp L part of canonical with roots}, $w$ also has no internal roots (except the anchor when $w$ is closed), which means $ \tilde{\bm s}^{\mathbb L}_{j,w}$ is also determined by \eqref{eq canonical L part in case when no internal root}. Therefore   \eqref{eq canonical L part in case when no internal root} is compatible with \eqref{eq decomp L part of canonical with roots} in this case according to \eqref{eq recursive of section of Li} and \eqref{eq section of Li pull back from modified base} as $S'=\emptyset$. The transversality of $ \tilde{\bm s}^{\mathbb L}_{j,v}$ on the zero loci in this case is equivalent to the transversality of the multisections  $T_{j,v}$ or $\overline{T_{j,v}}$ in \S \ref{sec construction of tame L part}.

In the case $v$ has at least one internal root which is not the anchor, we denote by $Z\subseteq \oPMb_{v}$ the zero loci of $\tilde{\bm s}^{\mathcal W}_{v}$ or $\overline{\tilde{\bm s}^{\mathcal W}_{v}}$ direct summed with $\bigoplus_{\substack{1\le s \le j-1\\ \nu(s)\in T^I(v)}}\tilde{\bm{s}}^{\mathbb L}_{s,v}$,
we first extend the restriction of $ \tilde{\bm s}^{\mathbb L}_{j,v}$ to $Z \cap \bigcup_{\Lambda\in \partial^* v}\oPMb_{\Lambda}\subset \oPMb_{v}$ (which is transverse by inductive hypothesis and the fact that $\operatorname{For}_{\Lambda \to \mathcal T_{S'}\Lambda}$ is a submersion) to a transverse multisection over $Z$ using Lemma \ref{lem extend section from sub}, then further extend it to any smooth multisection over the entire $\oPMb_{v}$.

\bibliographystyle{abbrv}
\bibliography{OpenBiblio}

\end{document}